\documentclass{article}
\usepackage[leqno]{amsmath}
\usepackage{amssymb}
\usepackage{amsfonts}
\usepackage{amsthm}
\usepackage{amscd}
\usepackage{latexsym}
\usepackage{stmaryrd}
\usepackage{mathrsfs}   	
\usepackage{amsbsy}		
\usepackage{graphicx}
\usepackage{scalefnt}
\usepackage{flafter}
\usepackage{bm}
\usepackage[all]{xy}
\usepackage[left,modulo]{lineno}
\usepackage{multirow}
\usepackage{circuitikz}		
\usepackage{lscape}
\usepackage{tkz-base}       
\usepackage{tikz-cd} 		
\usepackage{pstricks-add}   
\usepackage{hyperref}

\theoremstyle{plain}
\newtheoremstyle{myexstyle}
{}{}{}{}{\bfseries}{.}{ }{}

\newtheorem{theorem}{Theorem}[section]
\newtheorem{corollary}[theorem]{Corollary}
\newtheorem{definition}[theorem]{Definition}
\newtheorem{lemma}[theorem]{Lemma}

\newtheorem{notations}[theorem]{Notations}

\newtheorem{proposition}[theorem]{Proposition}
\theoremstyle{myexstyle}
\newtheorem{remark}[theorem]{Remark}
\newtheorem{example}[theorem]{Example}
\newtheorem{examples}[theorem]{Examples}

\renewenvironment{proof}[1][Proof]{\noindent\textbf{#1.} }{\ \rule{0.5em}{0.5em}}
\numberwithin{equation}{section}
\newcommand{\N}{\mathbb{N}}
\newcommand{\Z}{\mathbb{Z}}

\newcommand{\R}{\mathbb{R}}
\newcommand{\C}{\mathbb{C}}

\newcommand{\ad}{\operatorname{ad}}

\newcommand{\ap}{\rightarrow}
\newcommand{\rel}{{\rightarrowtriangle}}
\newcommand{\tto}{\rightrightarrows}

\newcommand{\Ker}{\ker}
\newcommand{\Hc}{{\mathcal H}}
\def\dis{\displaystyle}

 \def\a{\alpha}

\def\g{\gamma}

\def\o{\omega}

\begin{document}

\title{Partial Poisson Lie groups and  groupoids.\\ Application to Von Neumann algebras}
\author{Fernand Pelletier \& Patrick Cabau}
\date{}

\maketitle

\begin{abstract}
The purpose of this paper is to propose a version of the notion of convenient Lie groupoid as a generalization of this concept in finite dimension. 
The authors point out which obstructions appear in the infinite dimensional context and how an adapted notion of "bi-algebroid " in finite dimension (cf. \cite{MaXu94}) can be, nevertheless,  recovered. 
This paper is self-contained and recalls some important properties of "partial Poisson manifolds"  (cf. \cite{CaPe23}, Chapter~7)  and "Banach Poisson Lie groups"  (cf. \cite{Tum20}) needed for their purpose. 
The paper also gives  an illustration of these concepts  from all the results of A.~Odzijewicz and his collaborators on Lie groupoids and Von Neumann algebras.
\end{abstract}


\medskip

\noindent\textbf{Classification MSC 2020.--}  
22A22, 
22E65, 
46T05, 
57R2, 
70G45  
.

\tableofcontents

\bigskip


\maketitle


\section{Introduction}

In finite dimension, the theory of Poisson groupoids was introduced by Weinstein in  \cite{Wei87} as a generalization of symplectic groupoids initially defined in \cite{CDW87} and in parallel in  \cite{Zak90}. Poisson Lie groups and symplectic groupoids arise naturally in the integration of arbitrary Poisson manifolds (cf. \cite{CDW87} and \cite{Kar87}). Weinstein  noted that the Lie algebroid dual $A^*(G)$ of a Poisson groupoid $G$ has  itself a Lie algebroid structure.\\
This infinitesimal structure was further developed   in \cite{MaXu94}. In this paper, the authors  have pointed out a parallel between  Lie theory  for Poisson Lie  groups and  Poisson groupoids.    There is a one to one correspondence between Poisson Lie groups and Lie bialgebras for connected and simply connected  Lie groups. They  introduced and studied a natural infinitesimal invariant for Poisson groupoids $G$ which is a structure of Lie bialgebroid  on its Lie   algebroid $A(G)$.\\

A generalization  of Poisson manifold in the Banach setting was firstly proposed in \cite{OdRa03} and developed by A. Odzijewicz and his collaborators. In parallel, the authors  of this work proposed a generalization of such  a structure to the convenient setting called  \textit{partial convenient Poisson manifold} in \cite{PeCa19}. On the other hand, based on the notion of \textit{generalized Banach Poisson manifold}, A. B. Tumpach defined in \cite{Tum20}, a notion of Banach Lie-Poisson Lie group.\\
A comparison of all these generalizations can be found in \cite{GRT23}. In this paper, the interested reader can find many references on the topic.\\

In \cite{OdSl16}, A. Odzijewicz and A. Sli\.{z}ewska associated a Banach Lie groupoid  to $W^{\prime}$-algebras. Such a notion of Banach Lie groupoid was developed in \cite{BGJP19} and produces a link with  earlier research on some infinite-dimensional manifolds associated with Banach algebras. Later, in \cite{OJS18},  A.  Odzijewicz, G. Jakimowicz and  A.  Sli\.zewska  defined Sub-Poisson Groupoid again in the context of $W^{\prime}$-algebras and gave many examples in this framework.\\
 
The first purpose of this paper is to propose a context in which one can define a notion of  Banach  (eventually convenient) partial Poisson and sub-Poisson groupoid as a a generalization of Weinstein's arguments in finite dimension, which recover the results in \cite{OJS18}. 
Secondly, we show how the results of \cite{MaXu94} can be adapted to the infinite dimension setting.\\
 
This paper is self-contained and is organized as follows:
\begin{enumerate}
\item[$\bullet$] 
Section 2 summarizes all the essential properties of a partial Poisson manifold which are well known  (cf. \cite{CaPe23}, Chap.~7)  and some complements which are needed in the other sections.
\item[$\bullet$] 
Section 3 can be considered as a survey of all principal properties of Banach Lie groups.
\item[$\bullet$]
Section 4 shows how Weinstein's arguments used in \cite{Wei87} can be adapted to the infinite dimensional setting and which results are still true.
\item[$\bullet$] 
Section 5  is devoted to the context in which the results of \cite{MaXu94} can be recovered.
 \item[$\bullet$] 
Section 6 contains essentially the construction of Sub-Poisson structures on the cotangent bundle of a Lie groupoid.
\item[$\bullet$] 
In section 7 the previous results are extended to an ascending sequence of  finite dimensional Poisson groupoids and is illustrated by some examples.
\item[$\bullet$] 
Finally, in an Appendix, we summarize all properties  needed to understand the  results on groupoids associated with $W^{\prime}$-algebras which are exposed in this paper.
\end{enumerate}

\section{Some Essential Properties of Partial Poisson Manifolds}

This section  summarizes all the properties of a partial Poisson manifold which will be used in the last  section. Some of them can be found in \cite{CaPe23}, or  are completed,  but the others are new. 

\subsection{Relations in Linear Poisson and Symplectic Convenient Spaces }
\label{_RelationsInLinearPoissonSymplecticconvenientSpaces}

\subsubsection{Partial Linear Poisson Convenient Spaces}\label{__PartialLinearPoissonConvenientSpaces}

Let $\mathbb{E}$ be a convenient space. Denote by $\mathbb{E}^\prime$ its kinematic dual  and by $<.\;,.>$ the duality bracket. \\
Let $\mathbb{E}^\flat$ be a vector subspace of $\mathbb{E}^\prime$ with its own convenient structure where the inclusion of $\mathbb{E}^\flat \to \mathbb{E}^\prime$ is bounded. The duality bracket $<.\;,.>$ between $\mathbb{E}^\prime$ and $\mathbb{E}$ induces a  bounded bilinear  linear form on $\mathbb{E}^\flat\times\mathbb{E}$  which is left non degenerated and which will be also denoted by  $<.\;,.>$ when no confusion is possible in the context.\\

A subspace $\mathbb{Y}$ of any convenient space $\mathbb{X}$ will be call \emph{closed}\index{closed} if it is $c^\infty$-closed according to the terminology used in  \cite{KrMi97}.\\

More generally, we systematically adopt notations and terminology used in \cite{KrMi97}.

\begin{definition}
\label{D_PoissonLinear}
A \emph{partial linear Poisson structure}\index{partial!linear Poisson structure} on $\mathbb{E}$ is the data of 
\begin{enumerate}
\item[$\bullet$] 
a vector subspace $\mathbb{E}^\flat$ of $\mathbb{E}^\prime$  with its own convenient structure and such that the  inclusion of $\mathbb{E}^\flat$ in $\mathbb{E}^\prime$ is bounded ;
\item[$\bullet$] 
a bounded linear map $P:\mathbb{E}^\flat\to \mathbb{E}$ such that
\[
<\omega_1,P\omega_2> = -<\omega_2,P\omega_1>
\]
for all $\o_1$ and $\omega_2$ in $\mathbb{E}^\flat$.\\
We then say that $ \left( \mathbb{E}^\flat,\mathbb{E},P \right) $ is a \emph{partial linear Poisson space}\index{partial!linear Poisson space}.
\end{enumerate}
\end{definition}

If $P$ is a  partial linear Poisson structure on $\mathbb{E}$, so is $-P$. When no confusion is possible, if $P$ is a fixed partial linear Poisson structure on $\mathbb{E}$, we simply denote by ${\mathbb{E}}^-$ the space $\mathbb{E}$ provided with  the partial linear Poisson structure  $-P$. \\

Moreover, if $P$ is a convenient  isomorphism, then the  $2$-form $\Omega _P \in \bigwedge^2 \mathbb{E}^\prime$ defined by
\[
\Omega_P(u,v)=<P^{-1}(u), v>
\]
is a weak symplectic form on $\mathbb{E}$ if it is closed.\\
When  $\mathbb{E}^\flat=\mathbb{E}^\prime$  then $\Omega _P$ is a linear strong symplectic form. 
In general, when   $\mathbb{E}^\flat\not=\mathbb{E}^\prime$, $\Omega_P$ is  a linear  symplectic form on $\mathbb{E}$. Conversely,  if $\Omega$ is a weak symplectic linear form on $\mathbb{E}$  and let $\mathbb{E}^\flat$ the range of associated the linear operator $\Omega^{\flat}: \mathbb{E} \to \mathbb{E}^\prime$. Then $P=(\Omega^\flat)^{-1}:\mathbb{E}^\flat \to \mathbb{E}$ is a partial Poisson structure which is a convenient isomorphism. This justifies the following terminology:   
\begin{definition}
\label{D_Symplectic}  
Let $P:\mathbb{E}^\flat\to \mathbb{E}$ be a partial Poisson structure.\\
If  $P$ is a convenient isomorphism,  we say that $P$ is a \emph{partial linear symplectic structure}\index{partial!linear symplectic structure} on $\mathbb{E}$. When $\mathbb{E}^\flat=\mathbb{E}^\prime$ we will simply say that $P$ is  \emph{linear symplectic}\index{linear!symplectic}.
\end{definition}
 
If $P$ is a partial linear symplectic structure, then $-P$ is also of the same type. Again, if no confusion is possible, we denote simply $\mathbb{E}^-$ this opposite structure. 

\begin{definition}
\label{D_AnnhilitorOfASubspace}
For any 
subspace $\mathbb{F}$ of a convenient space $\mathbb{E}$,  the \emph{annihilator}\index{annihilator} of $\mathbb{F}$ is the subspace  
\[
\mathbb{F}^a:=
\{\alpha \in \mathbb{E}^\prime: \; 
\forall u\in \mathbb{F}, <\alpha,u>=0 \}.
\]
\end{definition} 
 
\begin{definition}
\label{D_Coisotropic} 
Let  $(\mathbb{E}^\flat,\mathbb{E},P)$ be a partial linear Poisson  space.
\begin{enumerate}
\item[1.] 
Given a vector subspace $\mathbb{A} \subset \mathbb{E}^\flat$, the vector space   
\[
\mathbb{A}^{\perp_P}:=
\{\omega\in \mathbb{E}^\flat:\;
\forall \alpha \in \mathbb{A}\, \;
 <\omega, P\alpha>=0\}
\] 
is called the \emph{orthogonal of $\mathbb{A}$ relative to $P$}\index{AperpP@$\mathbb{A}^{\perp_P}$ (orthogonal of $\mathbb{A}$ relative to $P$)}.\\
The space  $\mathbb{A}$ is called \emph{isotropic}\index{isotropic} if $\mathbb{A} \subset \mathbb{A}^{\perp_P}$.
\item[2.]  
If $\mathbb{F}$ is a  convenient  subspace of $\mathbb{E}$ (that is $\mathbb{F}$ is  $c^\infty$ closed) we set  
\[
\mathbb{F}^0:=\{\alpha \in \mathbb{E}^\flat:\;
\forall\; u\in \mathbb{F}, \; <\alpha,u>=0  \}.
\]
The \emph{orthogonal}\index{Fperp@$\mathbb{F}^{\perp_P}$} $\mathbb{F}^{\perp_P}$ of a convenient subspace $\mathbb{F}$ of $\mathbb{E}$ is the vector  space
\[
\mathbb{F}^{\perp_P}=P(\mathbb{F}^0).
\]
The space  $\mathbb{F}$ is called  \emph{coisotropic}\index{coisotropic}  if $\mathbb{F}^0$ is isotropic in $\mathbb{E}^\flat$.\\
\end{enumerate}
\end{definition}
\begin{lemma}
\label{L_FirstProperties}${}$
\begin{enumerate}
\item[1.]  
The kernel  of $P$ is always contained in $A^{\perp_P}$ for any subspace $\mathbb{A}$  of $\mathbb{E}^\flat$.
\item[ 2.] 
If $\mathbb{A}$ is a vector subspace of $\mathbb{E}^{\flat}$, then $\mathbb{A}^{\perp_P}$  is always  $c^\infty$ closed in $\mathbb{E}^\flat$.
\item[3.] 
If $\overline{\mathbb{F}}$ is the $c^\infty$-closure  of  a subspace $\mathbb{F}$ of $\mathbb{E}$, we have $\mathbb{F}^0=(\overline{\mathbb{F}})^0$. 
\item[4.] 
If  $\mathbb{F}$ is a  subspace of $\mathbb{E}$, then both $(\mathbb{F}^0)^{\perp_P}$ and $\mathbb{F}^0$ contain $\ker P$.
\item[5.]  
$\mathbb{F}$ is coisotropic  if  and only if 
\[
\forall \left( \omega_1,\omega_2  \right) \in \left( \mathbb{F}^0 \right) ^2, \; <\omega_1, P\omega_2>=0.
\]
\item[6.] 
If $ \mathbb{F}_1\subset \mathbb{F}_2$ and $\mathbb{F}_1$ is coisotropic, so is $\mathbb{F}_2$.
\end{enumerate}
\end{lemma}

\begin{notations}
\label{N_OmegaP} 
From now on, in this section, for any  $ \left( \mathbb{E}^\flat,\mathbb{E},P \right) $ be a partial linear Poisson space, we set
\[
\Omega_P(\alpha, \beta):=<\alpha, P\beta>.
\]
 Note that since  since the inclusion of $\mathbb{E}^\flat$ in $\mathbb{E}^\prime$  is bounded, the dual  bracket $<.\;,.>$ in restriction to $\mathbb{E}^\flat,\mathbb{E}$ also bounded. Thus, since $P$ is bounded, it follows that $\Omega_P$ is a skew-symmetric bilinear form on $\mathbb{E}^\flat$.
\end{notations}

\begin{proposition}
\label{P_PartialSymplectic}
Let $(\mathbb{E}^\flat,\mathbb{E},P)$ be a partial linear Poisson space\index{partial!linear Poisson space} such that $P$ is an isomorphism. Consider  the symplectic form $\Omega_P$ on $\mathbb{E}$ defined by  
\[
\Omega_P(u,v)=<P^{-1}(u), v>.
\]
If $\mathbb{F}$ is any $c^\infty$-closed subspace of $\mathbb{E}$, we have 
\[
\mathbb{F}^{\perp_P}=
\{u\in \mathbb{E},:\;
\forall v \in \mathbb{F}, \; \Omega_P(u,v)=0\}.
\]
Moreover, we have 
$\left( \mathbb{F}^{\perp_P} \right) ^{\perp_P}=\mathbb{F}$ if $\mathbb{F}$ is closed.
\end{proposition}

\smallskip
Now, following the classical terminology (cf. \cite{Wei88}), we have:
\begin{definition}
\label{D_Lagrangian} 
Let $(\mathbb{E}^\flat,\mathbb{E},P)$ be a partial linear symplectic space. A convenient subspace $\mathbb{F}$ of  $\mathbb{E}$ is called \emph{Lagrangian}\index{Lagrangian subspace} if  $P(\mathbb{F}^0)=\mathbb{F}$.\\
We say that a Lagrangian space $\mathbb{F}$ is a \emph{split Lagrangian}\index{split Lagrangian subspace} if  $\mathbb{F}$ supplemented  in $\mathbb{E}$ by a Lagrangian space.
\end{definition}
Note that in a partial symplectic structure, any Lagrangian space is also coisotropic\index{coisotropic}. However, any Lagrangian space is not supplemented in general. Even more, the reader will find in \cite{KaSw98} an example of linear symplectic Banach space in which any Lagrangian space is not supplemented.  

\begin{proposition}
\label{P_PF0perp} 
Let $(\mathbb{E}^\flat,\mathbb{E},P)$ be a partial linear Poisson space\index{partial!linear Poisson space}.  If $\mathbb{F}$ a convenient  subspace of $\mathbb{E}$, we have the following properties:
\begin{enumerate}
\item[1.]   
$P \left( (\mathbb{F}^0)^{\perp_P} \right) 
=\mathbb{F}\cap P(\mathbb{E}^\flat)$ and so $\mathbb{F}$ is coisotropic if and only if $\mathbb{F}^{\perp_P}\subset \mathbb{F}$.
\item[2.] 
The following properties are equivalent:
\begin{enumerate}
\item[(i)] 
$\mathbb{F}$ is coisotropic ;
\item[(ii)] 
$\mathbb{F}\cap {P(\mathbb{E}^\flat)}$ is coisotropic in $\mathbb{E}$.\\
\end{enumerate}
\end{enumerate}
\end{proposition}


\subsubsection{Linear Poisson (symplectic) relations}
\label{___LinearPoissonRelations}

\begin{definition} 
\label{D_PoissonMorphism}
For $i \in \{1,2\}$, let $ \left( \mathbb{E}_i^{\flat}, \mathbb{E}_i, P_i \right)$ be a partial linear Poisson space\index{partial!linear Poisson space}.\\
A bounded linear map $\phi:\mathbb{E}_1\to \mathbb{E}_2$ is called a \emph{Poisson morphism}\index{morphism!Poisson}\index{Poisson!morphism}  (resp. an \emph{anti-Poisson morphism})\index{morphism!anti-Poisson} if the adjoint $\phi^{\ast}$ of $\phi$ satisfies
\[
\left\{
\begin{array}
[c]{rcll}
 \phi^{\ast}(\mathbb{E}_2^{\flat}) &\subset& \mathbb{E}_1^{\flat}&   \\
P_2 &=&	\phi\circ P_1\circ \phi^\ast  
& (\textrm{resp. } P_2=-\phi\circ P_1\circ \phi^\ast).
\end{array}
.\right.
\]
Moreover, if  each  $P_i$ is partial  symplectic and a Poisson morphism (resp. anti-Poisson morphism),  $\phi$ is then called a  \emph{partial symplectic morphism}\index{partial!symplectic morphism} (resp. partial anti-symplectic morphism\index{partial!anti-symplectic morphism}).
\end{definition}

The product $\mathbb{E}_1\times \mathbb{E}_2$ can be provided with two linear partial Poisson structures:
\begin{enumerate}
\item[1.]
$P_{\mathbb{E}_1\times\mathbb{E}_2^+}:\mathbb{E}_1^{\flat}\times \mathbb{E}_2^{\flat}\to \mathbb{E}_1\times \mathbb{E}_2$ given by $P_{\mathbb{E}_1\times\mathbb{E}_2}=(P_1, P_2)$
\item[2.]
$P_{\mathbb{E}_1\times{\mathbb{E}}_2^-}:\mathbb{E}_1^{\flat}\times \mathbb{E}_2^{\flat}\to \mathbb{E}_1\times \mathbb{E}_2$ given by $P_{\mathbb{E}_1\times\mathbb{E}_2}=(P_1, -P_2)$.
\end{enumerate}

Note that if each  $P_i$ is a partial (resp. strong) symplectic so are $P_{\mathbb{E}_1\times\mathbb{E}_2^+}$ and $P_{\mathbb{E}_1\times{\mathbb{E}}^-_2}$. For simplicity, we denote by $\mathbb{E}_1^+\times\mathbb{E}^+_2$ and $\mathbb{E}_1^+\times\mathbb{E}^-_2$ this first and second structure respectively.\\

Recalling that the graph
\footnote{This notation is classical. However,  usually, in finite dimension in the classical papers on symplectic groupoids or Poisson groupoids (cf. \cite{Gui75}, \cite{CDW87},  \cite{AlDa90}, \cite{Wei88} among many others) the graph of such a map is denoted $G_\phi=\{(\phi(u), u)\; u\in \mathbb{E}_1\}.$} 
of $\phi$ is 
\[
G_\phi=\{ \left(  u, \phi(u) \right) ,\; u\in \mathbb{E}_1\}.
\]
we then have the following characterization of a Poisson morphism:
\begin{proposition}
\label{P_CharacterizationPoissonMorphism}
A bounded linear map  $\phi:\mathbb{E}_1\to\mathbb{E}_2$  such that  $\phi^\ast \left( \mathbb{E}_2^{\flat} \right) \subset \mathbb{E}_1^{\flat}$ is a partial Poisson morphism (resp. anti-Poisson morphism) if and only if its graph is coisotropic\index{coisotropic} in $\mathbb{E}_1\times\mathbb{E}_1^-$ (resp. $\mathbb{E}_1^-\times\mathbb{E}_2$).\\
Moreover, if $P_1$ and $P_2$ are  partial symplectic, $\phi$ is a partial symplectic  or partial anti-symplectic map, respectively,  if and only if its graph is Lagrangian in  $\mathbb{E}_1\times\mathbb{E}_1^-$ or in $\mathbb{E}^-\times\mathbb{E}_2$ respectively.
\end{proposition}

\begin{proof} 
Note that since $\phi$ is bounded $G_\phi$ is closed (cf. \cite{Bam16}, 1.1)
\begin{eqnarray}
 \label{eq_Gphi0}
 \{ \left( \phi^{\ast}(\alpha),-\alpha \right) ,\;\alpha\in \mathbb{E}_2^{\flat}\}.
\end{eqnarray}
is $c^\infty$ closed (cf. \cite{Bam16}, 1.1).\\ 
For the partial Poisson structure $P_{\mathbb{E}_1\times{\mathbb{E}}^-_2}$, we have
\[
P_{\mathbb{E}_1\times{\mathbb{E}}_2^-}\left((G_\phi)^0\right)=\{\left(P_1(\phi^{\ast}(\alpha), P_2(\alpha)\right), \; \alpha\in E_2^{\flat}\}.
\]
Now, if $\phi$ is a Poisson morphism, we have $P_2=\phi\circ P_1\circ \phi^{\ast}$ and so
$$P_{\mathbb{E}_1\times\overline{\mathbb{E}}_2}\left((G_\phi)^0\right)=\{\left(P_1\circ\phi^{\ast}(\alpha), \phi(P_1\circ\phi^{\ast}(\alpha))\right), \; \alpha\in E_2^{\flat}\}.$$
which is contained in $G_\phi$.\\
Conversely, if $P_{\mathbb{E}_1\times{\mathbb{E}}^-_2}\left((G_\phi)^0\right)$ is contained in $G_\phi$, this implies that $P_2=\phi\circ P_1\circ \phi^{\ast}$.\\
Assume that  $P_1$ and $P_2$ are  partial symplectic 
if $G_\phi$ is Lagrangian, then this implies that $\phi$ is a Poisson morphism.\\
Conversely, assume that $\phi$ is partial  symplectic, this means that $P_1$ and $P_2$ are isomorphisms, so $P_{\mathbb{E}_1\times{\mathbb{E}}^-_2}$ is an isomorphism from $\mathbb{E}^{\flat}_1\times \mathbb{E}^{\flat}_2$  to $\mathbb{E}_1\times \mathbb{E}_2$. Thus it remains to show that 
\[
P_{\mathbb{E}_1\times{\mathbb{E}}^-_2}^{-1}(G_\phi)\subset(G_\phi)^0.
\]
But we have
\[
\begin{array}{rcl}
P_{\mathbb{E}_1\times{\mathbb{E}}^-_2}^{-1}(G_\phi)
&=&\{(P_1^{-1}(u), -P_2^{-1}\circ \phi(u)),\; u\in \mathbb{E}_1\}  \\
 &=& \{(\phi^{\ast}\left(P_2^{-1}\circ\phi(u)\right), -P_2^{-1}\circ \phi(u)),\; u\in \mathbb{E}_1\}.
\end{array}
\]
which is included in $(G_\phi)^0$ from (\ref{eq_Gphi0}).
\end{proof}

More generally, following \cite{Gui75}, \cite{CDW87}, \cite{AlDa90} and \cite{Wei88}, we introduce the following notions: 

\begin{definition}
\label{D_LinearRelation}${}$
\begin{enumerate}
\item 
Any closed  subspace  $R \subset \mathbb{E}_1\times \mathbb{E}_2$ will be  called a \emph{linear  relation}\index{linear relation}.\\
Given a  linear relation $R \subset \mathbb{E}_1\times \mathbb{E}_2$, we will say that $R$ is a linear  relation from $\mathbb{E}_1$ to $\mathbb{E}_2$ and we will write  
 $R :\mathbb{E}_1\rightarrowtriangle \mathbb{E}_2$\index{RE1E2@$R :\mathbb{E}_1\rightarrowtriangle \mathbb{E}_2$ (linear relation from $\mathbb{E}_1$ to $\mathbb{E}_2$}.\\
If $\mathbb{E}_1=\mathbb{E}_2$, we simply say that $R$ is a \emph{relation} in $\mathbb{E}$.
\item 
Let $R:\mathbb{E}_1 \rel \mathbb{E}_2$ and $S:\mathbb{E}_2\rel\mathbb{E}_3$ be two relations. The composition $S\circ R$ is the relation $S\circ R:\mathbb{E}_1\rel \mathbb{E}_3$ given by
\[
S\circ R=\{(x,z)\in \mathbb{E}_1\times\mathbb{E}_2\;:\; (x,y)\in R \textrm{ and } (y,z)\in S\;\textrm{ for some } y\in \mathbb{E}_2 \}.
\]
\end{enumerate}
\end{definition}

In particular, the graph $G_\phi$ of a bounded linear map  $\phi :\mathbb{E}_1\to \mathbb{E}_2$ (which is always a convenient subspace of  $\mathbb{E}_1\times \mathbb{E}_2$) is a linear relation.\\
We will reserve the symbol $\to$ for a graph $G_\phi$ and the symbol $\rel$ will be used for a general relation. Of course, if $\phi:\mathbb{E}_1\to\mathbb{E}_2$ and $\psi :\mathbb{E}_2\to \mathbb{E}_3$ are two bounded linear map. The composition $G_\psi \circ G_\phi$ is $G_{\psi\circ \phi}$.\\

Now following Weinstein's terminology (\cite{Wei88}), according to Proposition \ref{P_CharacterizationPoissonMorphism}, we introduce:

\begin{definition}
\label{D_PoissonRelation} ${}$
\begin{enumerate}
\item 
A relation $R:\mathbb{E}_1\rel \mathbb{E}_2$ is called a \emph{Poisson relation}\index{Poisson!relation}, if $R$ is coisotropic according to the partial Poisson structure $P_{\mathbb{E}_1\times{\mathbb{E}}^-_2}$.
\item  
If for $i \in \{1,2\}$, $(\mathbb{E}^\flat_i, \mathbb{E}_i, P_i)$ is partial symplectic,  a relation $R:\mathbb{E}_1\rel \mathbb{E}_2$ is called a \emph{symplectic relation} if $R$ is lagrangian in $P_{\mathbb{E}_1\times{\mathbb{E}}^-_2}$.
\end{enumerate}
\end{definition}

As in \cite{Wei88}, we begin by the following result which is the version of Proposition~1.3.2 in our context and whose proof used the same type of arguments
 
\begin{theorem}
\label{T_R(C)} 
 For $i \in \{1,2\}$, let $(\mathbb{E}^\flat_i, \mathbb{E}_i,P_i)$ be a partial linear Poisson space and $R$ (resp. $C$) a coisotropic space of $\mathbb{E}_1\times\mathbb{E}_2$ (resp. $\mathbb{E}_2$). Then 
\[
R(C)=\{y\in \mathbb{E}_2\;:\; (x,y)\in R \textrm{ for some } x\in C\}.
\]
is a coisotropic space in $\mathbb{E}_2$.\\
Moreover, if each $P_1$ and $P_2$ are partial symplectic and $R$ and $C$ are Lagrangian, so is $R(C)$.
\end{theorem}

\begin{corollary}
\label{C_ImageOfCoisotropic}
For $i \in \{1,2,3\} $, let $P_i: \mathbb{E}^\flat_i\to \mathbb{E}_i$ be a partial linear Poisson structure on $\mathbb{E}_i$. Given  two Poisson relations $S:\mathbb{E}_1\rel\mathbb{E}_2$ and $T:\mathbb{E}_2\rel\mathbb{E}_3$,  
$T\circ S$ is a Poisson relation from $\mathbb{E}_1\rel \mathbb{E}_3$. Moreover,  if each $P_i$ is  partial symplectic and $S$ and $T$ are Lagrangian, so is $T\circ S$.
\end{corollary}
\begin{proof}
See  proof of Theorem~1.3.4 in \cite{Wei88}.
\end{proof}

\subsection{Basic Properties of Partial Poisson Manifolds} 
\label{__EssentialPropertiesPartialPoissonManifolds}
\emph{All theses properties can be also found in \cite{CaPe23}.}
\subsubsection{Partial Poisson Manifolds}
\label{___PartialPoissonManifolds}

Let $M$ be a convenient manifold modelled on the convenient space $\mathbb{M}$.
We denote by $p_{M}:TM\rightarrow M$ its kinematic  tangent bundle and by $p_{M}^{\prime}:T^{\prime}M\rightarrow M$ its
 kinematic cotangent bundle.\\
Consider a vector subbundle $p^\flat:T^{\flat}M\rightarrow M$ of $p_{M}^{\prime}:T^{\prime}M \to M$ such that $p^\flat:T^{\flat}M\rightarrow M$ is a convenient bundle and the canonical injection $\iota:T^{\flat}M\to T^{\prime}M$ is a convenient  bundle morphism. 
Such a bundle $p^{\flat}:T^{\flat}M\rightarrow M$ will be called a \emph{weak subbundle}\index{weak!subbundle} of $p_{M}^{\prime}:T^{\prime}M\rightarrow M$ and its typical fibre is denoted $\mathbb{M}^\flat$.\\
We denote by $<.\;,.>$ the canonical bilinear crossing between $T^{\prime}M$ and $TM$.  Since the inclusion $\iota$ of $T^\flat M$ in $T^{\prime}M$ is an injective bundle morphism, we have an induced bilinear pairing  $<.\;,.>^{\flat}$ on  $T^\flat M\times TM$  defined by 
\[
(\alpha, u)\mapsto<\alpha,u>^{\flat}=<\iota\circ\alpha,u>.
\] 
which is continuous and left non degenerate. In the following, since  no confusion is possible, this pairing will be also denoted $<.\;,.>$.\\

According to \cite {KrMi97}, Definition 48.5, we introduce, for any open set $U$ in $M$, the set $\mathfrak{A}(U)$\index{AmathcalU@$\mathfrak{A}(U)$}  of smooth functions $f\in C^\infty(U)$ such that, for any $k \in \mathbb{N}^\star$ and any $x$ in $U$, the derivative\\
$d^{k}f(x)\in L_{\operatorname{sym}}^{k}(T_{x}M,\mathbb{R})$ satisfies:
\begin{equation}
\label{eq_dkf0}
\forall (u_2,\dots,u_k) \in (T_xM)^{k-1},\;
d^{k}_xf(.,u_{2},\dots,u_{k}) \in T_{x}^{\flat}M.
\end{equation}

\begin{remark}
\label{R_LinearFunctionalOnTflatMBounded}
In particular, the previous condition implies  that, for any bounded linear functional $\lambda$ on $T_{x}^{\flat}M$, then $u\mapsto \lambda(d^{k}_xf(u,u_{2},\dots,u_{k}))$ is bounded on bounded sets of $T_xM$.
\end{remark}

The sets $\{\mathfrak{A}(U)\subset C^\infty(U),\; U\textrm{ open in } M \}$ define a sub-sheaf $\mathfrak{A}_M$ of the sheaf\\
$C^\infty_M=\{C^\infty(U),\; U\textrm{open set in }M\}$  (cf. \cite{CaPe23}). Of course, if $T^\flat M=T^\prime M$, then $\mathfrak{A}_M=C^\infty_M$.\\

We then have the following properties of $\mathfrak{A}(U))$ proved in \cite{CaPe23}:
\begin{proposition}
\label{P_AUalgebra}
Fix any open set  $U$ in $M$.
\begin{enumerate}
\item
The set $\mathfrak{A}(U)$\index{E_U@$\mathfrak{A}(U)$} is a subalgebra of $C^\infty(U)$. 
\item 
For each $k\in \mathbb{N}$, any local vector fields $X_1,\dots ,X_k$ on $U$ and any $f\in\mathfrak{A}(U)$,  the map\\
$x\mapsto  d^kf(X_1,\dots,X_k)(x)$ belongs to $\mathfrak{A}(U)$.
\end{enumerate}
\end{proposition}

Then $\{\mathfrak{A}(U)\subset C^\infty(U),\; U\textrm{ open set in } M \}$ defines a sub-sheaf $\mathfrak{A}_M$  of the sheaf\\
$C^\infty_M=\{C^\infty(U),\; U\textrm{ open in }M\}$  (cf. \cite{CaPe23}). Of course, if $T^\flat M=T^\prime M$, then $\mathfrak{A}_M=C^\infty_M$.\\


A morphism $P:T^{\flat}M\to TM$ is called \textit{skew-symmetric}\index{skew-symmetric!morphism} if it satisfies the relation
\[
<\xi,P(\eta)>=-<\eta,P(\xi)>
\]
for all sections $\xi$ and $\eta$ of $T^{\flat}M$. \\
Given such a morphism $P$,  for each open set $U$ in $M$, on $\mathfrak{A}(U)$ we define:
\begin{equation}
\label{eq_DefinitionLocalBracket}
\{f,g\}_{U}=-<df,P(dg)>=<dg,P(df)>.
\end{equation}

In these conditions, the relation (\ref{eq_DefinitionLocalBracket}) defines a skew-symmetric bilinear map
\[
 \{\;,\;\}_U:\mathfrak{A}(U)\times\mathfrak{A}(U)\to \mathfrak{A}(U)
 \]  
which satisfies the Leibniz identity\index{Leibniz identity}\index{identity!Leibniz}
\[
\{f,gh\}_{U}=g\{f,h\}_{P}+h\{f,g\}_{U}
\]
(cf. \cite{PeCa19}, 2.1).\\
In fact, the family of the bilinear maps $\{\;,\;\}_U$ defines  a  sheaf of brackets which will be denoted $\{\;,\;\}_P$ 
\begin{definition}
\label{D_PartialPoisson}
Let $p^\flat:T^{\flat}M\rightarrow M$ be a weak subbundle of
$p_{M}^{\prime}:T^{\prime}M\rightarrow M$.
\begin{enumerate}
\item 
Any skew-symmetric morphism  $P:T^{\flat}M\rightarrow TM$ a
is called an \emph{almost Poisson anchor}\index{almost!Poisson anchor} and $\{\;,\;\}_P$ an \emph{almost Poisson bracket}\index{almost!Poisson bracket}.
\item 
Given an almost Poisson anchor $P:T^\flat M\to TM$, we say that $P$ is a \emph{Poisson anchor} if the associated bracket $\{\;,\;\}_{P}$ satisfies the Jacobi identity\index{Jacobi identity}
\index{identity!Jacobi}
\[
\{f,\{g,h\}_{P}\}_{P}+\{g,\{h,f\}_{P}\}_{P}+\{h,\{f,g\}_{P}\}_{P}=0.\\
\]
In this case, $ \left( T^\flat M,TM, P,\{\;,\;\}_{P} \right) $ is called a \emph{partial Poisson structure}\index{partial!Poisson structure} on $M$.
\end{enumerate}
\end{definition}

As classically, given a partial Poisson structure $ \left( T^\flat M,TM,P,\{\;,\;\}_P \right) $, any function $f\in\mathfrak{A}(U)$ is called a \emph{Hamiltonian} and the associated vector field $X_{f}=P(df)$ is called a \emph{Hamiltonian vector field}.

We then have $\{f,g\}=X_{f}(g)$ and also $[X_{f},X_{g}]=X_{\{f,g\}}$ (see \cite{NST14}), which is equivalent to
\begin{equation}
\label{eq_Pdfdg}
P(d\{f,g\})=[P(df),P(dg)].\\
\end{equation}

If $ \left( T^\flat M,TM, P,\{\;,\;\}_{P} \right) $ is a partial Poisson structure on $M$, then, for any open set $U$, the space $\mathfrak{A}(U)$ provided with the Poisson bracket $\{\;,\;\}_P$ has a structure of Lie algebra, called classically  \emph{a Lie-Poisson  algebra}\index{Lie-Poisson algebra} and, according to (\ref{eq_Pdfdg}), the Poisson anchor $P$ gives rise to a Lie algebra morphism from $\mathfrak{A}(U)$ to the Lie algebra  $\mathfrak{X}(U)$ of smooth vector fields on $U$.

\subsubsection{Poisson bracket associated to a convenient Lie algebroid}${}$\\
\label{___PoissonBracketAndLieAlgebroid}
We recall some results of \cite{CaPe23}, Chap.~7 and,  as in $\S$~\ref{___PartialPoissonManifolds}, we consider the set $\mathfrak{A}(U)$.

Let $\pi:\mathcal{A}\to M$ be a convenient vector bundle and $\pi^{\prime}:\mathcal{A}^\prime\to M$ its associated dual convenient bundle.  We consider  closed  subbundle  $q^{\flat}:T^{\flat}\mathcal{A}^\prime\to \mathcal{A}^\prime$  of  $q_{\mathcal{A}^{\prime}}:T^\prime\mathcal{A}^\prime\to \mathcal{A}^\prime$.


\begin{definition}
\label{D_LinearFunctions}
Let $U$ be an open set in $M$ and denote by  $\mathcal{A}_U$ (resp. $\mathcal{A}^\prime_U$) the restriction of $\mathcal{A}$ (resp. $\mathcal{A}^\prime$ ) to $U$.
A function $\Phi\in C^\infty(\mathcal{A}^\prime_U)$ is called linear\index{linear!function}, if  the restriction of $\Phi $ to each fibre $\mathcal{A}^\prime_x$ is a linear map into $\mathbb{R}$
\footnote{Since $\Phi$ is smooth, its restriction to  $\mathcal{A}^\prime_x$  is linear and smooth, so  must be bounded (cf. \cite{KrMi97}).}.
\end{definition}

Let $\mathfrak{a}$ be a section of $\mathcal{A}_U$ and denote by $\Phi_\mathfrak{a}:\mathcal{A}^\prime\to \mathbb{R}$ the smooth map defined by
\begin{equation}
\label{eq_FrfPhia}
 \Phi_\mathfrak{a}(\sigma)=<\sigma,\mathfrak{a}\circ\pi_*(\sigma)>.
\end{equation}
It is clear that $\Phi_\mathfrak{a}$ is linear.\\

We denote by $\mathfrak{A}(\mathcal{A}^\prime_U)$ the set of smooth functions on  $\mathcal{A}^\prime_U $ which satisfy the properties of (\ref{eq_dkf0}). 

According to \cite{CaPe23}, 6.3.3, we have:

\begin{equation}
\label{eq_TflatA*}
T_\sigma^{\flat}\mathcal{A}^\prime
=
\{(\sigma,\eta, A)\in T_\sigma^\prime\mathcal{A}^\prime,\; 
\eta=d(f\circ\pi^\prime),\; 
A=\mathfrak{a}\circ\pi^\prime\; \mathfrak{a}\in \Gamma(\mathcal{A}_U),\;
f\in C^\infty(U)\}
\end{equation}
for any small enough open set $U$ around $x$.

Each function of type $\Phi_\mathfrak{a}$ and $f\circ \pi^\prime$ are such that  $d\Phi_\mathfrak{a}$ and $df\circ \pi^\prime$ belong to the set $\Gamma(T^\flat \mathcal{A}^\prime_U)$ of section of $(T^\flat \mathcal{A}^\prime_U)$,  for any section $\mathfrak{a}$ of $\Gamma(\mathcal{A}^\prime_U)$, and any function $f$ of $C^\infty(U)$.

Let $  C^\flat(\mathcal{A}^\prime_U)$ the set of smooth maps on  $\mathcal{A}^\prime_U$  whose restriction to each fibre $\mathcal{A}^\prime _x$ is constant. Such a function $\varphi$ naturally defines a function $\bar{\varphi}$ on $U$ such that $\varphi = \bar{\varphi}\circ \pi^\prime$.\\
Therefore, we can identify the subalgebra $\{f\circ \pi^\prime,\; f\in C^\infty(U)\}$ of $C^\flat(\mathcal{A}^\prime_U)$ with $C^\infty(U)$.

Now, it is clear that  the set $C^\flat(\mathcal{A}^\prime_U)$ of functions $F\in C^\infty(\mathcal{A}^\prime_U)$, whose differential belongs to $\Gamma(T^\flat \mathcal{A}^\prime_U)$, has a structure of algebra. Note that from (\ref{eq_dkf0}),   
if $F$ belongs to $\mathfrak{A}(\mathcal{A}^\prime_U)$, this clearly implies that $dF$ must be a
section of $T^\flat \mathcal{A}^\prime$. Since $dF$ is a smooth section of $T^\prime\mathcal{A}^\prime$ and that $T^\flat \mathcal{A}^\prime$ is closed convenient  subbundle of $T^\prime \mathcal{A}^\prime$, it follows that  $dF$ is a smooth section of $T^\flat\mathcal{A}^\prime_U$ and so $\mathfrak{A}(\mathcal{A}^\prime_U)$ is contained in
$C^\flat(\mathcal{A}^\prime_U)$. In fact, we have (cf. \cite{CaPe23}, Proposition~7.2):

\begin{proposition}
\label{P_A(Tflat A)}
Let $U$ be an open set which is a chart domain.
The algebra $\mathfrak{A}(\mathcal{A}^\prime_U)$  is equal to $C^\flat(\mathcal{A}^\prime_U)$. This  algebra is generated by functions of type $f\circ \pi_{*}$ and $\Phi_\mathfrak{a}$, for any $f\in C^\infty(U)$ and any $ \mathfrak{a}\in \Gamma(\mathcal{A}_U) $.
\end{proposition}

We introduce:
\begin{definition}
\label{D_LinearPoissonAnchor}
Consider a partial Poisson structure $ \left( T^\flat \mathcal{A}^\prime, \mathcal{A}^\prime,P,\{.,.\}_{P} \right)$  on $\mathcal{A}^\prime$. We say that $P$ is a linear Poisson anchor, or that the Poisson bracket $\{.,.\}_P$ is linear if, for any open set $U$ and any linear function $\Phi_1$ and $\Phi_2$ in $\mathfrak{A}(\mathcal{A}^\prime_U)$, the Poisson bracket $\{\Phi_1,\Phi_2\}_P$ is linear.
 \end{definition}

The following result is an adaptation of the classical result of equivalence between algebroid structures on a finite dimensional anchored bundle  and linear  Poisson structures on its dual (cf.  \cite{CaPe23}, Theorem~7.1).

\begin{theorem}
\label{T_LinkPoissonAlgebroid}
Let  $P: T^{\flat}\mathcal{A}^\prime\to T\mathcal{A}^\prime$ be a linear  Poisson anchor on $\mathcal{A}^\prime$. Then there exists a unique structure of convenient Lie algebroid  $ \left( \mathcal{A},\pi,M,\rho,[.,.]_P \right) $ characterized in the following way, for any open set $U$  in $M$:
\begin{eqnarray}
\label{eq_linAP1}
\forall \left( \mathfrak{a}_1,\mathfrak{a}_2 \right) \in \Gamma \left( \mathcal{A}_U \right) ^2,\; \Phi_{[\mathfrak{a}_1,\mathfrak{a}_2]_P}=\{\Phi_{\mathfrak{a}_1},\Phi_{\mathfrak{a}_2}\}_P \\
\label{eq_linAP2}
\forall f\in C^\infty (U), \forall \mathfrak{a}\in \Gamma \left( \mathcal{A}_U \right),\;
\{\Phi_\mathfrak{a},f\circ \pi_*\}_P=df(\rho(\mathfrak{a}))\circ \pi_*.
\end{eqnarray}
where $\rho(\mathfrak{a})= T\pi^\prime(P(\Phi_\mathfrak{a}))$.\\
Conversely,  a  structure of convenient Lie algebroid  $ \left( \mathcal{A},\pi,M,\rho,[.,.]_\rho \right) $ defines a unique linear Poisson anchor  $P:T^\prime\mathcal{A}^\prime\to T\mathcal{A}^\prime$ which gives rise to a  linear  Poisson bracket $\{.,.\}_\rho$ on $\mathfrak{A}(\mathcal{A}^\prime_U)$, for any open set $U$ in $M$, which is characterized by relations  (\ref{eq_linAP1}) and (\ref{eq_linAP2}). Moreover, the unique structure of convenient Lie algebroid  $ \left( \mathcal{A},\pi,M,\rho,[.,.]_P \right) $ associated to $P$  is the original one.
\end{theorem}

From Proposition~\ref{P_A(Tflat A)} and Theorem~\ref{T_LinkPoissonAlgebroid}, we obtain:

\begin{corollary}
\label{C_AlgebroidPoissonManifold}
Let $ \left( \mathcal{A},\pi,M,\rho,[.,.]_P \right) $  be a convenient Lie algebroid and $\{.,.\}_P$ the Poisson Lie bracket associated to this structure. Then $ \left( M,C^\flat \left( \mathcal{A}^\prime \right) , \{.,.\}_P \right) $ is a partial Poisson manifold.
\end{corollary}

\subsubsection{Characteristic distribution and foliation}
\label{__CharacteristicDistributionFoliation}

Given a partial Poisson structure $ \left( T^\flat M, TM, P \right) $, the family of vector spaces 
\[
D:=\{D_x=P(T_x^\flat M)\subset T_xM,\; x\in M\}
\]
is a distribution on $M$ called the \emph{characteristic distribution}\index{characteristic distribution}.\\
In finite dimension, when $T^{\flat} M=T^{\prime}M$, it is well known that the characteristic distribution gives rise to a Stefan-Sussman foliation\index{Stefan-Sussman foliation} and each \emph{leaf}\index{leaf} is a symplectic immersed manifold.\\ 
In the convenient  setting, it is not true in general  and even in the Banach setting (see Remark \ref {R_Leaf}).  However, we have the following sufficient conditions under which such an analogous result is true (cf. \cite{PeCa19}).
\begin{theorem}\label{T_FoliationPartialBanachPoissonManifold} 
Let $ \left( T^\flat M, TM, P\right) $ be a partial Banach Poisson manifold such that the kernel of $P$ is supplemented in each fibre $T^{\flat}_{x}M$ of $T^{\flat}M$ and $P(T^{\flat}M)$ is a closed
distribution. Then we have the following:
\begin{enumerate}
\item 
$D=P(T^{\flat}M)$ is integrable and the foliation defined by
$D$ is an almost symplectic foliation, i.e. on each leaf $N$ of this foliation, we have  a unique non degenerate smooth skew symmetric bilinear form $\omega_N$ on $TN$ such that
 for any $u$ and $v$ in $T_xN$, then $\omega_N(u, v)= < \alpha, P(\beta)>$ for any $\alpha, \beta\in T_x^\flat M$ with $P(\alpha)=u$ and $P(\beta)=v$.
\item 
On each maximal leaf $N$, let $\omega_N$ be  the natural weak symplectic form  on $N$. We denote by $\mathcal{E}_{\omega_N} (V)$ the bubset of $f\in C^\infty(V)$ such that $df$ belongs to the range $\omega_N^\flat(TV)$.  Then the restriction $f_N$ to $U\cap N$  of any  $f\in \mathcal{E}(U)$ belongs to $\mathcal{E}_{\omega_N}(U\cap N)$ and we have, for any $f$ and $g$ in $\mathcal{E}(U)$,
\[
{\{f_{|N},g_{|N}\}_{P}}_{|N}=\{f_{N},g_{N}\}_{\omega_{N}}.
\]
\end{enumerate}
\end{theorem}

\begin{remark}
\label{R_Leaf}${}$
\begin{enumerate}
\item[1.] 
In general, in the convenient setting, the characteristic distribution need not be integrable. Even in the Banach setting,  if the distribution $D$ is not closed,  we think that the $D$ is not integrable, but, unfortunately, we do not have any example. 
\item[2.]
The Theorem \ref{T_FoliationPartialBanachPoissonManifold} gives  sufficient conditions  for the integrability of the characteristic distribution and,  on each leaf
$N$ of  this foliation, the  $2$-form  $\omega_N$ is not  closed  \textit{a priori} (as asserted in \cite{PeCa19}). But, in fact, $\omega_N$ is  effectively closed. Indeed, according to Lemma~\ref{L_P(U)generates}, for  any $x\in M$, there exists an open set $U$ around $x$ such that,  any  $\sigma \in T_{x}^{\flat}M$ can be written $\sigma=d_{x}f$ for some local smooth function $f:U\to\mathbb{R}$. Thus, for $i \in \{1,2,3\}$, each $X_{i}\in T_{x}N$ can be written $X_{i}=P(d_{x}f_{i})$ for some $f_{i}$ locally defined around $x$. Each Hamiltonian field is tangent to the characteristic foliation and the tangent space of a leaf $N$ at $x$ is generated by the value at $x$ of the set of Hamiltonian field $\{X_f,\; f\in \mathcal{E}(U)\}$. According to the Cartan's formulae of the exterior derivative in the expression of $d\omega_N(X_{f_0}, X_{f_1},X_{f_2})$, the Jacobi identity appears twice for the triple $(f_0,f_1,f_2)$. Since the Poisson bracket $\{\;,\;\}_P$ satisfies the Jacobi identity,  it  follows that $\omega_N$ is closed.
\item[3.] 
When $(T^\flat M, TM, P)$ is a partial symplectic structure, even in the convenient setting, the characteristic distribution is $TM$ and then the characteristic foliation is reduced to one leaf that is $M$. So, the  argument  in Point 2. is then valid  and so $M$ is a weak symplectic manifold.
\end{enumerate}
\end{remark}

\subsubsection{Partial Lie Algebroids and Partial Poisson Manifolds}
\label{____PartialLieAlgebroidsAndPartialPoissonManifolds}

We refer to\cite{CaPe23}, 3.18.4  for the notion  of partial convenient Lie algebroid and we will use the same  notations. At first we have the fundamental result: (cf. \cite{CaPe23}, Lemma~7.3).

\begin{lemma}
\label{L_P(U)generates}
Let $(T^\flat M,M,P,\{.,.\}_P)$ be a partial Poisson structure.\\
For each $x\in M$, there exists an open neighbourhood $U$ of $x$ such that the vector space $\mathfrak{P}_x(U)=\{d_xf,\; f\in \mathfrak{A}(U)\}$ is equal to the fibre $T^\flat_xM$.
\end{lemma}

Then we have (cf. \cite{CaPe23}, Proposition~7.3 and Corollary~7.2).
\begin{proposition}
\label{P_PropertiesPartialPoissonManifold}
Let $(T^\flat M,P,\{.,.\}_{P})$ be a partial Poisson structure. We denote by $\mathfrak{P}_M$ the  sheaf of $\mathfrak{A}(U)$-modules generated by the set $\{df,f\in\mathfrak{A}(U)\}$.
 Then we have the following properties:
\begin{enumerate}
\item
We can define a sheaf of almost brackets\index{sheaf!of almost brackets} $[.,.]_{P}$ on the sheaf  $\mathfrak{P}_M$  by:
\begin{equation}
\label{eq_AlmostPoissonBracket}
[\alpha,\beta]_{P}=L_{P(\alpha)}\beta-L_{P(\beta)}\alpha-d<\alpha,P(\beta)>
\end{equation}
for any open set $U$ in $M$ and any sections $\alpha$ and $\beta$ in $\mathfrak{P}(U)$ where
$L_{X}$ is the Lie derivative. \\
Moreover,  $[.,.]_{P}$ satisfies:
\begin{equation}
\label{eq_DiffAlmostPoissonBracket}
\forall \left( f,g \right) \in \left( \mathfrak{A}(U) \right) ^2,\; [df,dg]_{P}=d\{f,g\}_{P}.
\end{equation}
\item $ \left( \mathfrak{P}_M,[.,.]_{P} \right) $ is a sheaf of Poisson-Lie algebras.\\
In particular,  $ \left( T^\flat M,p^\flat_M, M, P,\mathfrak{P}_M \right) $ is a strong  partial convenient  Lie algebroid and  $ \left( T^\flat M, M, P, [.,.]_P \right) $ is an pre-Lie algebroid that is $P$ satisfies $P([\alpha,\beta]_P]=[P(\alpha),P(\beta)]$ for all local sections  $\alpha, \beta$ de $T^\flat M$. Moreover, if $P$ is an injective morphism, then 
$ \left( T^\flat M,p^\flat_M, M, P,\mathfrak{P}_M \right) $ is a Lie algebroid.\\
\end{enumerate}
\end{proposition}

\begin{remark}
\label{R_About1}
In  fact, Assertion (1) of Proposition~\ref{P_PropertiesPartialPoissonManifold}, is also valid if we only assume that $P$ is an almost partial Poisson anchor. The proof of this assertion only uses the fact that $P$ is, in particular, an almost partial Poisson anchor (cf. Proof of Proposition~7.3 in \cite{CaPe23}, p.~317). Note that, in this case, we  only have a well defined almost Poisson bracket $\{f,g\}_{P}=<dg,Pdf>$ on each algebra $\mathfrak{A}(U)$ for any open set $U$ in $M$.
\end{remark}

\subsubsection{Jacobi Identity and Schouten Condition}
\label{__SchoutenCondition}

It is well known that, in finite dimension, if $P:T^*M\to TM$ is an almost Poisson anchor, then we have a skew symmetric contravariant  tensor  $\Lambda$ of type $(2,0)$ on $T^* M$ defined by
\begin{eqnarray}
\Lambda(\alpha,\beta)=<\beta,P\alpha>.\label{eq_Lambda}
\end{eqnarray}

Then $P$ defines a Poisson structure on $M$ if and only if the Schouten bracket $[\Lambda,\Lambda]$ vanishes identically on $M$ (cf. \cite{Lic77}). In fact, the property $[\Lambda,\Lambda]=0$ is equivalent to the Jacobi identity of the bracket $\{.,.\}_P$ naturally associated to $P$.\\

Now consider an almost Poisson anchor $P:T^\flat M\to TM$. \\
On the one hand,
as in finite dimension, one can associate to $P$ a skew symmetric contravariant  tensor  $\Lambda$ of type $(2,0)$ on $T^\flat M$ defined by
\begin{eqnarray}\label{eq_Lambda}
\Lambda(\alpha,\beta)=<\beta,P\alpha>.\label{eq_Lambda}
\end{eqnarray}
On the other hand, on any open set $U$ in $M$,
and for any $\sigma_1$ and $\sigma_2$ in $\mathfrak{P}(U)$, the map\\
$[P,P]:\mathfrak{P}(U)\times \mathfrak{P}(U)\to \mathfrak{X}(U)$
given  by
\begin{equation}
\label{eq_PP}
[P,P](\sigma_1,\sigma_2)=P([\sigma_1,\sigma_2]_P)-[P\sigma_1,P\sigma_2]
\end{equation}
Then $[P,P]$  is a tensor of type $(2,1)$  with takes values in $TM$. On the other hand,  as in finite dimension,  to $P$ is associated a skew symmetric   tensor $\Lambda$ on $T^\flat M$ of type $(2,0)$ also defined  by (\ref{eq_Lambda}). Thus we can consider the  tensor  $[\Lambda,\Lambda]$ of type $(3,0)$ on $T^\flat M$ defined by
\begin{eqnarray}
[\Lambda,\Lambda](\sigma_1,\sigma_2,\sigma_3)=<\sigma_3,[P,P](\sigma_2,\sigma_1),\label{eq_LambdLambda}
\end{eqnarray}
for any $\sigma_i\in \mathfrak{P}(U)$ for $i \in \{1,2,3\}$.

\begin{definition}
\label{D_SchoutenBracketLambda}
Let $P:T^\flat M\to TM$ be an almost Poisson anchor and $\Lambda$ the associated  contravariant tensor of type $(3,0)$. Then $2[\Lambda,\Lambda]$ is called the Schouten bracket of $\Lambda$.
\end{definition}
 Note that this Definition is justified in \cite{CaPe23}, 7.2.2. Now, as in finite dimension, we have (cf. \cite{CaPe23}, Theorem~7.2):
\begin{theorem}
\label{T_CharacterizationPoisson}
Let $P:T^\flat M\to TM$ be an almost Poisson anchor and $\Lambda$ the associated  contravariant tensor of type $(3,0)$.  Then $P$ defines a partial Poisson structure on $M$ if and only if $[\Lambda,\Lambda]=0$.
\end{theorem}

\subsection{Partial Poisson Bialgebroid and  its Associated Cohomology}

In this section, we  consider a partial Poisson manifold $ \left( M,\mathfrak{A}_M,\{.,.\}_{P} \right) $. We have seen that  $(T^\prime M,p^\flat_M, M, P,\mathfrak{P}_M)$ is a strong  partial Banach  Lie algebroid and  
$ \left( T^\flat,p^\flat_ M, M,P,[\cdot,\cdot]_P \right) $ is an almost Lie algebroid (cf. Proposition~\ref{P_PropertiesPartialPoissonManifold}, (2)).

\subsubsection{Partial Bialgebroid}
\label{____PartialBialgebroid}

According to \cite{CaPe23}, 3.18.6.4,
let $d_P$  be the associated  exterior differential associated to the strong partial Lie algebroid  $ \left( T^\prime M,p^\flat_M, M, P,\mathfrak{P}_M \right) $. Then,  if $X$ is a vector field defined on an open set $U$ in $M$,                                                                                                                                                                                                                                                                                                                                                                                                                                                                                                                                                                                                                                                            according to \cite{CaPe23}, Example~7.4, $X$ can be considered as an element of  $\bigwedge^1\Gamma^*(T^\flat M)$, so we have:
\begin{eqnarray}
\label{eq_dPX}
d_PX(\alpha, \beta)=L^P_\alpha<\beta,X>-L^P_\beta<\alpha, X>-<[\alpha,\beta]_P,X>.
\end{eqnarray}

Therefore, as for finite dimensional Poisson manifolds (\cite{Lic77}), we have (cf. \cite{CaPe23}, Proposition~7.8):
\begin{lemma}
\label{L_dPLambdaLambda}
Let $\Lambda $ be the Poisson tensor  associated to $P$. Then  with the previous notations we have:
\begin{enumerate}
\item
$d_PX=-L_X \Lambda$;
\item
$d_P[X,Y]=L_Xd_PY-L_Yd_PX$.
\end{enumerate}
\end{lemma}

\begin{remark}
\label{R_OtherCharacterizationdP}${}$
In finite dimension, recall that a Lie bi-algebroid is a pair of Lie algebroids $(\mathcal{A}, \pi,M,\rho,[.,.]_\mathcal{A})$ and $(\mathcal{A}^\prime, \pi^\prime,M,\rho^\prime,[.,.]_{\mathcal{A}^\prime})$ where $\pi^\prime:\mathcal{A}^\prime\to M$ is the dual bundle of $\pi:\mathcal{A}\to M$ with the following compatibility condition
\[
d_{\rho^\prime}[X,Y]_\mathcal{A}=L_Xd_{\rho^\prime}Y-L_Yd_{\rho^*}X,
\]
for all sections $X$ and $Y$ of $\mathcal{A}$  and $d_\rho$ (resp. $d_{\rho^\prime}$) is the exterior differential associated to the Lie algebroid $(\mathcal{A}, \pi,M,\rho,[.,.]_\mathcal{A})$ (resp. $(\mathcal{A}^\prime, \pi^\prime,M,\rho^\prime,[.,.]_{\mathcal{A}^\prime})$).
(cf. \cite{MaXu94} ). \\
In particular, if $M$ is provided with a Poisson structure defined by a bivector $\Lambda$ and a Poisson anchor $P:T^*M\to M$, then the pair  $(TM,p_{TM}, M,\operatorname{Id},[.,.])$ and $ (T^*M, p_{T^*M}, P,[.,.]_P)$ is a Lie bi-algebroid  
(cf. \cite{MaXu94}. According  to our notations, in this finite dimensional  context,  the compatibility  relation is exactly Proposition \ref{L_dPLambdaLambda} (2).\\
Unfortunately,  in our infinite dimensional context, even if $T^\flat M=T^\prime M$, the Poisson structure, in general, does not provide $T^\prime M$ with a Lie algebroid structure.
\end{remark}

Taking into account Proposition ~\ref{L_dPLambdaLambda}, we propose the following generalization of the notion of Lie bialgebroid in finite dimension:
\begin{definition}
Let $ \left( E,\pi, M,\rho,[.,.]_E \right) $ be a Banach Lie algebroid. Let $\pi^\prime:E^\prime\to M$ be the dual bundle of $E$. Assume that there exists a weak Banach subbundle $\pi^\flat: E^\flat\to M$ with the following properties:
\begin{itemize}
\item[$\bullet$] 
$E^\flat\subset E^\prime$, 
$\pi^\flat=\pi^\prime_{| E^\flat}$ and the inclusion $\iota: E^\flat\to E^\prime$ is a bundle morphism;
\item[$\bullet$] 
there exists a pre-Lie algebroid structure $(E^\flat, \pi^\flat, M,\rho^\flat,[.,;]_{E^\flat})$\footnote{cf. Proposition~\ref{P_PropertiesPartialPoissonManifold}.} on $E^\flat$ and also a strong Lie algebroid structure on $E^\flat$.
\end{itemize}
We denote by $d_\rho$ (resp. $d_{\rho^\flat}$)  the exterior differential on $E$ (resp. $E^\flat$) associated to these structures. We say that the pair $(E,E^\flat)$ is  a partial bialgebroid\index{partial!bialgebroid} if,  for any open set $U$ in $M$,   we have:
\[
d_{\rho^\flat}[X,Y]_E=L^{\rho^\flat}_X d_{\rho^\flat}Y- L^{\rho^\flat}_Y d_{\rho^\flat}X,
\]
for all sections $X$ and $Y$ of  $E_U$, considered as $1$-forms on $E^\flat$.
\end{definition}
\medskip
\begin{examples}
\label{exBialgebroid}${}$
\begin{enumerate}
\item[1.] 
Any finite dimensional  bialgebroid $ \left(  E,E^\prime \right) $ is a partial bialgebroid. 
\item[2.] From Proposition~\ref{L_dPLambdaLambda}, to any partial Poisson manifold 
$ \left( M,\mathfrak{A}_M,\{.,.\}_{P} \right) $ is associated a partial bialgebroid $ \left( TM, T^\flat M \right) $ (see also Theorem~\ref{T_CharacPartialBialgebra}).
\end{enumerate}
\end{examples}

\begin{proof}[Proof of Lemma \ref{L_dPLambdaLambda}] Fix some open set $U$ in $M$. All sections $\alpha$ and $ \beta$ of $T^\prime M$ and vector fields $X$ and $Y$ considered in this proof are assumed to be defined on $U$.\\

1. Using the expression of the Lie derivative of a $2$-form, we get:
\begin{equation}
\label{LXL}
L_X\Lambda(\alpha,\beta)=-<\beta,[P(\alpha), X]>+<\alpha,[P(\beta),X]>-X(<\beta,P(\alpha)>).
\end{equation}
On the other hand, using the exterior derivative of a $2$-form on a partial strong convenient Lie algebroid (cf. \cite{CaPe23}, 3.18.6.3), we have
\begin{equation}
\label{eq_dPX}
d_PX(\alpha,\beta)=P(\alpha)(<\beta,X>)-P(\beta)(<\alpha,X>)-<[\alpha,\beta]_P,X>.
\end{equation}
Using the value of $[\alpha,\beta]_P$, we get:
\[
d_PX(\alpha,\beta)=<\beta,[P(\alpha), X]>-<\alpha,[P(\beta),X]>+X(<\beta,P(\alpha)>)
\]
which ends the proof of (1).\\

2. From the value of a Lie derivative and (1), we have:
\[
\begin{array}
[c]{cl}
L_Xd_PY(\alpha,\beta)   & =X\{d_PY(\alpha,\beta)\}-d_PY(L_X\alpha,\beta)-d_PY(\alpha, L_X\beta)				\\
			             & =X\{Y\{<\beta,P(\alpha)>\}\}-X\{<L_Y\alpha,P(\beta)>\}+X\{<L_Y\beta,P(\alpha)>\} \\
                        & \; -Y\{<L_X\beta,P(\alpha)>\}+<L_XL_Y\alpha,P(\beta)>-<L_Y\beta,P(L_X\alpha)> \\
                        & \; - Y\{<\beta, P(L_X\alpha)>\}+<L_Y\alpha,P(L_X\beta)>-<L_XL_Y\beta, P(\alpha)>.
\end{array}
\]
By permutation of $X$ and $Y$, we obtain the corresponding expression of $L_Yd_PX$. By a direct calculation, we obtain easily that
\begin{equation}
\label{dPXY}
L_Xd_PY-L_Yd_PX=[X,Y]\{<\alpha,P(\beta)>-<L_{[X,Y]}\alpha, P\beta>+<L_{[X,Y]}\beta,P(\alpha),
\end{equation}
which is exactly $d_P[X,Y](\alpha, \beta)$.
\end{proof}

\subsubsection{Convenient Linear Partial Poisson  Spaces}
\label{___ConvenientPartialieLiePoissonSpaces} 
The concept of Banach-Lie Poisson space was introduced in  \cite{OdRa03} and generalized in \cite{Tum20} or in \cite{NST14}, Example~2.10.\\
In this section, we propose a generalization of this concept to the convenient setting. 
 
\begin{definition}\label{D_PartialLiePoissonConvenientSpace} 
Let $\mathfrak{g}$ be a convenient space.  Consider a convenient subspace $\mathfrak{g}^\flat$ of  $\mathfrak{g}^\prime$ so that the inclusion is bounded and   $P:T^\flat\mathfrak{g}:=\mathfrak{g}\times\mathfrak{g}^\flat\to T\mathfrak{g}=\mathfrak{g}\times\mathfrak{g}$ is an injective almost Poisson anchor.\\  
We say that $ \left( \mathfrak{g},\mathfrak{g}^\flat, P \right) $ is a  \emph{convenient partial Lie-Poisson space}\index{partial!Lie-Poisson space} if $\mathfrak{g}^\flat$ is provided with a Lie bracket $[.,.]^\flat$ such that, for all $x\in \mathfrak{g}$, 
 \begin{equation}
 \label{eq_Bracketgflat}
P([\alpha,\beta]^\flat)( x)=[P\alpha,P\beta](x)
\end{equation}
where $\alpha$ and $\beta$ are considered as functions in 
$\mathfrak{A}(\mathfrak{g})
:=\left\lbrace f\in C^\infty(\mathfrak{g}),: d_xf\in \{x\}\times \mathfrak{g}^\flat \right\rbrace$.
\end{definition} 
 
We have the following characterization:
\begin{proposition}
\label{P_CharacPartialBialgebra}
Consider a convenient subspace $\mathfrak{g}^\flat$ of  $\mathfrak{g}^\prime$ such that the inclusion is bounded and let $P:T^\flat\mathfrak{g}:=\mathfrak{g}\times\mathfrak{g}^\flat\to T\mathfrak{g}=\mathfrak{g}\times\mathfrak{g}$ be an almost Poisson anchor. Then  $ \left( \mathfrak{g},\mathfrak{g}^\flat, P \right) $ is a  partial Lie-Poisson space if and only if $P$ is a Lie Poisson anchor.\\
Moreover, in this case, 
$ \left( T\mathfrak{g}, T^\flat\mathfrak{g} \right) $ is a partial Lie bialgebroid.
\end{proposition} 

\begin{proof} 
Consider a convenient subspace $\mathfrak{g}^\flat$ of  $\mathfrak{g}^\prime$ such that the inclusion is bounded and let 
$P:T^\flat\mathfrak{g} \to T\mathfrak{g}$ 
be an almost Lie-Poisson anchor.  
We have an almost Lie bracket on the set $\mathfrak{A}(\mathfrak{g})
:=\{f\in C^\infty(\mathfrak{g}),\;: d_xf\in \{x\}\times \mathfrak{g}^\flat$ 
defined by $\{f,g\}_P=<dg, Pdf>$. Then if $\Lambda$ is  the skew symmetric contravariant  tensor   of 
type $(2,0)$ on $T^\flat\mathfrak{g}$ associated to $P$ as defined in (\ref{eq_Lambda}),
 the  bracket  $\{.,.\}_P$  satisfies the Jacobi identity if and only if the Schouten bracket vanishes: $[\Lambda,\Lambda]=0$ (cf. Theorem~\ref{T_CharacterizationPoisson}). But, according to (\ref{eq_PP}) and (\ref{eq_LambdLambda}), this condition is equivalent to 
\begin{equation}
\label{eq_PPequvLambadaLambda}
P(d\{f,g\})=[Pdf, Pdg]
\end{equation}
for all $f, g$ in $\mathfrak{A}(\mathfrak{g})$.  
This condition implies that,in particular, we have
\begin{equation}
\label{eq_PPequvLambdaLambdagprime}
P([\alpha,\beta] )=[P\alpha, P\beta]
\end{equation}
for all sections $\alpha, \beta$ of $T^\flat \mathfrak{g}$.
 Since $\mathfrak{g}^\flat$ can be identified with the set
\[
\left\lbrace f_\alpha\in \mathfrak{A}(\mathfrak{g}): \forall \alpha\in \mathfrak{g}^\flat ,\; \forall x\in \mathfrak{g},\;
f_\alpha(x)=<\alpha, x>
\right\rbrace
\]
then we have a Lie bracket on $\mathfrak{g}^\flat$ defined by
\begin{equation}
\label{eq_Bracketgflat}
[\alpha,\beta]^\flat=d\{f_\alpha,f_\beta\}.
\end{equation}
Thus we have a structure on Lie algebra on $\mathfrak{g}^\flat$  and the condition  (\ref{eq_Bracketgflat}) is satisfied.\\
 
Conversely,  assume that the assumptions of Definition~\ref{D_PartialLiePoissonConvenientSpace} are satisfied.  Since $P$ is injective, it follows that the almost Lie bracket  on $\mathfrak{g}^\flat$  defined by (\ref{eq_PPequvLambdaLambdagprime}) coincides with the given Lie bracket on $\mathfrak{g}^\flat$. Thus the relation (\ref{eq_PPequvLambadaLambda}) is satisfied and so the Schouten bracket of $P$ is zero, which ends the proof of the first part.\\
 
 Now under the assumption that   $ \left( \mathfrak{g},\mathfrak{g}^\flat, P \right) $ is a  partial Lie-Poisson space the last part is a direct consequence of  Proposition \ref{L_dPLambdaLambda}.
\end{proof}

\begin{example}
\label{ex-Anatol} 
Assume that $\mathfrak{g}^\flat=\mathfrak{g}^\prime$. According to \cite{OdRa03}, a Lie algebra $\mathfrak{g}$  is  a Lie  Poisson algebra if the set $C^\infty(\mathfrak{g})$ is provided with a Poisson Lie algebra $\{.,.\}$ such that considering its dual $\mathfrak{g}^\prime$ as a subset of  $C^\infty((\mathfrak{g})$, then $\{.,.\}$ induces a Lie bracket on $\mathfrak{g}^\prime$.  In fact, we have a natural  bounded map $P:\mathfrak{g}\times\mathfrak{g}^\prime \to \mathfrak{g} \times \mathfrak{g}$ such that $\{f,g\}=<dg,P(df)>$ which is given by $P\alpha(x)=-\operatorname{ad}^*_\alpha x$ where, as classically,  $\operatorname{ad}^*_\alpha x=-\alpha\circ \operatorname{ad}_x$ (cf. next Example) which satisfies the Jacobi identity (cf.  \cite{OdRa03}).
\end{example}

\begin{example}
\label{R_ComparisonBarabra}
Assume that  $ \left( \mathfrak{g},[.,.] \right) $ is a convenient Lie algebra. Then, to each $x\in \mathfrak{g}$, we have a bounded  linear map $\operatorname{ad}_x:\mathfrak{g}\to \mathfrak{g}$ given by $\operatorname{ad}_x(y)=[x,y]$. This gives rise to a bounded \footnote{Note that this action is continuous for the associated  bornological topology of $\mathbb{E}$.}
  coadjoint  action of $\mathfrak{g}$ onto $\mathfrak{g}\prime$ in an evident way. We also have a bounded  coadjoint  action 
\[
\operatorname{ad}^*:
\mathfrak{g}\times \mathfrak{g}^\prime
\to \mathfrak{g}^\prime
\]
given by 
$\operatorname{ad}^*_x\alpha=-\alpha\circ \operatorname{ad}_x$ for all $x\in \mathfrak{g}$ and $\alpha\in \mathfrak{g}^\prime$.\\
Consider a convenient subspace $\mathfrak{g}^\flat$ of  $\mathfrak{g}^\prime$ whose inclusion is bounded.  If $\mathfrak{g}^{\flat\prime}$ denote the dual of $\mathfrak{g}^\flat$, we have the inclusions:  
$ \mathfrak{g}\subset \mathfrak{g}^{\prime\prime}\subset \mathfrak{g}^{\flat\prime}.$  Assume that the restriction of the previous   coadjoint action 
 to $\mathfrak{g}\times \mathfrak{g}^\flat$  takes values in $\mathfrak{g}^\flat$ and is bounded\footnote{This condition is automatically satisfied if $\mathfrak{g}^\flat$ is closed in $\mathfrak{g}^\prime$.}. When there is no ambiguity, this action will still be denoted $\operatorname{ad}^*$.\\
When $\mathfrak{g}^\flat$ is a separating Lie algebra  for $\mathfrak{g}$, for any $x\in \mathfrak{g}$, then    the map $P_x(\alpha(x))=-\operatorname{ad}
 _{\alpha(x) }(x)$ is a well defined vector field on $\mathfrak{g}$  which satisfies the assumption of Definition~\ref{D_PartialLiePoissonConvenientSpace}. Then,  in the Banach setting, from \cite{Tum20}, Theorem~3.14, $(\mathfrak{g},\mathfrak{g}^\flat, P)$ is a partial Lie-Poisson space in the sense of  Definition~\ref{D_PartialLiePoissonConvenientSpace}. Thus this definition can be considered as an 
adaptation  of Definition~3.12 of \cite{Tum20}.
\end{example}


In  \cite{Tum20}, the reader will find other more particular and very interesting examples in the Banach setting.

\subsection{Poisson morphisms and Poisson maps}
\label{__PoissonMorphismsPoissonMaps}

\subsubsection{Poisson (resp.  symplectic) relations}

\label{___PoissonSymplecticRelations}

By analogy with the linear setting, in the context of convenient manifolds, a \emph{relation}\index{relation} $R:M_1\rel M_2$ is a closed weak submanifold  of $M_1\times M_2$. Any map $\phi:M_1\times M_2$ will be also considered as the relation given by its graph 
 $G_\phi=\{ \left( x,\phi(x) \right) , x\in M_1\}$.  If $S:M_2\rel M_3$ is also a relation, then the relation $S\circ R:M_1\rel M_3$ is again defined by
\begin{eqnarray}
\label{eq_CompositionRelations}
{}\;\;\;\;\;\;\;\;\;\;\;\;S\circ R=\{(x,z)\in M_1\times M_3,\;:\;   (x,y)\in R,\; (y,z)\in S,\; \textrm{with } y\in M_2\}.
\end{eqnarray}
The inverse of $R$ is the relation $R^{-1}:M_2\rel M_1$ given by
 $$R^{-1}=\{(y,x)\in M_2\times M_1 \textrm{ such that } (x,y)\in R\}.$$
 
As in \cite{Wei88}, we introduce:
\begin{definition}
\label{D_PoissonRelation} 
For $i \in \{1,2\}$, let $ \left( T^{\flat}M_i,TM_i,P_i,\{\;,\;\}_{P_i} \right) $  be a Poisson structure (resp.  symplectic structure). A relation $R:M_1\rel M_2$ is a Poisson relation, if $R$ is a closed weak submanifold of $M_1\times M_2$ which is coisotropic  for the partial  Poisson structure (resp. Lagrangian for the partial symplectic structure) $M_1\times M_2^-$. 
\end{definition}

Unfortunately, the composition of two Poisson relations is not in general a Poisson relation and even not a closed  submanifold without some adequate  assumptions.  In fact, if  $R:M_1\rel M_2$  and $S:M_2\rel M_3$ are Poisson relations, then the product $R\times S$ is a closed  submanifold of $M_1\times M_2\times M_2\times M_3$ and if $p_{1,3}$ is the projection of $M_1\times M_2\times M_2\times M_3$ onto $M_1\times M_3$ then according to (\ref{eq_CompositionRelations}), we have:
\[
S\circ R=p_{1,3}(R\times S\cap (M_1\times \Delta_2\times M_3))
\]
where $\Delta_2$ is the diagonal of $M_2\times M_2$.\\

\medskip
This is why is introduced the following conditions under which $S\circ R$ will be a weak closed submanifold in $M_1\times M_3$ (cf. \cite{Wei88}, Definition~1.3.7):

\begin{definition}
\label{D_AssumptionClean} ${}$
\begin{enumerate}
\item
A pair $ \left( N_1,N_2 \right) $ of  closed submanifolds of a convenient  manifold $M$ is called a \emph{clean pair}\index{clean pair} if $N_1\cap N_2$ is a  closed submanifold of $M$ such that $T(N_1\cap N_2)=TN_1\cap TN_2$\footnote{This notion was firstly introduced by R. Bott in his book \cite{Bot94}.}.
\item
Let $R:M_1\rel M_2$  and $S:M_2\rel M_3$ be  relations. 
 The pair of relations  $(R,S)$ is called a clean pair if:
\begin{enumerate}
\item[(i)] 
If $\Delta_2$ denotes the diagonal in $M_2\times M_2$, the intersection of the  closed  submanifolds $R\times S$ and $ D=M_1\times \Delta_2\times M_3$ is a clean pair in $M_1\times M_2\times M_2\times M_3$; 
\item[(ii)] 
$S\circ R$ is a closed submanifold of $M_1\times M_3$ and the  restriction to $(R\times S)\cap D$ of the projection of $p_{1,3}:M_1\times M_2\times M_2\times M_3\to M_1\times M_3$ is a submersion on $S\circ R$.
\end{enumerate}
\end{enumerate}
\end{definition}

\begin{remark} 
In \cite{Wei88}, A. Weinstein introduced the notion of clean pair and very clean pair for the finite dimensional context. Definition~\ref{D_AssumptionClean} corresponds to Weinstein's notion of very clean pair. But since we will only need  the assumptions of this Definition we simply use the terminology "clean pair" instead of "very clean pair".
\end{remark}
Note that from \cite{CaPe23}, Remark~1.26, if a pair $(R,S)$ of relations is such that  $R\times S$ and $D$ are transverse, then the condition (2)(i) of Definition~\ref{D_AssumptionClean} is satisfied,  but, in general, the condition (2) (ii) is not satisfied. The following Lemma gives examples of clean pair of relations:

\begin{lemma}
\label{L_Sudmertion}   
Consider a smooth  $\phi:M_1\to M_2$. If $p_i: M_2\times M_3\to M_i$ for $i \in \{2,3\}$ is the canonical projection,  let  $R:M_2\rel M_3$ be a relation  such that the  restriction of $p_3$ to $R$  is a diffeomorphism onto $M_3$ and $\phi_1(M_1)\subset p_2(R)$. If we identify $\phi$ with the relation from $M_1$ to $M_2$ given by its graph, then $(\phi ,R)$ is a clean pair. In particular, if $\psi$ is a submersion from $M_3$ to $M_2$ then $(\psi^{-1},\phi)$ is a clean pair.
\end{lemma}

\begin{proof} 
We consider the relation product
\[
\phi\times R 
=\left\lbrace
(x, \phi(x), y_2, y_3),\; x\in M_1, (y_2,y_3)\in R
\right\rbrace .
\]
From the assumptions, $\phi\times S$ is a closed submanifold of $M_1\times M_2\times M_2\times M_3$ modelled on $\mathbb{M}_1\times \mathbb{M}_3$. 
If  $\Delta_2$ is  the diagonal of $M_2\times M_2$ and $D= M_1\times \Delta_2\times M_3$, then we have 
\[
\begin{matrix}
  (\phi \times R)\cap D&=\{(x,\phi(x), \phi(x), y), \textrm{with } (\phi(x),y )\in R\}\hfill{}\\
  					&=\{(x,\phi(x),\phi(x), y),\; (x, y)\in M_1\times M_3\}\hfill{}.\\
\end{matrix}
\]
From the assumptions, it follows that   ($\phi \times R)\cap D$ is a closed submanifold in $M_1\times M_2\times M_2\times M_3$ modelled on $\mathbb{M}_1\times\mathbb{M}_3$.

Now, for $m=(x,\phi(x),\phi(x), y)\in (\phi \times R)\cap D$, we have 
\begin{equation}
\label{TphiRcapD}
\begin{matrix}
&T_m((\phi \times R)\cap D)\hfill{}\\
&{}\;\;\;\;\;\;\;\;\;\;\;\;=\{(u, v, -v, w)\in T_xM_1\times T_{\phi(x)}M_2\times T_{\phi(x)}M_2\times T_yM_2,  \textrm{ with } v=T_x\phi(u)\}\hfill{}\\
&{}\;\;\;\;\;\;\;\;\;\;\;= T_m(\phi \times R)\cap T_m D.\hfill{}
\end{matrix}
\end{equation}
Thus condition  (2) (i) of Definition \ref{D_AssumptionClean} is satisfied.\\

On the other hand the relation $(R\circ \phi)$ is 
\[
\{(x  ,y)\in M_1\times M_3 \textrm{ with } (\phi(x),y)\in R\}
\]
which, from assumption, is equal to $M_1\times M_3$.\\ 
Moreover, for any $m=(x,\phi(x),\phi(x), y)\in (\phi \times R)\cap D$, according to (\ref{TphiRcapD}), we obtain:
\[
T_m p_{1,3}(T_m((\phi \times R)\cap D))=T_x M_1\times T_y M_3=T_{(x,y)}(R\circ \phi).
\]
Now, as $p_{1,3}^{-1}(x,y)=\{(x,\phi(x),\phi(x),y)\}$  which is reduced to a point in $(\phi \times R)\cap D$, this trivially implies  that the   condition  (2) (ii) of Definition~\ref{D_AssumptionClean} is satisfied. 
Finally, if $\psi:M_3\to M_2$ is a submersion,  $R=\psi^{-1}$ satisfies clearly the assumptions required for $R$ in Lemma~\ref{L_Sudmertion}. This implies the last Assertion.
\end{proof}

Now, for a clean pair of  Poisson relations, we have:

\begin{theorem} \label{T_CompositionCleanPoissonRelations} 
Let $R:M_1\rel M_2$  and $S:M_2\rel M_3$ be  Poisson  (resp. symplectic) relations  such that the pair $(R,S)$ is clean. Then $S\circ R:M_1\rel M_3$ is a Poisson (resp. symplectic) relation.
\end{theorem}

 
\begin{proof} 
We only consider the case of Poisson relation. The symplectic case uses point by point the same type of argument. \\ 
Since the pair  $(R,S)$ is clean, it follows that for  $(x,y)\in R$ and $(y,z)\in S$, the tangent space of $R\times S\cap D$ at $m=(x,y,y,z)$ is the intersection of 
$T_m(R\times S)$ and $T_m D$. From Definition~\ref{D_AssumptionClean}, assumption (ii),  $T_m p_{1,3}$ is a surjective map from $T_m(R\times S)\cap T_m D$ onto $T_{(x,z)} S\circ R$. But $T_{(x,z)} (S\circ R)$ is exactly the composition  $(T_{(x,y)}S)\circ (T_{(y,z)}R)$ of the linear Poisson relations $T_{(x,y)}R: T_x M_1\rel T_y M_2$ and $T_{(x,y)}S: 
T_y M_2\rel T_z M_3$.  Thus, from  Definition \ref{D_PoissonRelation} and  Corollary \ref{C_ImageOfCoisotropic}, it follows that $S\circ R$ is a Poisson relation.
\end{proof} 

Let $R:M_1\rel M_2$ be a Poisson (resp. symplectic) relation between  two partial Poisson (resp. partial symplectic) manifolds. Given any coisotropic (resp. Lagrangian) submanifold $C$  in $M_1$,  following \cite{Wei88} (2.1), if $\{*\}$ is the trivial Poisson manifold reduced to a point with the zero Poisson structure, then  $C$ can be considered as a Poisson (resp. symplectic) relation $C:\{*\}\rel M_1$; since $C$ is coisotropic (resp. Lagrangian), $C$ can be identified with the graph of the relation  $C:\{*\}\rel M_1$.  Thus we introduce:
 
\begin{definition}
\label{D_RCclean} 
Let $R:M_2\rel M_1$ be a relation and $C$ a closed submanifold of $M_2$. 
We will say that $(R,C)$ is a clean pair if the pair of relations $(R,C)$ is a  clean pair.
\end{definition}

  
{\bf From now on, if $(R,C)$ is a clean pair, we identify  $R(C)$ with $ R\circ C$}. In this way,   Theorem~\ref{T_CompositionCleanPoissonRelations} has the following Corollary:
 
\begin{corollary}\label{C_CompositionRelationPoissonCoisotropic} 
Let $R:M_1\rel M_2$ be a Poisson (resp. symplectic) relation between  two partial Poisson (resp. partial symplectic) manifolds and a coisotropic (resp. Lagrangian) weak closed submanifold $C$  in $M_1$. If $(R,C)$ is  clean, then $R(C)$ is a coisotropic (resp. Lagrangian)  submanifold of $M_2$.
 \end{corollary}

\subsubsection{Coisotropic Submanifolds }
\label{____CoisotropicSubmanifolds}

Fix a partial convenient Poisson manifold 
$ \left( T^\flat M,TM, P,\{\;,\;\}_{P} \right) $. In this way, for each $x\in M$, $P_x:T^\flat_xM\to T_xM$ is a Poisson anchor.

\begin{definition} ${}$
\begin{enumerate}
\item[1.] 
An ideal $\mathcal{I}$ in a Poisson Lie algebra $\mathcal{E}$ whose bracket is $\{\;,\;\}$ is called \emph{coisotropic}\index{coisotropic} if $\mathcal{I}$ is closed under the bracket $\{\;,\;\}$.
\item[2.] 
A  closed submanifold $N$ of $M$\footnote{$N$ is a \emph{closed submanifold} of $M$ if it is a convenient  manifold,  the inclusion $i$ of $N$ into $M$ is smooth and $T_x i$ is injective with closed range for any $x\in N$.} is called \emph{a coisotropic submanifold} of $M$, relative to $P$ if $T_xN\subset T_xM$ is a coisotropic subspace of $T_xM$ relatively to $P_x$.
\end{enumerate}
\end{definition}
We have the following characterization of a coisotropic weak closed submanifold.

\begin{proposition}
\label{P_CharacterizationCoisotropicSub} 
Let $N$ be a  closed convenient  submanifold of $M$. Then the following properties are equivalent:
\begin{enumerate}
\item[(a)] 
$N$ is coisotropic.
\item[(b)] 
For any open set $U$ in $M$, the set $\mathcal{I}_N(U)$ of smooth functions of $\mathcal{E}(U)$ which vanish on $U\cap N$ is coisotropic in $\mathcal{E}(U)$.
\item[(c)] 
For any open set $U$ and any  $h\in \mathcal{I}_N(U)$, the Hamiltonian field $X_h$ is tangent to $N$.
\end{enumerate}
\end{proposition}

\begin{proof} 
At first, it is clear that the set $\mathcal{I}_N(U)$ is an ideal in  $\mathcal{E}(U)$. For $f\in \mathcal{I}_N(U)$ $f_{| U\cap N}\equiv 0$   then for each $x\in U$, $d_xf$  belongs to $(T_xN)^0\in T_x'M$.  Now from Lemma  \ref{L_P(U)generates}  it follows that $(T_xN)^0$ is generated by all the forms $\{d_xf, f\in \mathcal{I}_N(U)\}$ for all $x\in U$. According to Lemma~\ref{L_FirstProperties}, 5, the definition of $\{\;,\;\}_p$ and (\ref{eq_DefinitionLocalBracket}), $N\cap U$ is coisotropic   if and only if $\{f,g\}_P=0$ on $N\cap U$, which gives the equivalence  between (a) and (b).\\
The equivalence between (a) and (c) is a consequence of Proposition \ref{P_PF0perp} applied at each point $x\in U$ according to the definition of $X_h$ and Lemma  \ref{L_P(U)generates}.
\end{proof}

\begin{definition}
\label{D_PoissonMorphism} 
Consider two partial Poisson structures $(T^\flat M_i,TM_i, P_i,\{\;,\;\}_{P_i})$  for $i \in \{1,2\}$ 
A smooth map $\phi:M\rightarrow M^{\prime}$  is called a \emph{Poisson morphism} (resp. \emph{anti-Poisson morphism}) if
\begin{enumerate}
\item
\hfil{}
$\forall x\in M_1,\; \left( T_{\phi(x)}\phi \right) ^{\ast} (T^\flat _{\phi(x)}M_1)\subset T^\flat_xM_1$\\ where  $\left( T_{\phi(x)}\phi \right) ^{\ast}$  is the adjoint of the tangent map $T_x\phi$;
\item   
$T_x\phi: T_xM_1\to T_{\phi(x)}M_2$ is a (linear) Poisson morphism (resp anti-Poisson morphism) from $ \left( T^\flat_xM_1, T_xM_1, (P_1)_x \right) $ to $\left( T'_{\phi(x)}M_2, T_{\phi(x)}M_2,(P_2)_{\phi(x)} \right) $.
\end{enumerate}   
If, moreover, each $P_i$ is partial symplectic, a Poisson morphism (resp. anti-Poisson morphism) is called a  \emph{symplectic map}\index{symplectic map} (resp. \emph{anti-symplectic map}\index{anti-symplectic map}).
\end{definition}

Note that  the tangent map $T_x\phi:T_xM_1\to T_{\phi(x)}M_2$ is a linear Poisson morphism from $(T^\flat_xM_1, T_xM_1, (P_1)_x)$ to $(T^\flat_{\phi(x)}M_2, T_{\phi(x)}M_2,(P_2)_{\phi(x)} )$  means precisely that, for any $x\in M_1$, we have: 
\begin{eqnarray}
\label{eq_PoissonMap}
(P_2)_{\phi(x)}=T_x\phi\circ (P_1)_x\circ (T_x\phi )^{\ast}.
\end{eqnarray}
\medskip
As  in the context of partial linear Poisson structure, for $i \in \{1,2\}$, let $P_i:T^\flat M_i\to TM_i$ be a Poisson anchor on a Banach manifold $M_i$. Then we can provide the product $M_1\times M_2$ with two Poisson anchors:
\begin{enumerate}
\item
$P_{M_1\times M_2}:T^\flat M_1\times T^\flat M_2\to TM_1\times TM_2$ given by $P_{M_1\times M_2}=(P_1, P_2)$; 
\item
$P_{M_1\times M_2^-}:T^{\flat}M_1\times T^{\flat}M_2\to TM_1\times TM_2$ given by $P_{M_1\times{M}_2^-}=(P_1, -P_2)$ which will be  simply  denoted $M_1\times M_2^-$.
\end{enumerate}

Moreover, if each  $(T^\flat M_i,TM_i,P_i,\{\;,\;\}_{P_i})$ is a partial Poisson manifold on $M_i$, we obtain in this way two partial Poisson structures on $M_1\times M_2$ associated to these previous Poisson anchors. Moreover, if each  $ \left( T^\flat M_i,TM_i,P_i,\{\;,\;\}_{P_i} \right) $ is partial symplectic, so are the associated partial Poisson structures on $M_1\times M_2$.


By application of Proposition~\ref{P_CharacterizationPoissonMorphism}, the graph of a smooth map $\phi:M_1\times M_2$ is a closed submanifold of $M_1\times M_2$  and  from the definition of a coisotropic submanifold, 
 we obtain:
\begin{proposition}
\label{P_CharacterizationPoissonMap}
Consider two partial Poisson structures 
$ \left( T^\flat M_i,TM_i, P_i,\{\;,\;\}_{P_i} \right) $  for $i \in \{1,2\}$ 
and a smooth map $\phi:M_1\rightarrow M_2$  such that $(T_{x} \phi)^{\ast}(T^{\flat}_{\phi(x)}M_2)\subset T_x^{\flat}M_1$  for all $x\in M_1$.\\
Then $\phi$ is a Poisson morphism  (resp. anti-Poisson morphism) if and only if the graph of $\phi$ is a coisotropic  submanifold of $M_1\times {M}_2^-$ (resp.  $M_1\times {M}_2$). Moreover, if each partial Poisson structure is partial symplectic or partial anti-symplectic, then such a map $\phi$ is partial symplectic  (resp. partial anti symplectic)  if and only if the graph of $\phi$ is Lagrangian in $M_1\times {M}_2^-$ (resp.  $M_1\times {M}_2$).
\end{proposition}

\subsubsection{Links between Poisson Morphisms and Poisson Maps}
\label{____LinksBetweenPoissonMorphismsAndPoissonMaps}

The notion of Poisson map is defined as follows:
\begin{definition}
\label{D_PoissonMap} 
For $i \in \{1,2\}$, let 
\[
\{(\mathcal{E}_i (U_i), \{\;,\;\}_{U_i}),\;  \textrm{ for any open subset } U_i  \textrm{ in } M_i\}
\]
be a  Lie-Poisson algebras sheaf on $M_i$.\\
A map $\phi: M_1\to M_2$ is a \emph{Poisson map}\index{Poisson map} if,  for any open set $U_i$ in $M_i$ such that $\phi(U_1)\subset U_2$, the induced map $\phi^{\ast}:\mathcal{C}^{\infty}(U_2)
\to \mathcal{C}^{\infty}(U_1) $,  defined by $\phi^{\ast}(f):=f\circ \phi$, is such that
\[
\left\{
\begin{array}{rcl}
\phi^{\ast} \left( \mathcal{E}(U_2) \right) &\subset& \mathcal{E}(U_1)\\
\forall (f,g) \in  \mathcal{E}_2(U_2)^2,\;\{\phi^{\ast}(f),\phi^{\ast}(g)\}_{P_1}
&=&
\phi^{\ast}(\{f,g\}_{P_2}) .
\end{array}
\right.
\]
\end{definition}

\begin{proposition}
\label{P_PropertiesPoissonMap}  
Consider two partial Poisson structures $ \left( T^\flat M_i,TM_i, P_i,\{\;,\;\}_{P_i} \right)$  for $i \in \{1,2\}$  and a convenient Poisson map $\phi$.\\
Then, for any $x\in M_1$, we have  $(T_{x}^\ast  \phi)(T^{\flat}_{\phi(x)}M_2)\subset T_x^{\flat}M_1$ and the restriction of $T_{x}^\ast \phi$ to $T^{\flat}_{\phi(x)}M_2$ is a bounded map from $T^{\flat}_{\phi(x)}M_2$  to $T_x^{\flat}M_1$ respectively endowed with their own topology of convenient spaces.
\end{proposition}

\begin{proof} 
Fix some $x\in M_1$ and consider on open set $U_2$ around $\phi(x)$ in $M_2$ such that $\{df_2, \; f_2\in \mathcal{E}_2\}$ generates $T_{\phi(x)}^\flat M_2$.  Let $U_1$ be an open set around $x$ in $M_1$ such that $\phi^{\ast}(\mathcal{E}(U_2)) \subset \mathcal{E}(U_1)$.   Thus from the definition and properties of $\mathcal{E}(U_1)$, this implies that, for any $f_2\in \mathcal{E}(U_2)$, the map $T_{\phi(x)}^*\phi(df_2)$ belongs to $T_{x}^\flat M_1$ provided with its own convenient structure. This implies  that, for any  bounded linear functional $\lambda$ on  $T_{x}^\flat M_1$, $\lambda( T_{\phi(x)}^*\phi(df_2))$ is bounded for any $f_2$ such that $d_{\phi(x)}f_2$ belongs to a bounded set of $T_{\phi(x)}^\flat M_2$. Since $T_{\phi(x)}^\flat M_2$ is generated by $\{d_{\phi(x)}f_2, \; f_2\in \mathcal{E}_2\}$, it follows that $T_{\phi(x)}^*\phi$ is a bounded linear map from $T^{\flat}_{\phi(x)}M_2$  to $T_x^{\flat}M_1$ endowed with their own topologies of convenient spaces.
\end{proof}

\begin{remark}
\label{R_DifferenceWithTumpach} 
Definition~\ref{D_PoissonMap} of a Poisson map is quite  different from the definition given in \cite{Tum20}.  From Proposition~\ref{P_PropertiesPoissonMap}, the condition in the first assumption is in  Tumpach's Definition  but the second one (that is (\ref {eq_PoissonMap})) is not imposed and cannot be satisfied. On the other hand, the definition of a Poisson morphism (cf. 
Definition~\ref{D_PoissonMorphism}) does not require that the restriction of $T_{x} ^\ast \phi$ to $T^{\flat}_{\phi(x)}M_2$ is a bounded map from $T^{\flat}_{\phi(x)}M_2$ to $T_x^{\flat}M_1$ endowed with their own topology of convenient spaces. Note that both definitions of Poisson morphism in \cite{CaPe23} and Definition~\ref{D_PoissonMorphism} are coherent.
\end{remark}

The following result  gives relations between convenient Poisson morphisms and convenient  Poisson maps (cf. Theorem 7.53  and Remark 7.54 in \cite{CaPe23}):
\begin{theorem}
\label{T_CharacterizationPoissonMorphism} 
Consider two partial Poisson structures $ \left( T^\flat M_i,TM_i, P_i,\{\;,\;\}_{P_i} \right) $,  for $i \in \{1,2\}$, such that  $T^\flat M_1$ is closed in $T^\prime M_1$ and let $\phi:M_1\to M_2$ be a smooth map.
\begin{enumerate}
\item 
If $\phi$ is a  convenient Poisson morphism, then $\phi$ is also a convenient Poisson map.
\item  Assume that the typical fibre $\mathbb{M}^\flat_2$ of $T^\flat M_2$  is dense in $\mathbb{M}^\prime$  (typical fiber of $T^\prime M_2$), for the weak$^*$ topology.\\  
 Then   if $\phi$ is a convenient Poisson map  then $\phi$ is  also a convenient Poisson morphism.\\
 In particular, if  $P_2$ is partial symplectic or if $T^\flat M_2=T^\prime M_2$ the previous converse is true.
\end{enumerate}
\end{theorem}

\begin{proof}[Sketch of the proof]${}$\\
\emph{Point (i)}
Suppose that property ({\bf CPS}) is satisfied.  Since the problem is local, and according to compatible charts as previously,  we may assume that, for $i \in \{1,2\}$,  $U_i$ is a $c^\infty$ open set in $\mathbb{M}_i$ such that $TM_i=U_i\times \mathbb{M}_i$,  $T^\flat M_i=U_i\times \mathbb{F}_i$ and $\phi$ is a smooth map from $U_1\to U_2$ where
\begin{description}
\item[$\bullet$]
$T_x\phi \left( \{x\}\times \mathbb{M}_1 \right) \subset \{\phi(x)\}\times \mathbb{M}_2$;
\item[$\bullet$]
$T^*_x\phi \left( \{\phi(x)\}\times\mathbb{F}_2 \right) \subset \{x\}\times\mathbb{F}_1$.
\end{description}
Moreover, by assumption, $\mathbb{F}_1$ is closed in $\mathbb{M}_1^\prime$.\\

Since $\phi ^*(f)=f\circ \phi$,  we have
\[
d_x(f\circ \phi)=d_{\phi(x)} f\circ T_x\phi=T^*_x\phi \left( d_{\phi(x)}f \right).
\]
Since $f$ belongs to $\mathfrak{A}(U_2)$, the map $ d_xf$ belongs to $T^\flat_{y} M_2$. Thus,  
the previous relation implies that $d_x (f\circ \phi)$ belongs to $T_x^\flat M_1$ and is bounded since  $\mathbb{F}_1$ is closed in $\mathbb{M}_1^\prime$.

As in the proof of Proposition \ref{P_AUalgebra}, 2., by induction, we can show that
\[
d^k_x(f\circ\phi) \left( u_1,\dots, u_k \right)
\]
is a summation of terms of type:
\begin{equation}
\label {eq_dfTerms}
d^l_{\phi(x)}f
\left( 
T_x\phi \left( u^{\sigma^1} \right) ,T^2_x\phi \left( u^{\sigma^2} \right) ,\dots,T^i_x\phi \left( u^{\sigma^i} \right) ,\dots, T^{h}_x\phi \left( u^{\sigma^{h}} \right),
\right)
\end{equation}
where
\begin{description}
\item[$\bullet$]
the set  $\{\sigma^1,\dots,\sigma^h\}$ is a partition of the set $\{1,\dots,k\}$;
\item[$\bullet$]
for $\sigma^{i}$, one has:
\begin{description}
\item[--]
each $\sigma^{i}$ is a disjoint union of $\nu_i$ strictly increasing sequence of length $i$, that is
 $\sigma^i_j:= \Big(  \left( s^i_j \right) _1,\dots, \left( s^i_j \right) _i \Big)$  if $\nu_i\geq 1$;
\item[--]
otherwise $\sigma^i$ is an empty set
\end{description}
\item[$\bullet$]
for all $1\leq l\leq k$ and $h+l=k$, we have
\[
\left\lbrace
\begin{array}{rcl}
\nu_1+2\nu_2+\cdots+i\nu_i+\cdots+h\nu_h    &=& k  \\
\nu_1+\cdots+\nu_i+\cdots+\nu_h             &=& l
\end{array}
\right.;
\]
\item[$\bullet$]
If $\nu_i\geq 1$ and  $\sigma^i_j
=
\Big( \left( s^i_j \right) _1,\dots, \left( s^i_j \right)_i \Big) $, for $j \in \{1,\dots,\nu_i\}$, we set
\[
\begin{array}{rcl}
u_{\sigma^i_j}  &=&     \left( u_{(s^i_j)_1},\dots,u_{(s^i_j)_i} \right)  \\
T^i_x\phi \left( u^{\sigma^i} \right)
                &=&    \Big(  T^i_x\phi \left( u_{\sigma^i_1} \right) \dots, T^i_x\phi \left( u_{\sigma^i_{\nu_i}} \right)  \Big);
\end{array}
\]
\end{description}
Note  that the term (\ref{eq_dfTerms}) is completely defined  by such a  partition  $\{\sigma^1,\dots,\sigma^h\}$ and the summation in the expression of $d^k_x(f\circ\phi) \left( u_1,\dots, u_k \right)$ is  for all such partitions.\\
For the end of the proof we need the following Lemma:
\begin{lemma}
\label{L_diepsilonast}
Let $\phi$ be a smooth map from a $c^\infty$ open set $U$ around $0$ in convenient space $E_1$ into a convenient space $E_2$.\\
Assume that there exists a convenient space $F_2$ contained in $E_2^\prime$ with   bounded inclusion  and a closed convenient subspace $F_1$  of  $E^{\prime}_1$.  If the restriction of   $(T_x\phi)^*$   to $F_2$ is a
 bounded linear map from $F_2$ to $F_1$ for any $x \in U$ any $u_2,\dots, u_n\in E_1$, then  $ \left((T_x^n\phi)(.,u_2,\dots, u_n)\right)^*$ is a bounded linear  map from
  $F_2$ to $F_1$, for  integer $n$ and $x\in U$.
\end{lemma}
According to this Lemma,  we fix such a partition $\{\sigma^1,\dots,\sigma^h\}$. Then  $1$ belongs to  one and only one $\sigma^{i_0}_j:=((s^{i_0}_j)_1,\dots,(s^{i_0}_j)_{i_0})$  and  since this sequence is strictly increasing, we  $(s^{i_0}_j)_1=1$. The corresponding   term is
\[
T_x^{i_0}\phi \left( u^{\sigma^{i_0}_j} \right)
=
T^l_x\phi \left( u_{(s^{i_0}_j)_1},\dots,u_{(s^{i_0}_j)_{i_0}} \right).
\]
After  deleting $u_{(s^l_j)_1}$ in this term   we obtain  the bounded linear map
\[
\widehat{\phi_1}:u_1\mapsto T^{i_0}_x\phi(u_1,u_{(s^{i_0}_j)_2},\dots,u_{(s^{i_0}_j)_{i_0}}).
\]
By application of Lemma~\ref{L_diepsilonast}, the restriction of $\left( \widehat{\phi_1}\right)^*$ to  $\mathbb{F}_2$ is a bounded linear map into $\mathbb{F}_1$.\\
Associated to this partition, we consider again the corresponding term (\ref{eq_dfTerms}):

\begin{equation}
\label{eq_termpertition}
d^l_{\phi(x)}f \Big( T_x\phi \left( u^{\sigma^1} \right) ,T^2_x\phi \left( u^{\sigma^2} \right) ,\dots,T^i_x\phi \left( u^{\sigma^i} \right) ,\dots, T^{h}_x\phi \left( u^{\sigma^{h}} \right) \Big).
 \end{equation}
In this expression, all vectors  $T_x\phi \left( u^{\sigma^1} \right) ,T^2_x\phi \left( u^{\sigma^2} \right) ,\dots,T^i_x\phi \left( u^{\sigma^i} \right) ,\dots, T^{h}_x\phi \left( u^{\sigma^{h}} \right)$ belong to\\
 $T_{\phi(x)}M_1\equiv \mathbb{F}_2$, which we denote by $v_1,\dots, v_h$.
Among these vectors, we have one and only one which is equal to $\widehat{\phi_1}(u_1)$ and we may assume that $v_1=\widehat{\phi_1}(u_1)$.\\
Finally, after having deleted $u_1$ in (\ref{eq_termpertition}), we obtain a term of type
\[
\widehat{\phi_1}^* \Big( d^l_{\phi(x)}f \left( .,{v_2},\dots,\dots v_h \right) \Big) .
\]
which belongs to $\mathbb{F}_1$ from the construction of $(\phi_1)^*$. But such a  property is true for  any term of the decomposition of $d_x^k(f\circ\phi)(., u_2,\dots, u_k)$ for any $k\in \mathbb{N}$ and any $u_2,\dots, u_k$ in $T_xM_1$.
This implies that  $d_x^k(f\circ\phi)(., u_2,\dots, u_k)$ belongs to $T_xM_1$ for any $k\in \mathbb{N}$ and any $u_2,\dots, u_k$ in $T_xM_1$ and so  the proof is completed, according to Proposition~\ref{P_AUalgebra}.

\emph{The converse under assumption of Point (ii)}.\\
Note that the assumption of Point (ii) is equivalent to the assumption  $(\mathbb{M}^\flat)^{\operatorname{ann}}\cap \mathbb{M}=\{0\}$ according to the Bipolar Theorem  (cf. \cite{Rud91},  Theorems~3.7 and 3.12) which is  equivalent to the property that   $\mathbb{M}^\flat$ separates points in $\mathbb{M}$.\\
 
Fix some point $x\in M_1$ and let $U$ be an open set around $\phi(x)$ in $M_2$ such that set $\{d_{\phi(x)}f, f\in \mathcal{E}(U)\}$ is equal to
$T_{\phi(x)}^{\flat}M_2$ (cf. Lemma~\ref{L_P(U)generates}).  Since  $d_x(\phi_* f)=(T_x^{\ast}\phi)(df)$, and  $(T_{x}^{\ast} \phi)(T^{\flat}_{\phi(x)}M_2)\subset T_x^{\flat}M_1$,  the set 
$\{d_x(f\circ\phi)),  f\in \mathcal{E}(U)\}$  
generates $T^{\flat}_xM_1$.  From the definition of the Poisson bracket from a Poisson anchor, we have 
\begin{equation}
\label{eq_Poissonmap}
<d_x(g\circ \phi), T_x\phi\circ P_1\circ(T_{\phi(x)}^{\ast}\phi)(df_x(f\circ \phi))>=<d_{\phi(x)}g, P_2(d_{\phi(x)}f>
\end{equation}
for all $f, g\in \mathcal{E}(U)$.\\ 
Then, for any fixed $f\in  \mathcal{E}(U)$ in relation (\ref{eq_Poissonmap}, we obtain:
$$<d_{\phi(x)}g,T_x\phi\circ P_1\circ(T_{\phi(x)}^{\ast}\phi)(d_{\phi(x)}f- P_2d_{\phi(x)}f>=0$$ 
for any $ g \in \mathcal{E}(U)$.  Since the set $\{d_{\phi(x)}g, g\in  \mathcal{E}(U)\}$ generates $T_{\phi(x)}x^\flat M_2$, it follows that $T_x\phi\circ P_1\circ(T_{\phi(x)}^{\ast}\phi)(df_x(f\circ \phi))- P_2d_{\phi(x)}f$ belongs to $(T_{\phi(x)}^\flat M_2)^a  \cap T_yM_2$
but by assumption, as we have seen previously,   we have $(T_y^\flat M_2)^a  \cap T_yM_2=\{0\}$ for all $y\in  \phi(M_1)$,
so the proof is completed.\\
Now, if $P_2$ $P_2$ is partial symplectic  then $P_2$ is an isomorphism $T^\flat M_2$ to $TM$  this necessary implies that $(T_y^\flat M_2)^a  \cap T_yM_2=\{0\}$ for all $y\in M_2$. The same is also clear if  $T^\flat M_2= T^\prime M_2$.
\end{proof}

\begin{proof}[proof of Lemma \ref{L_diepsilonast}]
Fix some $x\in U$. By assumption the result is true for $n=1$.\\
Assume that the result is true for all $1\leq i< n$. Then according to the definition of $d_x^n\phi$ (cf. \cite{KrMi97}, 5.11), we have
 \[
 T^n_x\phi(u,u_2,\dots,u_n)
 =
 \lim_{t\rightarrow 0}
 \left(
 \frac{1}{t} \left( T_{x+tu}^{n-1}\phi \left( u_2 ,\dots, u_n \right)
 - T_x^{n-1}\phi \left( u_2,\dots, u_n \right)
 \right)
 \right).
\]
On the other hand, for any $\alpha\in F_2$, we obtain
 \[
 (T^n_x\phi(.,u_2,\dots,u_n))^*(\alpha)=\alpha \circ (T^n_x\phi(.,u_2,\dots,u_n)).
 \]
Now the map
\begin{equation}
\label{eq_alphaTxn}
 t\mapsto  \alpha\circ T_{x+tu}^{n-1}\phi(.,u_2,\dots,u_n)
\end{equation}
is a smooth map from  $E^{\prime}_2$ to $E^{\prime}_1 $.\\
Thus we have
\[
\displaystyle\frac{d}{dt}_{| t=0}\left( \alpha\circ\left( T_{x+tu}^{n-1}\phi(.,u_2,\dots,u_n)\right) \right)
=\alpha \circ \left( T_x^n\phi(.,u,u_2\dots,u_n) \right).
\]
But the inclusion of $F_2 $ in $E_2$  is bounded, $F_1$ is closed in $E_1^\prime$ and the map (\ref{eq_alphaTxn}) takes values in $F_1 $ from the inductive assumption. This implies that $ \alpha \circ \left(T_x^n\phi(.,u,u_2\dots,u_n)\right)$ belongs  to $F_1$  for any $\alpha\in F_2$. But the map $\alpha\mapsto
\alpha \circ \left(T_x^n\phi(.,u,u_2\dots,u_n)\right)$ is a bounded map from $F_2$ to $E^\prime_1$ which takes values in $F_1$. Since $F_1$ is closed it follows that this map is a linear bounded map from $F_2$ to $F_1$.
\end{proof}

\begin{remark}
\label{R_TheConverse}${}$
\begin{enumerate}
\item[1.]
In finite dimension, we have equivalence between Poisson map and  Poisson morphism. However, in the context of Theorem~\ref{T_CharacterizationPoissonMorphism}, (2), in general, $\mathbb{M}_2^\flat \not=\mathbb{M}_2^\prime $  and not dense in  $\mathbb{M}^\prime _2$ for its weak$^*$ topology and so the assertions  (1) and (2) are not  equivalent. \\
For instance,  it  is always  the case if  $T^{\flat } M_1=T^{\prime} M_1$, $\mathbb{M}_2$  is reflexive,  $\mathbb{M}_2^\flat \not=\mathbb{M}_2^\prime $   and not dense in  $\mathbb{M}^\prime _2$ for its weak$^*$ topology, and  $\phi$ is a surjective submersion.
\item[2.] 
If  a smooth map $\phi:M_1\to M_2$ is a  convenient  Poisson morphism from $ \left( T^\flat M_1,TM_1, P_1,\{\;,\;\}_{P_1} \right) $  to $ \left( T^\flat M_2,TM_2, P_2,\{\;,\;\}_{P_2} \right) $, $\phi$ can be also a Poisson map even if $T^\flat M_1$ is not a closed subbundle of $T^\prime M_1$ (cf. Theorem~\ref{T_ExistencePartialPoissonStructure}).
\end{enumerate}
\end{remark}

\subsubsection{Existence of a Partial Poisson Structure on the Source of a Submersion onto a Partial Poisson Manifold}
\label{___PullBackBySubmersion}

Given a surjective submersion on a partial Poisson manifold, we look for sufficient conditions on the structure of the source for which this submersion is a Poisson map.
More precisely, we have:

\begin{theorem}
\label{T_ExistencePartialPoissonStructure} 
Consider  a  partial Poisson structure $ \left( T^{\flat}M_1,TM_1,P_1,\{\;,\;\}_{P_1} \right) $   and $\phi:M_1\to M_2$ a surjective submersion.  Then we have
\begin{enumerate}
\item[(i)] 
$M_2$ can be provided with a (unique)   Lie-Poisson algebras sheaf
\[
 \{(\mathfrak{A}(U_2),\{\;,\;\}_{U_2}),\;\; U_2\textrm{ open set in } M_2\} 
\] 
such that  $\phi$ is a  Poisson map 
if and only if the relation $\phi^{-1}\circ \phi$ is a Poisson relation from $M_1$ to $M_1$.  
 \item[(ii)]  
Assume that  $M_1$ is connected, and  assume that  $(\ker T\phi)^0=(\ker T\phi)^a\cap T^\flat M_1$ is a closed convenient subbundle of $T^\flat M_1$ ($(\ker T\phi)^a$ is the  the annihilator of $\ker T\phi$). Then there exists a unique  partial Poisson structure $(T^\flat M_2, TM_2,P_2, \{\;,\;\}_{P_2})$ such that  $\phi$ is  Poisson morphism if and only if the relation $\phi^{-1}\circ \phi$ is a Poisson relation from $M_1$ to $M_1$. In this case $\phi$ is also a Poisson map.
  \item[(iii)] 
If $(T^{\flat}M_1,TM_1,P_1,\{\;,\;\}_{P_1})$   is partial symplectic or if  we have  $T^\flat M_1= T^{\prime}M_1$, then there exists a unique  partial Poisson  structure $(T^{\flat}M_2, TM_2,P_2, \{\;,\;\}_{P_2})$ such that  $\phi$ is  Poisson morphism if and only if the relation $\phi^{-1}\circ \phi$ is a   Poisson  relation from $M_1$ to $M_1$ and so, in this case , $\phi$ is also a Poisson map.
\end{enumerate}
\end{theorem}

\begin{remark}
\label{R_Caseanti-Poisson}
Under the same assumptions of Theorem~\ref{T_ExistencePartialPoissonStructure}, we obtain the same conclusion in (i), (ii) and (iii) when $\phi$ is an anti-Poisson morphism if and only if the graph of $\phi^{-1}\circ (-\phi)$ is a Poisson relation.
\end{remark}


 
\begin{proof} ${}$\\
(i) Fix some open set  $U_2$ in $M_2$. If  $U_1=\phi^{-1}(U_2)$, we consider  the set
\[
\mathcal{E}(U_2)=\{ g\in C^\infty(M_2): \phi^*(g)\in \mathcal{E}(U_1)\}.
\]
Clearly $\mathcal{E}(U_2)$ has a structure of algebra.  At first, we show that $ \mathcal{E}(U_2)$ can be provided with a Lie-Poisson structure. \\
The Poisson bracket on $M_1\times M_1^-$  is given by 
\[
\{ h_1,h_2\}_{(P_1,P_1^-)}= - < d_1 h_1, P_1(d_1 h_2)>+<d_2 h_1, P_1(d_2 h_2)>
\]
where $ d_1 h$ and  $ d_2 h$ are  the partial differential of $h$ relative to the first factor  and the second factor respectively.\\
Now,  in $M_1\times M_1^-$,  the graph of the relation $\phi^{-1}\circ \phi: M_1\rel M_1$ is the set  
\[
G=\{(x,y)\in M_1\times M_1^-: \phi(x)=\phi(y)\}.
\]   
Assume that $G$ is a Poisson relation, which means that $G$ is coisotropic in $M_1\times M_1^-$.  
By Proposition \ref{P_CharacterizationCoisotropicSub}, the ideal $\mathcal{I}_g(U_1\times U_1)$ is stable under the Poisson bracket on $M_1\times M_1^-$ in restriction to $U_1\times U_1$. \\
For $i \in \{1,2\}$,  to the  function $f_i\in \mathcal{E} (U_2)$ we associate $g_i(x,y)=f_i(\phi(x))-f_i(\phi(y))$ on $U_1\times U_1^-$. Thus $g_i$ vanishes on  $(U_1\times U_1)\cap G$ and the bracket the $\{g_1,g_2\}_{(P_1,P_1^-)}$ also vanishes 
  on  $(U_1\times U_1)\cap G$ and so
\[
\{ g_1,g_2\}_{(P_1,P_1^-)}(x,y)=\{ f_1\circ \phi,f_2\circ \phi\}_{(P_1,P_1^-)}(x)-\{ f_1\circ \phi,f_2\circ \phi\}_{(P_1,P_1^-)}(y)=0
\]
for all $x,y$ such that $\phi(x)=\phi(y)$. Thus $ \{ f_1\circ \phi,f_2\circ \phi\}_{(P_1,P_1^-)}$ is constant on each fibre of $\phi$. It follows that 
 $ \{ f_1\circ \phi,f_2\circ \phi\}_{(P_1,P_1^-)}$ belongs to  $ \mathfrak{A}(U_2)$ and only depends only on $f_1$ and $f_2$. Since $\phi$ is a submersion, in this way, we define a Poisson bracket on  $ \mathfrak{A}(U_2)$. \\
For the converse  assume that $\phi$ is a partial Poisson map. We first show that the graph of $\phi$ is coisotropic. Note hat at first, since  $\phi$ is smooth, for any $x\in M$, the linear map $T_x\phi:T_xM_1\to T_xM_2$  must be bounded. Thus, if $G_\phi$ denotes the graph of $\phi$,  its tangent space $T_{(x,\phi(x)} G_\phi=\{(u, T_x\phi(u)), u\in T_x M_1\}$ is a closed subspace of $T_xM_1\times T_{\phi(x)}M_2$ and so  $G_\phi$ is a closed submanifold of $M_1\times M_2$. We must show that $T_{(x,\phi(x))} G_\phi$ is a coisotropic  subspace of $T_xM_1\times T_{\phi(x)}M_2$ provided with the linear partial Poisson structure $(T_x^\flat M\times T_{\phi(x)}^\flat M_2, T_xM_1\times T_{\phi(x)}M_2, (P_1)_x\times (-P_2)_{\phi(x)}).$\\
Indeed, we have an open set $U$ around $x\in M_1$ such that $T_x^\flat M_1$ (resp. $T_{\phi(x)}^\flat  M_2$)  is generated by $\{d_xf_1,\; f_1\in \mathfrak{A}(U_1)\}$ (resp.  $\{d_{\phi(x)}f_2,\; f_2\in \mathfrak{A}(U_2)\}$) where $\phi(U_1)\subset U_2$. Therefore
 $(T_{(x,\phi(x))} G_\phi)^0=\{(d_x(f_2\circ \phi)),d_{\phi(x)}f_2),\; f_2\in \mathfrak{A}(U_2)\}$.
  Since $\phi$ is a Poisson map, it follows that 
\[
\begin{array}{cl}
	& <d_x(f_2\circ \phi), P_1 d_x(g_2\circ \phi)>- <d_{\phi(x)}f_2),P_2(d_{\phi(x)}g_2)>		\\
=	& \{f_2\circ\phi, g_2\circ \phi\}_{P_1}-\{f_2,g_2\}_{P_2}\circ \phi=0.
\end{array}
\]
Thus, from Lemma \ref{L_FirstProperties}, 5., $T_{(x,\phi(x))} G_\phi$ is coisotropic.\\
  
But, $(\phi^{-1},\phi)$ is clean pair from  Lemma \ref{L_Sudmertion}, and so from  Theorem \ref{T_CompositionCleanPoissonRelations}
the graph of  $\phi\circ \phi^{-1}$ is coisotropic.\\ 

(ii) 
 At the one hand, since $M_1$ is connected and $\phi$ is a submersion,  we have a convenient space $\mathbb{F}$ such that $\mathbb{M}_1$ is isomorphic to $\mathbb{M}_2\times \mathbb{F}$. In particular, $\ker T\phi$ is a subbundle of $TM$ with typical fibre $\mathbb{F}$. Let $(\ker T\phi)^a$ be the subbundle of $T^\prime M_1$ which is the  annihilator of $\ker T\phi$. Of course, $\ker T\phi$ is a the typical of a convenient bundle  isomorphic to $\mathbb{M}^\prime_2$. We consider the pull-back $\phi^\ast(T^\prime M_2)$ over $M_1$. Then $T^*\phi$ can be considered as a bundle morphism from $\phi^\ast(T^\prime M_2)$  to $T^\prime M_1$. Since $\phi$ is a submersion, for any $x\in M_1$, we must have $\ker T_x^\ast \phi=\{0\}$ and $T_x^\ast\phi(T_{\phi(x)}^\prime M_2)=(\ker T_x \phi) ^a$.  As $T^\ast \phi$ is a bundle morphism, this implies that $T^\ast \phi$ is an isomorphism from $\phi^\ast(T^\prime M_2)$  
to $(\ker T\phi)^a$ and so we can identify $(\ker T\phi)^a$  and $\phi^\ast(T^\prime M_2)$. Thus we have the following commutative diagram
\begin{equation}
\label{eq_diagramTphia}
\xymatrix{
{(\ker T\phi)^a}\ar[rd]  &
 \ar@{=}[l]_{T^\ast \phi} \phi^\ast(T^\prime  M_2)\ar[d]\ar[r]^{\widehat{\phi^\ast} }&T^\prime  M_2 \ar[d]\\
&M_1 \ar[r]^{\phi }         & M_2}
\end{equation}
where $\widehat{\phi^\ast}$ is a smooth convenient bundle morphism  such that the restriction $\widehat{\phi^\ast}_x$ to the fibre $\left(\phi^\ast(T^\prime M_2)\right)_x$ is an isomorphism onto $T_{\phi(x)}^\prime M_2$.

We denote $\iota_0$ the inclusion of $(\ker T\phi)^0$ into $(\ker T\phi)^a$.
\begin{lemma}
\label{L_TflatM2}
${}$
\begin{enumerate}
\item[(i)]  
For any $y\in M_2$, the vector space $T_y^\flat M_2=\widehat{\phi^\ast}\circ \iota_0 \left( (\ker T\phi)^0_x \right)$ if $\phi(x)=y$ does not depend on such a choice of $x$ and has a structure of convenient space isomorphic to $\mathbb{M}_2^\flat$.
\item [(ii)]
The range $T^\flat M_2=\widehat{\phi^\ast}\circ \iota_0 \left( (\ker T\phi)^0 \right) $ has a structure of  convenient weak subbundle of $T^\prime M_2$. In particular, $(\ker T\phi)^0$ is the pull-back of $ T^\flat M_2$ over $\phi$.
\end{enumerate}
\end{lemma}

\begin{proof}
${}$\\
(i) Note that by construction, for each $x\in M_1$,  there exists an open set $U_2$ in $M_2$ around $\phi(x)$ such that $(\ker T\phi)^0_x$ is generated by 
$\{ 
d_x(f_2\circ 
\phi),\; f_2\in \mathfrak{A}(U_2)
\}$. 
Indeed  the fibre  $(\ker T\phi)^a_x$,  is generated by the set $\{ d_x(f_2\circ \phi),\; f_2\in C^\infty(U_2)
\}$.

In particular, we can remark that 
\[
\forall f_2\in C^\infty(U_2),\;
\widehat{\phi^\ast} \left( d_x(f_2\circ \phi) \right)
= d_{\phi(x)}f_2.
\]
But, since  the fibre $T_x^\flat M_1$ is generated by 
$\{d_xf_1,\; f_1\in \mathfrak{A}(U_1)\}$, their intersection is generated by $\{ d_x(f_2\circ \phi),\; f_2\in \mathfrak{A}(U_2)\}$. \\
Fix some $y\in M_2$  and choose an open set $U_2$ in $M_2$ around $y$ such that $T_y^\prime M_2$ is generated by $\{ d_yf_2,\; f_2\in C^\infty(U_2)\}$ and consider $U_1=\phi^{-1}(U_2)$. According the the previous Remark, for each $x\in \phi^{-1}(y)$ the fibre  $(\ker T\phi)^0_x$ is generated by $\{ d_x(f_2\circ \phi),\; f_2\in \mathfrak{A}(U_2)\}$.  \\
Consider two points $x_1\not=x_2$ such that  $y=\phi(x_1)=\phi(x_2)$.  As $\phi$ is a submersion, the convenient space $\mathbb{M}_1$ is isomorphic to $\mathbb{M}_2\times \mathbb{F}$ and,  for $i \in \{1,2\}$, we  have charts $(V_i, \Phi_i)$ around $x_i$ in $M_1$  such that:\\

(1) $\; \Phi_i(x_i)=(0,0)\in \mathbb{M}_2\times \mathbb{F}$;\\

(2)	$\; \Phi_i(V_i)=\overline{W}\times \overline{V}\subset\mathbb{M}_2\times\mathbb{F} $  and $\Psi(W)=\overline{W}\subset \mathbb{M}_2$;\\

(3)	$\; \Psi\circ \phi\circ \Phi_i^{-1}=q$ where $q$ is the  first projection from $\mathbb{M}_2\times \mathbb{F}$ on $\mathbb{M}_2$.\\		

It follows that $T\Phi_i$ is a bundle isomorphism from ${TM_1}_{| V_i}$ to $\overline{W}\times \overline{V}\times \mathbb{M}_2\times \mathbb{F}$ and $T\Psi$ is an isomorphism from ${TM_2}_{|W}$ to $\overline{W}\times \mathbb{M}_2$ and so ${\ker T\phi}_{| V_i}$ is isomorphic to $\overline{W}\times \overline{V}\times \mathbb{F}$. Now from (2) and(3), according to
diagram (\ref{eq_diagramTphia}), this easily  implies that  we have an isomorphism $\Theta_i={T^\ast\Phi_i^{-1}}_{| V_i}$ from  $(\ker T\phi)^a_{| V_i}$   to  $\overline{W}\times \overline{V}\times \mathbb{M}^\prime_2$ over $\Phi_i$. Therefore, if we set $\theta=\Phi_2^{-1}\circ \Phi_1$, then $\Theta=\Theta_2^{-1}\circ \Theta_1$ is a convenient  isomorphism from  $
(\ker T\phi)^a_{| V_1}$ to  $(\ker T\phi)^a_{| V_2}$ 
such that,  in restriction to  $(\ker T\phi)^a_{| V_1}$, we have\\ 

(4) $\; \widehat{\phi^*}\circ \Theta={\widehat{\phi^*}}$.\\
Since the restriction of  $ \widehat{\phi^*}$ to each fibre is an isomorphism,  according to (1), we have\\
 
(5) $\; \Theta \left( d_{x_1}(f\circ \phi) \right) 
=d_{x_2}(f\circ \phi)$ which implies Point (i).\\

(ii) From Point (i), the  set  $T^\flat M_2=\displaystyle\bigcup_{y\in M_2}T_y^\flat M_2$ is well defined and  let $\iota_2:  T^\flat M_2 \to T^\prime M_2$ be the inclusion.\\
Now, since by assumption, $(\ker T\phi)^0$ is a closed convenient subbundle of $T^\flat M_1$ and the inclusion of $T^\flat M_1$ in $T^\prime M_1$ is a smooth convenient morphism, the inclusion of $(\ker T\phi)^0$ into $(\ker T\phi)^a$ is also a smooth convenient morphism. This gives rise to an injective bounded map from each fibre $(\ker T\phi)^0_x$ into $(\ker 
T\phi)^a_x$. Thus we may consider that the typical fibre $\mathbb{M}_2^\flat$ of  $(\ker T\phi)^0$ is contained in the typical fibre $\mathbb{M}_2^\prime$ of $(\ker T\phi)^a$ and the inclusion $\overline{\iota_0}$ is an injective bounded map. \\
Using the local context introduced in (2) and (3), without loss of generality, we can assume that $V_i$ is simply connected  and so $(\ker T\phi)^0{| V_i}$ is trivial. Thus we have a trivialization  $\Phi^\flat_i:(\ker T\phi)^0{| V_i} \to \overline{W}\times \overline{V}\times \mathbb{M}_2^\flat$;  
since $\widehat{\phi^*}\circ \iota_0$ is a smooth convenient  morphism from $(\ker T\phi)^0_{V_i}$ into $T^\prime M_2$ whose range is $T^\flat M_2$ we have the following commutative diagrams:
\begin{equation}
\label{eq_Diagram1}
\xymatrix {
(\ker T \phi)^0_{| V_i} \ar[rr]^{\Phi^\flat _i} \ar[dd] \ar[dr]^{\iota_0}  &	& 
\overline{W} \times \overline{V} \times \mathbb{M}_2^\flat \ar[dr]^{\overline{\iota_0}} 
|!{[dl];[dr]}\hole	
\ar[dd]
\\	
& (\ker T \phi)^a_{| V_i} 
\ar[rr]^{T^*\Phi_i^{-1}\hspace{40pt}} \ar[dd]      &	& \overline{W}\times\overline{V}\times\mathbb{M}_2^\prime \ar[dd] 												\\
V_i \ar[rr]^{\;\;\;\Phi_i} 
   \ar[dr]^{\operatorname{Id}}   
   && \overline{W}\times\overline{V}_i \ar[rd]^{\operatorname{Id}}		\\
   & V_i\ar[rr]^{\Phi_i} && \overline{W}\times\overline{V}_i 					\\
 }
\end{equation}
\begin{equation}
\label{eq_Diagram2}
\xymatrix {
(\ker T \phi)^a_{| V_i}
\ar[rr]^{T^\ast\Phi^{-1}} 
\ar[dd] \ar[dr]^{\widehat{\phi^*}}  &	& 
\overline{W}\times\overline{V}\times\mathbb{M}_2^\prime \ar[dr]^{q\times \operatorname{Id}} 
|!{[dl];[dr]}\hole	
\ar[dd]												\\	
& {T^\prime M_2}_{| W} 
\ar[rr]^{T^*\Phi_i^{-1}\hspace{40pt}} \ar[dd]   &	& \overline{W}\times\mathbb{M}_2^\prime \ar[dd] 		\\
V_i \ar[rr]^{\;\;\;\Phi_i} 
\ar[dr]^{\phi}   					&	& \overline{W}\times\overline{V}_i \ar[rd]^{q}					\\
   & W \ar[rr]^{\Psi} && \overline{W} 			\\
 }
\end{equation}    
Since $\widehat{\phi^*}\circ \iota_0$ is an injective  smooth convenient bundle morphism over $\phi_{| V_i}$ , it follows that $\widehat{\phi^*}\circ \iota_0(  (\ker T \phi)^0_{| 
 V_i})={T^\flat M_2}_{| W}$  and, according to (4), this range does not depend on the choice of such trivializations $\Phi_i^\flat$. Thus,  we can provide  ${T^\flat M_2}_{| W}$  with a structure of trivial convenient bundle over $W$. Indeed, let $W_i(x)=\Phi_i^{-1}({\overline{W}\times\{0\}})$ the closed supplemented submanifold  of $V_i$. Then the restriction of $\phi$ to $W_i(x)$ is a diffeomorphism  onto $W$ and so   
$\overline{\iota_0}\circ \Phi_i^\flat\circ(\widehat{\phi^*}\circ \iota_0)^{-1}:{T^\flat M_2}_{|W}\to \overline{W}\times \mathbb{M}_2^\flat\subset \overline{W}\times \mathbb{M}_2^\prime$ is then a trivialization over  $q\circ\Phi_i\circ 
(\phi_{|W_i(x)})^{-1}:W\to \overline{W}$.   \\

Therefore,  by merging   Diagrams (\ref{eq_Diagram1}) and (\ref{eq_Diagram2}),  and using composition  of $T^\ast \Psi^{-1}$ with $\iota_2$, we obtain the following commutative diagram:
\begin{equation}
\label{eq_Diagram3}
\xymatrix {
(\ker T \phi)^0_{| V_i}
\ar[rr]^{\Phi^\flat_i} 
\ar[dd]
\ar[dr]^{\widehat{\phi^*} \circ \iota_0}  &	& 
\overline{W}\times\overline{V}\times\mathbb{M}_2^\flat 
\ar[dr]^{q\times \overline{\iota_0}} 
|!{[dl];[dr]}\hole	
\ar[dd]												\\	
& {T^\flat M_2}_{| W} 
\ar[rr]^{T^*\Psi^{-1}\circ \iota_2\hspace{50pt}} \ar[dd]   &	& \overline{W}\times\mathbb{M}_2^\flat \ar[dd] 		\\
V_i \ar[rr]^{\hspace{20pt}\Phi_i} 
\ar[dr]^{\phi}   					&	& \overline{W}\times\overline{V}_i \ar[rd]^{q}					\\
   & W \ar[rr]^{\Psi} && \overline{W} 			\\
 }
\end{equation}
In particular, we have
\begin{equation}
\label{eq_TPsitrivialization}
\left\{
\begin{array}{rcll}
\Psi
&=&	q \circ\Phi_i\circ (\phi_{|W(x)})^{-1} 
&\textrm{ on  }W\\
T^\ast\Psi^{-1}\circ \iota_2
&=&	\left( q \times \overline{\iota_0} \right) \circ \Phi_i^\flat\circ(\widehat{\phi^*}\circ \iota_0)^{-1} 
&\textrm{ on } {T^\flat M_2}_{| W}
\end{array}
\right.
\end{equation}

By merging   Diagrams (\ref{eq_Diagram1})
 and (\ref{eq_Diagram2}), and using composition with $\iota_2$, 
we see that $\widehat{\phi^*}\circ \iota_0$ is a smooth morphism bundle  from $(\ker T\phi)^0_{| V_i}$ to $T^\prime M_2$ over $\phi$ such that 
\[
T^*\Psi^{-1}\circ (\widehat{\phi^*}\circ \iota_0)
=(q\times Id \circ \overline{\iota_0})\circ \Phi_i^\flat.
\]
  
According to this relation,   since ${T^\flat M_2}_{| W}=\widehat{\phi^*}\circ \iota_0((\ker T\phi)^0_{| V_i})$ in  $T^\prime M_2$ and $T^\ast \Psi^{-1}$ is a smooth isomorphism, it follows that the following diagram is commutative
  
\begin{equation}
\label{eq_Diagram4}
\xymatrix {
(\ker T \phi)^0_{| V_i}
\ar[rr]^{\Phi^\flat_i} 
\ar[dd]
\ar[dr]^{\widehat{\phi^*} \circ \iota_0}  &	& 
\overline{W}\times\overline{V}\times\mathbb{M}_2^\flat 
\ar[dr]^{q\times \overline{\iota_0}} 
|!{[dl];[dr]}\hole	
\ar[dd]												\\	
& {T^\prime M_2}_{| W} 
\ar[rr]^{T^*\Psi^{-1}\hspace{40pt}} \ar[dd]   &	& \overline{W}\times\mathbb{M}_2^\prime \ar[dd] 		\\
V_i \ar[rr]^{\hspace{20pt}\Phi_i} 
\ar[dr]^{\phi}   					&	& \overline{W}\times\overline{V}_i \ar[rd]^{q}					\\
   & W \ar[rr]^{\Psi} && \overline{W} 			\\
 }
\end{equation} 

and so  that  the restriction of $T^\ast \Psi^{-1}$ to ${T^\flat M_2}_{| W}$ is a trivialization.  
 
In fact,  by merging diagrams (\ref{eq_Diagram3}) for $i \in \{1,2\}$, we obtain
\begin{equation}
\label{eq_DiagramkerTphii}
\xymatrix{
	& 	(\ker T \phi)^0_{| V_1} 
	\ar[ld]_{\Theta^\flat} 
	\ar[rd]^{\widehat{\phi^*}\circ\iota_0} 
	\ar[dd]	&	&				\\
	(\ker T \phi)^0_{| V_2} 
	\ar[rr]^{\hspace{50pt}\widehat{\phi^*}\circ\iota_0}
	\ar[dd]		& 	 &	
	{T^\flat M_2}_{| W} \ar[rr]^{T^*\Psi^{-1}\circ \iota_2}\ar[dd] & & \overline{W}\times\mathbb{M}_2^\flat	\ar[dd]	\\
	&	V_1	\ar[ld]_{\theta} \ar[rd]^{\phi}			&	&				\\
	V_2 \ar[rr]^{\phi}		&			& 	W \ar[rr]^{\Psi}	& &	\overline{W}	
}
\end{equation}
where $\Theta^\flat=(\Phi_2^\flat)^{-1}\circ\Phi_1^\flat$ is a smooth isomorphism over $\theta=\Phi_2^{-1}\circ \Phi_1$. 
It follows that   such a trivialization of ${T^\flat M_2}_{| W}$ does not depend on the choice of  the trivialization of $(\ker T\phi)^0$ around any point $x$ in the fibre over  $y\in W$.\\
 
Note that we have $\Theta^\flat=\Theta\circ \iota_0$.\\
 
Consider two charts $ \left( W'_i,\Psi'_i \right) $ around some $y\in M_2$ for $i \in \{1,2\}$.  Fix some $x\in M_1$ such that $\phi(x)=y$. By same arguments as in (1), (2) and (3), we can find  a chart $(V'_i,\Phi'_i)$  around $x$ in $M_1$ such that, after shrinking $W_i$ if necessary, we have\\

(1') $\; \Psi'_i(y)=\Psi'(y)=0\in \mathbb{M}_2 \textrm{ and }\Phi'_i(x)=(0, 0)\in \mathbb{M}_2\times \mathbb{F}$;	\\

(2') $\;  \Psi_i(W'_i)=\overline{W}'_i\subset \mathbb{M}_2 \textrm{ and } \Phi'_i(V'_i)=\overline{W}'_i\times \overline{V}'_i$;\\

(3') $\; \Psi'_i\circ \phi\circ (\Phi_i')^{-1}=q$ where $q$ is the  first projection of $\mathbb{M}_2\times \mathbb{F}$ on $\mathbb{M}_2$.\\

We again assume that $V'_i$ and $W'_i$ are simply connected. Since $(\ker T\phi)^0$ is a convenient bundle, it follows that we have the trivialization 
\[
 \left( \Phi'_i \right) ^\flat: (\ker T\phi)^0_{| V'_i}\to \overline{W}'_i\times \overline{V}'_i\times\mathbb{M}_2^\flat.
 \]
Then the composition $(\Phi'_2)^\flat\circ[(\Phi'_1)^\flat]^{-1}$ in restriction to  $(\overline{W}'_1\cap  \overline{W}'_2)\times (\overline{V}'_1\cap \overline{V}'_2)\times \mathbb{M}_2^\flat$ 
 is a diffeomorphism of type 
$$(\bar{x},\bar{z},\bar{\alpha})\mapsto (\bar{x},\Pi_{\bar{x}}(\bar{z}), A_{\bar{x}}(\bar{\alpha})),$$
where  $\bar{x}\mapsto \Pi_{\bar{x}}$ (resp. $\bar{x}\mapsto A_{\bar{x}})$ is a field over $\overline{W}'_1\cap  \overline{W}'_2$ of local diffeomorphisms of $\mathbb{F}$ over 
$\overline{V}'_1 \cap\overline{V}'_2$ (resp. automorphisms of $\mathbb{M}_2^\flat$).  As in the upper situation,  we have an analog  commutative  diagrams as (\ref{eq_Diagram3} ) and by 
superposition  of these diagrams and using the transition map (5') we obtain a transition map for  ${T^\flat M_2}_{| W}$ and ${T^\flat M_2}_{| W'}$  which show that $T^\flat M_2$ is a convenient subbundle of $T^\prime M_2$.
 Now, by construction, we have $(\ker T\phi)^0_x:=\{(x,\widehat{\phi^*}^{-1}_x(\alpha)),\; \alpha\in T^\flat M_2$ which ends the proof of Point (ii).
\end{proof}

\noindent\emph{End of the proof of Theorem \ref{T_ExistencePartialPoissonStructure}, Point(ii)}


At first, from the proof of Point (i) of Lemma \ref{L_TflatM2}  the fiber $T^\flat M_2=\mathfrak{P}_y(U_2)$ for all $y\in M_2$. Now, since $ (\ker T\phi)^0$ is the pull-back of $T^\prime M_2$, this implies that $T_{x}^\ast\phi=\widehat{\phi^\ast}_x^{-1}$.   
According to notations in the proof of Lemma  \ref{L_TflatM2} we have $\theta=\Phi_2^{-1}\circ \Phi_1$, so $T\theta=T\Phi_2^{-1}\circ T\Phi_1$ and so $T^\ast \theta= T^\ast \Phi_1\circ T^\ast \Phi_2^{-1}=\Theta^{-1}$. 
From diagram (\ref {eq_DiagramkerTphii}), (4) 
and since $\Theta^\flat=\Theta\circ \iota_0$, it follows that for any $\alpha\in T^\flat_y M_2$,  if $y=\phi(x_1)=\phi(x_2)$ then $T^\ast _{x_1}\phi(\alpha)=\Theta^\flat \circ  T^\ast _{x_2}\phi(\alpha)= \Theta^{-1}$. Therefore,  for any $\alpha\in T^\flat_y M$,  
if $P^0_1$ denotes the restriction of $P_1$ to $(\ker T\phi)^0$ we have
\[
(P_1^0)_{x_2}\circ T_{x_1}^\ast\phi=(P^0_1)_{x_1}\circ \Theta^\flat\circ T_{x_1}^\ast \phi= (P^0_1)_{x_1}\circ T^*\theta^{-1}\circ T_{x_1}^\ast \phi.
\]
But $T_{x_2}\phi=T\theta\circ T_{x_1}\phi$. Thus, for any $\alpha\in T_yM_2$, we have 
\[
T_{x_2}\phi \circ (P_1^0)_{x_2}\circ T_{x_1}^\ast\phi(\alpha)
=(T_{x_1}\phi \circ T\theta )\circ (P^0_1)_{x_1}
\circ T_{x_1}^\ast\phi(\alpha).
\]
By setting
 $(P_2)_{y}= T\phi\;\circ P_1\circ\; \widehat{\phi^*}^{-1}_x\circ  \iota_0 $, we get a  partial Poisson structure  $(T^\flat M_2, TM_2, P_2,\{\;,\;\}_{P_2})$. 

(iii) If  $ \left( T^{\flat}M_1,TM_1,P_1,\{\;,\;\}_{P_1} \right) $ is partial symplectic, this means that $P$ is an isomorphism and so we have a smooth  non-degenerate $2$-form $\Omega_P$ on $T M_1$ defined by $\Omega_P(u,v):=<P^{-1}(u), v>$. 
Since $\Omega$ is  non degenerate, for all $x\in M_1$,  the orthogonal 
$ \left( \ker T_x\phi \right) ^\perp
=\{u\in T_xM_1,\; 
\forall v\in \ker T_x\phi\, \; \Omega_P(u,v)=0\}$ 
defines a convenient  subbundle $(\ker T\phi)^\perp$ of $TM_1$ and so $(\ker T\phi)^0=P^{-1} \left( (\ker T\phi)^\perp \right) $ is a convenient subbundle of $T^\flat M_1$. Then we can apply (ii).\\

If  $T^\flat M_1=T^{\prime}M_1$, then $(\ker T\phi)^0=(\ker T\phi)^a$ is a convenient subbundle of $T^\flat M_1$ and so we also can apply (ii). 
\end{proof}

\subsubsection{Dual Pairs}
\label{____DualPairs}
We end this section by a generalization of the notion of dual pairs introduced in \cite{CaWe99}. 

\begin{definition}
\label{D_DualPairs}  
Consider a  partial symplectic manifold $ \left( T^\flat M, TM, P \right)$ and,  for $i \in \{1,2\}$, a  Poisson map $\phi_i$ from $M$ to the partial  Poisson manifold\\ 
$(T^\flat M_i, TM_i, P_i)$. We say that $(\phi_1,\phi_2)$ is a \emph{dual pair}\index{dual pair} if, for any $x\in M$, we have
$\ker T_x\phi_1=(\ker T_x\phi_2)^{\perp_P}$ and $\ker T_x\phi_2=(\ker T_x\phi_1)^{\perp_P}$.
\end{definition}
 
\begin{proposition}
\label{P_DualPairs} 
Consider a  partial symplectic manifold $ \left( T^\flat M, TM, P,\{\;,\;\}_p \right) $.  
For $i \in \{1,2\}$,  let $\phi_i$ be a  Poisson map  from $ \left( T^\flat M, TM, P,\{\;,\;\}_P \right)$ to the partial  Poisson manifold $ \left( T^\flat M_i, TM_i, P_i ,\{\;,\;\}_{P_i} \right) $ such that  $ \left( \phi_1,\phi_2 \right) $ is a dual pair.  Then, for each open set $U_i$ in $M_i$ and any $f_i$ in  the Poisson-Lie algebra $(\mathfrak{A}(U_i),\{\;,\;\}_{U_i})$ we have
$\{\phi_1^{\ast}f_1,\phi_2^{\ast} f_2\}_P=0$ on $ \phi_1^{-1}(U_1)\cap \phi_2^{-1}(U_2)$.
\end{proposition}

\begin{proof} 
In the one hand, under our assumptions, from the proof of Theorem~\ref{T_ExistencePartialPoissonStructure},  (ii), we have $T^\flat M_i=\displaystyle\bigcup_{y\in M_i} \mathfrak{P}_y(U_i)$ where $\mathfrak{P}_y(U_i)=\{d_yf,\; y\in \mathfrak{A}_i(U_i)\}$.
On the other hand, the characteristic foliation is reduced to $M$, so by  Remark~\ref{R_Leaf}, (2) and (3), we have a weak symplectic form $\omega$ such that $\omega(u, v)=<P^{-1}(u), v>$ for all $u$ and $v$ in $T_xM$ and all $x\in M$. On the other hand, for any $\alpha\in T_x^\flat  M$, we have an open set $U$ around $x$ in $M$ and a smooth map $f:U\to \mathbb{R}$ such that $\alpha=d_xf$ and since $P:T^\flat M \to TM$ is an isomorphism, we have $\omega _x( u, v)=\omega_x(X_{f_u},X_{f_v})= d_x\{f_u,f_v\}_P$ if $P^{-1}(u)=d_xf_u$ and $P^{-1}(v)=d_x f_v$. \\ 
Thus, according to Proposition~\ref{P_PartialSymplectic}, we have 
\[
(\ker T_x\phi_i)^{\perp_P}=\{P(\{d_xf),\;  f\in \mathfrak{A}(U):\;
\forall v\in \ker T_x\phi_i,\;  < d_xf, v>=0 \}.
\]   
But since $\phi_i$ is a Poisson map, this implies that  if $f_i$ belongs to $\mathfrak{A}(U_i)$ then $\phi_i^*f_i$ belongs to $ \mathcal{E}_i \left( \phi_i^{-1}\left( U_i \right) \right) $ and then each $d_x \left( \phi_i^{\ast}f_i \right) $ belongs to $(\ker T_x \phi_j)^0$ for $i\not=j$  and $x\in \phi_1^{-1}(U_1)\cap \phi_2^{-1}(U_2)$ and then if $ \left( \phi_1,\phi_2 \right) $ is a dual pair, if follows that $\{\phi_1^{\ast}f_1,\phi_2^{\ast}f_2\}_P=0$ on  $\phi_1^{-1}(U_1)\cap \phi_2^{-1}(U_2)$.
\end{proof}

\begin{remark} 
\label{R_CasePhiantiPoisson} 
The conclusion of Proposition~\ref{P_DualPairs} is also true  when $(\phi_1,-\phi_2)$ or $(-\phi_1,\phi_2)$ is a dual pair.
\end{remark}

\section{Partial Poisson (resp. Symplectic) Convenient Lie Groups}
\label{__PartialPoissonSymplecticConvenientLieGroups}

This section contains a survey of definitions and properties of a convenient Lie group which are either contained in \cite{CaPe23} or are an adaptation of results on Banach Lie groups contained in \cite{Tum20} (see also \cite{GRT23}).

\subsection{Preliminaries and Notations} 

\begin{definition}
\label{D_ConvenientLiegroup} 
A \emph{convenient Lie group}\index{convenient!Lie group}\index{Lie group} $\mathsf{G}$ is a convenient manifold endowed with a group structure for which the multiplication ${\bf m}:\mathsf{G}\times \mathsf{G}\to \mathsf{G}$ and the inversion ${\bf i}:\mathsf{G}\to \mathsf{G}$ are smooth mappings.
\end{definition}

The unit element of the group $ \mathsf{G}$ will be denoted $e$ and the inverse of an element $g$ by $g^{-1}$.\\
For $g \in \mathsf{G}$, $L_{g}:x\mapsto g.x$ is the \emph{left translation}\index{left translation} and $R_{g}:x\mapsto x.g$ is the \emph{right translation}\index{right translation}.

 The tangent bundle to $ \mathsf{G}$ will be  the kinematic tangent bundle $T \mathsf{G}$ and the vector fields, smooth sections of this bundle, are the kinematic vector fields whose set is denoted $\mathfrak{X}(\mathsf{G})$.

\begin{definition}
\label{D_LeftInvariantVectorFields}
 A vector field $X$ on $ \mathsf{G}$ is said to be
left invariant%
\index{vector field!left invariant}
if
\[
\forall g\in \mathsf{G},
\left( L_{g} \right)  _{\ast}X=X
\]
where $\left(  L_{g}\right)  _{\ast}X=TL_{g}\circ X\circ L_{g^{-1}}$.
\end{definition}

Since $\left(  L_{g}\right)  _{\ast}\left[  X,Y\right]  =\left[  \left(
L_{g}\right)  _{\ast}X,\left(  L_{g}\right)  _{\ast}Y\right]  $, the set of
all left invariant vector fields on $ \mathsf{G}$ forms a Lie subalgebra of
$\mathfrak{X}(\mathsf{G})$ denoted $\mathfrak{g}$.

\begin{definition}
\label{D_LieAlgebraOfALiegroup}
$\mathfrak{g}  $ is called the Lie algebra\index{Lie algebra} of the Lie group $ \mathsf{G}$.
\end{definition}

$\mathfrak{g}  $ is in bijective correspondence with the tangent
space $T_{e}g$ \textit{via} the linear isomorphism%
\[
\begin{array}
[c]{ccc}
\mathfrak{g} & \to & T_{e}\mathsf{G}\\
X & \mapsto & X_{e}
\end{array}
\]
So $T_{e} \mathsf{G}$ becomes a Lie algebra for the bracket still denoted $\left[.,.\right]  $ and defined by
\[
\left[  X_{e},Y_{e}\right]  =\left[  X,Y\right]  _{e}
\]

\begin{proposition}
Let $\varphi: \mathsf{G}\to \mathsf{H}$ be a smooth homomorphism \index{homomorphism!Lie group} of Lie groups. \\
Then $T_{e}\varphi:\mathfrak{g}  \to \mathfrak{h}  $ is a Lie algebra homomorphism\index{homomorphism!Lie algebra}.
\end{proposition}

\begin{definition}
\label{D_OneParameterSubgroup}
Let $\left( \mathsf{G},.\right)$ be a Lie group with Lie algebra $\mathfrak{g}$. A \emph{$1$-parameter subgroup}\index{$1$-parameter subgroup} of $ \mathsf{G}$ is a smooth Lie group homomorphism 
$\alpha: \left(  \mathbb{R},+ \right) 
\to \left( \mathsf{G} ,.\right) .$
\end{definition}

So a $1$-parameter subgroup of $ \mathsf{G}$ is a smooth curve $\alpha$ in $\mathsf{G}$ such that
\[\left\{
\begin{array}
[c]{c}%
\forall\left(  s,t\right)  \in\mathbb{R}^{2},\ \alpha\left(  s+t\right)
=\alpha\left(  s\right)  .\alpha\left(  t\right)  \\
\alpha\left(  0\right)  =e
\end{array}
.\right.
\]

\begin{proposition}
\label{P_CharacterizationIntegralCurveInLieAlgebra}
Let $\alpha:\mathbb{R} \to \mathsf{G}$ be a smooth curve where $\alpha\left(  0\right)  =e$ and let $X\in\mathfrak{g}$. We denote $\hat{X}$ the left invariant vector field such that $\hat{X}_{e}=X$. Then the following assertions are equivalent:
\begin{enumerate}
\item
$\alpha$ is a $1$-parameter subgroup of $ \mathsf{G}$ and $X=\left.
\dfrac{\partial}{\partial t}\right\vert _{t=0}\alpha\left(  t\right)  $;
\item
$\alpha$ is an integral curve of the left invariant vector field
$\hat{X}$;
\item
$\operatorname{Fl}^{\hat{X}}(t,x)  =x.\alpha(t)  $ is the unique global flow of the vector field $\hat{X}$.
\end{enumerate}
\end{proposition}

\begin{definition}
\label{D_ConvenientLieGroupExponentialMap} 
Let $\mathsf{G}$ be a Lie group with Lie algebra
$\mathfrak{g}$. We say that $\mathsf{G}$ admits an exponential map\index{exponential map} if there exists a smooth mapping $\exp:\mathfrak{g} \to  \mathsf{G}$ such that 
$t \mapsto \exp(tX) $ is the unique $1$-parameter subgroup with tangent vector $X$ at $0$.
\end{definition}

As in the Banach framework, we have the following properties:

\begin{description}
\item[\textbf{(ExpC1)}] $\exp\left(  0\right)  =e$;

\item[\textbf{(ExpC2)}] $T_{0}\exp=\operatorname{Id}_{\mathfrak{g}}$
\end{description}

where \textbf{(ExpC2)} is a consequence of $\left.  \dfrac{d}{dt}\right\vert _{t=0}%
\exp\left(  tX\right)  =X$.\\

Unfortunately, for a kinematic vector field, integral curves need not exist locally\footnote{cf. \cite{KrMi97}, 32.12.}, so the existence of such a map is not guaranteed.\\

Let $\mathsf{G}$ be a convenient Lie group with Lie algebra $\mathfrak{g}$.
\begin{definition}
\label{D_ConvenientLieGroupConjugation}
For any $g\in \mathsf{G}$, we define the \emph{conjugation}\index{conjugation} or \emph{inner automorphism}\index{inner automorphism}
\[
\begin{array}
[c]{cccc}
\operatorname{conj}_{g}: & \mathsf{G} & \to & \mathsf{G}\\
	& x & \longmapsto & g.x.g^{-1}.
\end{array}
\]
\end{definition}

\begin{definition}
\label{D_AdjointRepresentation}
The adjoint representation\index{representation!adjoint} is the mapping
\[
\begin{array}
[c]{cccc}%
\operatorname{Ad}: & \mathsf{G} & \to & \operatorname{GL}\left(
\mathfrak{g}\right) \\
& g & \longmapsto & T_{e}\left(  \operatorname{conj}_{g}\right).
\end{array}
\]
The adjoint representation of the Lie algebra $\mathfrak{g}$ is the map
$\operatorname{ad}=T_{e}\operatorname{Ad}$.\\
We then have:
\begin{equation}
\label{eq_adLieBracket}
\forall\left(  X,Y\right)  \in\mathfrak{g}^{2},\operatorname{ad}\left( X \right)  Y=\left[  X,Y\right] .\\
\end{equation}
\end{definition}

\subsubsection{Coadjoint Action on a Convenient Subspace of a Dual and $1$-cocycle}\label{___CoadjointActionOnSubspace}

For this section, we refer to \cite{Tum20}, 2.5. But for the sake of completeness, we recall this classical context.\\

Let $\mathsf{G}$ be a convenient Lie group and $\Phi: \mathsf{G}\to \operatorname{Aff}(\mathbb{E})$ a representation into the affine group of a vector space $\mathbb{E}$.  As 
$\operatorname{Aff}(\mathbb{E})$ is a semi-direct product of the linear group $\operatorname{GL}(V)$ and the translation group on $\mathbb{E}$ identified with $\mathbb{E}$,  $\Phi$ can be written as a pair $(\phi,\Theta)$ where $\phi:\mathsf{G}\to \operatorname{GL}(\mathbb{E})$ and $\Theta:\mathsf{G}\to \mathbb{E}$. Since $\Phi$ is a group morphism it follows that $\phi$ is a group morphism and $\Theta$ satisfies
\begin{equation}
\label{eq_1CocylceGroup}
\Theta(gh)=\Theta(g)+\phi(g)\left(\Theta(h)\right)
\end{equation}
for all $g$ and $h$ in $\mathsf{G}$. In this case,  $\Theta$ is called \emph{${\bf 1}$-cocyle on $\mathsf{G}$ relative to $\phi$}\index{OneCocycle@$1$-cocyle}. \\

The derivative $T_e\Phi $ is a Lie algebra morphism from $\mathfrak{g}$ to the Lie algebra $\mathfrak{aff}(\mathbb{E})$ of $ \operatorname{Aff}(\mathbb{E})$  which can be written as a pair $(T_e \phi, T_e\Theta)$ where $T_e\phi$ and $T_e\Theta$ gives rise to a Lie algebra morphism $d\phi$ from $\mathfrak{g}$ to $\mathfrak{gl}(\mathbb{E})$ and $d\Theta$ from $\mathfrak{g}$ to $\mathbb{E}$ respectively which satisfy the relation
\begin{equation}
\label{eq_1CocycleLieAlgebra}
d\Theta([x,y])
=d\phi(x) \left( d\Theta(y) \right)
-d\phi(y) \left( d\Theta(x) \right)
\end{equation}
for all $x$ and $y$ in $\mathfrak{g}$. In this case,  $d\Theta$ is called{ \bf ${\bf 1}$-cocyle on $\mathfrak{g}$ relative to $d\phi$}. \\


In particular, consider a Lie group  $\mathsf{G}$  and a bi-invariant  subbunbdle $T^\flat \mathsf{G}$  of $T\mathsf{G}$. If we denote by $\mathfrak{g}^\flat$ the fibre $T_e^\flat\mathsf{G}$, then $\mathfrak{g}^\flat$ is invariant by the coadjoint action $\mathsf{G}$ on $T^\prime \mathsf{G}$  (formally, same proof as in Proposition 5.6 in \cite{Tum20}).

\begin{proposition}
\label{P_1CocylcleRelAd} 
Under the previous assumptions,  a map $\lambda: \mathfrak{g}\to  \bigwedge^2\mathfrak{g}^\flat$   is called  a $1$-cocycle relative to the coadjoint action of $\mathfrak{g}$ onto $\mathfrak{g}^\flat$ if,  we have the following properties:
\begin{enumerate}
\item[(1)]
there exists $\mathsf{p}:\mathfrak{g}\times\mathfrak{g}^\flat \to \mathfrak{g}$ such that
$\lambda(x)(\alpha,\beta)=<\alpha, \mathsf{p}(x)(\beta)>$, for all $x\in \mathfrak{g}$ and $\alpha, \beta$ in $\mathfrak{g}^\flat$;
\item[(2)] 
for all $x, y$ in $\mathfrak{g}$ and $\alpha, \beta$ in $\mathfrak{g}^\flat$, we have:
\begin{equation}
\label{eq_1Cocycle}
\lambda([x,y])=\operatorname{ad}^{(2,0)}_x(\lambda(y))-\operatorname{ad}^{(2,0)}_y(\lambda(x)).
\end{equation}
\end{enumerate}
\end{proposition}

Then, in the Banach setting, from \cite{Tum20}, Theorem~3.14, $ \left( \mathfrak{g},\mathfrak{g}^\flat, P \right) $ is a partial Lie-Poisson space in the sense of Definition~\ref{D_PartialLiePoissonConvenientSpace}. Thus this Definition can be considered as an 
adaptation of Definition~3.12 in \cite{Tum20}.

\subsubsection{Partial Bialgebras}

The notion of Lie bialgebra in finite dimension\footnote{cf. \cite{Kos04}.} appears as a particular case of  Lie bialgebroid (see \cite{MaXu94} for instance). A generalization to the Banach setting can be found in \cite{Tum20}. In the next paragraph,  we will  propose the following  natural adaptation which we call \emph{partial bialgebra}\index{partial!bialgebra}.
 
\begin{definition}
\label{D_PartialLieBialgebra} 
Let $(\mathfrak{g},[.,.])$ be a convenient Lie algebra and $\mathfrak{g}^\flat$ be a  convenient subspace of $\mathfrak{g}^\prime$ where the inclusion $\mathfrak{g}^\flat \to \mathfrak{g}^\prime$ is bounded. We will say that 
$ \left( \mathfrak{g},\mathfrak{g}^\flat \right) $ is a \emph{partial Lie bialgebra}\index{partial!Lie bialgebra} if:
\begin{enumerate}
\item[(1)]  
the restriction  of the  coadjoint action 
 to $\mathfrak{g}\times \mathfrak{g}^\flat$ takes values in $\mathfrak{g}^\flat$ and is bounded;
\item[(2)] 
there exists  a Lie algebra bracket $[.,.]^\flat$ on $\mathfrak{g}^\flat$ such that  the dual map of $[.,.]^\flat:\mathfrak{g}^\flat\times\mathfrak{g}^\flat\to \mathfrak{g}^\flat$
 restricts to a $1$-cocycle $\lambda:\mathfrak{g}\to \bigwedge^2\mathfrak{g}^{\flat\prime}$ relative to the coadjoint action of $\mathfrak{g}$ onto $\mathfrak{g}^\flat$.
\end{enumerate}
\end{definition} 
 
\begin{remark}
\label{R_CocycleCondition} 
According to \cite{Tum20}, Remark~7.3, by formal  same arguments adapted to the convenient setting, the $1$-cocycle condition can be written:
\begin{equation}
\label{eq_1Cocycle2}
\begin{matrix}
<[x,y],[\alpha,\beta]^\flat>=&<y,[\operatorname{ad}_x^*\alpha,\beta]^\flat>+<y,[\alpha,\operatorname{ad}_x^*\beta]^\flat>\hfill{}\\
					&-<x,[\operatorname{ad}_y^*\alpha,\beta]^\flat>-<x,[\alpha,\operatorname{ad}_y	a\beta]^\flat>.\hfill{}\\
\end{matrix}
\end{equation}
\end{remark}

\subsection{Bi-invariant Partial Poisson (resp. symplectic) Structures}
We begin by fixing some notations.\\
As in the Banach setting used in \cite{Tum20},  on a convenient Lie group $\mathsf{G}$ the left (resp. right) translation $L_g$ (resp. $R_g$) gives rise to a left  (resp. right) smooth  action of $\mathsf{G}$ on $T^\prime\mathsf{G}$ given by $(L_g)^*(\alpha)=\alpha\circ TL_g$ (resp. $R_g^*(\alpha)=\alpha\circ TR_g$) and also a smooth  action  denoted $L_g^{**}$ (resp. $R^{**}_g$) on the bidual  $T^{\prime\prime}\mathsf{G}$. \\
The smooth adjoint action  
$\operatorname{Ad}_g=L_g\circ R_{g^{-1}}$ 
of $\mathsf{G}$ on its  Lie algebra $\mathfrak{g}$ induces on the dual $\mathfrak{g}^{\prime}$  a smooth coadjoint action  $\operatorname{Ad}^*_g=L_g^*\circ R_{g^{-1}}^*$ and a smooth action $\operatorname{Ad}_g^{**}=L_g^{**}\circ R_{g^{-1}}^{**}$.\\
 
According to the context of Banach Poisson Lie groups exposed in  \cite{Tum20}, we introduce:

\begin{definition}
\label{D_PoissonLIeGroup} 
Given  a convenient Lie group $\mathsf{G}$,  a partial Poisson  (resp. symplectic) structure  
$ \left( T^\flat \mathsf{G}, T\mathsf{G}, P,\{\;,\;\}_P \right) $ is called a \emph{bi-invariant partial Poisson (resp. symplectic) structure}\index{bi_invariant partial Poisson} if  we have:
\begin{enumerate}
\item[{\bf (BiT)}] 
$T^\flat \mathsf{G}$ is bi-invariant that is $L_g^*(T_{gh}^\flat \mathsf{G})=T_h^\flat \mathsf{G}$ and $R_g^*(T_{hg}^\flat \mathsf{G})=T_h^\flat \mathsf{G}$ for all $g,h,\in \mathsf{G}$.
\item[{\bf (BiP)}] 
$P$ is bi-invariant that is $P_h\circ L_g^*=TL_g^{-1}\circ P_{gh}$ and $P_h\circ R_g^*=T R_g^{-1}\circ P_{hg}$ for all $g,h,\in \mathsf{G}$.
\end{enumerate}

Under theses conditions, we say that  
$ \left( T^\flat \mathsf{G}, T\mathsf{G}, P,\{\;,\;\}_P \right) $ is a convenient partial Poisson (resp. symplectic) Lie group.
\end{definition}

\begin{remark} 
When $T^\flat \mathsf{G}$ is bi-invariant, if we set $\mathfrak{g}^\flat:=T_e^\flat\mathsf{G}$, then the map  $(TL)^*: T^\flat \mathsf{G}\to G\times \mathfrak{g}^\flat$ (resp. $(R)^*: T^\flat \mathsf{G}\to G\times \mathfrak{g}^\flat$ ) given by 
$TL^*(g,\alpha)=(g, L_g^*(\alpha))$ (resp. $TR^*(g,\alpha)=(g, R_g^*(\alpha))$) is a convenient bundle isomorphism
\end{remark}

As we have already seen, to $P$ is associated a partial Poisson tensor $\Lambda $ defined by 
 $\Lambda(\alpha,\beta)=<\alpha, P(\beta)>$. Therefore, if $P$ is bi-invariant, we have: 
\begin{equation}
\label{eq_BiinvariantLambda}
\Lambda_{gh}(\alpha,\beta)=\Lambda_h(L_g^*\alpha, L_g^*\beta) +\Lambda_g(R_h^*\alpha ,R_h^*\beta).
\end{equation}

\begin{notations}
\label{Lg**}${}$\\
We will denote by $T^{\flat\prime}M$ the convenient dual bundle of $T^\flat M$.
Since the restriction to $T_h\mathsf{G}\subset T_h^{\prime\prime}\mathsf{G}$ of the adjoint $L^{**}g$ to $T_h\mathsf{G}$ is $TL_g$. However for the for the sake of clarity,   as in \cite{Tum20}, for any $\Lambda\in \bigwedge^kT^{\flat\prime}\mathsf{G}$ we set 

$(L_g^{**} \Lambda)_{gh}(\alpha_1,\dots,\alpha_k):= \Lambda_h(L_g^* \alpha_1,\dots ,L_g^* \alpha_k) $\\ 
and for analog reason we also set 

$(R_g^{**} \Lambda)_{hg}(\alpha_1,\dots,\alpha_k):= \Lambda_h(R_g^* \alpha_1,\dots ,R_g^* \alpha_k) $.\\
Since we have

$Ad_g=TR_{g^{-1}}\circ TL_g$ (resp. $Ad_{g^{-1}}=TL_{g^{-1}}\circ TR_g$)\\
we  set 

$Ad_g^{**}\Lambda_e= R_{g^{-1}}^{**}\circ L_g^{**} \Lambda_e$ (resp. $Ad_{g^{-1}}^{**}\Lambda_e=L_{g^{-1}}^{**}\circ R_g^{**} \Lambda_e$).\\
\end{notations}

In this way, the relation (\ref{eq_BiinvariantLambda}) can be written:
\begin{equation}
\label{eq_BiinvariantLambda1}
\Lambda_{gh}=L_g^{**}\Lambda_h +R_h^{**}\Lambda_g.
\end{equation}
\medskip
The following result generalizes  some essential results on Banach Poisson Lie groups proved in \cite{Tum20}:

\begin{proposition}\label{P_CharacterizationPartialPoissonLiegroup}  
Let $\mathsf{G}$ be a convenient Lie group and 
$ \left( T^\flat \mathsf{G},T\mathsf{G},P,\{\;,\;\}_P \right) $ any partial Poisson structure on $\mathsf{G}$. 
The following assertions are equivalent:
\begin{enumerate}
\item[(1)] 
$ \left( T^\flat \mathsf{G}, T\mathsf{G}, P,\{\;,\;\}_P \right) $ is a convenient partial Poisson (resp. symplectic)  Lie group.
\item [(2)] 
The multiplication $\bf{m}:\mathsf{G}\times \mathsf{G}\to \mathsf{G}$ is a Poisson (resp. partial symplectic) map when $\mathsf{G}\times \mathsf{G}$ is provided with the partial Poisson (resp. symplectic) structure product.
\item [(3)] 
For the Poisson structure $\mathcal{G}\times\mathcal{G}\times \mathcal{G}^-$ on $\mathcal{G}\times\mathcal{G}\times\mathcal{G}$, the graph $G_{\bf m}$  of the multiplication is coisotropic (resp. Lagrangian)
\item[(4)] $T^\flat \mathsf{G}$ is  a bi-invariant subbundle of $T^\prime\mathsf{G}$ and  the partial Poisson tensor $\Lambda$ associated to $P$ satisfies, for all $g,h\in \mathsf{G}$,
\begin{equation}
\label{eq_biinvarianceLambda}
\Lambda_{gh}=L_g^{**}\Lambda_h \textrm{ and } \Lambda_{hg}=R_g^{**}\Lambda_h.
\end{equation}
\end{enumerate}
\end{proposition}

\begin{remark} 
Note that Proposition \ref{P_CharacterizationPartialPoissonLiegroup}, (2)   implies that the Definition 5.3 of a Banach Poisson Lie group in \cite{Tum20} is coherent with our Definition \ref{D_PoissonLIeGroup} in the Banach setting.
\end{remark}

\begin{proof} 
We only look for the partial Poisson situation. For the partial symplectic framework, the same argument works too.\\
(1)$\Leftrightarrow $(2)\footnote{This proof is an adaptation of the proof of Proposition 5.4 in \cite{Tum20}.} 
In the one hand, for the partial Poisson structure on $\mathsf{G}\times \mathsf{G}$,  we have $T_{gh}^\flat \mathsf{G}\times\mathsf{G}=T_g^\flat \mathsf{G}\times T_h^\flat \mathsf{G}$. If $(X_g,X_h)\in T_{gh}^\flat \left( \mathsf{G}\times\mathsf{G} \right) $, then we have: 
\begin{equation}
\label{eq_Tm}
T_{(g,h)}\mathbf{m} (X_g,X_h)=TL_g(X_h)+TR_h(X_g).
\end{equation}
 The bi-invariance of $T^\flat \mathsf{G}$ means that, for any $\alpha \in T_h^\flat\mathsf{G}$ and $\alpha\circ T_gR_h \in 
T_{gh}^\flat \mathsf{G}$ then $\alpha\circ T_hL_g$ belongs to $T_g^\flat\mathsf{G}$. It follows that the relation (\ref{eq_Tm}), for all $g,h\in \mathsf{G}$ and $X_g,X_h \in T_{gh}^\flat \mathsf{G}\times\mathsf{G}$, is equivalent to ({\bf BiT}).
On the other hand, let $(\phi,\psi)$ be a smooth function around $(g,h)\in \mathsf{G}\times\mathsf{G}$. We denote by $X^{\phi}$ and $X^{\psi}$ the Hamitonian vector fields associated to $\phi$ and $\psi$ respectively. Then, for the partial product Poisson $P\times P$, the Hamiltonian field of $(\phi,\psi)$ is $(X^\phi,X^\psi)$. Then according to (\ref{eq_Tm}), we must have
\[
T_{(g,h)}{\bf m}(X^\phi,X^\psi)=T_hL_g(X^\phi)+T_gR_h(X^\psi).
\]
Now by definition of the product Poisson structure on  $\mathsf{G}\times\mathsf{G}$, we must have: 
\begin{equation}
\label{eqPtimesP}
(P\times P)(d\phi,d\psi)(g,h)=X^{(\phi,\psi)}_{(g,h)}=T_hL_g(X^\psi)+T_gR_h(X^\psi)=TL_p(P_h(d\phi))+TR_h( P_g(d\psi))
\end{equation}
${\bf m}$ is a Poisson morphism if and only if 
$$T_{(g,h)}{\bf m} (X^{(\phi,\psi)})=T_hL_g(X^\phi)+T_gR_h(X^\psi)$$
which ends the proof of the announced  equivalence.\\

The equivalence 
(2)$\Leftrightarrow $(3) is a consequence of Proposition \ref{P_CharacterizationPoissonMap}.\\

For the equivalence (1)$\Leftrightarrow $(4), we have only to prove the equivalence between ({\bf BiP}) and the  Relation (\ref{eq_biinvarianceLambda}) which is an easy consequence of the definition of $\Lambda$ from $P$ by  using (\ref{eq_BiinvariantLambda1}).
\end{proof}

\subsection{Multiplicative Admissible $2$-form and Lie Bialgebra}
\label{___Multiplicative1Cocycle} 

In this section, we fix a  bi-invariant  convenient subbundle $T^\flat \mathsf{G}$ of $T^\prime \mathsf{G}$ and we set $\mathfrak{g}^\flat=T_e^\flat\mathsf{G}\subset \mathfrak{g}^\prime$. The convenient dual of  $\mathfrak{g}^\flat $ will be denoted $\mathfrak{g}^{\flat\prime}$.\\

\subsubsection{Multiplicative admissible $2$-forms}

 According to \cite{CaPe23}, Definition~7.12,  for $k\geq 2$, {\it an admissible section $\Lambda \in\bigwedge^kT^{\flat \prime}\mathsf{G}$} is a section for which there exists a skew-symmetric morphism   $P:(T^{\flat }\mathsf{G})^{k-1}\to T\mathsf{G}$ such that
\[
\Lambda(\alpha_1,\dots,\alpha_k)=\alpha_1(P(\alpha_2,\dots, \alpha_k))
\]
 for all $(\alpha_1,\alpha_2,\dots,\alpha_k)\in (T^{\flat \prime}\mathsf{G})^{k}$.\\
For the sake of simplicity, such a section 
will be called \emph{ an admissible $k$-form}. Note that the partial Poisson tensor $\Lambda$ associated a partial Poisson anchor $P:T^\flat\mathsf{G}\to T\mathsf{G}$ is an $2$-admissible form.\\

According to Property  (\ref{eq_biinvarianceLambda}), as in finite dimension for  the notion of multiplicative tensor (cf. for example \cite{Dri83}, \cite{LuWe90}, \cite{Marl98} and \cite{Kos04}), we introduce:
\begin{definition}
\label{D_multipivative} 
An admissible $k$-form  $\Lambda$ is called \emph{multiplicative}\index{multiplicative $k$-form} if it satisfies:
\begin{equation}
\label{eq_biinvarianceLambda2}
\Lambda_{gh}=L_g^{**} \Lambda_h +R_h^{**}\Lambda_g
\end{equation}
for all $(g,h)\in \mathsf{G}^2$.
\end{definition}

\begin{remark}
\label{R_LambdaMultiplicative}${}$
\begin{enumerate}
\item[1.] 
If $\Lambda$ is multiplicative, for $h=g=e$, from the relation (\ref{eq_biinvarianceLambda2}) we obtain $\Lambda_e=2\Lambda_e$ and so $\Lambda _e=0$.
\item[2.]   
If $\Lambda_e=0$, for any $v\in T_e\mathsf{G}$,  choose any vector field on $\mathsf{G}$ such that $X(e)=v$. Since we have (cf. \cite{CaPe23}, Chap.~7)
\begin{equation}
\label{eq_LXOmega}
{L}_X\Lambda(\alpha_1,\dots,\alpha_k)(e)
=d_e\{\Lambda(\alpha_1,\dots,\alpha_k)\}(X)-\displaystyle \sum_{i=1}^k \Lambda_e(\alpha_1,\dots, L_X\alpha_i,\dots,\alpha_k)
\end{equation}
it follows that $\mathcal{L}_X\Lambda(\alpha_1,\dots,\alpha_k)(e)
=d_e\left(\Lambda(\alpha_1,\dots,\alpha_k)\right)(X)$ and so only depends on the $1$-jet of $\Lambda$ at point $e$ and the value $X(e)$. Thus, as in finite dimension, the differential $d\Lambda_e$ is intrinsically defined. 
\end{enumerate}
\end{remark}

As in finite dimension, (cf. \cite{Dri83}), we have the following properties:

\begin{proposition}
\label{P_MultiplicativeLambda}
Given any admissible $k$-form $\Lambda$, we associate to $\Lambda$  two bounded fields
$$\Lambda_L:\mathsf{G}\to \bigwedge^k\mathfrak{g}^{\flat\prime} \textrm{ and } \Lambda_R:\mathsf{G}\to \bigwedge^k\mathfrak{g}^{\flat\prime}$$
defined, for all $g,h\in \mathsf{G}$, by ${(\Lambda_L)}_g=L_{g^{-1}}^{**}\Lambda_g $ and ${(\Lambda_R)}_g=R_{g^{-1}}^{**}\Lambda_g$.
We then have the following equivalences:
\begin{enumerate}
\item[(i)] 
$\Lambda$ is multiplicative.
\item[(ii)] 
$\Lambda_L$ satisfies for all $g,h\in \mathsf{G}$:
\[
{(\Lambda_L)}_{gh}={(\Lambda_L)}_{h}+Ad_{h^{-1}}^{**}{(\Lambda_L)}_{g}.
\]
\item[(iii)]
$\Lambda_R$ satisfies for all $g,h\in \mathsf{G}$:
\[
{(\Lambda_R)}_{gh}=Ad_{g}^{**}{(\Lambda_R)}_{h}+{(\Lambda_R)}_{g}.
\] 
\end{enumerate}
\end{proposition}

\begin{remark} 
\label{R_ProprInFinite}
In finite dimension, when the Lie group $\mathsf{G}$ is connected, the following properties: 
\begin{enumerate}
\item[(iv)] 
$\Lambda_e=0$ and $\mathcal{L}_X\Lambda$  are left invariant whenever $X$ is a left invariant vector field.
\item[(v)] 
$\Lambda_e=0$ and $\mathcal{L}_X\Lambda$  are right invariant whenever $X$ is a right invariant vector field.
\end{enumerate}
are equivalent to properties (i), (ii) and (iii)
(cf.  \cite{Dri83} and \cite{Kos95}).\\
Properties (iv) or (v) are equivalent to any of  Properties (i), (ii) and (iii) in this Proposition for  connected  convenient lie group with an exponential maps for which $\mathsf{G}$ is generated by the range of the exponential map, which is in particular the case for  connected and simply connected Banach Lie groups.  Unfortunately, in the general case for convenient Lie groups, such properties are not true since  the previous assumption is not satisfied.
\end{remark}

\begin{proof}
According to Notations~\ref{Lg**}, for all $(g,h) \in \mathsf{G}^2$, we have
\[
\begin{array}{rcl}
{(\Lambda_L)}_{gh}-{(\Lambda_L)}_{h}-Ad_{h^{-1}}^{**}{(\Lambda_L)}_g
&=&
L_{(gh)^{-1}}^{**} \Lambda_{gh}-L_{h^{-1}}^{**} \Lambda_h-L_{h^{-1}}^{**}\circ R_h^{**}\circ L_{g^{-1}}^{**}\Lambda_g \\
	&=&L_{(gh)^{-1}}^{**}\left(\Lambda_{gh}-L_g^{**}\Lambda_h-R_h^{**}\Lambda_g\right)
\end{array}
\]
which implies the equivalence between (i) and (ii). \\
A formally comparable calculus shows the equivalence between (i) and (iii).
\end{proof}

\subsubsection{1-cocycles and  Lie Bialgebras}\label{___1-cocycle}

For this  paragraph, we refer to \cite{Tum20}, 2.5 and 4.1.  For the sake of completeness, we adapt these results to the convenient context.\\

Let $\mathsf{G}$ be a convenient Lie group and $\Phi: \mathsf{G}\to \operatorname{Aff}(\mathbb{E})$ a representation into the Affine group of a vector space $\mathbb{E}$.  As 
$\operatorname{Aff}(\mathbb{E})$ is a semi-direct product of the linear group $GL(V)$ and the translation group on $\mathbb{E}$ identified with $\mathbb{E}$,  $\Phi$ can be written as a pair $(\phi,\Theta)$ where $\phi:\mathsf{G}\to GL(\mathbb{E})$ and $\Theta:\mathsf{G}\to \mathbb{E}$. Since $\Phi$ is a group morphism it follows that $\phi$ is a group morphism and $\Theta$ satisfies
\begin{equation}
\label{eq_1CocylceGroup}
\Theta(gh)=\Theta(g)+\phi(g)\left(\Theta(h)\right)
\end{equation}
for all $g$ and $h$ in $\mathsf{G}$. In this case,  $\Theta$ is called \emph{${\bf 1}$-cocyle on $\mathsf{G}$ relative to $\phi$}. \\

The derivative $T_e\Phi $ is a Lie algebra morphism from $\mathfrak{g}$ to the Lie algebra $\mathfrak{aff}(\mathbb{E})$ of $ \operatorname{Aff}(\mathbb{E})$  which can be written as a pair $(T_e\Phi, T_e\Theta)$ where $T_e\phi$ and $T_e\Theta$ gives rise to a Lie algebra morphism $d\phi$ from $\mathfrak{g}$ to $\mathfrak{gl}(\mathbb{E})$ and $d\Theta$ from $\mathfrak{g}$ to $\mathbb{E}$ respectively which satisfies the relation

\begin{equation}
\label{eq_1CocycleLieAlgebra}
d\Theta([x,y])=d\phi(x)\left(d\Theta(y)\right)-d\phi(y)\left(d\Theta(x)\right)
\end{equation}
for all $x$ and $y$ in $\mathfrak{g}$. In this case, $d\Theta$ is called {\bf ${\bf 1}$-cocyle on $\mathfrak{g}$ relative to $d\phi$}. 

\begin{example}
\label{Ex_LambdaRCocycle}  
Consider a Lie group  $\mathsf{G}$ and a bi-invariant subbundle $T^\flat \mathsf{G}$  of $T\mathsf{G}$. If we denote by $\mathfrak{g}^\flat
$ the fibre $T_e^\flat\mathsf{G}$. Then $\mathfrak{g}^\flat$ is invariant by the coadjoint action $\mathsf{G}$ on $T^\prime \mathsf{G}$  (formally, same proof as in \cite{Tum20}, Proposition~5.6). Let $\Lambda$ be an admissible multiplicative $ 2$-form and $\Lambda_R:\mathsf{G}\to \bigwedge^2\mathfrak{g}^\flat$ given by
\[
(\Lambda_R)_g=R_{g^{-1}}^{**}\Lambda_g
\] 
 for all $g\in \mathsf{G}$.\\
Then $\Lambda_R$ is a $1$-cocycle on $\mathsf{G}$ with respect to the coadjoint action $Ad^{**}$ on  the vector space $\bigwedge^2\mathfrak{g}^\flat$  that is
 \begin{equation}
 \label{eq_1CocyleLambdaR}
(\Lambda_R)_{gh}=(\Lambda_R)_g+Ad^{**}(g)(\Lambda_R)_h
\end{equation}
for all $g,h$ in $\mathsf{G}$
(Formally again, the proof is the same as for Proposition 5.7  Assertion (2) in \cite{Tum20}). 
Therefore, we have 
\begin{equation}
\label{eq_AdCocycleAlgebra}
T_e\Lambda_R([x,y])
=\operatorname{ad}^{**}_x(T_e\Lambda_R(y))-\operatorname{ad}^{**}_y(T_e\Lambda_R(x))
\end{equation}
which means exactly:
\begin{equation}
\label{eq_AdCocycleAlgebra2}
\begin{array}{rcl}
T_e \Lambda_R([x,y])(\alpha,\beta)
   &=& T_e\Lambda_R(y)(\operatorname{ad}^{*}_x\alpha,\beta)+ T_e\Lambda_R(y)(\alpha,\operatorname{ad}^{*}_x\beta)\\
   &~~~~~& -T_e\Lambda_R(x)(\operatorname{ad}^{*}_y\alpha,\beta)- T_e\Lambda_R(x)(\alpha,\operatorname{ad}^{*}_y\beta).
\end{array}
\end{equation}
\end{example}

This example justifies the following  context and Definition.\\`

 Consider a convenient Lie algebra $(\mathfrak{g},[.,.])$, a convenient subspace  $\mathfrak{g}^\flat$ of $\mathfrak{g}^\prime$ whose inclusion is bounded and the coadjoint action  in restriction to $\mathfrak{g}\times \mathfrak{g}^\flat$  which  takes values in $\mathfrak{g}^\flat$ and is bounded.\\

For any bounded map bilinear map  $\Lambda\in  \bigwedge^2\mathfrak{g}^\flat$ we say that \emph{ $\Lambda$ is admissible} (cf. \cite{CaPe23}, Definition~7.12) if there exists a map  $P:=\Lambda^\flat: \mathfrak{g}^\flat \to \mathfrak{g}$ given by $\Lambda^\flat(\alpha)=\Lambda(., \alpha)$. Note that this map is a bounded linear map.
Given such an admissible $\Lambda$ we can define a natural coadjoint action on $\bigwedge^2\mathfrak{g}^{\flat\prime}$ defined by:
\begin{equation}
\label{eq_Ad20}
\operatorname{ad}^{(2,0)}_x\Lambda(\alpha,\beta)=\Lambda(\operatorname{ad}^*_x\alpha,\beta)+ \Lambda(\alpha,\operatorname{ad}^*_x\beta)
\end{equation}
for all $\Lambda\in \bigwedge^2\mathfrak{g}^{\flat\prime}$.\\

\begin{definition}
\label{D_1CocylcleRelAd} 
Under the previous assumptions, a map $\lambda: \mathfrak{g}\to  \bigwedge^2\mathfrak{g}^\flat$   is called  a $1$-cocycle relative to the coadjoint action of $\mathfrak{g}$ onto $\bigwedge^2\mathfrak{g}^{\flat\prime}$ if, for all $x\in \mathfrak{g}$ and $\alpha, \beta$ in $\mathfrak{g}^\flat$,
\begin{enumerate}
\item[(i)]  
there exists $\mathsf{p}:\mathfrak{g}\times\mathfrak{g}^\flat \to \mathfrak{g}$ such that
$\lambda(x)(\alpha,\beta)=<\alpha, \mathsf{p}(x)(\beta)>$;
\item[(ii)]
we have the following relation: 
\begin{equation}
\label{eq_1Cocycle}
\lambda([x,y])=\operatorname{ad}^{(2,0)}_x(\lambda(y))-\operatorname{ad}^{(2,0)}_y(\lambda(x)).
\end{equation}
\end{enumerate}
\end{definition}

\emph{Note that the relation (\ref{eq_AdCocycleAlgebra2}) implies that  the map $\lambda(x)=T_e\Lambda_R(x)$ is a $1$-cocycle relative to the coadjoint action of $\mathfrak{g}$ onto $\bigwedge^2\mathfrak{g}^{\flat\prime}$}.\\

The notion of Lie bialgebra in finite dimension\footnote{See \cite{Kos04}.} appears as a particular case of  Lie bialgebroid (see \cite{MaXu94} for instance). A generalization to the Banach setting can be found in \cite{Tum20}. We propose the following  natural adaptation:
 
\begin{definition}
\label{D_PartialLieBialgebra} 
Let $(\mathfrak{g},[.,.])$ be a convenient Lie algebra and $\mathfrak{g}^\flat$ be a  convenient space of $\mathfrak{g}^\prime$ whose inclusion is bounded. We will say that $(\mathfrak{g},\mathfrak{g}^\flat)$ is a \emph{partial Lie bialgebra}\index{partial!Lie bialgebra} if:
\begin{enumerate}
\item[(i)]  
the restriction  coadjoint action 
 to $\mathfrak{g}\times \mathfrak{g}^\flat$ takes values in $\mathfrak{g}^\flat$ and is bounded;
\item[(ii)]
there exists a Lie algebra bracket $[.,.]^\flat$ on $\mathfrak{g}^\flat$ such that  the dual map of $[.,.]^\flat:\mathfrak{g}^\flat\times\mathfrak{g}^\flat\to \mathfrak{g}^\flat$ restricts to a $1$-cocycle $\lambda:\mathfrak{g}\to \bigwedge^2\mathfrak{g}^{\flat\prime}$ relative to the coadjoint action of $\mathfrak{g}$ onto $\mathfrak{g}^\flat$.
\end{enumerate}
\end{definition} 
 
\begin{remark}
\label{R_CocycleCondition} 
According to \cite{Tum20}, Remark 7.3, by formal  same arguments adapted to the convenient setting, the $1$-cocycle condition can be written:
\begin{equation}
\label{eq_1Cocycle2}
\begin{matrix}
<[X,Y],[\alpha,\beta]^\flat>=&<Y,[\operatorname{ad}_X^*\alpha,\beta]^\flat>+<Y,[\alpha,\operatorname{ad}_X^*\beta]^\flat>\hfill{}\\
					&-<X,[\operatorname{ad}_Y^*\alpha,\beta]^\flat>-<X,[\alpha,\operatorname{ad}_Y\beta]^\flat>.\hfill{}\\
\end{matrix}
\end{equation}
\end{remark}

\subsubsection{Lie Bialgebra Associated to a Multiplicative $2$-form}

The following result gives a link between  admissible  multiplicative  $2$-forms on a Lie group $\mathsf{G}$  and $1$-cocycles:

\begin{theorem}
\label{T_CharacPartialBialgebra} 
Let  $\Lambda $ be an admissible  multiplicative  $2$-form  and  the associated  skew-symmetic morphism   $P:T^{\flat }\mathsf{G}\to T\mathsf{G}$. Assume that $P_e$ is injective. 
Then $[\alpha, \beta]_P= d_e\Lambda(\alpha,\beta)$ defines an almost bracket on $\mathfrak{g}^\flat$ which is a Lie bracket if and only if 
$ \left( T^\flat \mathsf{G}, T\mathsf{G}, P,\{\;,\;\}_P \right) $ is a convenient partial Poisson  Lie group. In this case, 
$ \left( T\mathsf{G}, T^\flat\mathsf{G} \right) $ is a partial Poisson partial bialgebroid and $d_e\Lambda$ is a $1$-cocycle on $\mathfrak{g}$ relative to the $\operatorname{ad}^{**}$ action on $\bigwedge^2\mathfrak{g}^\flat$.
\end{theorem}

\begin{proof} 
At first, recall that  the bi-invariance of $T^\flat\mathsf{G}$ implies that $T^\flat\mathsf{G}\equiv\mathsf{G}\times\mathfrak{g}^\flat$ \textit{via} right or left translation on $\mathsf{G}$.\\ 
In the one hand, from  the right invariance of  $T^\flat\mathsf{G}$, the map $P_R:\mathsf{G}\times\mathfrak{g}^\flat\to TG$  where $(P_R)_g:=P\circ R_{g^{-1}}$ is well defined.\\
On the other hand,  from the definition of $\Lambda_R$ from $\Lambda$ (cf. Proposition \ref{P_MultiplicativeLambda}), we have
 $$(\Lambda_R)_g(\alpha,\beta)=<\alpha,( P_R)_g\beta>$$
 for all $\alpha, \beta\in \mathfrak{g}^\flat$.
 
We will need  the following result:
\begin{lemma}
\label{L_LambdaR}
Under the assumptions of Theorem~\ref{T_CharacPartialBialgebra},
for any $g\in \mathsf{G}$, we have:
$P_e\equiv 0$ and 
$$d\left(\Lambda(\alpha,\beta)\right)(X_e)=T_e\Lambda_R(\alpha,\beta)(X_e)=<\alpha, T_e P_R(X_e)(\beta)>$$
for any local sections  $\alpha, \beta$ of $T_e^\flat G$ around $e$ and any $X_e\in T_e\mathsf{G}$.\\
\end{lemma}

\begin{proof}[Proof of Lemma \ref{L_LambdaR}] 
Fix some neighbourhood $V$ of $g$ in $\mathsf{G}$. For all $h\in V$ we have

  $\Lambda_h(\alpha,\beta)=(\Lambda_R)(R^*_{h^{-1}}\alpha, R^*_{h^{-1}}\beta)=<R^*_{h^{-1}}\alpha, (P_R)_h(\beta)>$.
\\
Since $\Lambda_R$ is a map from $\mathsf{G}$ to $ \bigwedge^2\mathfrak{g}^\flat$, it follows that we have:

\begin{equation}
\label{eq_dLambda}
\begin{array}{rcl}
d\left(\Lambda(\alpha,\beta)\right)(X_g)
  &=&<T_g\widetilde{\alpha}(X_g), (P_R)_g(\beta)>+<\alpha, T_g P_R(X_g)(\beta)\\
  &~~~~~&+ <\alpha, P_R(T_g \widetilde{\beta}(X_g))>
\end{array}
\end{equation}
where $\widetilde{\alpha}(h)=R_{h^{-1}}\alpha(h)$ and $\widetilde{\beta}(h)=R_{h^{-1}}\beta(h)$,  for all $h\in V$.\\
But  we have $\Lambda_e(\alpha, \beta)=<\alpha, P_e\beta>=0$, for all $\alpha$ and $\beta$ in $\mathfrak{g}^\flat$.  According to section 
\ref{__PartialLinearPoissonConvenientSpaces}, this implies that  $(\mathfrak{g}^\flat)^{\perp_{P_e}}=\{0\}$ and so $P_e(\beta)$ belongs to $P_e(\mathfrak{g}^\flat)^{\perp_{P_e}}=0$. Since $P_e$ is injective, it follows that $P_e(\beta)=0$ for all  $\beta\in \mathfrak{g}^\flat$. Thus from (\ref{eq_dLambda}), since  we have $(P_R)_e=P_e$ we obtain:
\begin{equation}
\label{eq_dLmabda1}
d\left(\Lambda(\alpha,\beta)\right)(X_g)=<\alpha, T_g P_R(X_g)(\beta)>.
\end{equation}
In particular, for $g=e$, we have $(P_R)_e=P_e=0$ and so we get:
\begin{equation}
\label{eq_dLmabda1}
d\left(\Lambda(\alpha,\beta)\right)(X_e)=<\alpha, T_e P_R(X_e)(\beta)>.
\end{equation}
But we can note that since $(\Lambda_R)_e=\Lambda_e=0$, the second member of (\ref{eq_dLmabda1}) is nothing but that $T_e\Lambda_R(\alpha,\beta)(X_e)$,
which ends the proof of Lemma  \ref{L_LambdaR}.
\end{proof}

Let $\mathfrak{A}(\mathsf{G}):=\{ \phi\in C^\infty(\mathsf{G}),\; d_g\phi\in T_g^\flat \mathsf{G}\}$. Then $\mathfrak{g}^\flat$ can be identified with the set
\[
\{\phi_\alpha\in \mathfrak{A}(\mathsf{G}),\; \phi_\alpha(e)=0, \;d_g\phi_\alpha=R_{g^{-1}}^*\alpha,>,\; \forall \alpha\in \mathfrak{g}^\flat,\; \forall g\in \mathsf{G}\}.
\]
Since $P$ is an almost Poisson anchor, we have an almost Lie Poisson bracket $\{\phi,\psi\}_P$ on  $\mathfrak{A}(\mathsf{G})$ given by
\[
\{\phi,\psi\}_P=<d\psi, P(d\phi)>
\]
which is  Poisson bracket if and only if $P$ is a Poisson anchor (cf Proposition~\ref{P_PropertiesPartialPoissonManifold} and Remark~\ref{R_About1}).\\ 
Now, from Remark \ref{R_About1}, $\{\phi,\psi\}_P$ induces a an almost Lie bracket $[\alpha,\beta]_P=d\{\phi_\alpha,\phi_\beta\}$ on $\mathfrak{g}^\flat$ which satisfies  the Jacobi identity.\\
From the definition of  $\Lambda_R$ as a map from $G$ to $\bigwedge^2\mathfrak{g}^\flat$ it follows that $T_e\Lambda $ is map from $\mathfrak{g}$ to $\bigwedge^2\mathfrak{g}^\flat$. Therefore,  the almost  bracket on $\mathfrak{g}$ satisfies the relations:
\[
[\alpha,\beta]_P=d\{\phi_\alpha,\phi_\beta\}=d_g\Lambda(R_{g^{-1}}^*\alpha, R_{g^{-1}}^*\beta)=T_e\Lambda_R(\alpha,\beta).
\]
Lemma \ref{L_LambdaR} implies that  $[\alpha, \beta]_P= d_e\Lambda(\alpha,\beta)$.\\ 
Now, if $P$ is a Poisson anchor, this implies that the bracket on $\mathfrak{A}$ satisfies the Jacobi identity  and from Proposition \ref{P_PropertiesPartialPoissonManifold} the almost bracket $[.,.]_P$ also satisfies the Jacobi identity since $P$ is injective by assumption.
Therefore, the almost Lie bracket induced on $\mathfrak{g}$ also satisfies the Jacobi identity.\\

Conversely, is the almost bracket on $\mathfrak{g}$ satisfies the Jacobi identity.\\ 
At first note that from  Remark \ref{R_About1}, we have an almost Lie bracket $[.,.]_P$  on  $T^\flat\mathfrak{g}$ which satisfies:
\[
[d\phi,d\psi]_P=d\Lambda(d\phi, d\psi).
\]
For any function $\phi_1,\phi_2,\phi_3$ in $\mathfrak{A}$, we set: 
\[
J_P(\phi_1,\phi_2,\phi_3)=\{\phi_1\{\phi_2,\phi_1\}_P\}_P+\{\phi_2\{\phi_3,\phi_1\}_P\}_P+\{\phi_3\{\phi_1,\phi_2\}_P\}_P.
\]
Note that, for any $g\in \mathsf{G}$ and any $\phi_i$ there exists $\alpha_i\in \mathfrak{g}^\flat$ such that  $d_g\phi_i=R_g^*\alpha_i$ for $i \in \{1,2,3\}$.\\
With all these properties we have sucessively:
\[
\begin{matrix}
d_g\{\phi_2,\phi_1\}_P&=d_g\Lambda(d\phi_2, d\phi_3)\hfill{}\\
			&=d_g\Lambda(R_g^*\alpha_2,R_g^* \alpha_3)\hfill{}\\
			&=d_e\Lambda_R(\alpha_2,\alpha_3)\hfill{}\\
			&=[\alpha_2,\alpha_3]_P.\hfill{}
\end{matrix}
\]
So we obtain
\[
d_g\left(\{\phi_1\{\phi_2,\phi_1\}_P\}_P\right)=d_e\Lambda_R(\alpha_1,[\alpha_2,\alpha_3]_P)=[\alpha_1[\alpha_2,\alpha_3]_P]_P.
\]
It follows that $d_gJ_P(\phi_1,\phi_2,\phi_3)$ is the Jacobi expression of the bracket on $\mathfrak{g}^\flat$ for $(\alpha_1,\alpha_2,\alpha_3)$.  Therefore
\[
d_gJ_P(\phi_1,\phi_2,\phi_3)=0
\]
for all $g\in \mathsf{G}$ and $\phi_1,\phi_2, \phi_3$ in $\mathfrak{A}$. By the way, $J_P$ must be must be constant and this value does not depend on the choice $(\phi_1,\phi_2,\phi_3)$ in particular of their values at $g$.  Thus $J_P$ must be identically zero, which ends the announced equivalence.\\

Now if $P$ is a Poisson anchor, from Proposition \ref{L_dPLambdaLambda},  $(T\mathsf{G}, T^\flat\mathsf{G})$ is a partial Poisson partial bialgebroid.\\
It remains to show that $d_e\Lambda$ is a $1$-cocycle on $\mathfrak{g}$ relative to the $\operatorname{ad}^{**}$ action on $\bigwedge^2\mathfrak{g}^\flat$.\\
 From Lemma \ref{L_LambdaR}, we have
\[   
[\alpha, \beta]_P= d_e\Lambda(\alpha,\beta)=T_e\Lambda_R(\alpha,\beta).
\]
The last assertion is then a direct consequence of  Example~\ref{Ex_LambdaRCocycle}.
\end{proof}

\subsubsection{Partial Bialgebras  and Manin triples}
\label{___ManinTriple}
In \cite{Tum20},  the reader can find a  definition of a Manin triple  in the Banach setting and the link with $1$-cocycles and bialgebras.  We can note that, in this context,  a Manin triple is a triple of Banach Lie algebras $(\mathfrak{g},\mathfrak{g}_+, \mathfrak{g}_-)$ such that  $\mathfrak{g}=\mathfrak{g}_+\oplus \mathfrak{g}^-$ and $\mathfrak{g}_+$ and $\mathfrak{g}_-$ are in non degenerated-duality. Such a context is also true for the notion of Banach bialgebra used in \cite{Tum20} and we have a kind of equivalence between both notions.\\
Unfortunately, in our setting of  partial Poisson  manifold, for the  associated notion of partial bialgebra, we have a very weak notion of duality between $\mathfrak{g}$ and $\mathfrak{g}^\flat$ since the bounded linear form $\mathfrak{g}\times \mathfrak{g}^\flat$ is only non degenerate on the first argument but it is in general degenerate on the second one.  It could be possible to  define a  notion of "partial" Manin triple, but we have no very original examples in opposition of the rich context exposed in \cite{Tum20}.

\subsection{On Direct Limits of  (partial) Banach Poisson-Lie Groups} 
\label{___OnDirectLimitsOfPartialBanachPoissonLieGroupsAndExamples}

Any Banach Poisson-Lie group in the sense of \cite{Tum20} is in particular a partial (convenient) Poisson-Lie group. In \cite{Tum20}, the author points out many  interesting examples of Banach Poisson-Lie  group  and so they   are also partial convenient Poisson-Lie groups. We do not list here all these examples. Interested readers can directly consult this paper.
 
In this paragraph, direct limits  of ascending sequences of  (partial) Banach Poisson Lie groups will give examples of convenient Poisson-Lie groups which {\bf are not  Banach Poisson-Lie groups}.
 
\subsubsection{Direct Limits of some Ascending Sequences of Banach (Partial) Poisson-Lie Groups}\label{___DirectLimitBanachLieGroups} 
 
At first, we need to recall some results on Direct limit  of ascending sequences of Banach Poisson manifolds and Banach Lie groups. For more details see \cite{CaPe23}, 7.6  and 5.8.\\ 
 
Let  
$ \left( M_{i},\mathfrak{A}(M_i),\{.,.\}_{P_{i}} \right) _{i\in \N}$ be  a sequence of partial Poisson Banach manifolds where
$p_{i}^{\flat}:T^{\flat}M_{i}\to M_{i}$ is a Banach subbundle of
$p_{M_{i}}^{\ast}:T^{\ast}M_{i}\to M_{i}$ and $P_{i}:T^{\flat}M_{i}\to TM_{i}$ is the associated Poisson anchor. We denote by
$\mathbb{M}_{i}$ the Banach space on which $M_{i}$ is modelled, and by $\mathbb{F}_{i}$ the model of the typical fibre of the subbundle $p_{i}^{\flat}:T^{\flat}M_{i}\to M_{i}$ and we assume that $\mathbb{F}_{i}$ is a Banach
subspace of the dual $\mathbb{M}_{i}^{\ast}$ of $\mathbb{M}_{i}$. For any $i\in \N$, we denote by $\varepsilon_i^{i+1}:M_i\to M_{i+1}$ the inclusion map.

\begin{definition}
\label{D_DirectSequenceOfPartialPoissonBanachManifolds}
The sequence 
$ \left( M_{i},\mathfrak{A}(M_i),\{.,.\}_{P_{i}}\right) _{i \in \N}$ is called a \emph{direct sequence of partial Poisson manifolds}\index{direct sequence!of partial Poisson Banach manifolds} if 
\begin{description}
\item[\textbf{(DSpPM1)}]
$ \left( M_{i} \right) _{i \in \N}$ is an ascending sequence of Banach $C^{\infty }$-manifolds, where $M_{i}$ is modelled on the Banach space $\mathbb{M}_{i}$ such that $\mathbb{M}_{i}$ is a supplemented Banach subspace of $\mathbb{M}_{i+1}$ and such that $M_{i}$ is a closed submanifold of $M_{i+1}$; 
\item[\textbf{(DSpPM2)}]
for  each $i\in \mathbb{N}$, the map  
\[
\varepsilon_i^{i+1}:
 \left( T^\flat M_i, TM_i,  P_i,\{\;,,\;\}_{P_i} 
 \right) 
\to 
 \left(  T^\flat M_{i,+1} TM_{i+1},  P_{i+1}i,\{\;,,\;\}_{P_{i+1}} 
 \right) 
\]
is a Poisson map; 
\item[\textbf{(DSpPM3)}]
around each $x\in M$, there exists an ascending  sequence of charts 
$ \left( U_{i},\phi_{i} \right) _{i\in\N}$ such that we have the following commutative  diagram:
\[
\xymatrix{
T^\flat U_i \ar[d]  &T^\flat U_{i+1}\ar[d]\ar[l]_{T^*\varepsilon_i^{i+1}} \\
U_i \ar[r]_{\varepsilon_i^{i+1} }         &U_{i+1}}
\]
\end{description}
\end{definition}

\begin{remark}
\label{R_PropertyOfDirectLimitFibre}${}$
\begin{enumerate}
\item[1.]Since 
$ \left( M_i,\varepsilon_i^{i+1} \right) _{i \in \N}$ is an ascending  sequence of Banach manifolds, we may assume that $\mathbb{M}_{i}$ is a Banach subspace of $\mathbb{M}_{i+1}$. By abuse of notation, we continue to denote by $\varepsilon_{i}^{i+1}$ the natural inclusion of $\mathbb{M}_{i}$ in $\mathbb{M}_{i+1}$ and  $(\varepsilon_{i}^{i+1})^{\ast}:\mathbb{M}_{i+1}^{\ast}\to\mathbb{M}_{i}^{\ast}$ 
the adjoint operator. 
Note that $(\varepsilon_{i}^{i+1})^{\ast}$ is surjective. Since
$\mathbb{F}_{i}$ is the typical fibre of $T^{\flat}M_{i}$,  
we must have 
$(\varepsilon_{i}^{i+1})^{\ast}(\mathbb{F}_{i+1})=\mathbb{F}_{i}$  
and  each $\mathbb{F}_i$ is a Banach subspace of $\mathbb{M}_I^\prime$.
\item[2.] 
In fact, $\textbf{(DSpPM2)}$ implies  ${P_{i+1}}_{| T^\flat M_I}=P_i$.
\end{enumerate}
\end{remark}

At first, we have: 

\begin{proposition}
\label{P_ConvenientTangentBundleFrechetCotangentBundleBundleMorphismDirectSequence}
Let $ \left( M_{i},\mathfrak{A}(M_i),\{.,.\}_{P_{i}}\right) _{i\in\N}$ be a direct sequence of partial Poisson Banach manifolds and let $M$ be the direct limit of the sequence $ \left( M_i \right) _{i \in \N} $. We then have:
\begin{enumerate}
\item
$p_{M}:TM\to M$ is a convenient bundle which is the kinematic bundle of $M$.
\item
$p'_M:T'M\to M$ and $p^{\flat}:T^{\flat}M\to M$
are locally trivial convenient bundles over $M$. Moreover $p'_M:T'M \to M$ is the kinematic cotangent bundle of $M$.
\item
There exists a canonical bundle morphism $P:T^{\flat}M\to TM$ whose typical fibre  $\underleftarrow{\lim}\mathbb{F}_i$ if $\mathbb{F}_i$ is the typical fibre of $T^\flat M_i$
characterized, for $x_i\in M_{i}$, by
\begin{equation}
\label{eq_LinkPiPi+1}
{P}(x_i,\xi)=P(x_i,\underleftarrow{\lim}_{k\geq i}(\xi_{k}))=\underrightarrow
{\lim}_{i \leq k}(P_{k}(x_i,\xi_{k})).
\end{equation}
Moreover, $P$ is skew-symmetric, relatively to the canonical dual pairing between $T'M$ and $TM$ in restriction to $T^{\flat}M\times TM$.
\end{enumerate}
\end{proposition}

Finally we have:

\begin{theorem}
\label{T_PartialPoissonStructureOnDirectLimitOfPoissonBanachManifolds}
Let $\left( M_{i},\mathfrak{A}(M_{i}),\{.,.\}_{P_{i}} \right) _{i\in\N}$ be a direct sequence  of partial Poisson Banach manifolds and
$M=\underrightarrow{\lim}(M_{i})$.\\
There exists a weak subbundle $p^{\flat}:T^{\flat}M\to M$ of $p'_M:T'M\to M$ and a skew-symmetric morphism $P:T^{\flat
}M\to TM$ such that $ \left( M,\mathfrak{A}_P(M),\{.,.\}_{P} \right) $ is a partial
Poisson manifold with the following characterization:\\
If ${\varepsilon}_{i}:M_{i}\to M$ is the canonical injection then ${\varepsilon}_{i}$ is a Poisson morphism from $ \left( M_{i},\mathfrak{A}(M_{i}),\{.,.\}_{P_{i}} \right) $ to $ \left( M,\mathfrak{A}(M),\{.,.\}_{P} \right) $ for all $i\in\mathbb{N}$, and we have
\[
\mathfrak{A}_P(M)
=\underrightarrow{\lim}\mathfrak{A}(M_i)
\]
\[
\{.,.\}_{P}=\underrightarrow{\lim}\{.,.\}_{P_{i}}.
\]
\end{theorem}

Therefore $ \left( T^\flat M, M, P, \{.,.\}_P \right) $ is {\bf   not \textit{a priori}} a partial Poisson manifold since its Poisson bracket is  defined only on the  sheaf of subalgebras $\mathfrak{A}_P(M)=\underrightarrow{\lim}\mathfrak{A}(M_i)$ of sheaf $\mathcal{A}_M$ of smooth functions associated  to  the convenient bundle $T^\flat M$. However:

\begin{lemma} 
Let $ \left( U_i, \Phi_i \right) _{i \in \N} $ be an ascending sequence of charts. Set $U=\underrightarrow{\lim}U_i, \Phi=\underrightarrow{\lim}\Phi$. 
Then, for any smooth function $f:U\to M$, if $f_i=f_{| U_i}$, then $f=\underrightarrow{\lim}f_i$.
\end{lemma}
\begin{proof} 
Recall that if $M_i$ is modelled on the Banach space $\mathbb{M}_i$, we may assume that $\mathbb{M}_i$ is a supplemented subspace  of $\mathbb{M}_{i+1}$ and let $\mathbb{M}=\underrightarrow{\lim} \mathbb{M}_i$ which is a convenient space. Without loss of generality, we may assume that  $U_i$ is a $c^\infty$-open set of $\mathbb{M}_i$.Then from \cite{CaPe23}, Lemma~5.2,    we must have $f=\underrightarrow{\lim}f_i$.
\end{proof}

It follows that 
$ \left( T^\flat M, M, P, \{.,.\}_P \right) $ is {\bf really} a partial Poisson manifold.\\

We will apply these results to the following context of Banach (partial) Poisson-Lie groups.\\ Note that the direct limit of an ascending sequence of Banach manifolds can be can been as a non Hausdorff manifold. But an ascending sequence of finite dimensional manifolds is always a Haussdorff manifold (\cite{Glo03}).\\
However for an ascending sequence
\[
\mathsf{G}_1\subset \cdots\subset \mathsf{G}_i\subset\mathsf{G}_{i+1}\subset\cdots
\]
where $\mathsf{G}_i$ is a Banach Lie subgroup of $\mathsf{G}_{i+1}$, the associated sequence of Lie algebras
\[
\mathfrak{g}_1\subset \cdots\subset \mathfrak{g}_i\subset\mathfrak{g}_{i+1}\subset\cdots
\]
is a strong ascending sequence of  closed Banach spaces and so the direct limit $\mathfrak{g}=\underrightarrow {\lim}\mathfrak{g}_i$ is a Hausdorff locally convex space and so a Hausdorff convenient space. Since $C_{\mathbb{K}}^\omega$
 differentiability implies convenient smoothness,   by \cite{Dah11}, Theorem~4.11, we obtain:
\begin{theorem}\label{T_DirecLlimitBanachLiegroups} 
Consider an ascending sequence $\mathsf{G}_1\subset \cdots\subset \mathsf{G}_i\subset\mathsf{G}_{i+1}\subset\cdots$ of 
Banach Lie groups such that $\mathsf{G}_i$ is a Banach Lie subgroup of $\mathsf{G}_{i+1}$. Assume that, for each $i\in\mathbb{N}$, there exists a norm $||.||_i$ on $\mathfrak{g}_i$ 
such that $||[X,Y]||_i\leq ||X||_i \;||Y||_i$ and  the norm of  the morphism $\epsilon_i^{i+1}:\mathfrak{g}_i \to \mathfrak{g}_{i+1}$ canonically  associated to the inclusion $\varepsilon_i^{i+1}$ of $\mathsf{G}_i$ in $\mathsf{G}_{i+1}$ is  at most $1$.\\
Then $\mathsf{G}=\underrightarrow {\lim}\mathsf{G}_i$ is a convenient Lie group whose Lie algebra is $\mathfrak{g}$.  Moreover, 
$\exp_G:=\underrightarrow {\lim}\exp_{\mathsf{G}_i}$ is a convenient  local diffeomorphism  on  a $c^\infty$-open set of $0$.
\end{theorem}

Note that this result is especially true for sequences of finite dimensional Lie groups (cf. \cite{Glo03}).\\

In a more general context, as  Corollary of Theorem~\ref{T_DirecLlimitBanachLiegroups}, we have the following result  
 (cf. \cite{CaPe23}, 5.8):
\begin{corollary}
\label{C_DirectLimitOfLinearGroupsOnAscendingSequencesOfSupplementedBanachSubspaces}
Let $(\mathbb{E}_{n})_{n\in\N}$ be an ascending sequence of Banach spaces such that $\mathbb{E}_n$ is a supplemented Banach subspace of $\mathbb{E}_{n+1}$. \\
Then $\mathbb{E}=\bigcup\limits_{n\in\N}\mathbb{E}_{n}$ is an \textrm{LB}-space and $\mathcal{L}(\mathbb{E})=\bigcup\limits_{n\in\N}\mathcal{L}(\mathbb{E}_{n})$ is also an \textrm{LB}-space, where $\mathcal{L}(\mathbb{E}_{n})$ is the Banach space of continuous linear operators of $\mathbb{E}_{n}$. Moreover, $\mathcal{GL}(\mathbb{E})=\bigcup\limits_{n\in\mathbb{N}}\mathcal{GL}(\mathbb{E}_{n})$ has a structure of convenient Lie group modelled on $\mathcal{L}(\mathbb{E})$, where $\mathcal{GL}(\mathbb{E}_{n})$ is the Banach Lie group of linear continuous automorphisms of $\mathbb{E}_{n}$.
\end{corollary}

\subsubsection{Ascending Sequences of Banach (Partial) Poisson-Lie Groups}
\label{____AscendingSequencesOfBanachPartialPoissonLieGroup}

We have seen that any Banach Poisson-Lie group is in particular a Banach partial Poisson-Lie group, so we only consider sequences of Banach partial Poisson-Lie groups. Thus naturally, an  ascending sequence of Banach (partial) Poisson Lie groups  $ \left( T^\flat\mathsf{G}_i, \mathsf{G}_i, P_i \right)_{i\in \N}$  is an ascending sequence of Banach  (partial) Poisson manifolds. 
 
\begin{theorem}\label{T_DirectLimitPartialPoissonGroups} 
Consider an ascending sequence of Banach (partial) Poisson Lie groups 
$\{(T^\flat\mathsf{G}_i, \mathsf{G}_i, P_i)\}_{i\in \N}$. Assume that the ascending sequence of groups $ \left( \mathsf{G}_i \right) _{i\in \N}$ satisfies the assumption of Theorem~\ref{T_DirecLlimitBanachLiegroups}.\\
Then the direct limit  
$\mathsf{G}=\underrightarrow {\lim}\mathsf{G}_i$ can be endowed in fact with a (partial) Poisson-Lie group structure. 
\end{theorem}
 
\begin{proof} 
It remains to show that the partial Poisson manifold $ \left( T^\flat \mathsf{G}, \mathsf{G}, P \right) $ is a (partial) Poisson Lie group.  
 
In complement to the conclusion of Theorem \ref{T_DirecLlimitBanachLiegroups}, we have the following results:\\
 
Assume that $g=\underrightarrow {\lim}g_i$. By using classical results on direct limits and maps and their differentials, we have the following relations and results :\\
  
The left (resp. right) translation $L_g$ (resp. $R_g$)  is $L_g=\underrightarrow {\lim}L_{g_i}$ (resp. $R_g=\underrightarrow {\lim}R_{g_i})$    and analog result for $\mathsf{Ad}_g$ on $\mathsf{G}$. \\
   
Since any $x\in \mathfrak{g}$ can be written $x=\underrightarrow {\lim}x_i$ with $x_i\in \mathfrak{g}_i$,  one has  $\mathsf{ad}_x=\underrightarrow {\lim}\mathsf{ad}_{x_i}$ and $\mathsf{ad}_x^* \underleftarrow {\lim} \mathsf{ad}_{x_i}^*$.\\

If we set $\mathfrak{g}^\flat=T_e^\flat\mathsf{G}$, we have $\mathfrak{g}^\flat= \underleftarrow {\lim}\mathfrak{g}_i^\flat$ (cf. \cite{CaPe23} p.~351).    \\  
  
The bi-invariance of each $T^\flat\mathsf{G}_i$ implies the bi-invariance of $T^\flat \mathsf {G}$.\\

From Remark~\ref{R_PropertyOfDirectLimitFibre} 2.,  the bi-invariance or each $P_i$  implies  the bi-invariance of $P$.\\

This completes the proof for direct limits of sequences of  partial Banach Poisson-Lie groups.
\end{proof}

\subsubsection{Some Examples of Convenient Partial  Poisson Lie Groups}\label{___ConvenientNotBanachPoissonLiegroups}

\begin{example}
\label{ex_UnC}
{\sf Structure of Poisson-Lie group on $\mathsf{U}(n)$.}\\ 
We first recall how to build a structure of Poisson-Lie group on $\mathsf{U}(n):=\mathsf{U}(n,\mathbb{C})$. Let $\mathfrak{u}(n)$ (cf.\cite{Tum20})\footnote{This situation is a particular case of a more general context exposed in \cite{Kos04} or \cite{LuWe90} for instance.}
the Lie algebra of the unitary group $\mathsf{U}(n)$, that is  the space of skew-symmetric matrices and $\mathfrak{b}(n)$ the Lie algebra  of the Borel group $ \mathsf{B}(n):= \mathsf{B}(n, \mathbb{C})$, that is the space of upper triangular matrices with real coefficient on diagonal.
Then the Lie algebra of the group $\mathsf{GL}(n,\mathbb{C})$ is $\mathfrak{gl}(n)$ of all complex matrices and the  Iwasawa decomposition gives rise to the decomposition 
\[
\mathfrak{gl}(n)=\mathfrak{u}(n)\oplus \mathfrak{b}(n).
\]
Consider  the non degenerate symmetric bilinear continuous map $\langle.,.\rangle $ given by
$$\langle A,B\rangle=\Im\mathsf{Tr}(AB)$$
where  $\Im\mathsf{Tr}(AB)$ denotes the imaginary part of the trace of the matrix $AB$.
 Then $(\mathfrak{u}(n),\mathfrak{b}(n))$ is a bialgebra to which is associated a canonical Poisson Lie structure on  $\mathsf{U}(n)$. More precisely, in this finite dimensional situation, $\mathfrak{b}(n))$ is isomorphic to the Lie algebra $\mathfrak{u}(n)^*$. From the decomposition $\mathfrak{gl}(n)=\mathfrak{u}(n)\oplus \mathfrak{b}(n)$ we denote by $p_1$ (resp. $p_2$) the canonical projection of  $ \mathfrak{gl}(n)$ on $\mathfrak{u}(n)$ (resp. $\mathfrak{b}(n)$  the associated  multiplicative Poisson tensor  $\Lambda_R^n:  \mathsf{U}(n) \to \bigwedge \mathfrak{b}(n)$\footnote{$\mathfrak{b}(n)$ being identified with $\mathfrak{u}(n)^*$.} is characterized by (cf. \cite{LuWe90})
\begin{equation}
\label{eq_LambdaRUn}
(\Lambda_R^n)_g(\alpha_1,\alpha_2):=\Im\mathsf{Tr}\left(p_2(\mathsf{Ad}_{g^{-1}}\alpha_1)p_1(\mathsf{Ad}_{g^{-1}}\alpha_2)\right)
 \end{equation}
 The corresponding multiplicative $2$-form on  $\mathsf{U}(n)$ is\footnote{For notations, see Example~\ref{Ex_LambdaRCocycle}.} 
\[
(\Lambda^n)_g(\alpha_1,\alpha_2)=R^{**}(\Lambda_R^n)_g(\alpha_1,\alpha_2)
\]
and the  Poisson anchor $P_n:T^*\mathsf{U}(n)\to T\mathsf{U}(n)$ is the map
\[
(g,\alpha)\mapsto (\Lambda^n)_g(.,\alpha).
\]
\end{example}

\begin{example}
\label{ex_Non BanachPoissonLie group}
{\bf  A  convenient  Poisson-Lie group.}\\
We have a canonical embedding from $\mathsf{U}(n)$ into $\mathsf{U}(n+1)$ and associated, canonical embedding from $\mathfrak{u}(n)$ (resp. $\mathfrak{b}(n))$) in $\mathfrak{u}(n+1)$ (resp. $\mathfrak{b}(n+1)$) (cf.  proof of Proposition \ref{P_inclusion BoundedSets}).  
Consider  the Poisson anchor $P_n$ (resp. $P_{n+1}$) Poisson anchor  $P:T^*\mathsf{U}(n)\to T\mathsf{U}(n)$ (resp. $P_{n+1}:T*\mathsf{U}(n+1)\to T\mathsf{U}(n+1)$) as previously defined. It is clear for the previous  construction that $\{(T^\flat\mathsf{U}(n), \mathsf{U}(n), P_n)\}_{n\in \mathbb{N}}$ is an ascending sequence of finite dimensional Poisson-Lie groups whose direct limit is $\mathsf{U}(\infty)$ and which is so  provided with a {\bf convenient Poisson-Lie group structure $(T^*\mathsf{U}(\infty), \mathsf{U}(\infty), P_\infty)$ which is not a Banach Poisson Lie group}.  Note that if $f$ and $g$ are smooth maps on $\mathsf{U}(\infty)$ and if $g\in \mathsf{U}(\infty)$ then $g$ belongs to some $\mathsf{U}(n)$ and $\{f,g\}_{P_\infty}(g)=\{f,g\}_{P_n}(g)$ which does not depends of the integer $n$ such that $g$ belongs to $\mathsf{U}(n)$.
\end{example}

\begin{example}\label{ex_StrictPartialPoissonLiegroup} 
{\sf A partial Banach Poisson-Lie group which is not a Banach Poisson-Lie group.}\\
We consider the context of \cite{TuGo23}.\\
Let $\mathcal{H}$ be a separable complex Hilbert space provided with an Hilbert basis  $\left(\;\vert i\rangle \right) _{i\in \Z}$. We consider the group $\mathsf{U}(\mathcal{H})$ of skew-hermitian bounded operators and we denote by $\mathfrak{u}(\mathcal{H})$ its Lie algebra and  $\mathfrak{u}_1(\mathcal{H})$ the the Lie subalgebra of trace class of elements of $\mathfrak{u}(\mathcal{H})$. We also consider the Lie algebra of algebra of upper triangular trace-class operators :
\[
\mathfrak{b}_1^+ (\mathcal{H}):=\{\alpha\in \mathsf{L}_1 (\mathcal{H}) \footnote{cf. Appendix~\ref{__BoundedOperatorsOnAHilbertSeparableComplexSpaceAndVonNeumanAlgebras}.}\;:\; :\;\alpha \vert i\rangle \in \textrm{span}\{ \;\vert j\rangle ,j\leq i\} \textrm{ and } \langle i\vert \alpha\vert i\rangle \in \mathbb{R},\; \forall i\in \mathbb{Z}\}
\]
Then the bilinear map
\[
\begin{matrix}
\mathfrak{u}(\mathcal{H})\times \mathfrak{b}_1^+(\mathcal{H})&\to& \mathbb{R}\hfill{}\\
\hfill  (A,B)&\to&\Im\mathsf{Tr}(AB)\hfill{}
\end{matrix}
\]
is a duality pairing between $\mathfrak{u}(\mathcal{H})$ and $ \mathfrak{b}_1^+(\mathcal{H})$ (cf. \cite{TuGo23}, Proposition 10). By the way, it can be considered as a subspace of the dual  $\mathfrak{u}^*(\mathcal{H})$.\\
On the other hand, the map $\Phi: \mathsf{L}_1 (\mathcal{H})\to \mathfrak{u}^*(\mathcal{H})$ given by $\Phi(a)(x)=\Im\mathsf{Tr}(ax)$. Then the kernel on $\Phi$ is $\mathfrak{u}_1(\mathcal{H})$ and we obtain a 
continuous injective linear map from $\mathsf{L}_1 (\mathcal{H})/\mathfrak{u}_1(\mathcal{H})$ into $\mathfrak{u}^*(\mathcal{H})$ which is the predual  $\mathfrak{u}_*(\mathcal{H})$ of $\mathfrak{u}(\mathcal{H})$ is 
preserved by the coadjoint action of $\mathsf{U}(\mathcal{H})$. Moreover $\Phi(\mathfrak{b}_1^+ (\mathcal{H}))$ is a proper dense subspace of  $\mathsf{L}_1 (\mathcal{H})/\mathfrak{u}_1(\mathcal{H}$ (cf. \cite{TuGo23}, Proposition~13).

The Lie-Poisson structure on $\mathsf{U}(\mathcal{H})$ is built from the decomposition
\[
\mathsf{L}_2(\mathcal{H})=\mathfrak{u}_2(\mathcal{H})\oplus \mathfrak{b}_2^+(\mathcal{H})
\]
where $ \mathfrak{u}_2(\mathcal{H})=\mathfrak{u}(\mathcal{H})\cap \mathsf{L}_2(\mathcal{H})$ and $\mathfrak{b}_2^+(\mathcal{H})$ is the Lie algebra of  Hilbert-Schmidt upper triangular operators with real coefficients on the diagonal. Then we have a Banach Poisson-Lie structure on $\mathsf{U}(\mathcal{H})$ defined in the following way (cf. \cite{TuGo23}, Theorem~14):

$T^\flat \mathsf{U}(\mathcal{H})=\bigcup_{g\in \mathsf{U}(\mathcal{H}}R_{g^{-1}}\mathfrak{u}_*(\mathcal{H})$ which the predual bundle of $T\mathsf{U}(\mathcal{H})$

the Poisson structure is defined by $\Lambda_R:\mathsf{U}(\mathcal{H})\to  \bigwedge^2\mathfrak{u}_*(\mathcal{H})$ \footnote{cf. Proposition~\ref{P_MultiplicativeLambda}.} given by 
\begin{equation}
\label{eq_LambdaRH}
(\Lambda_R^{\mathcal{H}})_g([\alpha_1],[\alpha_2])=\Im\mathsf{Tr}\left(p_2(\mathsf{Ad}_{g^{-1}}\alpha_1)p_1(\mathsf{Ad}_{g^{-1}}(p_2(\alpha_2))\right)
\end{equation}
where $[\alpha_l]=\Phi (\alpha_l)\in \mathsf{L}_1 (\mathcal{H})/\mathfrak{u}_1(\mathcal{H})\equiv \mathfrak{u}_*(\mathcal{H})$ for $l=1,2$. We denote by $P_\mathcal{H}$ the associated Poisson anchor.\\

 Now,  considering $\mathcal{H}$ provided with a fixed Hilbert basis  $\left( \;\vert i\rangle \right)_{i\in \Z}$, for each $n\in \N$, we have a canonical embedding of $\mathsf{U}(n)$ in $\mathsf{U}(\mathcal{H})$ which is a closed Banach Lie subgroup. By the way, $\mathfrak{b}(n)$  this gives rise to an embedding of Banach space of $\mathfrak{b}(n)$ in  $\mathfrak{b}_2^+(\mathcal{H})$ which identified with embedding.
Note that  $\Phi(\alpha)$ belongs to $\mathfrak{u}^*(n)$ and so  the class $[\alpha]$ is  reduced to $\alpha$. Moreover, in one hand,  according to the relations (9) and (10) in \cite{TuGo23}, $\mathfrak{u}(n)$  and  $\mathfrak{b}(n)$  are invariant by conjugation by a unitary operator. In the other hand,    according to (\ref{eq_LambdaRUn}) and  (\ref{eq_LambdaRH})  and the previous identifications, $\Lambda^n_R$ is the restriction of $\Lambda_R^{\mathcal{H}}$ to $\mathfrak{b}(n)$. 
We denote by $P_\mathcal{H}:T^*\mathsf{U}(\mathcal{H})\to T\mathsf{U}(\mathcal{H})$ the Poisson anchor associated to $\Lambda_R^{\mathcal{H}}$, it follows that $P_n$ is nothing but the restriction of $P_\mathcal{H}$ to $T^*\mathsf{U}(n)$ 
and so $ \left( T^*\mathsf{U}(n),\mathsf{U}(\mathcal{H}), P_\mathcal{H} \right) _{\vert T^*\mathsf{U}(n))})$ is a {\bf Banach partial Poisson-Lie group  which is not a Banach Poisson-Lie group}.
\end{example}

\begin{remark}
\label{ex_complement}${}$
\begin{enumerate}
\item[1.] 
As in finite dimension, we can define the notion of convenient partial Poisson-Lie group as follows:\\
let $(T^\flat \mathsf{G}\mathsf{G},P)$ be a Poisson Lie group.  A closed supplemented subgroup $\mathsf{H}$ of $\mathsf{G}$ will be called a \emph{convenient partial Poison-Lie subgroup} of $\mathsf{G}$ if there exists a convenient partial Poisson Lie group $ \left( T^\flat \mathsf{H}, \mathsf{H}, Q \right) $ on $\mathsf{H}$ such that $T^\flat \mathsf{H}$ is a closed supplemented subbundle of $T^\flat \mathsf{G}$ whose  typical fibre $\mathfrak{h}^\flat$ is a Lie subalgebra of $\mathfrak{g}^\flat$. By the way, $ \left( T^*\mathsf{U}(n), \mathsf{U}(n), P_\mathcal{H} \right) $ is a Banach Lie-Poisson subgroup of  $T_*\mathsf{U}(\mathcal{H}), \mathsf{U}(\mathcal{H}, P_\mathcal{H})$. 
\item[2.] 
From Example \ref{ex_Non BanachPoissonLie group} and Example \ref{ex_StrictPartialPoissonLiegroup}, it follows easilly that $\mathsf{U}(\infty)$ can be considered as a subset of  $\mathsf{U}(\mathcal{H})$,  the inclusion is a continuous smooth injective map  and  $\mathsf{U}(\infty)$ is dense in $\mathsf{U}(\mathcal{H})$. It is a convenient weak submanifold of $\mathsf{U}(\mathcal{H})$ but it is not a convenient partial Poisson Lie subgroup of  $\mathsf{U}(\mathcal{H})$.
\end{enumerate}
\end{remark}

\section{Convenient Poisson-Lie Groupoids}
\label{__ConvenientPoissonLieGroupoids}

In this section, we recall the definition of a topological groupoid  and some generalization, from \cite{BGJP19}, of these definitions to convenient Lie groupoids and associated Lie algebroids, in particular in the Banach setting. The reader will find the definition of partial Poisson-Lie groupoid and symplectic Lie groupoid and also their principal properties in the last paragraph.

\subsection{Topological Groupoids}
\label{___TopologicalGroupoids}

\begin{definition}
\label{D_TopologicalGroupoid}
A topological groupoid $\mathcal{G} \rightrightarrows M$ is a pair $(\mathcal{G},{M})$ of topological spaces such that  $\mathcal{G}$ may not be Hausdorff but $M$ is Hausdorff,  with the following structure maps:
\begin{description}
\item[\textbf{(TGr1)}]
two surjective open continuous maps ${\bf s}:  \mathcal{G}\rightarrow M$ and ${\bf t}: \mathcal{G}\rightarrow M$ called \emph{source}\index{source} and \emph{target}\index{target} maps, respectively;
\footnote{Each point $g\in G$ can be seen as an arrow $g : \textbf{s}(g)\rightarrow{\bf t}(g)$ which joins ${\bf s}(g)$ to ${\bf t}(g)$.}
\item[\textbf{(TGr2)}]
a continuous map  ${\bf m}:\mathcal{G}^{(2)}\rightarrow\mathcal{G}$, where $\mathcal{G}^{(2)}:=\{(g,h)\in\mathcal{G}\times\mathcal{G}\mid {\bf s}(g) = {\bf t}(h)\}$ is provided with the induced topology from the  product topology on $\mathcal{G}\times \mathcal{G}$, called \emph{multiplication}\index{multiplication!in a groupoid} denoted ${\bf m}(g,h)=gh$ and satisfying an associativity relation  in the sense that the product $(gh)k$ is defined if and only if $ g(hk)$  is defined and in this case we must have $(gh)k = g(hk)$;
\item[\textbf{(TGr3)}]
a continuous  embedding ${\bf 1} : M \rightarrow\mathcal{G}$ called \emph{identity section}\index{identity section!of a groupoid} which satisfies
$g{\bf 1}_x = g$ for all $g\in{\bf s}^{-1}(x)$, and ${\bf 1}_x g = g$ for all $g\in{\bf t}^{-1}(x)$
(what in particular implies ${\bf s}\circ{\bf 1}=\operatorname{id}_{M}={\bf t}\circ{\bf 1}$);
\item[\textbf{(TGr4)}]
a homeomorphism  ${\bf i} : \mathcal{G}\rightarrow \mathcal{G}$,  denoted ${\bf i}(g)=g^{-1}$ called \emph{inversion}\index{inversion!in a groupoid}, which satisfies
$gg^{-1}={\bf 1}_{{\bf t}(g)}$, $\;g^{-1}g={\bf 1}_{{\bf s}(g)}$ (what in particular implies ${\bf s}\circ{\bf i}={\bf t}$, $\; {\bf t}\circ {\bf i}={\bf s}$).
\end{description}
The space $ M$ is  called the \emph{base} of the groupoid\index{base!of a groupoid}, and $\mathcal{G}$
is called the \emph{total space}\index{total space} of the groupoid.
\end{definition}
For any $(x,y) \in M^2$, we denote $\mathcal{G}(x,-):={\bf s}^{-1}(x)$,  $\mathcal{G}(-,y):={\bf t}^{-1}(y)$, and
\[
\mathcal{G}(x,y):=\{g\in\mathcal{G}\mid {\bf s}(g)=x,\ {\bf t}(g)=y\}=\mathcal{G}(x,-)\cap\mathcal{G}(-,y).
\]
The \emph{isotropy group}\index{isotropy group} at $x\in M$ is the set
\[
\mathcal{G}(x)=\{ g\in \mathcal{G}\mid {\bf s}(g)=x={\bf t}(g)\}=\mathcal{G}(x,x)\subseteq\mathcal{G}
\]
and the \emph{orbit}\index{orbit} of $x\in M$ is the set
\[
\mathcal{G}.x=\{{\bf t}(g)\mid g\in {\bf s}^{-1}(x)\}={\bf t}(\mathcal{G}(x,-))\subseteq M.
\]
For any $g\in \mathcal{G}$, if ${\bf s}(g)=x$ and ${\bf t}(g)=y$,
we define its corresponding left translation\index{left translation}
\[
L_g: \mathcal{G}(-,x)\rightarrow \mathcal{G}(-,y), \quad h\mapsto gh,
\] and similarly its right translation\index{right translation}
\[
R_g:\mathcal{G}(x,-)\rightarrow \mathcal{G}(y,-), \quad h\mapsto hg.
\]
Each of these maps $L_g$ and $R_g$ is a homeomorphism.

The topological groupoid $\mathcal{G}\rightrightarrows M$ is called \emph{{\bf s}-connected}\index{groupoid!{\bf s}-connected} if its {\bf s}-fibres $\mathcal{G}(x,-)$ are connected for all $x\in M$, and \emph{{\bf s}-simply connected}\index{groupoid!{\bf s}-simply connected} if all its {\bf s}-fibres are connected and simply connected.

The topological groupoid is called  \emph{transitive}\index{groupoid!transitive} if the map
$({\bf s},{\bf t}): \mathcal{G}\rightarrow M\times M$
is surjective.

A \emph{topological morphism between  topological groupoids} $ \mathcal{G}\rightrightarrows M$ and $\mathcal{H}\rightrightarrows N$ is given by a pair of  continuous maps $\Phi:\mathcal{G}\rightarrow \mathcal{H}$ and $\phi:M\rightarrow N$ which are compatible with the structure maps, that is:
\begin{itemize}
\item[$\bullet$]
$\forall g\in \mathcal{G},\; \phi({\bf s}(g))={\bf s}(\Phi(g))$, $\;\;\phi({\bf t}(g))={\bf t}(\Phi(g))$  and  $\Phi(g^{-1})=\Phi(g)^{-1}$;
\item[$\bullet$]
$\forall(g,g')\in\mathcal{G}^{(2)}, \;  \Phi(gg')=\Phi(g)\Phi(g')$;
\item[$\bullet$]
$\forall x \in M, \; \Phi({\bf 1}_x)={\bf 1}_{\phi(x)}$.
\end{itemize}
Note that $\phi$ is uniquely determined by $\Phi$ and therefore a morphism between the groupoids $ \mathcal{G}\rightrightarrows M$ and $\mathcal{H}\rightrightarrows N$ is given only by a continuous map $\Phi:\mathcal{G}\rightarrow \mathcal{H}$ which satisfies the above compatibility conditions.

A \emph{subgroupoid}\index{subgroupoid} of $\mathcal{G}\rightrightarrows M$  is a groupoid  $\mathcal{H}\rightrightarrows N$ such that  $\mathcal{H}\subset \mathcal{G}$
and the inclusion $\iota:\mathcal{H}\rightarrow \mathcal{G}$ is a topological morphism of groupoids.
A subgroupoid $\mathcal{H}\rightrightarrows N$ of $\mathcal{G}\rightrightarrows M$ is called a \emph{wide subgroupoid}\index{subgroupoid!wide}  if $N=M$.

\subsection{Convenient Lie groupoids}
\label{___ConvenientLieGroupoids}
As in \cite{BGJP19}, for a Banach-Lie algebroid, the definition of a n.n.H. convenient Lie groupoid requires the following analogue basic facts whose proof requires adapted arguments.

\begin{proposition}
\label{P_BasicPropertiesOfTopologicalGroupoids}
Let $\mathcal{G}\rightrightarrows M$ be a topological groupoid  satisfying the following conditions:
\begin{enumerate}
\item[(i)]
$ \mathcal{G}$ is a n.n.H convenient manifold and $M$ is a Hausdorff convenient  manifold;
\item[(ii)]
the map ${\bf s}:\mathcal{G}\to M$ is a submersion;
\item[(iii)]
the map ${\bf i}: \mathcal{G}\to \mathcal{G}$ is a smooth diffeomorphism.
\end{enumerate}
Then ${\bf t}$ is also a  submersion and so, for any $x\in M$, each fibre $\mathcal{G}(x,-)$ and $\mathcal{G}(-,x) $ are n.n.H.  convenient submanifolds of $\mathcal{G}$ and $\mathcal{G}(x,-)$ is Hausdorff if and only if  $\mathcal{G}(-,x) $ is Hausdorff.  Moreover we also have:
\begin{enumerate}
\item
the topological space $\mathcal{G}^{(2)}$ is a n.n.H. convenient submanifold of $\mathcal{G}\times \mathcal{G}$;
\item
the map ${\bf 1} : M \to\mathcal{G}$ is smooth and ${\bf 1}(M)$ is a  convenient  submanifold of  $\mathcal{G}$.
\end{enumerate}
Moreover, if  $ \mathcal{G}$ is a  Hausdorff manifold, then, for each $x\in M$, each fibre $\mathcal{G}(x,-)$ and $\mathcal{G}(-,x) $ are   Hausdorff submanifolds of $\mathcal{G}$  and $\mathcal{G}^{(2)}$ is a  Hausdorff submanifold of $\mathcal{G}\times \mathcal{G}$.
\end{proposition}

\begin{proof}
At first, since from (ii), ${\bf s}$ is a submersion (cf. \cite{CaPe23}, Definition~1.13) and, from  (iii), ${\bf i}$ is a diffeomorphism,  it follows that ${\bf t}={\bf s}\circ {\bf i}$, is a  submersion (cf. \cite{CaPe23}, Proposition~1.15).
Therefore, from \cite{CaPe23}, Theorem~1.14, the fibres $\mathcal{G}(x,-)$ and $\mathcal{G}(-,x) $ are split  n.n.H. convenient immersed submanifolds of $\mathcal{G}$, for any $x\in M$.
Now assume that $\mathcal{G}(x,-)$ is Hausdorff.
We have $\mathcal{G}(-,x)={\bf i}(\mathcal{G}(x,-))$. 
Since ${\bf i}$ is a diffeomorphism of $\mathcal{G}$ this implies that $\mathcal{G}(-,x)$ is also Hausdorff.
The same argument can be applied for the converse.

Consider the map $({\bf s},{\bf t}):\mathcal{G}\times \mathcal{G}\to M\times M$.
If $\mathcal{D}$ denotes the diagonal in $ M\times M$, we have
 $\mathcal{G}^{(2)}=({\bf s},{\bf t})^{-1}(\mathcal{D})$.
Therefore, again from  \cite{CaPe23}, Theorem~1.14, $\mathcal{G}^{(2)}$ is a n.n.H. convenient submanifold of $\mathcal{G}\times \mathcal{G}$ since ${\bf s}$ and ${\bf t}$ are   submersions.
Note that  if  $\mathcal{G}$ is Hausdorff so are $ \mathcal{G}(x,-)$ and $\mathcal{G}(-,x) $ for all $x\in M$ and the same is true for $\mathcal{G}^{(2)}$.  \\
The proof of the last part is formally the same as in \cite{BGJP19}, Proposition~3.1.
\end{proof}

\begin{definition}
\label{D_ConvenientLieGroupoid}
A \emph{n.n.H. convenient Lie groupoid}\index{convenient!Lie groupoid} is a topological groupoid  $\mathcal{G}\rightrightarrows M$
satisfying the following conditions:
\begin{description}
\item[\textbf{(CLGr1)}]
$ \mathcal{G}$ is a n.n.H. convenient manifold and $M$ is a Hausdorff convenient manifold;
\item[\textbf{(CLGr2)}]
the map ${\bf s}:\mathcal{G} \to M$ is a
submersion;
\item[\textbf{(CLGr3)}]
the map ${\bf i}: \mathcal{G} \to \mathcal{G}$ is  smooth;
\item[\textbf{(CLGr4)}]
the multiplication ${\bf m}:\mathcal{G}^{2} \to \mathcal{G}$  is smooth.
\end{description}
A convenient Lie  groupoid $\mathcal{G}\rightrightarrows M$ is called split\index{groupoid!split} if, for every $(x,y) \in M^2$, the set $\mathcal{G}(x,y)$ is a  submanifold of $\mathcal{G}$.\\
A  convenient Lie morphism between the  convenient Lie groupoids $ \mathcal{G}\rightrightarrows M$ and $\mathcal{H}\rightrightarrows N$ is a  topological morphism $\Phi:\mathcal{G}\to\mathcal{H}$ which is a smooth map.
\end{definition}

When $ \mathcal{G}$  is modelled on a Banach space (resp. Fr\'echet space), $\mathcal{G}$ is called a Banach-Lie groupoid (resp. Fr\'echet-Lie groupoid).

\subsection{Examples of Convenient Lie Groupoids}
\label{___ExamplesOfConvenientLieGroupoids}
We will adapt to the convenient context  some classical examples about finite-dimensional Lie groupoids (cf. \cite{BGJP19}).

\subsubsection{Convenient Lie Groups}
\label{____ExGroupoids_BanachLieGroups}
Any convenient Lie group $G$ is a convenient Lie groupoid: the set of  arrows $\mathcal{G}$ is the set $G$ and the set of objects $M$ is reduced to the singleton $\{{\bf 1}\}$, where ${\bf 1}\in G$ is the unit element.

\subsubsection{Convenient Lie Pair Groupoid}
\label{____ExGroupoids_BanachLiePair}
Given a convenient manifold $M$, let $\mathcal{G}:=M \times M$ and let ${\bf s}$ and ${\bf t}$ be the Cartesian projections of $M\times M$ on the first and the second factor, respectively.
The multiplication map ${\bf m}$ and the inverse ${\bf i}$ are respectively  defined as  ${\bf m}((x,y),(y,z))=(x,z)$ and ${\bf i}(x,y)=(y,x)$ and so they are smooth. Finally, the map ${\bf 1}: x \mapsto (x,x)$ is also smooth.
We then obtain a convenient-Lie groupoid $M\times M \rightrightarrows M$.

\subsubsection{General Linear Banach-Lie Groupoids}
\label{____ExGroupoids_GeneralLinearBanachLieGroupoids}
Let $\pi:\mathcal{A}\to M$ be a Banach vector bundle.
The general linear Banach  groupoid $\mathcal{GL}(\mathcal{A}) \rightrightarrows M$ is the  Banach groupoid such that $\mathcal{GL}(\mathcal{A})$ is the set of  linear isomorphisms $g:\mathcal{A}_x \to \mathcal{A}_y$ between each pair of fibres $(\mathcal{A}_x,\mathcal{A}_y)$. The source map and the target map are obvious and the multiplication is the composition of linear isomorphisms and ${\bf 1}_x=\operatorname{Id}_{\mathcal{A}_x}$.\\

\subsubsection{Disjoint Union of n.n.H. Convenient Lie groupoids}
\label{____ExGroupoids_DisjointUnionOfnnHBanachLieGroupoids}
Let $\{\mathcal{G}_\lambda \rightrightarrows M_\lambda\}_{\lambda\in \textbf{L}}$ be a family of n.n.H. convenient-Lie groupoids.
We denote $\mathcal{G}:=\bigsqcup\limits_{\lambda\in \textbf{L}}\mathcal{G}_\lambda$ and $M:=\bigsqcup\limits_{\lambda\in \textbf{L}}M_\lambda$.
Since we consider here disjoint unions, the structure of the n.n.H. convenient-Lie groupoid $\mathcal{G} \rightrightarrows M$ is clearly defined by the collection of structure maps for each particular $\mathcal{G}_\lambda \rightrightarrows M_\lambda$ for $\lambda\in \textbf{L}$.\\
For example, if we consider a finite sequence of convenient  bundles $\mathcal{A}_i \rightarrow M_i$, for $i \in \{1,\dots,n\}$,
then we have the natural structure of a n.n.H. convenient-Lie groupoid  $\bigsqcup\limits_{i=1}^n \mathcal{GL}(\mathcal{A}_i) \rightrightarrows \bigsqcup\limits_{i=1}^n  M_i$.\\
More generally, given any convenient-Lie n.n.H. groupoid $\mathcal{G}\rightrightarrows M$, if $M=\bigsqcup\limits_{\lambda \in \textbf{L}} M_\lambda$ is a partition of $M$ into open $\mathcal{G}$-invariant submanifolds, then the corresponding groupoids $\mathcal{G}_\lambda \rightrightarrows M_\lambda$ obtained by restriction are n.n.H. Banach-Lie groupoids and $\mathcal{G} \rightrightarrows M$ is the disjoint union of the family $\{\mathcal{G}_\lambda \rightrightarrows M_\lambda\}_{\lambda \in \textbf{L}}$.

\subsubsection{Action of a Convenient Lie Group}
\label{____ExGroupoids_ActionOfABanachLieGroup} 
To each  smooth action 
$
\begin{array}{cccc}
A:	& G \times M	& \to	& M \\
	&	(g,x)		& \mapsto & g.x
\end{array}
$ 
of a convenient Lie group on a convenient manifold, is associated a convenient Lie groupoid $\mathcal{G} \rightrightarrows M$ defined in the following way:
\begin{itemize}
\item[$\bullet$]
$\mathcal{G}:=M\times G$;
\item[$\bullet$]
${\bf s}(x,g):=x$ and ${\bf t}(x,g):= A(g,x)=g.x$;
\item[$\bullet$]
if $y=g.x$ then ${\bf m}((y,h),(x,g)):=(x,hg)$;
\item[$\bullet$]
${\bf i}(x,g):=(g.x,g^{-1})$;
\item[$\bullet$]
${\bf 1}_x=(x,{\bf 1})$.
\end{itemize}

The smoothness of the action clearly implies the smoothness of all these previous maps.\\

It is easily seen that for any $x_0,y_0\in M$ one has
\[
\mathcal{G}(x_0,y_0)=\{x_0\}\times\{g\in G\mid g.x_0=y_0\},
\]
and in particular the isotropy group at $x_0$ is
\[
\mathcal{G}(x_0)=\{x_0\}\times G(x_0)
\]
where $G(x_0):=\{g\in G\mid g.x_0=x_0\}$.

It is worth pointing out that  the above groupoid needs not split.
To obtain a specific example, let $\mathcal{X}$ be any real Banach space with a closed linear subspace $\mathcal{X}_0$ which has no supplement in $\mathcal{X}$, e.g. the space of all bounded sequences of real numbers $\mathcal{X}=\ell^\infty \left( \mathbb{R} \right) $ with its subspace $\mathcal{X}_0=c_0$ consisting of the sequences which converge to $0$ (cf. \cite{CaPe23}, Example~1.2 ).
Then the abelian Banach-Lie group $G= \left( \mathcal{X},+ \right) $ acts smoothly transitively on the Banach manifold $M=\mathcal{X} \backslash \mathcal{X}_0$ by  $A(g,x+\mathcal{X}_0):=g+x+\mathcal{X}_0$ for all $g,x \in \mathcal{X}$.
If we define the groupoid $\mathcal{G}:=M\times G \rightrightarrows M$ as above, then the isotropy group at the point $\mathcal{X}_0\in M$, that is,
\[
\mathcal{G}(\mathcal{X}_0)=\{\mathcal{X}_0\} \times \mathcal{X}_0 \subseteq \mathcal{G}
\]
is not a submanifold of $\mathcal{G}=M \times \mathcal{X}$ since $\mathcal{X}_0$ is not a supplemented subspace of $\mathcal{X}$.

\subsubsection{The Gauge Groupoid of a Banach Principal Bundle}
\label{____ExGroupoids_TheGaugeGroupoidOfAPrincipalBundle}
A Banach principal bundle is a locally trivial fibration $\pi: P\to M$ over a connected Banach manifold whose typical fibre is a Banach-Lie group $G$.
We then have a right action of $G$ on $P$ whose corresponding quotient $P/G$ is canonically diffeomorphic to $M$.
We get a right diagonal  action of $G$ on $P\times P$ in an evident way.
The \emph{gauge groupoid}\index{gauge groupoid} is the set of orbits of this action, that is, the quotient $\mathcal{G}:=(P\times P)/G$ provided with the quotient topology.
The source (resp. target) of the equivalence class of $(v,u)$ is  $\pi(u)$ (resp. $\pi(v)$) (cf. \cite{Marl08}).

\subsubsection{Fundamental Groupoid of a  Banach Manifold}
\label{____ExGroupoids_FundamentalGroupoidOfABanachManifold}
Recall that to any connected topological space $M$ is associated its corresponding \emph{fundamental topological groupoid} $\Pi(M)\rightrightarrows M$ where  $\Pi(M)$ is the set of homotopy classes of paths with fixed end points with source map  (resp. target map) the origin (resp. the end) of a class of paths, where the composition is the concatenation of paths whenever  it is defined and the inverse is the opposite class.
When $M$ is a connected  Banach manifold, then each source fibre  is a universal covering of $M$.
In particular, such a fibre is a Banach manifold. Moreover, it is also a principal bundle over $M$ with structural group $\pi_1(M)$. Then the gauge groupoid of this principal bundle can be identified with $\Pi(M)$ and so we get a structure of n.n.H. Banach-Lie groupoid on $\Pi(M)$.

\subsection{Additional Properties of Banach-Lie Groupoids}
\label{___SomePropertiesOfBanachLieGroupoids}%
We recall below some properties of a Banach-Lie groupoid (cf. \cite{BGJP19}, Theorem~3.3). 

\begin{theorem}
\label{T_OrbitBanachGroupoid}
Let  $\mathcal{G}\rightrightarrows M$ be a n.n.H. (pure) Banach-Lie groupoid and define ${\bf t}_x$ as the restriction of ${\bf t}$ to $\mathcal{G}(x,-)$.\\
For any $x \in M$, the following assertions hold.
\begin{enumerate}
\item
For all $x\in M$ and any $y\in \mathcal{G}.x$ the set $\mathcal {G}(x,y)$ is a closed
submanifold of $\mathcal{G}$. In particular, the  isotropy group $\mathcal{G}(x)$ is a Banach-Lie group and its Lie algebra is isomorphic to $\ker T_{{\bf 1}_x}{\bf s}\cap \ker T_{{\bf 1}_x}{\bf t} $.
\item
For any $x\in M$ for which $\mathcal{G}(x,y)$ is a  Banach submanifold of $\mathcal{G}$, for all $y\in \mathcal{G}.x={\bf t}_x(\mathcal{G}(x,-))$, then the orbit $\mathcal{G}.x$ is a pure Banach manifold whose inclusion map $\mathcal{G}.x \hookrightarrow{} M$ is smooth.
\item
Under the assumption of (2), ${\bf t}_x: \mathcal{G}(x,-) \to \mathcal{G}.x$ is a Banach principal $\mathcal{G}(x)$-bundle.
\end{enumerate}
\end{theorem}
\begin{remark}
\label{R_AssumptionPoint2}
Note that the assumptions in (2) are satisfied for all $(x,y)\in M^2$ as soon as $\mathcal{G}\rightrightarrows M$ is split.
\end{remark}
\begin{proof}[Sketch of the proof]
To prove 1., we consider the distribution
\[
\Delta_g=\ker T_g{\bf s}\cap \ker T_g{\bf t}
  \quad \text{for all }g\in \mathcal{G}(x,-),
\]
and we get  a Banach subbundle of $\Delta\subseteq T\mathcal{G}(x,-)$ which is integrable. Thus, from\footnote{This theorem was proved in the context of (Hausdorff) Banach manifolds but all the arguments used in the proof of this result are based on local charts, so the theorem also holds for n.n.H. Banach manifolds.} \cite{CaPe23}, Theorem~8.1, there exists a partition of $\mathcal{G}(x,-)$ into  closed immersed n.n.H. Banach manifolds and each one is modelled on the Banach space $\Delta_{{\bf 1}_x}$.\\
Moreover, since $\Delta_g=\ker T_g{\bf t}_x$, the maximal leaf through $g\in \mathcal{G}(x,-)$ is a connected component of ${\bf t}^{-1}({\bf t}(g))$ and so is a closed subset of the  Banach manifold $ \mathcal{G}(x,-)$.
In particular,  the isotropy group $\mathcal{G}({{\bf t}(g)})$ is a union of such leaves and so is a closed immersed n.n.H. Banach manifold. Since each point in  $\mathcal{G}({{\bf t}(g)})$ is closed, thus it follows from  \cite{Bou67} that $\mathcal{G}({{\bf t}(g)})$ is in fact Hausdorff.
This implies that $\mathcal{G}({{\bf t}(g)})$ has the structure of a Banach-Lie group.\\

For the proof of 2., we consider the map
\[
\begin{array}
[c]{cccc}
\Psi\colon  & \mathcal{G}(x,-)/\mathcal{G}(x)   & \to 		 & \mathcal{G}.x		\\
            & g\cdot\mathcal{G}(x)  	        & \mapsto 	 & {\bf t}(g).
\end{array}
\]
It is easy to prove that  $\Psi$ is bijective. We can show that, under the assumption that  $\mathcal{G}(x,y)$ is a  Banach submanifold of $\mathcal{G}$, for all $y\in \mathcal{G}.x$, the quotient set $\mathcal{G}(x,-)/\mathcal{G}(x)$ has the structure of a smooth manifold for which the map $\Psi\colon \mathcal{G}(x,-)/\mathcal{G}(x) \to M$ is smooth (cf. proof of Theorem 3.3 in \cite{BGJP19}).\\

The last Point 3. is a consequence of \cite{BGJP19}, Lemma 2.5.
\end{proof}

Unfortunately,  the properties of a Banach-Lie groupoid described in Theorem \ref{T_OrbitBanachGroupoid} are not true in general for a convenient Lie groupoid even when it is split. The essential reason is that we have no general theorem of existence of a local flow for a vector field in the convenient setting.  We will show that, under some strong assumptions on projective or direct limits of Banach-Lie groupoids, such properties will be valid.

\subsection{Convenient Lie Algebroid Associated to a Convenient Lie Groupoid}
\label{___ConvenientLieAlgebroidAssociatedToAConvenientLieGroupoid}
In the same way as in the  context of Banach-Lie groupoid, to a convenient  Lie  groupoid is associated a natural convenient Lie algebroid on  $M$. Now we briefly recall its construction.\\
Denote by  $T^{\bf s}\mathcal{G}$  the vertical bundle associated to the submersion ${\bf s}: \mathcal{G}\to M$ and by $\Gamma_{\operatorname{inv}}^r(\mathcal{G})$  the vector space  of  right invariant vector fields. $\Gamma_{\operatorname{inv}}^r(\mathcal{G})$ can be provided with the Lie bracket of vector fields as a Lie subalgebra of the Lie algebra of vertical vector fields. The tangent map $T{\bf t}$ induces a morphism of Lie algebras from $\Gamma_{\operatorname{inv}}^r(\mathcal{G})$ to   the Lie algebra $\Gamma(M)$ of vector fields on $M$.
We consider the pull-back $\mathcal{AG}\to M $ by ${\bf 1}:M\to\mathcal{G}$  of  the convenient bundle $T^{\bf s}\mathcal{G}\to {\bf 1}(M)$  which can be considered as the restriction of  $T^{\bf s}\mathcal{G}$ to ${\bf 1}(M)$ identified with $M\subset\mathcal{G}$.\\
Note that since $M$ is Hausdorff it follows that $ \mathcal{AG}$ is a Hausdorff convenient bundle.

On the one hand, we have a canonical convenient bundle morphism ${\bf 1}_*:\mathcal{AG}\to T^{\bf s}\mathcal{G} $ from the bundle  $\mathcal{AG}\to M $ to the bundle $T^{\bf s}\mathcal{G}\to {\bf 1}(M)$ which is an isomorphism in restriction to each fibre.
On the other hand, if  $\Gamma_{M}(T^{\bf s}\mathcal{G})$  is  the vector space of smooth sections of $T^{\bf s}\mathcal{G}\to {\bf 1}(M)$, then $T_{\bf 1}{\bf t}$ gives rise to an isomorphism from
 $\Gamma_{M}(T^{\bf s}\mathcal{G})$  onto $\Gamma_{\operatorname{inv}}^r(\mathcal{G})$ which is also a morphism of Lie algebras.
Now, the morphism  ${\bf 1}_*$ gives rise to a Lie bracket on the vector space $\Gamma(\mathcal{AG})$ of smooth sections of  $\mathcal{AG}\to M $, which will be denoted by $[.,.]$.
Finally,  by composition, we obtain an isomorphism of Lie algebras between $\Gamma_{M}(T^{\bf s}\mathcal{G})$ and  $\Gamma_{\operatorname{inv}}^r(\mathcal{G})$. Now, $\mathcal{AG}$ is provided with the anchor $\rho=T{\bf t}_{| M}$. From the property of Lie bracket of vector fields on $\mathcal{G}$, it follows that the Lie bracket $[.,.]$  on $\mathcal{AG}$ is localizable by using the same arguments as in the proof of \cite{BGJP19}, Theorem~4.19. Moreover, since  this Lie bracket of vector of  kinematic vector fields only depends on the $1$-jets of such vector fields,  it follows that  $(\mathcal{AG}, M,\rho,[.,.])$ is a convenient Lie algebroid (cf. \cite{CaPe23}, Definition~3.78).\\

In a dual way, we can consider the vertical bundle $T^{\bf t}\mathcal{G}$ associated to the submersion ${\bf t}$ and the Lie algebra $\Gamma_{\operatorname{inv}}^l(\mathcal{G}$ of left invariant vector fields. Then the tangent map $T{\bf s}$ induces a morphism of Lie algebras from $\Gamma_{inv}^l(\mathcal{G})$ to the Lie algebra $\mathfrak{X}(M)$ of vector fields on $M$. As previously, we consider the restriction $\overline{\mathcal{A}}\mathcal{G}$ of  $T^{\bf s}\mathcal{G}$ to ${\bf 1}(M)$ identified with $M$. If $\rho':=T{\bf s}_{| M}$ is then an anchor for  $\overline{\mathcal{A}}\mathcal{G}$ and the isomorphism between $\Gamma_{\bf 1}(\mathcal{G}^{\bf t})$ and $\Gamma_{inv}^l(\mathcal{G})$ provides the set of sections of  $\overline{\mathcal{A}}\mathcal{G}$ with a Lie bracket denoted $\overline{[\;,\;]}$; so we get a convenient Lie algebroid $(\overline{\mathcal{A}}\mathcal{G}, M,\overline{\rho}, \overline{[\;,\;]})$.
Since $T{\bf i}$  transforms a right invariant vector field into a left invariant vector field, this map induces an isomorphism of Lie algebroid from  $(\mathcal{AG}, M,\rho,[\;,\;])$ to $(\overline{\mathcal{A}}\mathcal{G}, M,\overline{\rho}, \overline{[\;,\;]})$.

\subsection{Partial Poisson Convenient Groupoids}\label{__PartialPoissonBanachGroupoids}
Consider a Banach Lie groupoid $\mathcal{G}\tto M$. The mutiplication ${\bf m}$ is a map from $\mathcal{G}^{(2)}$ to $\mathcal{G}$. Since  $\mathcal{G}^{(2)}$ is a closed submanifold of $\mathcal{G}\times\mathcal{G}$ it follows that the graph ${G}_{\bf m}$ of ${\bf m}$ is a closed submanifold of $\mathcal{G}\times\mathcal{G}\times\mathcal{G}$.

\begin{definition}${}$
\label{PartialPoissonBanachLieGroupoid}
\begin{enumerate}
\item[1.] 
A groupoid  $\mathcal{G}\tto M$ is called a \emph{partial Poisson Banach-Lie groupoid}\index{partial!Poisson Banach-Lie groupoid} if there exists a  partial Poisson structure associated to some $P:T^{\flat}
\mathcal{G}\to T\mathcal{G}$ such that 
 for the Poisson structure $\mathcal{G}\times\mathcal{G}\times \mathcal{G}^-$ on $\mathcal{G}\times\mathcal{G}\times\mathcal{G}$, the submanifold $G_{\bf m}$ is coisotropic.
\item[2.] 
A groupoid  $\mathcal{G}\tto M$ is called a \emph{partial symplectic}  (resp.  strong symplectic) \emph{Banach-Lie groupoid} if there exists an isomorphism $P$ from $T^\flat\mathcal{G}$ to $T\mathcal{G}$ (resp. $T^{\prime}\mathcal{G}$ to $T\mathcal{G}$) such that for the associated  Poisson 
structure $\mathcal{G}^+\times\mathcal{G}^+\times \mathcal{G}^-$ on $\mathcal{G}\times\mathcal{G}\times\mathcal{G}$, the submanifold $G_{\bf m}$ is Lagrangian.
\end{enumerate}
\end{definition}

For a partial Poisson Lie convenient groupoid, we have the following properties. In fact, the properties of a partial symplectic Banach-Lie groupoid are stronger than the context  of partial Poisson convenient Lie groupoid. For this reason, we will give these results in another theorem.

\begin{theorem}
\label{T_PropertiesPoissonGroupoid}
Let  $\mathcal{G}\tto M$ be a partial  Poisson Banach-Lie  groupoid. Then we have the following properties:
 \begin{enumerate}
\item[(i)] 
The submanifold $M$ is coisotropic in $\mathcal{G}$.
\item[(ii)] 
The inversion in $\mathcal{G}$ is an anti-Poisson morphism.
\item[(iii)] 
There exists a unique  Lie Poisson algebras sheaf on $M$ such that ${\bf s}$ (resp. ${\bf t}$) is a Poisson morphism (resp. anti-Poisson morphism). Moreover if $T^\flat \mathcal{G}=T^{\prime}\mathcal{G}$ or if $ \left( T^\flat \mathcal{G}, T\mathcal{G}, P \right) $ is  partial symplectic, then there exists a unique  partial Poisson structure on $M$ such ${\bf s}$ (resp. ${\bf t}$) is a Poisson morphism (resp. anti-Poisson morphism).
\end{enumerate}
\end{theorem}

\begin{proof} 
The proof of this theorem is an adaptation of the proof of Theorem 4.2.3 of \cite{Wei88} to our context.\\
1. (i)  Note that, according to axiom ({\bf TGr3}) in Definition \ref{D_TopologicalGroupoid}, $M$  is the set
 $$\{x\in \mathcal{G},\; \textrm{such that } yx=y \textrm{ for some } y \in \mathcal{G}) \}.$$  Let  $C$ be the diagonal $\Delta_\mathcal{G}$ of $\mathcal{G}^2$ and $R:\mathcal{G}^2\rel \mathcal{G}$ whose graph is 

$\{(yx, y ,x),\; (y,x)\in \mathcal{G}\times \mathcal{G}\}.$
 
Then 
$R(C)=\{x\in \mathcal{G} \textrm{ such that } (y,y,x) \in R \textrm{ for some } (y,y)\in \Delta_\mathcal{G}\}$\\ 
which is $M$ according to the previous characterization.\\
In the one hand,   $\operatorname{Id}_\mathcal{G}$ is a Poisson morphism, so its graph $\Delta_\mathcal{G}$ is coisotropic in $\mathcal{G}^-\times \mathcal{G}^+$.  
On the other hand, since $G_{\bf m}$ is coisotropic in $\mathcal{G}^+\times \mathcal{G}^+\times \mathcal{G}^-$, it follows that  $R$,  whose graph is $G_{\bf m}^{-1}$, is also coisotropic in $\mathcal{G}^-\times \mathcal{G}^+\times \mathcal{G}^+$. Thus by  Corollary~\ref{C_CompositionRelationPoissonCoisotropic}, $R(C)=M$ will be coisotropic in $\mathcal{G}^+$ if we show that $(R,C)$ is a clean pair.\\

According Definition~\ref{D_AssumptionClean} and Definition~\ref{D_RCclean}, if $\Delta_{\mathcal{G}^2}$ denotes the diagonal of  ${\mathcal{G}^2}$, we must show at first  that  $(\Delta_\mathcal{G}\times R)\cap (\Delta_{\mathcal{G}^2}\times \mathcal{G})$ is a closed submanifold of $(\mathcal{G}^2)\times\mathcal{G}^2\times \mathcal{G}$ whose tangent bundle is the intersection $(T \Delta_\mathcal{G}\times TR)\cap (T\Delta_{\mathcal{G}^2}\times T\mathcal{G})$.\\
But we have:
\[
(\Delta_\mathcal{G}\times R)\cap (\Delta_{\mathcal{G}^2}\times \mathcal{G})=\{(y,y,y, y,x)\in C\times R\; \textrm{such that } yx=y \}
\] 
which is the range of the embedding of $\mathcal{G}$ in $\mathcal{G}^2\times \mathcal{G}^2\times \mathcal{G}$ 
\[
{\bf j}:y\mapsto (y,y,y,y,{\bf s}(y)).
\] 
Now the tangent space of this submanifold at a point $(y,y,y,y,x)$ is the range $T_y{\bf j}$. To show that this space is exactly 
\[
 (T_{(y,y,y,y,x)}\Delta_\mathcal{G}\times TR)\cap (T_{(y,y,y,y,x)}\Delta_{\mathcal{G}^2}\times T\mathcal{G})
\]
we have only to show that $T_{(y,y,x)}R\cap (T{(y,y)}\Delta_\mathcal{G}\times T_x \mathcal{G})$ is a set of type 
\[
\{(u,u,v):\; u\in T_y\mathcal{G},\; v=T_y{\bf s}(u)\}.
\]
At first, any vector in   $T_{(y,y,x)}R\cap (T{(y,y)}\Delta_\mathcal{G}\times T_x \mathcal{G})$ is of type $(u,u,v)$ for  $u\in T_y\mathcal{G}$ and $v\in T_x\mathcal{G}$. Let 
 $x(t)$ and $y(t)$ be two smooth curves in $\mathcal{G}$  defined on $]-\varepsilon ,\varepsilon [$  such that $x(0)=x, y(0)=y$, $\dot{x}(0)=v$, $\dot{y}(0)=u$ and $(\dot{yx})
 (0)=u$. Since $y(0)x(0)=y(0)$, this implies that $x(0)$ belongs to $M$. We can write
 
 $x(t)={\bf t}({\bf L}_{y(t)^{-1}}(yx(t))$.\\
Thus we obtain:
 
 $v=\dot{x}(0)=T_y{\bf t}(u)$.\\
Since $\displaystyle\frac{d}{dt}(y(t))_{t=0}= \displaystyle\frac{d}{dt}(y(t)x(t))_{t=0}$,  the last condition 2. (ii) of Definition~\ref{D_AssumptionClean} is clear. This ends the proof of (i).  \\
 
(ii) Once more, the graph of the inversion is of type $R(C)$ where $R:\mathcal{G}\rel \mathcal{G}\times \mathcal{G}$ whose graph is $\{(gh, g ,h), \;:\;g, h\in \mathcal{G}\}$ 
 and $C=M$ since 
\[
R(C)=\{(g,h)\in \mathcal{G}\times \mathcal{G}\textrm{ such that } (x, g,h)\in R \textrm{ for some } x\in M\}.
\]
Now, $M$ is coisotropic in $\mathcal{G}^+$ and  $R$ is coisotropic in $\mathcal{G}^+\times \mathcal{G}^-\times\mathcal{G}^-$. Thus if the pair $(R,C)$  is clean, then $R(C)$ is coisotropic in $\mathcal{G}^-\times\mathcal{G}^-$ which is equivalent to say that $R(C)$, that is the graph of ${\bf i}$, is coisotropic in $\mathcal{G}^+\times \mathcal{G}^+$ and so  ${\bf i} $ is an anti-Poisson morphism.\\

The intersection 
$(M\times R) \cap 
(\Delta_\mathcal{G}\times \mathcal{G}^2)$ is the set 
\[
\{(x,x,g,h)\in M\times R\}
\] which is the range of the embedding of $M\times\mathcal{G}^2$ into $\mathcal{G}^2\times\mathcal{G}^2$
$${\bf j}: (x,g,h)\mapsto (x,x,g,h).$$
The same type of argument used in the proof of point (i) permits to show that
$$T(M\times R)\cap (\Delta_\mathcal{G}\times \mathcal {G}^2)=T(M\times R)\cap T(\Delta_\mathcal{G}\times \mathcal {G}^2)$$ and is  left to the reader.
Finally, the last condition 2. (ii) of Definition \ref{D_AssumptionClean} is clear and so the proof (ii) is complete.\\

(iii) We have ${\bf s}(g)={\bf s}(g')$ if and only if, there exists $h\in \mathcal{G}$   $g'=hg$ (for instance $h=g^{-1}g'$). So the graph of ${\bf s}\circ {\bf s}^{-1}$ can be written 
$R( C)$ with $R: \mathcal{G}\times \mathcal{G}\rel \mathcal{G}$  whose graph is $\{(h,g, gh),\; (h,g)\in\mathcal{G}^2\}$  and $C=\mathcal{G}$. Indeed, we have 
 $${\bf s}\circ {\bf s}^{-1}=\{(g,g')\in \mathcal{G}^2 \textrm{ such that } g'=gh \textrm{ for some } h \in g\}.$$
 Since $g$ is coisotropic in $\mathcal{G}^+$ and $R$ is coisotropic $\mathcal{G}^+\times\mathcal{G}^+\times \mathcal{G}^-$, if the pair $(R,C)$ is clean the $R(C)$ is 
 coisotropic in $\mathcal{G}^+\times \mathcal{G}^-$ and so ${\bf s}\circ {\bf s}^{-1}$ is a Poisson relation. The result follows from Theorem 
 \ref{T_ExistencePartialPoissonStructure}. Now the proof of the property of cleanliness of $(R,C)$ is elementary and is left to the reader. Since  ${\bf t}={\bf s}\circ {\bf i}$  and ${\bf i}$ is an anti-Poisson morphism, it follows that ${\bf t}$  is an anti-Poisson morphism. 
\end{proof}

\begin{remark}
\label{R_SheafLiePoissonIniii}
In the context of Theorem~\ref{T_PropertiesPoissonGroupoid}, if $T^\flat \mathcal{G}=T^{\prime}\mathcal{G}$ or if $(T^\flat \mathcal{G}, T\mathcal{G}, P)$ is  partial symplectic, the unique  partial Poisson structure on $M$ such ${\bf s}$ (resp. ${\bf t}$) is a Poisson morphism (resp. anti-Poisson morphism) is such that the  Lie-Poisson algebras  sheaf on $M$ built in Point (i) is exactly the Poisson Lie algebras sheaf  associated to this  partial Poisson structure on $M_2$. However even if the partial Poisson  structure on $\mathcal{G}$ is symplectic the unique induced Poisson structure on $M$ is not in general symplectic even in finite dimension (cf. \cite{AlDa90}).
\end{remark}

Now, for a partial symplectic Lie convenient groupoid, we have the following properties:
\begin{theorem}\label{T_PropertiesSymplecticgroupoid}  
Let $\mathcal{G}\tto M$ be a partial  symplectic convenient Lie groupoid. 
Then we have: 
\begin{enumerate}
\item[(i)] 
The submanifold $M$ is Lagrangian in $\mathcal{G}$.
\item[(ii)] 
The inversion in $\mathcal{G}$ is an anti-symplectic map. 
\item[(iii)] 
There exists a unique strong Poisson structure on $M$ such that ${\bf s}$ (resp. ${\bf t}$) is a Poisson morphism (resp. anti-Poisson morphism).
\item[(iv)] 
There exists a weak symplectic $2$-form $\omega$  such that $\omega^\flat: T\mathcal{G}\to T^{\prime}\mathcal{G}$ whose range is $T^\flat M$ and $\omega^\flat_{| T^\flat M}=P$.
\item[(v)] 
For each $g\in\mathcal{G}$, each vector space $T_g\mathcal{G}({\bf s}(g),-)$ and $T_g\mathcal{G}(-, {\bf t}(g))$ are orthogonal symplectic subspaces of the symplectic space $(T_g\mathcal{G},\omega)$.
\item[(vi)] 
If, moreover, $\mathcal{G}\tto M$ is a split  Banach-Lie groupoid,  $\omega$ induces a weak symplectic form on  each orbit of $\mathcal{G}$ in $M$.
\end{enumerate}
\end{theorem}

Following \cite{BGJP19}, 6.2  a convenientcLie groupoid $\mathcal{G}\tto M$ is called  \emph{transitive} if the map 
$({\bf s},{\bf t}):\mathcal{G}\ap M\times M$
is a surjective submersion. In this case, the action of $\mathcal{G}$ on $M$ has only one orbit reduced to $M$.
Thus by application of Theorem \ref{T_PropertiesSymplecticgroupoid}, we obtain:

\begin{corollary}\label{C_TransitivePartialBanachgroupoid}
Let $\mathcal{G}\tto M$ be a transitive partial symplectic Banach-Lie groupoid. Then  the weak symplectic $2$-form $\omega$  on $\mathcal{G}$ induces a weak symplectic form on $M$.
\end{corollary}

\begin{remark}
\label{R_FiberNonSymplectic} 
Consider a partial symplectic Banach-Lie  groupoid $\mathcal{G}\tto M$.
 The $2$-form induced by $\omega$ on $T_g\mathcal{G}( {\bf s}(g),-)$ (resp. $T_g\mathcal{G}(-, {\bf t}(g))$ is not symplectic since the intersection of these two tangent spaces is isomorphic to the Lie algebra of the isotropic group of ${\bf s}$ under the assumptions  of Theorem~\ref{T_OrbitBanachGroupoid}. In this context, this is true if and only if the isotropy group of ${\bf s}(g)$ is reduced to the identity. Moreover, the partial Poisson structure on $M$ cannot be weak symplectic in general from from Point (vi).
\end{remark}

Note that from Point (i) and (iii)  in Theorem~\ref{T_PropertiesSymplecticgroupoid}, Proposition~\ref{P_DualPairs} and Remark~\ref{R_CasePhiantiPoisson}, we obtain:

\begin{corollary}
\label{C_Involution} 
Let $\mathcal{G}\tto M$ be a partial  symplectic Banach-Lie  groupoid. For any open set $U$ in $M$, let $(\mathcal{E}_M(U),\{\;,\;\}_U)$ be the Poisson Lie algebra on $U$ associated to the partial Poisson structure on $M$ and $\{\;,\;\}_P$ the Poisson bracket on $\mathcal{G}$ associated to the partial symplectic structure on $\mathcal{G}$. Then any pair $(f_1,f_2)$ in $\mathcal{E}_M(U)$ are in involution, that is $\{f_1,f_2\}_U=0$.
\end{corollary}

\begin{proof}[Proof of Theorem~\ref{T_PropertiesSymplecticgroupoid}] The proof of Point (i) is formally the same as in Theorem~\ref{T_PropertiesPoissonGroupoid}, (i) but under the assumption that $G_{\bf m}$ is Lagrangian. 
Since, in this case, $\Delta_\mathcal{G}$ is also Lagrangian in $\mathcal{G}^-\times\mathcal{G}^+$, we use here the second part of Corollary ~\ref{C_CompositionRelationPoissonCoisotropic}.  In fact, in this case, the  Lie-Poisson algebras sheaf on $M_2$ is exactly the sheaf obtained from the  partial Poisson
structure on $M$ (cf. Remark~\ref{R_SheafLiePoissonIniii}).\\
Point (ii) and (iii) are exactly the result of point (ii) and (iii) of Theorem~\ref{T_PropertiesPoissonGroupoid}.\\
Point (iv) is a consequence of Theorem~\ref{T_FoliationPartialBanachPoissonManifold} 
applied to the characteristic foliation whose leaves are the connected components of $\mathcal{G}$.\\
For Point (v), we will adapt the proof of Theorem  3.6 of \cite{Marl05}.  
According to Point (iv), we have a weak symplectic form $\omega$ on $\mathcal{G}$ such that:
 $\omega(w,w')=<P^{-1}(w), w'>$ for any $w, w'$ in $T_g\mathcal{G}$ and any $g\in \mathcal{G}$. In this way,  on $\mathcal{G}^+
 \times\mathcal{G}^+\times\mathcal{G}^-$, the partial symplectic structure  is an isomorphism ${P}^3$ from $(T^\flat \mathcal{G})^3$ to  $(T \mathcal{G})^3$ given by 
\[
P^3(w_1,w_1,w_3)
= \left( P(w_1),P(w_2),-P(w_3) \right) .
\]
Thus, we also have a weak symplectic form $\omega^3$ on  $(\mathcal{G})^3$  associated to ${P}^3$ given by
\begin{eqnarray}
\label{E_OP3}
\omega^3((w_1,w_2, w_3),(w'_1,w'_2,w'_3))=\omega(w_1,w'_1)+\omega(w_2,w'_2)-\omega(w_3,w'_3).\\
\end{eqnarray}
 
For each $g\in \mathcal{G}$, we set $A_g=\mathcal{G}({\bf s}(g),-)$ and $B_g=\mathcal{G}(-, {\bf t}(g))$ the fibre  of ${\bf s}$  and  ${\bf t}$ respectively.\\
Fix some $k\in \mathcal{G}$. For any $(g,h)\in\mathcal{G}^{(2)}$, we set $k=gh$. Then we have   and ${\bf R}_h(A_h)=A_k$ and ${\bf L}_g(B_g)=B_k$. Now, for $u\in T_gA_g$ 
and $v\in T_hA_h$, if $w_1$ and $w_2$ belongs to $T_k\mathcal{G}$, the vectors $(0,u, w_1)$ and $(0, v,w_2)$ belongs to $T_{(g,h, k)}\mathcal{G}^3$ and are tangent to 
the graph $G_{\bf m}$ of ${\bf m}$ if and only if 
\begin{eqnarray}
\label{E_TRhTLg}
w_1= T{\bf R}_h(u)\textrm{ and }w_2=T{\bf L}_g(v).
\end{eqnarray}

Now by assumption, $G_{\bf m}$ is Lagrangian in $\mathcal{G}^+\times\mathcal{G}^+\times\mathcal{G}^-$. Thus,  according to  Proposition~\ref{P_PartialSymplectic}, the definition of a Lagrangian manifold, the relations (\ref{E_OP3}) and (\ref{E_TRhTLg}), we must have:
\[
\omega(u,0)+\omega(0,v)-\omega( T{\bf R}_h(u),T{\bf L}_g(v)=0,
\]
for all $u\in T_g A_g$ and $v\in T_hB_h$. Since $T{\bf R}_h(T_hA_h)= T_k A_k$ and $T{\bf L}_g(T_gB_g)=T_k B_k$, this implies that
\begin{center}
 $T_kB_k=  (T_kA_k)^{\perp_P}$ and  $T_kA_k=  (T_kB_k)^{\perp_P}$
\end{center}
according to Proposition~\ref{P_PartialSymplectic} which ends the proof of Point (v).\\

(vi) If, moreover, $\mathcal{G}\tto M$ is split, for each $x\in M$, from Theorem~\ref{T_OrbitBanachGroupoid}, $\mathcal{G}(x,-)$ is a closed split submanifold of $\mathcal{G}$. 
For any $g\in \mathcal{G}(x,-)$, let $H_g$ be a supplement of $T_g\mathcal{G}(x,-)\cap T_g\mathcal{G}(-,x)$ in $T_g\mathcal{G}(-,x)$. Then from Point (v) of Theorem~\ref{T_PropertiesSymplecticgroupoid}   and Point (ii) of Proposition \ref{P_PartialSymplectic}, $\omega$ induces 
on $H_g$ a symplectic form. Since we have a canonical isomorphism from $T\mathcal{G}(x,-)/T\mathcal{G}(x)$  to $T\mathcal{G}.x$ (cf Proof of Point (iii) of Theorem 3.3 in 
\cite{BGJP19}), by  using the same arguments as in the proof of Proposition~2.4.5 in \cite{PeCa19}, we obtain a  symplectic form $\widetilde{\omega}$ on $T\mathcal{G}.x$. This means that the set of connected components of the orbit of $\mathcal{G}$ is exactly  the characteristic foliation of the partial Poisson structure on $M$. 
Since this partial Poisson structure is strong, form Point (iii) of Theorem,  it follows that the almost symplectic form on each orbit is closed (cf. Theorem~\ref{T_FoliationPartialBanachPoissonManifold} ). 
\end{proof}

\section{Banach Bialgebroids and Banach-Lie Groupoids}
\label{__BanachBialgebroidsAndBanachLieGroupoids}

This section is devoted to the generalization to the convenient  setting of results contained in \cite{MaXu94}.

\subsection{Properties of the Partial Conormal Bundle of a Coisotropic Manifold}

Let $ \left( T^\prime M, TM, P,\{\;,\;\}_{P} \right) $ be a partial
Poisson manifold.  If $C$ is a closed submanifold of $M$, we denote by $T^\prime_CM$ (resp. $T^\flat_C M$) the restriction of   $T^\prime M$ (resp. $T^\flat M$) to $C$ that is the pull-back of  $T^\prime M$ (resp. $T^\flat M$) over the inclusion of $C$ in $M$.

\begin{definition} ${}$
\begin{enumerate}
\item[(1)]
The  \emph{conormal}\index{conormal}   $\mathcal{N}^a(C)$ is the subbundle of $T^\prime_C M$ defined by:
\[
\mathcal{N}^a(C)
=\{(x,\eta)\in T^\prime_C M:\;
\forall u\in T_xC,\,  <\eta, u>=0\}.
\]
\item[(2)] 
The \emph{partial  conormal}\index{partial conormal}  $\mathcal{N}^0(C)$ of $T^\flat_C M$ defined by:
\[
\mathcal{N}^0(C)
=\{(x,\eta)\in T^\flat_C M:\;  
\forall u\in T_xC,\, <\eta, u>=0\}.
\]
\end{enumerate}
\end{definition}

\begin{remark}
\label{R_N0C}
It is clear that $\mathcal{N}^a(C)$ is a   subbundle of $T^\prime_CM$,  but in general,  $\mathcal{N}^0(C)$ {\bf will not be a bundle} over $C$. \\
\end{remark}


Now, for each open set $U$ in $M$, from Proposition~\ref{P_PropertiesPartialPoissonManifold}, the sheaf of brackets $[\;,\;]_P$ provides each module $\Gamma(T^\prime M_U)$ with a structure of Lie algebra.
\begin{theorem}
\label{T_PartialConormalC} 
Let $C$ be a coisotropic submanifold of $M$. 
\begin{enumerate}
\item[(1)] 
The partial conormal $\mathcal{N}^0(C)$ is exactly $(TC)^\perp$. If  $T^\flat M=T^\prime M$  or  if  $P$ is partial symplectic and $C$ is Lagrangian, then $\mathcal{N}^0(C)$ is a subbundle of $T^\flat M$.  
\item[(2)]
If $\mathcal{ N}^0(C)$ is a  closed sub-bundle of $T^\flat_CM$, it  can be  provided with a Lie algebroid structure whose anchor  is  the restriction of $P$ to $\mathcal{N}^0(C)$.
\end{enumerate}
\end{theorem}



\begin{proof}${}$\\
(1) From the definition \ref{D_Coisotropic} we have  $(T_xM)^{\perp_{P_x}}=P_x((T_xM)^0)$ and so we have  $\mathcal{N}^0(M)=(TM)^\perp$. At first assume that $T^\flat M=T^\prime M$ then $\mathcal{N}^a(C)=\mathcal{N}^0(C)$ and so is a subbundle of $T^\prime M$. Assume now that $P$ is partial symplectic.  
Since $C$ is Lagrangian for each $x\in C$ we have $P_x(\mathcal{N}_x^0(C))= T_xC$. Thus $\mathcal{N}^0(C)=P^{-1}(TC)$ and since $P$ is an isomorphism, it follows that   $\mathcal{N}^0(C)$ is a subbundle of $T^\flat_CM$.

\noindent
(2) Let $P_C$ be the restriction of $P$ to $\mathcal{N}^0(C)$. Since $C$ is coisotropic, we have already seen that  $P(\mathcal{N}^0(C)) \subset TC$ and so $P_C$ takes values in $TC$.
 We must firstly show that the bracket defined in (\ref{eq_AlmostPoissonBracket}) restricted to sections of $\mathcal{N}^0(C)$ takes values in $\mathcal{N}^0(C)$.
Fix some open set $U$ in $M$.  Using the relation $L_X\omega=d\omega(X,\;)+d(\omega(X))$ and  the definition of $[\;,\;]_P$,   for any $\alpha$ and $\beta$ in $\Gamma(T^\flat M_U)$ we have 
$$[\alpha,\beta]=d\beta(P\alpha,\;)-d\alpha(P\beta,\;)-d\beta(P\alpha,\;)-d<\alpha,P\beta>.$$
 It follows that for $X\in T(C\cap U)$  and any  section $\alpha$ and $\beta$ of $\mathcal{N}^0(C)$ over $C\cap U$ we have $<\alpha,P\beta>\equiv 0$ on $C\cap U$ and so 
 we obtain $[\alpha,\beta](X)=d\beta(P\alpha,X)-d\alpha(P\beta,X)$. But since $P\alpha$ and $X$ are tangent to $C$ so is the bracket $[P\alpha, X]$ then we get 
 $d\beta(P\alpha,X)=0$. By  same arguments we also have $d\alpha(P\beta,X)=0$. Thus  $[\alpha,\beta]$ belongs to $\mathcal{N}^0(C)$. Since the Jacobiator of $[\;,\;]_P$ is 
 zero on sections of $T^\flat M$ and the restriction of $[\;,\;]_P$ to section of $\mathcal{N}^0(C)$ takes values  in  $\mathcal{N}^0(C)$, this  Jacobiator is also zero on sections of $
 \mathcal{N}^0(C)$. Finally the Leibniz rule on $\Gamma(T^\flat M_U)$ implies the Leibniz rule in restriction to $\Gamma(\mathcal{N}^0(C\cap U))$. \end{proof}

\subsection{The partial Dual Lie Algebroid of  the  Lie Algebroid of a Banach Lie-Poisson Groupoid}

Let $ \left( \mathcal{AG}, M,\rho,[\;,\;] \right)  $  be the  Lie algebroid associated to a groupoid  $\mathcal{G}\tto M$.
We denote by $\mathcal{AG}^\prime\to M$ the dual bundle of $\mathcal{AG}\to M$.   We have the following properties:

\begin{theorem}
\label{T_PartialConormal} 
Assume that $\mathcal{G}\tto M$ is a convenient Lie-Poisson groupoid  such that partial conormal $\mathcal{N}^0(M)$  is a split convenient subbundle of $T^\flat_M\mathcal{G}$. Then we have
\begin{enumerate}
\item[(1)] If $T^\flat \mathcal{G}=T^\prime \mathcal{G}$ the previous assumption is always true and then  $\mathcal{N}0(M)=\mathcal{N}^a(M)$  is canonically isomorphic to the dual $\mathcal{AG}^\prime\to M$ of the Lie algebroid $\mathcal{AG}\to M$.
\item[(2)] The inclusion of  $\mathcal{N}^0(M)$ in the conormal bundle $\mathcal{N}^a(M)$ is a bundle morphism.
\item[(3)] $\mathcal{N}^0(M)$  can be  provided with a Lie algebroid structure whose anchor  is  the restriction of $P$ to $\mathcal{N}^0(M)$.\\
\end{enumerate}
\end{theorem}

\begin{corollary}\label{C_PartialCononormalsymplectic} 
If a Banach Lie groupoid $\mathcal{G}\tto M$ satisfies one of the following assumptions:
\begin{enumerate}
\item[(i)]   
$\mathcal{G}\tto M$ is a Lie-Poisson groupoid such that $T^\flat \mathcal{G}=T^\prime\mathcal{G}$;
\item[(ii)]  
The kernel of $P$ in restriction to $M$ is a split subbundle of $T_M^\flat \mathcal{G}$  and its range is $T_M\mathcal{G}$;
\item[(iii)]  
$\mathcal{G}\tto M$ is a partial symplectic Lie groupoid;
\end{enumerate}
then $\mathcal{N}^0(M)$ is a split subbundle of $T^\flat \mathcal{G}$. \\
If we identify $\mathcal{AG}^\prime$ with  $\mathcal{N}^a(M)$ then $\mathcal{N}^0(M)$ is a weak subbundle of  $\mathcal{AG}^\flat$ which has a structure of convenient  Lie algebroid.
\end{corollary}

\begin{remark}\label{R_OtherConditionAssumprionTheorem} According to the decomposition  (\ref{eq_TMGDecomposition}), if $\mathcal{G}$ is an Hilbert groupoid, or if $TM$ is finite dimenisonal or finite codimensional, the assumptions of Theorem \ref{T_PartialConormal}  are also satisfied.
\end{remark}

Under the assumptions of Theorem \ref{T_PartialConormal}, if we identify $\mathcal{AG}^\prime$ with  $\mathcal{N}^a(M)$ then $\mathcal{N}^0(M)$ is a weak subbundle of  $\mathcal{AG}^\flat$ which has a structure of Banach Lie algebroid. So we introduce:

\begin{definition}
\label{D_PartialDual} 
Assume that $\mathcal{G}\tto M$ is a {\bf Banach} Lie-Poisson groupoid  such that partial conormal $\mathcal{N}^0(M)$  is a split  Banach  subbundle of $T^\flat_M\mathcal{G}$. 
Then  $\mathcal{N}^0(M)$ is called \emph{the partial dual of $\mathcal{AG}$} and is denoted $\mathcal{AG}^\flat$.
\end{definition}




\begin{proof} [Proof of Theorem \ref{T_PartialConormal} ]${}$\\
(1) Assume that $T^\flat \mathcal{G}=T^\prime \mathcal{G}$. Then $\mathcal{N}^0(M)=\mathcal{N}^a(M)$. At first since ${\bf s}:\mathcal{G}\to M$, is a submersion, as  $M$  is   considered as a  submanifold of $\mathcal{G}$, it follows that
\begin{equation}\label{eq_TMGDecomposition}
T_M\mathcal{G}=T_M^{\bf s}\mathcal{G}\oplus TM=\mathcal{AG}\oplus TM.
\end{equation}
If  $p_M^{\bf s}:T_M\mathcal{G}\to T\mathcal{G}^{\bf s}$ is the associated projection,  we obtain an isomorphism $\widehat{p^{\bf s}_M}:T\mathcal{G}^{\bf s}\to T_M\mathcal{G}/TM$. 
This implies that  $T^\prime_M\mathcal{G}$ is isomorphic to $ T^\prime_M\mathcal{G}^{\bf s}\oplus T^\prime M$. Thus   $\mathcal{N}^a(M)$ can be identified with  to $T^\prime_M\mathcal{G}^{\bf s}$ and is a split subbundle of $T^\flat\mathcal{G}$.  Since  $\mathcal{AG}\equiv T^\prime_M\mathcal{G}^{\bf s}$,  therefore, $\mathcal{N}^0(M)$ can be identified with $\mathcal{AG}^\prime $.\\

(2) For simplicity, assume that $M$ is connected and we use the context of the previous arguments.  

Let $\mathbb{G}^{\bf s}$ the typical fibre of $T_M\mathcal{G}^{\bf s}$ and $\mathbb{M}$ the typical model of $M$. Then, for each $x\in M$, there exists a chart $(U,\psi)$ of $x$ in $\mathcal{G}$ such that $\psi(U)=\bar{U}_1\times\bar{U_2}\subset \mathbb{M}\times \mathbb{G}^{\bf s}$. In this way, we get a trivialization:
\[
T^\prime\psi^{-1}: T^\prime\mathcal{G}_{U}\to (\bar{U}_1\times\bar{U}_2)\times \left( \mathbb{M}^\prime\oplus (\mathbb{G}^{\bf s})^\prime\right).
\]
In this chart, we have $(T^\prime\psi)^{-1}\left(\mathcal{N}^a(M)_{| U}\right)=\bar{U}_1\times \bar{U}_2\times (\mathbb{G}^{\bf s})^\prime$. 
Without loss of generality we may assume that $\bar{U}_1$ and $\bar{U}_2$ are simply connected. Now, since $\mathcal{N}^0(M)$ is a closed subbundle of $T^\flat \mathcal{G}$,  there exists a local trivialization $\Psi$ of $T^\flat \mathcal{G}_U$ such that $\Psi (T^\flat \mathcal{G}_U)=\bar{U}_1\times \bar{U}_2\times \mathbb{G}^\flat$ if $\mathbb{G}^\flat$ is the typical fibre of $T^\flat \mathcal{G}_U$.
 On the other hand, since $\mathcal{N}^0(M)$ is a closed  subbundle of $T_M^\flat \mathcal{G}$, if we identify $\mathbb{G}^\flat$ with $T_x^\flat \mathcal{G}$, we have a decomposition $ \mathbb{G}^\flat=\mathbb{N}^\flat\oplus \mathbb{H}$ where $\mathbb{N}^\flat$ is the typical fibre of  $\mathcal{N}^0(M)$.  Thus we may assume that $\Psi (T^\flat \mathcal{G}_U)=(\bar{U}_1\times \bar{U}_2)\times (\mathbb{N}^\flat\oplus \mathbb{H})$.  
 But since we have $\mathcal{N}^0(M)=\mathcal{N}^a(M)\cap T^\prime_M \mathcal{G}$, we must have $\mathbb{N}^\flat= (\mathbb{G}^{\bf s})^\prime\cap \mathbb{G}^\flat$. Thus, in these local coordinates associated to these trivializations,  the inclusion of $\mathcal{N}^0(M)_{| U}$ into $\mathcal{N}^a(M)_{| U}$ is the map 
\[
( y, \eta, 0)\mapsto (y,\eta,0).
\]
Now consider  another chart $(U',\psi')$ around $x$  and another trivialization $
 \Psi'(T^\flat \mathcal{G}_U)=\bar{U}'_1\times \bar{U}'_2\times \mathbb{G}^\flat$  with the same properties and decompositions.  Since $\mathcal{N}^0(M)$ (resp.  $\mathcal{N}^a(M)$) is a split subbundle of 
 $T^\flat\mathcal{G}$ (resp. $T^\prime\mathcal{G}$), the transition maps $\Psi'\circ \Psi^{-1}$ (resp. $(T^\prime\psi')^{-1}\circ (T^\prime\psi)$) preserves the vector space $(\mathbb{G}^{\bf 
 s})^\prime\cap \mathbb{G}^\flat$ (resp. $ (\mathbb{G}^{\bf s})^\prime$), which ends the proof of (2).\\
 Point (3) is a direct application of Theorem \ref{T_PartialConormalC} (2).
\end{proof}

\begin{proof}[Proof of Corollary \ref{C_PartialCononormalsymplectic}]
This corollary is a direct consequence of Theorem \ref{T_PropertiesSymplecticgroupoid}, (1), Theorem \ref{T_PartialConormalC} (1) and Theorem \ref{T_PartialConormal} if we show that under the assumption (i)   (ii) or (iii) , $\mathcal{N}^0(M)$ is a split subbundle of $T^\flat _M\mathcal{G}$. \\

1. Note that according to the proof of Point (1) of Theorem~\ref{T_PartialConormal},  under assumption (i), $\mathcal{N}^0(M)=\mathcal{N}^a(M)$ and so it is closed split subbundle of $T^\flat _M\mathcal{G}=T^\prime_M\mathcal{G}$. \\

2. Under the assumption (ii),  we denote by $K^\flat M$ the kernel of $P$ in restriction to $T^\flat_M\mathcal{G_M}$ and $H^\flat M$ a a complementary  of $K^\flat M$ in $T_M^\flat \mathcal{G}$. We denote by $P_H$ the restriction of $P$ to $H^\flat M$. Since $P$ is surjective, it follows that $P_H:H^\flat M \to T_M\mathcal{G}$ is an isomorphism. Taking in account the decomposition (\ref{eq_TMGDecomposition}), we set $H^0 M= P_H^{-1}(TM)$ and $K^0 M=P_H^{-1}(\mathcal{AG})$. By construction, for any $\alpha \in K^0 M$ (resp. $\alpha \in H^0 M$), we have $<\alpha, u>=0$,  for all $u\in  TM$ (resp. $u\in  \mathcal{AG})$).  Note that $\mathcal{N}^0(M)=K^\flat M\oplus H^0 M$ and so $\mathcal{N}^0(M)$ is a split subbundle of $T^\flat \mathcal{G}$.\\

3. Under assumption (iii),    $M$  is a particular case of the context of assumption (ii).
\end{proof}

We  are now in situation to prove the following result which generalizes a result in finite dimension for Poisson Lie groupoids (cf. \cite{MaXu94}):
\begin{theorem}\label{T_BialgebroidPoissongroupoid}  
Let  $\mathcal{G}\tto M$ be  a  Banach Lie-Poisson groupoid  such that partial conormal $\mathcal{N}^0(M)$ is a split  Banach subbundle  of $T^\prime_M\mathcal{G}$. Then $(\mathcal{AG}, \mathcal{AG}^\flat)$ is a partial bialgebroid.
\end{theorem}

\begin{proof} 
Since $T^\flat \mathcal{G}$ has a partial Poisson structure, from Proposition \ref{L_dPLambdaLambda},  for any open set $U$ in $\mathcal{G}$,   we have:
\begin{equation}
\label{dPXY2}
d_P[X,Y]=L_Xd_PY-L_Yd_PX
\end{equation}
for all vector fields  $X$ and $Y$ on $U$   considered as $1$-forms on $T^\flat \mathcal{G}$.\\
On the other hand, we have seen in the proof of Theorem \ref{T_PartialConormalC}, 2., for each open set $U$ in $\mathcal{G}$ that the restriction of $[.,.]_P$ to $\mathcal{AG}^\flat_{U\cap M}$ takes values in $\mathcal{AG}^\flat_{U\cap M}$.
According to  (\ref{LXL}) and Proposition  \ref{L_dPLambdaLambda} Point 1, if $X$ is a section of $\mathcal{AG}_{U\cap M}$  it follows that $d_PX$ in restriction to $\mathcal{AG}^\flat_{U\cap M}$  is a $2$-form on  $\mathcal{AG}^\flat_{U\cap M}$.
Finally from (\ref{dPXY}) and Proposition  \ref{L_dPLambdaLambda} Point 2, it follows that the restriction of each member of (\ref{dPXY2})  to  $\mathcal{AG}^\flat_{U\cap M}$  induces a $2$-form on $\mathcal{AG}^\flat_{U\cap M}$ which satisfies relation (\ref{dPXY2}). This  ends the proof according to the definition of a partial bialgebroid.
\end{proof}

\begin{remark}
\label{R_NotTrueInConvenient}
All these previous results are not true in general in the convenient setting. The essential reason is that even if $\mathbb{F}$ is a split convenient subspace of a convenient 
space $\mathbb{E}$ then if $\mathbb{E}=\mathbb{F}\oplus \mathbb{G}$, then the canonical  projection of $\mathbb{G}$ onto $\mathbb{E}/\mathbb{F}$ is not a convenient isomorphism in general.  

In particular,  $\mathbb{F}^*$ is not isomorphic to the annihilator of $\mathbb{G}$ and so $\mathcal{N}^a(M)$ cannot be identified with $\mathcal{AG}^\prime$.\\ 
Thus the whole proof of Theorem \ref{T_PartialConormal}  does not work.
However, in the convenient setting,  if $T^\flat \mathcal{G}=T^\prime \mathcal{G}$ and $P=T^\flat  \mathcal{G}\to T\mathcal{G}$ is a convenient bundle isomorphism, that means that it is a strong symplectic convenient groupoid, so the typical fibre $\mathbb{G}$ of $\mathcal{G}$ is reflexive and, in particular, $\mathcal{AG}$ is isomorphic to  $\mathcal{AG}^\prime$. In fact, we also can show that  $(\mathcal{AG},\mathcal{AG}^\prime)$ is a partial bialgebroid (formally analog proof as Theorem~\ref{T_BialgebroidPoissongroupoid}).  Note that we will see that Theorem~\ref{T_PartialConormal}  is true for direct limit of finite dimensional Poisson groupoids.
\end{remark}

\section{Sub-Poisson Groupoid Structures}
\label{__SubPoissonGroupoidStructures}
In this section, we propose a context  which is a generalization of classical results  in finite dimension to the convenient setting. These concepts  are illustrated by \cite{OJS18}.

\subsection{Partial Symplectic Poisson Structure on the Cotangent Bundle of a Convenient Lie Groupoid}
Let $\mathcal{G}\tto M$ be a convenient  Lie groupoid.  As in finite dimension, the  conormal    $\mathcal{N}^a(M)$ is the subbundle of $T^\prime_M \mathcal{G}$ defined by:
$$\mathcal{N}^a(M)=\{(x,\eta)\in T^\prime_ M\mathcal{G}, \textrm{ such that } <\eta, u>=0,\;\; \forall u\in T_xM\}.$$
It is clearly a closed subbundle of $T^\prime M$. From the decomposition (\ref{eq_TMGDecomposition}), we have 
\begin{equation}\label{eq_VirtualDualAG}
T^\prime_M\mathcal{G}=\mathcal{N}^a(M)  \oplus ( \mathcal{AG})^a
\end{equation}
where $( \mathcal{AG})^a$ is the annihilator of $\mathcal{AG}$. Then $\mathcal{N}^a(M)$ is then  isomorphic to  the dual bundle of $ \mathcal{AG}$. 
 we will call $\mathcal{N}^a(M)$  the{ \bf  dual bundle }    of $ \mathcal{AG}$ and  will be denoted $\mathcal{AG}^\prime$. Note that the typical fibre of $\mathcal{AG}^\prime$ is isomorphic to $\mathbb{M}^a$ and so to $\mathbb{G}^\prime$    and  the typical fibre of $\mathcal{AG}^a$ is isomorphic to $\mathbb{G}^a$ and also $\mathbb{M}^\prime$. As  in finite dimension, we have:

\begin{theorem}
\label{T_GroupoidT'G} 
Let $\mathcal{G}\tto M$ be a convenient  Lie groupoid. Then we have
\begin{enumerate}
\item[1.]  $T^\prime \mathcal{G}\tto\mathcal{AG}^\prime$ has  a groupoid structure.
\item[2.]  $T^\prime \mathcal{G}\tto \mathcal{AG}^\prime $ has a canonical  partial symplectic groupoid and so $(\mathcal{AG},\mathcal{AG}^\prime$) is a  partial Lie bialgebroid.
\end{enumerate}
\end{theorem}

\begin{proof} 
For this proof  we need local charts in this context. For simplicity, we assume that $M$ is connected. Let $\mathbb{G}$ the typical fibre of $T_M^{\bf s}\mathcal{G}$ and $\mathbb{M}$ the typical model of $M$.  Note that since ${\bf s}_\mathcal{B}$ and ${\bf t}_\mathcal{B}$  are diffeomorphisms on $U_1$ and $V_1$ respectively on their range, it follows ${ F_\mathcal{B}:=\bf t}_{\mathcal{B}}\circ {\bf s}_{\mathcal{B}}^{-1}$ is a diffeomorphism from $U_1$ to $V_1$. Therefore  $\mathbb{G}$ is also the typical fibre of $T_M^{\bf t}\mathcal{G}$. 
Fix some $y_0\in M$.  There is a chart $(V,\psi)$ of $y_0 \in\mathcal{G}$  with domain simply connected and such that $\psi(V)=\bar{V}_1\times \bar{V}_2 \subset \mathbb{M}\times \mathbb{G}$ and $\psi(y_0)=(0,0)$. In this way $\psi $ is a pair $(\psi_1,\psi_2) $ and we set $\psi_i(y)=\bar{y}_i$ for $i \in \{1,2\}$ and so $\bar{y}_i:V\to \bar{V}_i$ is a smooth map.\\
If $T^\ast \psi$ denotes the adjoint of $T\psi$, we get a trivialization
\[
(T^\prime\psi)^{-1}: T^\prime\mathcal{G}_{| V}\to \bar{V}_1\times \bar{V}_2 \times ( \mathbb{M}\times \mathbb{G})^\prime.
\]
But we have 
\[
( \mathbb{M}\times \mathbb{G})^\prime=\mathbb{M}^\a\times \mathbb{G}^a \equiv \mathbb{G}^\prime \times\mathbb{M}^\prime
\]
where $\mathbb{E}^a$ denotes the annihilator of any closed subspace $\mathbb{E}$ of $\mathbb{M}\times \mathbb{G}$.  \\
If we set $V_1=\psi^{-1}(\bar{V}_1\times \{0\})$,  then $(V_1,(\psi_1)_{| V_1})$ is a chart around $z\in M$ and,  in this chart,  we have:  

$(T^\prime\psi)^{-1}(\mathcal{AG}^a_{| V_1})=\bar{V}_1\times  \mathbb{G}^\prime$,

$(T^\prime\psi)^{-1}(\mathcal{AG}^\prime_{| V_1})=\bar{V}_1\times  \mathbb{M}^\prime$

On the one hand, for any $\xi\in T^\prime M_{| V}$, $(T^\prime\psi)^{-1}(\xi)$ can, be written a a pair $(\bar{\xi}_1,\bar{\xi}_2)$ with $\bar{\xi}_1:V\to \mathbb{G}^\prime $ and $\bar{\xi}_2:V\to \mathbb{M}^a\prime$ are smooth maps. \\
 
 Since ${\bf t} : \mathcal{G}\to M$ is a submersion, around  any $g_0\in \mathcal{G}$, we have  a chart $(\mathcal{V},\Psi)$ with simply connected domain and    such that 
$$\Psi(\mathcal{V})=\bar{\mathcal{V}}_1\times \bar{\mathcal{V}}_2\subset \mathbb{M}\times\mathbb{G}$$
and $\Psi( g_0)=0$.  Again $\Psi$ is a pair $(\Psi_1,\Psi_2)$.  If  $\mathcal{V}_1=\Psi^{-1}(\bar{\mathcal{V}}_1\times\{0)\})$ and $\mathcal{V}_2=\Psi^{-1}(\{0\}\times \bar{\mathcal{V}}_2$, the expression of   ${\bf t}_{| \mathcal{V}}$ in the associated  local coordinates is the map $(\bar{y},\bar{g})\mapsto \bar{y}$.\\

Assume the ${\bf t}(g_0)=y_0$. After shrinking the domains $\mathcal{V}$ and $V$,  if necessary,  and by  composing with a local diffeomorphism in $\mathbb{M}$ which fixes $0\in \mathbb{M}$ and a diffeomorphism of $\mathbb{G}$ which fixes $0\in \mathbb{G}$, we can choose these charts $(V, \psi)$ and $(\mathcal{V},\Psi)$ such that  $\bar{\mathcal{V}}_1=\bar{V}_1$. In the associated 
coordinates, if $\bar{y}=\psi_1({\bf t}(g))$ and $\bar{g}_2=\Psi_2(g)$,  the expression of ${\bf t}$ is  still the map $(\bar{y}, \bar{g}_2)\mapsto \bar{y}$ where $\bar{y}:\mathcal{V}\to \mathbb{M}$ and $\bar{g}_2:\mathcal{V}\to \mathcal{G}$ are smooth maps.

At the level of the tangent bundle of $\mathcal{G}$, according to the previous notations, if  we have a trivialization
\[
T\Psi: T\mathcal{V}\to (\bar{{V}}_1\times \mathbb{M})\times( \bar{\mathcal{V}}_2\times \mathbb{G})
\] 
and the following commutative diagram
\begin{equation}
\label{eq_DiagramTPSi}
\begin{tikzcd}
T\mathcal{G}_{|\mathcal{V}} \arrow[r, "T\Psi"] \arrow[d, "T{\bf t}"]
&  (\bar{{V}}_1\times \mathbb{M})\times( \bar{\mathcal{V}_2} \times \mathbb{G}) \arrow[d, "pr_{V_1}"] \\
TM_{| V}\arrow[r,  "T\psi_{| TV} " ]
& \bar{V}_1\times \mathbb{M}
\end{tikzcd}
\end{equation}
with  $T\Psi\left((g, Y)\right)=\left((\bar{y},\bar{Y}_1),(\bar{g}_2,\bar{Y}_2)\right)$ and 
$T\Psi\circ T{\bf t}\left((g, Y\right))=(\bar{y}, \bar{Y}_1))$
where $\bar{y}: \mathcal{V} \to  \mathcal{V}_1$ and $\bar{Y}_1: \mathcal{V}\to \mathbb{M}$ are smooth maps. \\

Now, ${\bf s}$ is also a submersion and so we have an analogue argument of chart $(U,\phi)$ of $x_0={\bf s}(g_0)$ such that $\Phi(U)=\bar{U}_1\times \bar{U}_2\subset\mathbb{M}\times\mathbb{G}$ and a chart $(\mathcal{U}, \Phi)$ around $g\in \mathcal{G}$ such that $\Phi(\mathcal{U})=\bar{U}_1\times\bar{\mathcal{U}}_2$ and with the same properties of the pairs $(V, \psi)$ and $(\mathcal{V},\Psi)$.
 As we have seen (modulo the choice of a bissection $\mathcal{B}$),  there exists a diffeomorphism $F:U_1\to V_1$ such that $F((y_0)=x_0$.\\

 On the one hand, for any $\xi$, $(T^\prime\psi)^{-1}(\xi)$ can be written as a pair $(\bar{\xi}_1,\bar{\xi}_2)$ with $\bar{\xi}_1:V_1\to \mathbb{G}^\prime $ and $\bar{\xi}_2:V_1\to \mathbb{M}^\prime$. \\

On the  other hand,  if $1_M\to \mathcal{G}$ is the  unit map for $T\mathcal{G}\tto M$ and will be identified as a submanifold of $\mathcal{G}$, then each $ Z\in T_x \mathcal{G}$ can be written in a unique way as $Z=(a, X)$ where $X\in T_xM$ and $a\in\mathcal{AG}_x$. Then the unit map $\widetilde{1}:\mathcal{AG}^\prime\to T^\prime\mathcal{G}$ is defined by 
$$<\widetilde{1}(\xi) ,(a, X)>=<\xi, a>$$

According to the previous notations, it follows that in local coordinates  $\widetilde{\bf 1}$ is the  identity  of $ \bar{V}_1\times  \mathbb{M}^\prime$. It follows that 
$1_M\to T^\prime \mathcal{G}$ is an embedding  and so $\widetilde{1}(\mathcal{AG}^\prime)$ is the subbundle of $T_M\mathcal{G}$  which is the annihilator of $TM$.

{\it Without loss of generality, under the previous context,} we may assume that $\mathcal{V}_1=\bar{V}_1\subset \mathbb{M}$ and $\mathcal{V}_2=\bar{\mathcal{V}}_2\subset  \mathbb{G}$. By the way and $T_\mathcal{V}\mathcal{G}:=T\mathcal{G}_{|\mathcal{V}}=(V_1\times  \mathcal{V}_2\times\mathbb{M}\times \mathbb{G})$ and $T{\bf t}_{| T_\mathcal{V}\mathcal{G}}=pr_{V_1}$. 
 In particular each $Z\in T_V\mathcal{G}$ is a pair $(Z_1, Z_2)$ where $Z_1:V\to \mathbb{M}$ and $Z_2:V\to \mathbb{G}$ are smooth maps. Moreover, according to diagram (\ref{eq_DiagramTPSi}), $T{\bf t}=p_{V_1}$.

In a dual way, we have  $T^\prime_{\mathcal{V}}\mathcal{G}=V_1\times  \mathcal{V}_2 \times \mathbb{G}^\prime\times \mathbb{M}^\prime$ and and in particular:
\begin{equation}
\label{eq_Dual}
(\mathcal{AG}^a_{| V_1})=\bar{V}_1\times  \mathbb{M}^\prime  \textrm{ and } \mathcal{AG}^\prime_{| V_1}=\bar{V}_1\times  \mathbb{G}^\prime.
\end{equation}

Therefore, for any $g=(y,g_2)\in \mathcal{V}$ any $\xi\in T_g^\prime \mathcal{V}$ can be written as a pair $(\xi_1,\xi_2)\in \{(y,g_2)\}\times \mathbb{G}^\prime\times \{(y,g_2)\}\times \mathbb{M}^\prime$ where  $\xi_1$ is a smooth map from 
$\mathcal{V}$ to $\mathbb{G}^\prime$ and $\xi_2$ a smooth map from $\mathcal{V}$ to $\mathbb{M}^\prime$. 

Also, for any $\xi\in T^\prime_{U_1} \mathcal{G}$ , its components on $\mathcal{AG}_{| U_1}$ is defined by its values on any section $a$ of  $\mathcal{AG}_{| U_1}$.\\
 
 On the other hand, for a section $b$  of $\mathcal{AG}_{| V_1}$, that is $b: V_1\to \mathbb{G}$, then $TR_g(b)$ belongs to $T{\bf s}^{-1}({\bf s}(g))$ and, for any section $a$ of $\mathcal{AG}_{| U_1}$, then $a-T{\bf t}(a)$ is a 
 section of $\ker T{\bf t}_{| U_1}$ and so $TL_g(a-\rho(a))$ belongs to $T{\bf t}^{-1}({\bf t}(g))$. Moreover,  the maps $((b,g)\mapsto TR_g(b)$ and $ (a,g)\mapsto TL_g(a-T{\bf t}(a))$ are locally smooth in the convenient setting (one can use smooth curves).\\
 
{\it According to previous arguments, as in finite dimension, the source and the target} 
$\widetilde{\bf s}, \widetilde{\bf t}: T^\prime \mathcal{G}\tto \mathcal{AG}^\prime$    are  well defined  in the following way:

\begin{equation}
\label{eq_bfsbft}
<\widetilde{\bf s}(\xi),a>(g)=<\xi,TL_g({a-\rho(a))}>  \textrm{ and } <\widetilde{\bf t}(\xi),b>=<\xi,TR_g (b)>,
\end{equation}
for all $\xi\in T_g^\prime\mathcal{G}$, $a \in \mathcal{AG}_{{\bf s}(g)}$ and $b \in \mathcal{AG}_{{\bf t}(g)}.$\\

{\it We must show that and $\widetilde{\bf s}$ and $\widetilde{\bf t}$ are submersions.}  
We first show that $\widetilde{\bf t}$ is a  submersion.\\
In fact, in the previous context, if $\xi$ belongs to $T_g^\prime \mathcal{G}$, then $\xi$ is a pair $(\xi_1,\xi_2)$ where  $\xi_1$ is a smooth map from $\mathcal{V}$ to $\mathbb{G}^\prime$ and $\xi_2$ a smooth map from $\mathcal{V}$ to $\mathbb{M}^\prime$. But  we have $<\widetilde{\bf t}(\xi),b>=<\xi, TR_g(b)=<\xi_1,TR_g(b)>$  and  
since $TR_g(b)$ belongs to $\mathbb{G}$ and $\xi_2$ belongs to $\mathbb{M}^\prime=\mathbb{G}^a$. It follows that $\widetilde{\bf t}$ takes values in $\mathbb{G}^\prime\equiv 
\mathcal{AG}^\prime_g$ which is smooth since for any smooth curve $\g:\R \to V_1$ the map $t\mapsto R^*_{g(\g(t)}(\xi(\g(t)$  is a smooth  section $\g^*( T^\prime {\bf t}^{-1}(M))
\equiv \R \times \mathbb{M}^\prime$
 Note that, for any $\eta\in \mathcal{AG}^\prime_{\bf t}(g)$,  if we set  $\xi=(R_g^{-1})^*\eta$, then $\widetilde{\bf t}(\xi)=\eta$ and so $\widetilde{\bf t}$ is surjective. Therefore, $\widetilde{\bf t}$ is a submersion.\\
 
Let ${\bf i}:\mathcal{G}\to \mathcal{G}$ be the inversion in $\mathcal{G}$. Then $T{\bf i}$ is the inversion for the Lie  groupoid structure $T\mathcal{G}\tto TM$. The adjoint $\widetilde{\bf i}:={\bf i}^\ast:=T^\prime {\bf i}$ is a diffeomorphism of $T\prime \mathcal{G}$. Since ${\bf i}\circ L_g=R_g$ it follows easily that $R^*_g=L_g^* \circ \widetilde{\bf i}$,  and from the definition of $\widetilde{\bf s}$, it is easy to see that  $\widetilde{\bf s}=\widetilde{\bf t}\circ \widetilde{\bf i}$ and the proof is left as an exercise to the reader. From  Proposition \ref{P_BasicPropertiesOfTopologicalGroupoids}, it follows that $\widetilde{s}$ is a submersion.\\

As in finite dimension, the multiplication $\widetilde{\bf m}$ on $T^\prime \mathcal{G}$ is given by 
\[
<\xi\bullet\eta , X\bullet Y>:=<\xi, X>+<\eta, Y>
\]
if $\widetilde{\bf t}(\xi)=\widetilde{\bf s}(\eta)$. 
Note that from Proposition~\ref{P_BasicPropertiesOfTopologicalGroupoids},  we know that the set of such pairs $(\xi,\eta)$ is a submanifold  $(T^\prime\mathcal{G})^{(2)}$ of $T^\prime\mathcal{G}\times T^\prime\mathcal{G}$.  If ${\bf s}(g)={\bf t}(h)$, on the one hand,  the bilinear map $T_g\mathcal{G}\times T_h\mathcal{G}\to T_{\bf m}(g,h)\mathcal{G}$: $(X,Y)\mapsto X+Y$ is a bounded surjective map. On the other hand, 
  the graph of this bilinear map $T_g\mathcal{G}\times T_h\mathcal{G}\to T_{\bf m}(g,h)\mathcal{G}$: $(X,Y)\mapsto X+Y$ can be considered a a  relation from $T_g\mathcal{G}\times 
  T_h\mathcal{G}$ to $ T_{\bf m}(g,h)\mathcal{G}$. From the surjectivity of the multiplication $T{\bf m}$  in $T^\prime \mathcal{G}$ and the definition of the multiplication $\widetilde{\bf 
  m}$ in $T^\prime\mathcal{G}$, it follows that	the graph of $\widetilde{\bf m}$ in $T^\prime\mathcal{G}$ is  $Ann(G_{T{\bf m}})$. 
The associativity of $T{\bf m}$ implies the associativity of $\widetilde{\bf m}$.\\  
Finally, under the previous local context (even in the convenient setting), from its definition,  it is easy to see  that the map $(\xi\eta)\mapsto \xi\bullet\eta$ is a smooth map from $(T^\prime\mathcal{G})^{(2)}$ to  $T^\prime\mathcal{G}$ which ends the proof.

\begin{remark}
\label{R-GrapkMultiplicationTprimeG}
$T^\prime \mathcal{G}\tto \mathcal{AG}^\prime$ is a $\mathcal{VB}$-groupoid dual of $\mathcal{VB}$-groupoid $T\mathcal{G}\tto TM$. 
\end{remark}

2. Let $\omega$ be the canonical $2$-form on  $T^\prime \mathcal{G}$. Therefore, for any $(\xi,g)\in T^\prime \mathcal{G}$, any $X\in  T_{(\xi,g)}T^\prime \mathcal{G}$ can be written as a pair $(v,v')$ where $v$ belongs to $T_g\mathcal{G}$ and $v'$ in $T^\prime \mathcal{G}$. We then have:
\begin{equation}
\label{eq_omega}
\omega_{(\xi,g)}(X,Y)=<w',v>-<v',w>,
\end{equation}
if $X=(v,v')$ and $Y=(w,w')$. Let $T^\flat (T^\prime \mathcal{G})$ be the range of $\omega^\flat: T(T^\prime\mathcal{G})\to T^\prime (T^\prime\mathcal{G})$. Thus, $P:=(\omega^\flat)^{-1}:T^\flat (T^\prime \mathcal{G})\to T (T^\prime \mathcal{G})$ is an isomorphism.\\

It follows that $P$ is a Poisson anchor and so $T^\prime \mathcal{G}\tto \mathcal{N}^a (M) $ is  partial symplectic groupoid.\\
Finally,  Proposition~\ref{L_dPLambdaLambda}  implies that  $(\mathcal{AG},\mathcal{AG}^\prime$) is a  partial Lie bialgebroid.
\end{proof}

\subsection{Sub-Poisson Structure on a  Subgroupoid of  the Cotangent Bundle of a Lie Groupoid}
In the Banach setting used in \cite{OJS18}, the authors introduce a notion of  sub-Poisson structure (p.~29). The following definition  is  an adaptation of the previous one: 
\begin{definition}
\label{D_SuBPoissonGroupoid} 
A convenient Lie groupoid $\mathcal{G}\tto M$ has a \emph{sub-Poisson groupoid structure}\index{sub-Poisson groupoid} if  there exists a Poisson anchor $P:T^\flat \mathcal{G} \to T\mathcal{G}$ such that $T^\flat \mathcal{G}$ can be provided with a convenient subgroupoid  structure $T^\flat \mathcal{G} \tto \mathcal{AG}^\flat$    of $T^\prime \mathcal{G}\tto \mathcal{AG}^\prime$ where $\mathcal{AG}^\flat$ is a weak subbundle of $\mathcal{AG}^\prime $ and there exists a    bundle  morphism 
 $P^\flat:\mathcal{AG}^\flat \to TM$ such that the following diagram is commutative:
\begin{equation}
\label{eq_DiagralSubPoisson}
 \xymatrix{
  T^\flat\mathcal{G}\ar@<-.5ex>[d] \ar@<.5ex>[d]\ar[r]^P&T\mathcal{G}\ar@<-.5ex>[d] \ar@<.5ex>[d]\\
 \mathcal{AG}^\flat \ar[r]^{P^\flat} &TM}
 \end{equation}
and  $P$ is a groupoid  morphism over $P^\flat$.\end{definition}

We have the following  characterization:
\begin{proposition}
\label{P_SubPoissonGroupoid}
Consider a convenient Lie groupoid $\mathcal{G}\tto M$ provided with a Poisson anchor $P:T^\flat \mathcal{G}\to T\mathcal{G}$.  Assume  that the partial conormal set $\mathcal{ N}^0(M)$ is a  closed sub-bundle of $T^\flat_M\mathcal G$,  and, if  $P^\flat $ denotes the restriction of $P$  to $\mathcal{ N}^0(M)$,   the following diagram is commutative.
\begin{equation}
\label{eq_DiagralSubPoissonProp}
 \xymatrix{
  T^\flat\mathcal{G}\ar@<-.5ex>[d]_{\widetilde{\bf s}_{T^\flat \mathcal{G}}}\ar@<.5ex>[d]^{\widetilde{\bf t}_{| T^\flat \mathcal{G}}}\ar[r]^P&T\mathcal{G}\ar@<-.5ex>[d]_{T{\bf s}}\ar@<.5ex>[d]^{T{\bf t}}\\
 \mathcal{N}^0(M)\ar[r]^{P^\flat} &TM}
\end{equation}
Then the convenient Lie groupoid $\mathcal{G}\tto M$ is a  sub-Poisson groupoid  if and only if       $\mathcal{G}\tto M$  has a partial Poisson structure.\\
\end{proposition}

\begin{corollary}
\label{C_RegularPoissonGroupoid}  
Under any of the assumptions of  Corollary~\ref{C_PartialCononormalsymplectic},  any  Lie groupoid is a sub-Poisson groupoid if and only if it  is a partial Poisson groupoid.
\end{corollary}

\begin{remark}
\label{R_SubPoissonPartial Poisson}
We can consider Proposition~\ref{P_SubPoissonGroupoid} as some kind of generalization of a classical characterization  of Poisson groupoid  (cf. \cite{MaXu94}, Proposition~8.1).  From Proposition~\ref{P_SubPoissonGroupoid}, a sub-Poisson  groupoid has a partial Poisson structure.  In particular, all examples of partial Poisson -Lie groupoids given in section \ref{___SomePropertiesOfBanachLieGroupoids} are sub-Poisson Lie groupoids.  But,  a  priori,  the converse is not  true in general without the assumption that $\mathcal{ N}^0(M)$ is a  closed sub-bundle of $T^\flat_M\mathcal G$.  Unfortunately we have no example of  partial Poisson-Lie groupoid for which $\mathcal{ N}^0(M)$ {\bf  is not } a  closed sub-bundle of $T^\flat_M\mathcal G$.
\end{remark}
\begin{proof} [Proof of Proposition \ref{P_SubPoissonGroupoid}]${}$ 

Assume that  $T\mathcal{G} \tto \mathcal{AG}^\prime $ is a sub-Poisson Lie groupoid.  Since $P$ is a groupoid morphism, it  satisfies the  condition 
\begin{equation}
\label{eq_CompatibilityPm}
   P\circ \widetilde{\bf m}(\xi,\eta)=T{\bf m}(P(\xi), P(\eta)).
\end{equation}
for all $(\xi,\eta)\in T^\flat\mathcal{G}^{(2)}$. 
   Now, since we have 
\[
<(\xi,\eta,-\widetilde{\bf m}(\xi,\eta) ), (X,Y,T{\bf m}(X,Y)>=0
\] 
for all $(\xi,\eta)\in T_{(g,h)}^\flat\mathcal{G}^{(2)}$ and $(X,Y) \in T_{(g,h)}\mathcal{G}^{(2)}$, the same is true  for all $(\xi,\eta)$ and $(\xi',\eta')$ in $T_{(g,h)}^\flat\mathcal{G}^{(2)}$ as $T^\flat \mathcal{G}\tto \mathcal{AG}^\flat$ is a sub-groupoid of $T^\prime \mathcal{G}\tto 
  \mathcal{AG}^\prime$.  By application of this relation to the first member of  (\ref{eq_CompatibilityPm}), it follows that $T^\prime \mathcal{G}\tto \mathcal{AG}^\prime$ is a partial Poisson groupoid. On the other hand, $\mathcal{AG}^\flat $ is a subgroupoid of $\mathcal{AG}^\prime$ (in particular $\mathcal{N}^0(M)=\mathcal{AG}^\flat$ is a closed subbundle of $T_M^\flat \mathcal{G}$) and so $P^\flat$ must be the restriction of $P$ to $\mathcal{AG}^\flat $ and the Diagram (\ref{eq_DiagralSubPoisson}) is then the same as the diagram (\ref{eq_DiagralSubPoissonProp}). \\

Conversely, assume that $\mathcal{G}\tto M$  has a partial Poisson structure.  In the one hand, since the graph of  ${\bf m}$ is coisotropic, this implies that the restriction $\widetilde{\bf m}^\flat$ of $
   \widetilde{\bf m}$ to $T_{g,h)}^\flat\mathcal{G}^{(2)}$ takes values in $T^\flat \mathcal{G}$,  and its  graph is a closed submanifold of $\{T^\flat\mathcal{G}\}^3$. It follows that the projection of the graph of $\widetilde{\bf m}^\flat$ onto $T^\flat \mathcal{G}$ is smooth which means that $\widetilde{\bf m}^\flat$ is smooth. Note that the property of coisotropy of the graph means 
   $$<(\xi, P(\eta),-P(\xi\bullet\eta), \xi', \eta',-\xi'\bullet\eta'>=0$$
for  all $(\xi, \eta)$ and $(\xi',\eta')$ in $(T^\flat\mathcal{G})^{(2)}$.  But, for any $g\in \mathcal{G}$, the Annihilator of $T^\flat\mathcal{G}$ is zero.  Since $\widetilde{\bf m}$ is surjective, it follows that This implies that $  P(\xi\bullet\eta)=(P(\xi) +P(\eta))$ and so $P$ satisfies the relation (\ref{eq_CompatibilityPm}).\\

On the other hand, as  $M$ is coisotropic and  $\mathcal{ N}^0(M)$ is a  closed sub-bundle of 
   $T^\flat_M\mathcal G$, from Theorem~\ref{T_PartialConormalC}, $\mathcal{N}^0(M)$ has a structure of Lie algebroid whose anchor is precisely $P^\flat$. In fact, it is a subalgebroid of   $\mathcal{AG}^\prime$  (cf. Proof of Theorem~\ref{T_PartialConormalC}).\\
Now, from Diagram~(\ref{eq_DiagralSubPoissonProp}), we have  $P^\flat(\widetilde{\bf t}_{|T^\flat\mathcal{G}})(\xi)= T{\bf t}\circ P(\xi)$ for all $\xi\in T^\flat\mathcal{G}$.   Recall that, 
   for any $\xi \in T_g^\prime \mathcal{G}$, we have $ \widetilde{\bf t}(\xi)(R_g^*(\xi))=\widetilde{\bf t}(\xi)$ which belongs to $(\mathcal{AG)}^\prime_{{\bf t}(g)}$. Therefore,  $\widetilde{\bf t} (\circ R_g^*)^{-1}=Id_{\mathcal{AG}^\prime_{{\bf t}(g)}}$ for all $ g\in \mathcal{G}$. This implies  that
    $$\widetilde{\bf t}_{| T^\flat \mathcal{G}}\circ (R_g^*)^{-1}=Id_{\mathcal{AG}^\flat_{{\bf t}(g)}}$$
     and it follows easily that $\widetilde{\bf t}_{| T^\flat \mathcal{G}}$ is a submersion. The same arguments can be applied to $\widetilde{\bf s}_{| T^\flat \mathcal{G}}$. Finally, $T^\flat\mathcal{G}\tto \mathcal{AG}^\flat$ is a subgroupoid of $T^\prime\mathcal{G}\tto \mathcal{AG}^\prime$
As  $P$ satisfies the relation (\ref{eq_CompatibilityPm}) as we have seen in the first part of the proof, it follows that $P$ is a groupoid morphism, which ends the proof.
\end{proof}

\subsection{The Cotangent Bundle of a Lie group as Lie  Groupoid} 
\label{___TheCotangentBundleOfALieGroupAsLieGroupoid}

\subsubsection{Some Preliminaries}\label{___PreliminariesCotangentBundleOfALieGroupAsLieGroupoid}
If $\mathsf{G}$ is a Lie group,  we have a natural  right (resp. left) action on  its cotangent bundle $T^\prime \mathsf{G} $   given by the map  
$(g, (h,\xi))\mapsto (gh,T_e^* L_g(\xi)$ (resp. $ (g, (h,\xi))\mapsto (gh, T_e^*R_h(\xi)$) and so a coadjoint action $\mathsf{G}$ on $T^\prime \mathsf{G}$.

This gives rise to an isomorphism $\widetilde{R}$ (resp. $\widetilde{L}$) from  $\mathfrak{g}^\prime\times G$ (resp. $\mathsf{G}\times\mathfrak{g}^\prime$) to  $T^\prime \mathsf{G}$, defined by $\widetilde{R}(g,\xi) =(\widetilde{\bf s}_g(\xi),g)$ (resp.  $\widetilde{L}(g,\xi) =(g,\widetilde{\bf t}_g(\xi))\;$, 
which respects the submersion $\widetilde{\bf s}$ (resp. $\widetilde{\bf t}$),  that is the following  diagrams are commutative:

\begin{equation}
\label{eq_RL}
 \xymatrix{
 \mathfrak{g}^\prime\times\mathsf{G} \ar[d] \ar[r]^{\widetilde{R}} &T^\prime \mathsf{G}\ar[d] ^{\widetilde{\bf s}}\\
\mathfrak{g}^\prime \ar[r]_{Id_{\mathfrak{g}^\prime}}         &\mathfrak{g}^\prime}
\;\;\;  \xymatrix{
\mathfrak{g}^\prime\times \mathsf{G} \ar[d] \ar[r]^{\widetilde{L}} &T^\prime \mathsf{G}\ar[d] ^{\widetilde{\bf t}}\\
\mathfrak{g}^\prime \ar[r]_{Id_{\mathfrak{g}^\prime}}         &\mathfrak{g}^\prime}
\end{equation}
Note that arguments used in the context of finite dimension in  \cite{AGMM90}, 2.3,  also true in the convenient setting (see also \cite{KrMi97} for such a context) for regular convenient Lie groups. Thus, from now on, we assume that $\mathsf{G}$ is  provided with an exponential map  for instance, $\mathsf{G}$ is a regular convenient Lie group.  
We have a coadjoint action of $\mathsf{G}$ on $\mathfrak{g}^\prime$, characterized by $g\mapsto \mathsf{Ad}^*_g(\xi)$.
  As $\widetilde{L}$ is an isomorphism  from $T^\prime\mathsf{G}$ to $\mathsf{G}\times\mathfrak{g}^\prime$,  we can provide $T^\prime \mathsf{G}$ with a semi-direct product given by:
\[
(g,\xi).(h,\eta):=(gh, \mathsf{Ad}^*_{h^{-1}}(\eta)+\xi).
\]
The coadjoint action of $\mathsf{G}$ on $\mathfrak{g}^\prime$ induces a smooth action of  the Lie algebra $\mathfrak{g}$ on $\mathfrak{g}^\prime$ in the following way:
 any $X\in \mathfrak{g}$  gives rise to a vector field $X^\sharp$ on $\mathfrak{g}^\prime$ whose value at $\xi\in \mathfrak{g}^\prime$ is characterized by 
\begin{equation}
\label{eq_FondamentalCoadjointVF}
 \langle X^\sharp_\xi, Y\rangle:=\langle\xi,[X,Y]\rangle=\mathsf{ad}^*_X(Y)
\end{equation}
 for any $Y\in\mathfrak{g}$.
  Thus, for any $\xi \in \mathfrak{g}^\prime$ we can define  a linear $2$ form $\omega_\xi$ on $\mathfrak{g}$ by 
\begin{equation}
\label{eq_KKSform}
 \sigma_\xi(X,Y)=\langle\xi,[X,Y]\rangle=\mathsf{ad}^*_X(Y).
\end{equation}
The kernel of $\sigma_\xi$ is the set of vectors $X$ such that $ \langle X^\sharp_\xi, Y\rangle=0$,  for all $Y\in \mathfrak{g}$. 
Thus, this kernel is the Lie algebra  $\mathfrak{g}_\xi$ of  the stabilizer $\mathsf{G}_\xi$  of the coadjoint action at $\xi$. Let $\mathcal{O}_\xi=\{\mathsf{Ad}_g(\xi),\; g\in \mathsf{G}\}$ be the coadjoint orbit through $\xi$. Unfortunately, in general, $\mathcal{O}_\xi$ does not have a  convenient manifold structure. Even if $\mathsf{G}$ is a Banach Lie group, $\mathcal{O}_\xi$ does not have a Banach manifold structure.  It is  true  if $\mathfrak{g}_\xi$ is a supplemented Banach subalgebra of $\mathfrak{g}$. Of course, it is always true in the Hilbert setting for a finite dimensional Lie group $\mathsf{G}$. In particular, in this last case, $\sigma$ is a symplectic form in restriction to each orbit $\mathcal{O}_\xi$ and  defines a Poisson structure on $\mathfrak{g}
 ^\prime$ which is the famous \emph{Kostant-Kirolov-Souriau Poisson structure} on $\mathfrak{g}^\prime$. We will see in the next section a generalization of such a result in the Banach setting.\\
 
As in \cite{AGMM90}, we consider the map $\kappa^r:T\mathsf{G}\to \mathfrak{g}$ (resp. $\kappa^l: T\mathsf{G}\to 
\mathfrak{g}$) defined by $\kappa_g^r=T_eR_{g^{-1}}$ (resp. $ \kappa_g^l=T_eL_{g^{-1}})$.  Note that these maps are $1$-forms with values in $ \mathfrak{g}$ and note that they  satisfy the Maurer-Cartan equations: 
\[
d\kappa^r+\dfrac{1}{2}\kappa^r\wedge\kappa^r=0 \text{~~(resp. } d\kappa^l-\dfrac{1}{2}\kappa^l\wedge \kappa^l=0 \text{)}.
\]

Note that $\kappa^r$ (resp. $\kappa^l$) is right invariant (resp. left invariant). In fact, 
we have   $\kappa^l={\bf s}$ and $\kappa^r={\bf t}$.\\

We consider the following diagram:

\[
\xymatrix{
T(T^\prime\mathsf{G}) \ar[d]_{\pi_{T^\prime \mathsf{G}}} \ar[r]^{T\pi_\mathsf{G}} &T\mathsf{G}\ar[d]^{\pi_\mathsf{G}} \\
T^\prime\mathsf{G} \ar[r]_{\pi }        &\mathsf{G}}
\]
Under these notations, we have:
 
 $\widetilde{R}={\bf s}\times \pi$ (resp. $\widetilde{L}=\pi\times{\bf t}$);

 $\widetilde{\bf s}_g=((\kappa^l_g)^{-1})^*$ (resp. $\widetilde{\bf t}_g=((\kappa^r_g)^{-1})^*$).\\
   
 
 Let $\lambda$ be the Liouville form on $T^\prime \mathsf{G}$. 
 Then the canonical symplectic form on $T^\prime \mathcal{G}$ is $\omega=-d\lambda$. 
To show that  $T^\prime\mathsf{G}$ has a structure of partial symplectic groupoid, if $pr_1$  (resp. $pr_2$) is the projection of $T^\prime \mathsf{G}\times T^\prime \mathsf{G}$ on the $1^{\text{st}}$ factor (resp. $2^{\text{nd}}$ factor), from Remark ~\ref{R_LambdaMultiplicative} 
we have to show that
\begin{equation}
\label{eq_omegaG}
  pr_1^*\omega+pr_2^*\omega={\bf m}^*\omega 
\end{equation}
Since $\omega=-d\lambda$, it is sufficient to show the relation (\ref{eq_omegaG}) for $\lambda$.\\
  Given $\xi\in T^\prime_g\mathsf{G}$ (resp. $\eta\in T_h^\prime \mathsf{G}$)) and $\Xi\in  T_\xi(T^\prime\mathsf{G})$ (resp. $\Theta\in T_\xi(T^\prime\mathsf{G})$), we have:
  $$  T^\prime\mathsf{G}(\Xi)\in T_g^\prime  \mathsf{G} \;\;\;\;({\rm resp. }\;T^\prime\mathsf{G}(\Theta)\in T_h^\prime\mathsf{G}).$$

Now, from the expression of the multiplication on the tangent space of a groupoid,   
we have 
$T_{(g,h)}\widetilde{\bf m}(\Xi,\Xi_2)=\Xi_1+\Xi_2$ which gives rise to the following:
$$\widetilde{\bf m}^*_{(g,h)}\lambda(\Xi_1, \Xi_2)=\lambda_g(\Xi_1)+\lambda_h(\Xi_2).$$
This  implies that $T^\prime \mathsf{G}$ has a partial symplectic structure defined by the  Poisson anchor $P=(\omega^\flat	)^{-1}:T^\flat(T^\prime \mathsf{G}):=	\omega^\flat(T(T^\prime\mathsf{G}))\to T(T^\prime \mathsf{G})$.					

\subsubsection{Properties of the Cotangent Bundle of a  Lie Group as Lie Groupoid} 
\label{____PropertiesOfTheCotangentBundleOfALieGroupAsLieGroupoid}
We consider the structure of partial  symplectic groupoid on $T^\prime\mathsf{G}\tto \mathfrak{g}^\prime$ as described in the previous section. 

 Essentially, by  application of Theorem~\ref{T_PropertiesSymplecticgroupoid}, we have: 
 
 \begin{theorem}
 \label{T_Cotangent bundleG} 
 Let $\mathcal{G}$ be a regular convenient Lie group. Its cotangent bundle $T^\prime \mathsf{G}$ has a structure of partial symplectic Lie groupoid over its  dual Lie algebra $\mathfrak{g}^\prime$ and we have:
\begin{enumerate}
\item[(1)] 
$\mathfrak{g}^\prime$ considered as the zero section of $T^\prime \mathsf{G}$ is Lagrangian in $T^\prime\mathsf{G}$.
\item[(2)] 
The inversion in $T^\prime\mathsf{G}$ is an anti-symplectic map. 
\item[(3)] 
Let $\omega$ the canonical symplectic  $2$-form on $T^\prime \mathsf{G}$ and the associated partial  Poisson structure $P:T^\flat(T^\prime\mathsf{G}):=\omega^\flat(\mathsf{G})
\to T\mathsf{G}$. In fact,  $T^\prime \mathsf{G}\tto \mathfrak{g}^\prime$ has a {\bf sub-Poisson groupoid structure}. In particular, $T^\flat(T^\prime\mathsf{G})\tto\mathfrak{g}^\prime\times \mathfrak{g}^\flat$ is a Lie groupoid (cf.  left vertical part of Diagram (\ref{eq_DiagralSubPoissonProp})).
\item [(4)] 
For each $g\in\mathcal{G}$, each vector space $T_g(\widetilde{\bf s}^{-1}(g))$ and $T_g(\widetilde {\bf t}^{-1}(g))$ are orthogonal symplectic subspaces of the  linear symplectic space $(T_g(T^\prime\mathsf{G}),\omega)$.
\item[(5)] 
There exists a unique Poisson anchor   $P_{\mathfrak{g}^\prime}=T{\bf s}\circ P\circ T\widetilde{\bf s}=T{\bf t}\circ P\circ T\widetilde{\bf t}$ on $\mathfrak{g}$ such that ${\bf s}$ (resp. ${\bf t}$) is a Poisson morphism (resp. anti-Poisson morphism). More precisely, if $\sigma$ is  the induced $2$ form on $\mathfrak{g}^\prime$ by $\omega$ we have:
\begin{enumerate}
\item[(a)] $\sigma$  is the Poisson bivector associated to $P_{\mathfrak{g}^\prime}$.
\item[(b)] For each coadjoint orbit $\mathcal{O}_\xi$ in $\mathfrak{g}^\prime$, for any $\xi\in \mathcal{O}$, then $\sigma_\xi$ is exactly the form defined by relation  (\ref{eq_KKSform}) and its kernel $\mathfrak{g}_\xi$ is a closed Lie subalgebra of $\mathfrak{g}$.
\end{enumerate}
\item[(6)] 
If $\mathsf{G}$ is a {\bf Banach Lie group}  such that $\ker\sigma_\xi$ is split in $\mathfrak{g}$, for all $\xi\in \mathfrak{g}^\prime$, we have:
\begin{enumerate}
\item[(c)] for any $\xi\in \mathfrak{g}^\prime$, the group of isotropy $\mathsf{G}_\xi$ of $\xi$ under the coadjoint action is a Lie subgroup of $\mathsf{G}$ whose Lie algebra is $\mathfrak{g}_\xi$.
\item[(d)] Each orbit $\mathcal{O}_\xi$ is a split submanifold of $\mathfrak{g}^\prime$ which is diffeomorphic to the homogeneous space $\mathbf{G}/\mathsf{G}_\xi$.
\item[(e)] $\omega$ induces a weak symplectic form  on  each orbit $\mathcal{O}_\xi$.  
\item[(f)] The pair $({\bf s}, -{\bf t})$ is a dual pair (cf. Definition \ref{D_DualPairs}).
\end{enumerate}
\end{enumerate}
\end{theorem}
 
\begin{remark}
\label{R_GeneralizationFiniteDimensionGroupoidStructureOfTheCotangentBundle}${}$
\begin{enumerate}  
\item[1.]
Theorem~\ref{T_Cotangent bundleG} can be considered as a generalization to the Banach setting of a classical result in finite dimension for the groupoid structure of the cotangent bundle of a finite dimensional Lie group $\mathsf{G}$ and the link with the Kostant-Kirilov-Souriau  Poisson structure on the dual of the Lie algebra of $\mathsf{G}$ (see for instance  \cite{Marl05}, 1).\
\item[2.] 
In general, in the convenient setting, the last Assertion (6) of Theorem \ref{T_Cotangent bundleG} is not true without these complementary assumptions. However, in $\S$  \ref{___ExamplesDirectLimitPoisson}, we exhibit a context in which this result is true without these additional assumptions.
\end{enumerate}
\end{remark}

\begin{proof}${}$\\
Assertions (1) and (2) are direct applications of Theorem~\ref{T_PropertiesSymplecticgroupoid}.\\
Assertion(3) will be a consequence of Proposition~\ref{P_SubPoissonGroupoid} if we show that the diagram (\ref{eq_DiagralSubPoissonProp}) is commutative in this context. \\
 
Since $\widetilde{R}$ is an isomorphism, $\overline{\lambda}=\widetilde{R}^*\lambda$ is the Liouville $1$-form on $\mathfrak{g}^\prime \times\mathsf{G}$ and so $\overline{\lambda}( X,\xi)=\langle \xi, X \rangle$. Therefore,  $\overline{\omega}=\widetilde{R}^*\omega$
is the canonical symplectic form on $\mathfrak{g}^\prime \times\mathsf{G}$ and so 
if $\mathfrak{X}_1=(\eta_1, X_1)$  and $\mathfrak{X}_2=(\eta_2,X_2)$ belong to $T_\xi\mathfrak{g}^\prime\times T_g\mathsf{G})$ we have 
$$\overline{\omega}(\mathfrak{X}_1, \mathfrak{X}_2)=\langle \eta_1, (X_2)_e\rangle-\langle \eta_2, (X_1)_e \rangle+\langle \xi,[X_1, X_2]\rangle$$
where  $(X_i)_e\in \mathfrak{g}$ is such that $T_eR_g((X_i)_e)= X_i$ for $i \in \{1,2\}$. \\
Thus, for $\mathfrak{X}=\mathfrak{X}=(\eta, X))\in  T_\xi\mathfrak{g}^\prime\times T_g\mathsf{G})$, we  have:
\begin{equation}
\label{eq_Omega}
 (\overline{\omega})^\flat( (\mathfrak{X})=(\mathfrak{i}(\eta+\xi \circ  \mathsf{ad}_{X_e}),-X)\in T^\prime(\mathfrak{g}^\prime\times G)=\mathfrak{g}^\prime\times \mathfrak{g}^{\prime\prime}\times T^\prime \mathsf{G},
\end{equation}
 where $\mathfrak{i}$  is the canonical inclusion of $\mathfrak{g}$ into $\mathfrak{g}^{\prime\prime}$ and   we set $\mathfrak{g}^\flat=\mathfrak{i}(\mathfrak{g})$.  
 This implies that the range of $(\overline{\omega})^\flat$  is $(\mathfrak{g}^\prime\times \mathfrak{g}
^\flat)\times T^\prime \mathsf{G}$. Thus  
$T^\flat(T^\prime \mathsf{G}):=\omega^{\flat}(T(T^\prime \mathsf{G}))$ is equal to $T^*\widetilde{R}((\mathfrak{g}^\prime\times \mathfrak{g}
^\flat)\times T^\prime \mathsf{G})$ and so  is a closed subbundle of $T^\prime (T^\prime\mathsf{G})$. If  
$\bar{P}:= \left( (\overline{\omega})^\flat  \right) ^{-1}$ and 
$P=\left( ({\omega})^\flat \right) ^{-1}$ the following diagram is commutative:

\begin{equation}
\label{eq_DiagramPoissonT'Gs}
  \begin{tikzcd}
        T^\flat(T^\prime\mathsf{G})\arrow[rr, "T^*\widetilde{R}"near end ]\arrow[dr, "P"]\arrow[dd, "T\widetilde{\bf s}_{| T^\flat(T^\prime \mathsf{G})}", swap]& &\mathfrak{g}^\prime\times\mathfrak{g}^\flat\times T^\prime\mathsf{G}\ \arrow[dr, "\overline{P}"]\arrow[dd, "p_{\mathfrak{g}^\prime\times\mathfrak{g}^\flat}", near end] & \\
        & T(T^\prime \mathsf{G})& &\mathfrak{g}^\prime \times \mathfrak{g}^\prime\times T\mathsf{G}\arrow[ll, "T\widetilde{R}", crossing over, near end]  \arrow[dd, "p_{\mathfrak{g}^\prime}\times {\bf s}"] \\
       \mathfrak{g}^\prime\times \mathfrak{g}^\flat \arrow[rr, "Id_{\mathfrak{g}^\prime\times\mathfrak{g}^\flat}", near end]\arrow[dr, "P_{\mathfrak{g}^\prime}"] & & \mathfrak{g}^\prime\times\mathfrak{g}^\flat\arrow[dr, "\overline{P}_{\mathfrak{g}^\prime}"] & \\
        & \mathfrak{g}^\prime\times\mathfrak{g}^\prime \arrow[from=uu, "{T{\bf s}}", crossing over, near start, swap] & &\mathfrak{g}^\prime\times\mathfrak{g}^\prime\arrow[ll, "Id_{\mathfrak{g}^\prime\times\mathfrak{g}^\prime}"]
    \end{tikzcd}
    \end{equation}
where ${\bf s}=T_eL:T\mathsf{G}\to \mathfrak{g} $. 
Note by same arguments,  that we have also the diagram 
\begin{equation}
\label{eq_DiagramPoissonT'Gt}
  \begin{tikzcd}
        T^\flat(T^\prime\mathsf{G})\arrow[rr, "T^*\widetilde{L}"near end ]\arrow[dr, "P"]\arrow[dd, "T\widetilde{\bf t}_{| T^\flat(T^\prime \mathsf{G})}", swap]& &\mathfrak{g}^\prime\times\mathfrak{g}^\flat\times T^\prime\mathsf{G}\ \arrow[dr, "\overline{P}"]\arrow[dd, "p_{\mathfrak{g}^\prime\times\mathfrak{g}^\flat}", near end] & \\
        & T(T^\prime \mathsf{G})& &\mathfrak{g}^\prime \times \mathfrak{g}^\prime\times T\mathsf{G}\arrow[ll, "T\widetilde{L}", crossing over, near end]  \arrow[dd, "p_{\mathfrak{g}^\prime}\times {\bf t}"] \\
       \mathfrak{g}^\prime\times \mathfrak{g}^\flat \arrow[rr, "Id_{\mathfrak{g}^\prime\times\mathfrak{g}^\flat}", near end]\arrow[dr, "P_{\mathfrak{g}^\prime}"] & & \mathfrak{g}^\prime\times\mathfrak{g}^\flat\arrow[dr, "\overline{P}_{\mathfrak{g}^\prime}"] & \\
        & \mathfrak{g}^\prime\times\mathfrak{g}^\prime \arrow[from=uu, "{T{\bf t}}", crossing over, near start, swap] & &\mathfrak{g}^\prime\times\mathfrak{g}^\prime\arrow[ll, "Id_{\mathfrak{g}^\prime\times\mathfrak{g}^\prime}"]
    \end{tikzcd}
    \end{equation}

Now, note that  it is clear that   $\mathcal{N}^0(G)$ is nothing but else than $\mathfrak{g}
^\prime\times\mathfrak{g}^\flat$.
By superposing  the Diagrams (\ref{eq_DiagramPoissonT'Gs}) and (\ref{eq_DiagramPoissonT'Gt}) the  vertical  right face of this new diagram implies that its left vertical face gives rise to the Diagram \ref{eq_DiagralSubPoisson} in this context. Therefore, the assumptions of Proposition \ref{P_SubPoissonGroupoid}, which ends the proof of Assertion (iii).\\

Assertion (4) is a direct application of 
Theorem \ref{T_PropertiesSymplecticgroupoid}  Point(iv).\\

For Assertion (5),   the existence, unicity  and the expression of  $P_{\mathfrak{g}^\prime}$ is an application of Theorem \ref{T_PropertiesSymplecticgroupoid}, (iii). \\

 {\it From now on, until the end of the proof of Assertion(5), we identify  $\mathfrak{g}^\flat$ and $\mathfrak{g}$ }.\\
 
 Now, we are in the context of Theorem~\ref{T_LinkPoissonAlgebroid} with $\mathcal{A}= T^\prime \mathsf{G}$,  provided with  the Lie bracket of vector fields and the anchor being the identity. Therefore, according to  
  \cite{CaPe23}, relation (7.19), for  $\mathfrak{X}=(\eta, X)\in T^\flat_{(\xi, g)}(T^\prime \mathsf{G})$, we have: 

\begin{equation}
\label{eq_PCanonicalSymplectic}
P(\eta, X)=(\eta+\xi\circ\mathsf{ad}_X,-X)=(\eta+\ad_X^*\xi,-X)
\end{equation}

Since $\mathfrak{g}^\prime\times \mathfrak{g}^\prime$ is identified with $\mathfrak{g}^\prime\times\mathfrak{g}^\prime\times \{0\}$,  so  $P_{\mathfrak{g}^\prime}$ is  a map from $\mathfrak{g}^\prime \times \mathfrak{g}$ to $\mathfrak{g}^\prime \times \mathfrak{g}^\prime$ which is given by
\begin{equation}
\label{eq_Pg'}
(P_{\mathfrak{g}^\prime})_\xi (X)=\mathsf{ad}_X^*\xi
\end{equation}
In the one hand, from (\ref{eq_Omega}), the  $2$-form $\omega_{\mathfrak{g}^\prime}$ induced by  $\omega$ on $\mathfrak{g}^\prime$  is given by 
 $$(\omega_{\mathfrak{g}^\prime})_\xi(X,Y)=\mathsf{ad}_X^*\xi, Y\rangle=\langle \xi,[X, Y]\rangle=\sigma_\xi(X,Y)$$ (cf. (\ref{eq_KKSform})).
 On the other hand, we have   
 $$\omega_\mathfrak{g}^\prime(X,Y)=\langle P_{\mathfrak{g}^\prime}(X), Y\rangle.$$
 Since  $\left(P_{\mathfrak{g}^\prime}(\{\xi\}\times\mathfrak{g})\right)^0=\{X\in \mathfrak{g}:\langle\mathsf{ad}_X^*\xi, Y\rangle=0\;, \forall Y\in \mathfrak{g}\}$
then $\left(P_{\mathfrak{g}^\prime}(\{\xi\}\times\mathfrak{g})\right)^0=\{X\in \mathfrak{g} :\mathsf{ad}_X^*\xi=0\}=\ker\sigma_\xi$. Since is the pullback of a closed $2$-form is also closed, it follows that   $\sigma $ is closed and thus,  its kernel $\mathfrak{g}_\xi$ is a Lie algebra. This ends the proof of Assertion (5).\\ 

For the proof of Assertion (6), we assume that $\mathsf{G}$ is a Banach Lie group and for all $\xi\in \mathfrak{g}^\prime$ the Lie algebra  $\mathfrak{g}_\xi$ is split. {\it Then the proof is an application of Theorem \ref{T_PropertiesSymplecticgroupoid}, (v). Nevertheless,  we will give a complete proof in this context in order to show  the link between the Lie groupoid approach and the coadjoint action}.
 Under the previous assumptions, it follows that  the Lie algebra of the group of isotropy is 
$G_\xi$ of $\xi$ under the coadjoint action is $\mathfrak{g}_\xi$ and so  a Banach Lie subgroup of $\mathsf{G}$. This implies that the orbit $\mathcal{O}_\xi$ is a Banach submanifold of $
\mathsf{G}$ since its tangent space at $\xi$ (and  so at every point of $\mathcal{O}_\xi$) is isomorphic to $\mathfrak{g}/\mathfrak{g}_\xi$ and so is supplemented. Therefore, 
$\mathcal{O}_\xi$  is a submanifold  of $\mathfrak{g}^\prime$ and is diffeomorphic to $\mathsf{G}/\mathsf{G}_\xi$ (cf. \cite{BGJP19},  Example 3.5 or  \cite{Belt06}).\\

Finally according to  
Definition \ref{D_DualPairs}, the pair $({\bf s}, -{\bf t})$ is a dual pair.
\end{proof}

\begin{remark} ${}$
\begin{enumerate}
\item[1.]
$\mathfrak{g}^\prime\times \mathfrak{g}^\flat$ is the Lie algebroid of the Lie groupoid  $T^\flat(T^\prime\mathsf{G})$ and the anchor is $P_{\mathfrak{g}^\prime}$. Thus when $G$ is a Banach Lie group, the assumption "$\mathfrak{g}_\xi$ is split" is equivalent to $\ker P_{\mathfrak{g}^\prime}$ is split at $\xi$. Therefore,  under the assumptions of Assertion (6), if $\mathsf{G}$ is connected,  there is identity between coadjoint orbits and orbits of the Lie algebroid $\mathfrak{g}^\prime\times \mathfrak{g}^\flat$ on $\mathfrak{g}^\prime$.
\item[2.]
In general, in the convenient setting, the last Assertion (6) of Theorem~\ref{T_Cotangent bundleG} is not true without these complementary assumptions. However, we exhib in $\S$  \ref{___ExamplesDirectLimitPoisson} a context in which this result is true without these additional assumptions 
\end{enumerate}
\end{remark}

\section{Ascending Sequences of Finite Dimensional   Banach (Partial) Poisson-Lie Groupoids and Examples}
This section proposes an adaptation of some previous results to ascending sequences of partial Banach Poisson Lie groupoids with some examples.

\subsection{Ascending Sequences of  Banach (partial) Poisson-Lie Groupoids}  
As in $\S$~\ref{___DirectLimitBanachLieGroups} and for the same reasons,  we only consider sequences of Banach partial Poisson Lie groupoids. Thus,  naturally, an  ascending sequence of Banach Poisson Lie groupoids   $\{(T^\flat\mathcal{G}_i, \mathcal{G}_i, P_i)\}_{i\in \N}$  is an ascending sequence of Banach (partial) Poisson manifolds (cf.  Definition \ref{D_DirectSequenceOfPartialPoissonBanachManifolds}) where 

$\bullet\;$ each $\mathcal{G}_i \tto M_i$ is a finite dimensional Lie groupoid ;

$\bullet\;$  $ \left( M_i \right) _{i\in \N}$ is an ascending sequence of manifolds;

 $\bullet\;$ each $\mathcal{G}_i$ is provided with source and target maps ${\bf s}_i$, ${\bf t}_i$, inverse ${\bf i}_i$, multiplications ${\bf m}_i$ and unity maps ${\bf 1}_i$;
 
 $\bullet\;$ for each $i$, any of these maps at level $i$ is the restriction of the corresponding map at level ${i+1}$.\\
 
At first, recall the following results of ascending sequences of finite dimensional Lie groupoids:
 
\begin{proposition}
\label{P_DirectLimitGroupoid}
Let  $ \left( \mathcal{G}_i\rightrightarrows M_i \right) _{i\in \N}$ be a direct sequence of finite dimensional Banach-Lie groupoids. Then we have:
\begin{enumerate}
\item[(1)]
$\mathcal{G}=\underrightarrow{\lim}\mathcal{G}_{i}$  and  $M=\underrightarrow{\lim}M_i$  are   Hausdorff convenient manifolds.
\item[(2)]
${\bf 1}=\underrightarrow{\lim}\mathbf{1}_{i}$ is a smooth injective closed  immersion from  $M=\underrightarrow{\lim}M_{i}$ into $\mathcal{G}$;
\item[(3)]
${\bf s}=\underrightarrow{\lim}\mathbf{s}_{i}$ (resp. ${\bf t}=\underrightarrow{\lim}\mathbf{t}_{i}$) is a submersion from $\mathcal{G}$ on to $M$;
\item[(4)]
${\bf m}=\underrightarrow{\lim}\mathbf{m}_{i}$ is a smooth map from  the convenient   manifold $\mathcal{G}^{(2)}=\underrightarrow{\lim}\mathcal{G}^{(2)}_{i}$ to $\mathcal{G}$;
\item[(5)]
${\bf i}=\underrightarrow{\lim}\mathbf{i}_{i}$ is a smooth diffeomorphism of $\mathcal{G}$.
\end{enumerate}
In fact, we obtain a natural structure of convenient  groupoid  $\mathcal{G}\tto M$ which is called the {\bf direct limit  Lie groupoid of the ascending sequence of 
$ \left( \mathcal{G}_i\rightrightarrows M_i \right) _{i\in \N}$}. Its Lie algebroid  is the direct  limit of the  sequence 
$ \left( \mathcal{AG}_i, M_i,\rho_i,[\;,\;]_i \right) _{i\in \N}$ of Banach Lie algebroids.
\end{proposition}
 
\begin{remark} 
Since $\mathcal{G}$ (resp. $M$) is a direct limit of finite dimensional manifolds, it follows that the typical model of $\mathcal{G}$ (resp.$M$)  is isomorphic to $\mathbb{R}^\infty$ which is a convenient reflexive vector space (cf. \cite{KuMcK89}).
\end{remark}

\begin{theorem}
\label{T_ComplementaryPropertiesOfAscendingLimitGroupoids}
Let  $ \left(  \mathcal{G}_i\rightrightarrows M_i \right) _{i\in \N}$ be an ascending  sequence  of  finite dimensional Lie groupoids and $\mathcal{G}\tto M$ its  direct limit  Lie groupoid (cf. Proposition~\ref{P_DirectLimitGroupoid}).
Then, for each $x=\underleftarrow{\lim}x_i\in M$, we have:
\begin{enumerate}
\item[(1)]
The isotropy group $\mathcal{G}(x)$ is a convenient Lie group which is the direct limit of the sequence of isotropy groups 
$ \left( \mathcal{G}_i(x_i) \right) _{i \in \N}$.
\item[(2)]
The orbit $\mathcal{G}.x$ is a closed immersed  convenient submanifold of $M$ which is the direct  limit of the sequence of orbits 
$ \left( (\mathcal{G}_i).x_i \right) _{i \in \N}$.
\item[(3)]
The restriction ${\bf t}_x:={\bf t}_{| \mathcal{G}(x,-)}: \mathcal{G}(x,-)\to M$ is a $\mathcal{G}(x)$ principal bundle which is the projective limit of the sequence of principal bundles $({\bf t}_i)_{x_i}:  \mathcal{G}_i(x_i,-)\to M_i.$
\item[(4)]
The associated convenient Lie algebroid $(\mathcal{AG},M,\rho,[\;,\;])$ is split and, for any $x\in M$,
its orbit $\mathcal{G}.x$ is a closed  immersed submanifold of $M$ whose connected components are orbits of the  Banach-Lie algebroid $\mathcal{AG}$.
\item[(5)] $\mathcal{AG}^\prime=\underrightarrow{\lim} \mathcal{AG}_i^*$ is the dual of $\mathcal{AG}$ and $(\mathcal{AG}, \mathcal{AG}^\prime)$ is a  partial bialgebroid.\\
\end{enumerate}
\end{theorem}

 By application of  Theorem~\ref {T_PartialPoissonStructureOnDirectLimitOfPoissonBanachManifolds} in the one hand, and application of  
  Theorem~\ref{T_ComplementaryPropertiesOfAscendingLimitGroupoids}  to Theorem~\ref{T_FoliationPartialBanachPoissonManifold}, on the other hand,  we obtain:  
 
\begin{corollary}\label{C_ComplemnetaryFoliationFinitePartialBanachPoissonManifold} 
Let  $ \left( \mathcal{G}_i\rightrightarrows M_i \right) _{i\in \N}$ be an ascending  sequence  of  finite dimensional Poisson (resp.  symplectic) Lie groupoids and $ \left( T^\flat \mathcal{G}, T\mathcal{G}, P\right) $ be  
the associated   partial convenient direct limit  manifold  defined in Proposition~\ref{P_ConvenientTangentBundleFrechetCotangentBundleBundleMorphismDirectSequence}. Then in each case,  $T^\flat 
\mathcal{G}=T^\prime \mathcal{G}$ and in the symplectic case, $\mathcal{G}$ is a (strong)  symplectic \footnote{Since the symplectic form $\omega$ on $\mathcal{G}$ is the direct limit of $\omega_i$ and as each $\omega_i^\flat$ is an isomorphism so the direct limit of these  operators is an isomorphism.} groupoid. Moreover all the Assertions of Theorm \ref{T_PropertiesPoissonGroupoid} (resp.  Therorem\ref {T_PropertiesSymplecticgroupoid}) are true and in particular the last one without anymore assumptions.\\
\end{corollary}

\subsection{Examples}\label{___ExamplesDirectLimitPoisson} ${}$
\begin{enumerate}
\item[1.] 
Consider  $(\mathcal{G}_i\rightrightarrows M_i)_{i\in \mathbb{N}}$ be an ascending  sequence  of  finite dimensional Poisson Lie groupoids.   Then  the associated  sequence  
$ \left( T^*\mathcal{G}_i\rightrightarrows \mathcal{AG}^* \right) _{i\in \N}$ of cotangent groupoids. It is also an ascending sequence of finite dimensional  symplectic Lie groupoids. If $\mathcal{G}=\underrightarrow{\lim}\mathcal{G}_{i}$,  then  we have  $T^\prime\mathcal{G}=\underrightarrow{\lim}T^*\mathcal{G}_{i}$ and $\mathcal{AG}
^\prime=\underrightarrow{\lim}\mathcal{AG}_{i}^*$ and we have a structure of (strong)  symplectic convenient Lie groupoid $T^\prime \mathcal{G}\tto \mathcal{AG}^\prime$. 
\item[2.] 
Let $ \left( \mathsf{G}_i \right) _{i\in \N}$ be an ascending sequence of finite dimensional Lie groups. Then the  sequence of cotangent bundles $ \left( T^\prime \mathsf{G}_i \right) _{i\in N}$ is also an ascending sequence of finite dimensional bundles, each one having a structure of symplectic Lie  groupoid (cf. Theorem~\ref{T_Cotangent bundleG} or  \cite{Marl05}, 1).
If $\mathsf{G}=\underrightarrow{\lim}\mathsf{G}_{i}$  is also a  convenient Lie  group provided with an exponential map (cf. Theorem~\ref{T_DirecLlimitBanachLiegroups}), then   $T^\prime\mathcal{G}=\underrightarrow{\lim}T^*\mathcal{G}_{i}$ is the cotangent bundle of $\mathsf{G}$ (cf. 
Proposition~\ref{P_ConvenientTangentBundleFrechetCotangentBundleBundleMorphismDirectSequence}). It follows that $T^\prime\mathsf{G}$ has also a structure of convenient Lie groupoid on $\mathfrak{g}
^\prime\times\mathfrak{g}^\flat$ (cf. Theorem \ref{T_Cotangent bundleG}). Note that the adjoint action of $\mathsf{G}$ on $\mathfrak{g}^\prime$ is the direct limit of the action of $\mathsf{G}_i$ over $\mathfrak{g}_i^\prime$ and so $\mathcal{O}_\xi$ through $\xi=\underrightarrow{\lim} \xi_i$ is $\underrightarrow{\lim}\mathcal{O}_\xi$ and so is a split convenient submanifold of $\mathsf{G}$. It follows that all assertions of Theorem~\ref{T_Cotangent bundleG}  are true, in particular, the last ones  without any more assumptions.\\

For a  more precise example, take $\mathcal{G}_n=\mathsf{GL}(n, \C)$.  The following description refers to \cite{Bab10}.\\
Using the left translation on $\mathsf{GL}(n,\C)$,  the tangent space  $T\mathsf{GL}(n,\C)$  is isomorphic to $\mathsf{GL}(n,\C)\times\mathfrak{gl}(n,\C)$ \textit{via} the right action of  $\mathsf{GL}(n,\C)$. By the duality bracket  $(A,B)\mapsto 
\mathsf{Tr}(AB)$, we can identify  $\mathfrak{gl}^*(n,\C)$ with $ \mathfrak{gl}(n,\C)$ and so the cotangent bundle of $\mathsf{GL}(n,\C)$ is isomorphic to $\mathsf{GL}(\infty)\times\mathfrak{gl}
(\infty)$. \textit{Via} the duality bracket  $(A,B)\mapsto \mathsf{Tr}(AB)$, we can identify $T^\prime\mathcal{GL}(n,\C)$ with  $\mathsf{GL}(n,\C)\times\mathfrak{gl}(n,\C)$. Therefore, the coadjoint orbit of $\mathsf{GL}(n,\C)$ on $\mathfrak{gl}^*(n,\C)$ can be identified with the adjoint orbit of  $\mathsf{GL}(n,\C)$ on $\mathfrak{gl}(n,\C)$. 
Of course, the canonical symplectic form on $T^*(\mathsf{GL}(n,\C)\times\mathfrak{gl}(n,\C))$ gies rise to    a symplectic  groupoid structure on $\mathsf{GL}(n,\C)\times\mathfrak{gl}
(n,\C)\tto \mathfrak{gl}(n,\C)\times\mathfrak{gl}(n,\C)$. In particular, all the assertions of the theorem, even the last one, is true without any more assumptions.

On the one  hand, each $A\in \mathfrak{gl}(n,\C)$ can be considered  as a linear map on $\C^n$  and let $J_A$ be its normal Jordan form. Then we have $A=gJ_Ag^{-1}$ for some $g\in \mathsf{GL}(n,\C)$. So the adjoint orbit  $\mathcal{O}_A$ of $A$ is equal to  $\mathcal{O}_{J_A}$. Thus any adjoint orbit is parametrized by all Jordan matrices.

On  the one hand, the Kirilov-Kostant-Souriau symplectic form on $\mathcal{O}_A$ is defined in the following way (cf. \cite{Bab10}):
\begin{equation} 
\label{eq_KKSGLNC}
\sigma_A(X_1,X_2)=\mathsf{Tr}(\dot{g}_1g\dot{A}_2)=-\mathsf{Tr}(\dot{g}_2g\dot{A}_1)
\end{equation}
where, for $i \in \{1,2\}$,  $X_i$ is  a vector  tangent  to a curve  on $[0, \varepsilon[$ of type  $A_i(t)=g_i(t)J_Ag_i^{-1}(t)$ given by:
\[
g_i(0)=g,\; \;\dot{g}_i:=\dfrac{d}{dt}_{| t=0} g_i(t),\;\; \dot{A}_i:=\dfrac{d}{dt}_{| t=0} A_i(t).
\]

It follows that  the Lie algebra $\mathsf{GL}(\infty,\C)=displaystyle\bigcup_{n\in \N} \mathsf{GL}(n\C)$  is $\mathfrak{gl}(\infty,\C)=\displaystyle\bigcup_{n\in \N} \mathfrak{gl}(n\C)$. 
Therefore,  for $A$ and $B$ in $\mathfrak{gl}(\infty,\C)$, the product is well defined since we have an integer  $n$ such that $A$ and $B$ belongs to the same $\mathfrak{gl}(n,\C)$ (\textit{via} the successive inclusions $\iota_n$ of $\mathfrak{gl}(n\C)$ into  $\mathfrak{gl}(n\C)$  and so the bilinear form 
\[
\mathsf{Tr}(AB):= \mathsf{Tr}_n(AB)
\]
 is well defined if  $\mathsf{Tr}_n$ is the bilinear operator trace in $\mathfrak{gl}(n\C)$.  
This gives rise to an operator of duality between $\mathfrak{gl}(\infty,\C)$ and $\mathfrak{gl}^*(\infty,\C)$ which is compatible with the previous inclusion $\iota_n$ and so is compatible with the direct limit. For analogue reasons, the adjoint actions  are compatible  with the inclusions of $\mathsf{GL}(n\C)$ into  $\mathsf{GL}(n\C)$ and $\iota_n$ and so any coadjoint orbit can be identified with an adjoint orbit which implies that the same result is true for the direct limits. Thus the coadjoint orbits are parametrized by the union of all Jordan matrices of any finite dimension. Since each such orbit is contained in some $\mathfrak{gl}(n,\C)$,  the Kirillov-Kostant-Sauriau symplectic form is also given by  the formulae (\ref{eq_KKSGLNC}).
\end{enumerate}

\appendix
\section{Bounded Operators on a Hilbert Separable Complex Space and Von Neuman Algebras}
\label{__BoundedOperatorsOnAHilbertSeparableComplexSpaceAndVonNeumanAlgebras}

Von Neumann algebras were originally introduced by John von Neumann. For a complete exposition on Von Neumann algebras, see Dixmier's book \cite{Dix69}.

\subsection{Basic preliminaries}\label{__PrelinnariesBoundeOperator}

The following results  on the space of bounded operators on a complex separable Hilbert space are classical  but, for the sake of completeness, we refer to the concise and complete notes  \cite{Hou13} or also  \cite{Dix69}.

\subsubsection{Basic Properties of Bounded Operators on a Separable Complex Hilbert Space}

Let $\mathcal{H}$ be a separable complex Hilbert space. We denote by $\langle\;|\;\rangle$ its inner product and $||\;||$ its associated norm.  Let $\mathsf{L}_\infty(\Hc)$ be the \emph{Lie algebra Banach space of bounded linear operators} on $\Hc$ provided with the classical norm  $||A||_\infty=\sup_{||x||\leq 1}$ and the Lie bracket of commutateur 
 $[A,B]=A B-B A$ where $AB$ means the composition $A\circ B$. 
 
 We will  enumerate some  classical properties of bounded operators. 
 
 The adjoint of $A\in \mathsf{L}_\infty(\Hc)$ is denoted $A^*$;  
 
 $A\in \mathsf{L}_\infty(\Hc)$ is called{ \it self-adjoint} if $A=A^*$;
 
 $p\in \mathsf{L}_\infty(\Hc)$ is called {\it a projection } if $p^2=p=p^*$;
 
 $A\in \mathsf{L}_\infty(\Hc)$ is called{ \it positive} if $\langle Ax| x\rangle\geq 0$ for all $x\in \Hc$; in this case, $A$ is self-adjoint;
 
 for $A, B\in \mathsf{L}_\infty(\Hc)$ we say that $A\geq B$ if  $A-B$ is positive;
 
 $u\in \mathsf{L}(\Hc)$ is an \emph{isometry}\index{isometry} if  $uu^*=\operatorname{Id}$;
 
  $u\in \mathsf{L}(\Hc)$ is \emph{unitary}\index{unitary operator} if $uu^*=u^*u=\operatorname{Id}$;
 
   $u\in \mathsf{L}(\Hc)$ is a \emph{partial isometry} if $u^*u$ is a projection.
   
\begin{lemma}
\label{L_InitialDomain} 
If $u$ is a partial isometry then so is $u^*$. The subspace $u^*\Hc$ is then closed and called the \emph{initial domain}\index{initial domain} of $u$, the subspace $u\Hc$  is also closed and called the\emph{final domain}\index{final domain} of $u$.
\end{lemma}
 
If $A$ is positive then $A=A^*$; it has a square root  $B:=A^{\frac{1}{2}}$ i.e. an operator $B$ such  that $B^2=A$.
  
For any $A\in \mathsf{L}(\Hc)$, we set $|A|=(A^* A)^{\frac{1}{2}}$.
  
 \begin{lemma}
 \label{L_PolarDecompositon}  
For any $A\in \mathsf{L}(\Hc)$, there exists a partial isometry $u $ such that $A = u|A|$, and that $u$ is unique subject to the condition that its initial domain is $ (\Ker A)^\perp$. The final domain of this $u$ is $A\Hc = (\Ker A^*)^\perp$. Such a decomposition is called \emph{the polar decomposition of $A$}\index{polar decomposition}.
\end{lemma}
  
\subsubsection{Some Classical Banach Sets of Bounded operators} 

The \emph{trace}\index{trace} of a positive operator is 
\[
\operatorname{Tr}A:=\dis\sum_{n}^\infty\langle e_n | A e_n\rangle
\]
where $ \left( e_n \right) $ is an orthonormal basis of $\Hc$. This value does not depend on the choice of the orthonormal basis.\\

The set $\mathsf{L}_2(\Hc)$  of \emph{Hilbert-Schmidt operators}\index{Hilbert-Schmidt operator} is the subspace of $\mathsf{L}_\infty(\Hc)$ of operators $A$  such that $||A||_2:=(\operatorname{Tr}(AA^*))^{\frac{1}{2}}<\infty$. It is a Banach Lie algebra  with the norm $||\;||_2$ and the Lie bracket of commutator. It is an ideal of $\mathsf{L}_\infty(\Hc)$.\\

  For any $A\in \mathsf{L}_\infty(\Hc)$, the square root of $(AA^*)$ is a well defined bounded operator and so the set $\mathsf{L}_1(\Hc)$ of {\it trace-class operators} is the subspace of operators $A$ in  $\mathsf{L}_\infty(\Hc)$ such that
\[
||A||_1:=\operatorname{Tr}(AA^*)^{\frac{1}{2}}<\infty
\]
 It is a Lie algebra for the norm $\vert\vert\;||_1$ and with the Lie bracket of commutator
 Note that we have $|\operatorname{Tr}A|\leq ||A||_1$ and so $\mathsf{L}_2(\Hc)\hookrightarrow \mathsf{L}_1(\Hc)$.\\
 
For any $1<p<\infty$, let $ \mathsf{L}_p(\Hc)$ the subspace of $A\in \mathsf{L}_\infty(\Hc)$ such that
 $$\vert\vert A||_p:=\left(\operatorname{Tr}(AA^*)^{\frac{p}{2}}\right)^{\frac{1}{p}}$$
 It is a Lie algebra for the norm $||\;||_p$ and the Lie bracket of commutator.\\
 
 We have the following classical properties:\\

$\bullet\;$  for $1<p<2<q<\infty$, we have
 \begin{equation}
 \label{eq_ChainInclusionLop}
 \mathsf{L}_1(\Hc)\hookrightarrow  \mathsf{L}_p(\Hc)\hookrightarrow \mathsf{L}_2(\Hc)\hookrightarrow \mathsf{L}_q(\Hc)\hookrightarrow  \mathsf{L}_\infty(\Hc).
 \end{equation}
Note that $\mathsf{L}_p(\Hc)$ is dense in $\mathsf{L}_2(\Hc)$ and a closed subspace of $\mathsf{L}_\infty(\Hc)$ but is not supplemented.\\

$\bullet\;$  The bilinear map $(A,B)\mapsto \operatorname{Tr}(AB)$ defines 

 a strong duality\footnote{A  non degenerate bounded bilinear $(x,y)\mapsto <x,y>$ between convenient spaces $\mathbb{E}$ and $\mathbb{F}$  is called a \emph{weak pairing}\index{weak pairing}. Such a paring is called \emph{strong}\index{strong} if the map $x\mapsto <x,\;>$ (resp. $y\mapsto <\;,y>$) is a convenient  isomorphism from $\mathbb{E}$ (resp. $\mathbb{F}$) to the dual of $\mathbb{F}$ (resp. of  $\mathbb{E}$).} pairing  between  $\mathsf{L}_2(\Hc)$ with itself;\\
 
 a weak duality pairing between $\mathsf{L}_1(\Hc)$ and $\mathsf{L}_\infty(\Hc)$ and, by the way, $\mathsf{L}_\infty(\Hc)$ can be identified with the dual of  $ \mathsf{L}_1(\Hc)$;\\

 a strong  duality pairing between  $\mathsf{L}_p(\Hc)$ and $ \mathsf{L}_q(\Hc)$ if $\dfrac{1}{p}+\dfrac{1}{q}=1$,  and so the dual of $\mathsf{L}_p(\Hc)$ can be identified with $\mathsf{L}_q(\Hc)$ and conversely.\\

\subsection{Finite Rank  and Compact Operators} 
In this paragraph, the orthonormal basis $ \left( e_n \right) $ for $\Hc$ is fixed. Let  $\mathsf{L}_{\operatorname{fin}}(\Hc)$ be the set of operators $A\in  \mathsf{L}_\infty(\Hc)$ whose range is finite dimensional. On the other hand,  $A\in  \mathsf{L}_\infty(\Hc)$ is called \emph{compact}\index{compact operator} if  the image of the closure of the unit ball in $\Hc$ is relatively compact. We denote by  $\mathsf{L}_{K}(\Hc)$  the set of compact operators on $\Hc$. The following results  on the space $\mathsf{L}_{K}(\Hc)$ are classical but for the sake of completeness, we refer to the concise and complete notes of Bourbaki.\\ 

 $\bullet\;$ $\mathsf{L}_{K}(\Hc)$ is the closure of $\mathsf{L}_{\operatorname{fin}}(\Hc)$  in $\mathsf{L}_\infty(\Hc)$ relative to the norm $||\;\||_\infty$. In particular, $\mathsf{L}_K(\Hc)$ is closed in  $\mathsf{L}_\infty(\Hc)$. \\
 
 $\bullet\;$   $\mathsf{L}_K(\Hc)$ is an ideal in the algebra $\mathsf{L}_\infty(\Hc)$. In particular,  $\mathsf{L}_K(\Hc)$ is a Lie subalgebra of $\mathsf{L}_\infty(\Hc)$.\\
 
 $\bullet\;$ $\mathsf{L}_{2}(\Hc)\subset \mathsf{L}_{K}(\Hc)$.\\

$\bullet\;$ $\mathsf{L}_1(\Hc)=\left(\mathsf{L}_K(\Hc)\right)^\prime$ and $\left(\mathsf{L}_1(\Hc)\right)^\prime =\mathsf{L}_\infty(\Hc)$.\\

\begin{remark}
\label{R_linkWithSequences} 
Consider the map $\Phi:\ell^\infty(\mathbb{N})\to \mathsf{L}_\infty(\Hc)$ given by:
\[
\Phi(\lambda_n)(x)=\dis\sum_{n=1}^\infty \lambda_n\langle x\vert e_n\rangle e_n.
\]
Then we have:\\

$\Phi^{-1}(\mathsf{L}_{\operatorname{fin}}(\Hc))=c_{00}(\mathbb{N})$, that is the set of finite complex sequences;\\

$\Phi^{-1}(\mathsf{L}_K(\Hc))=c_{0}(\mathbb{N})$,  that is the set of complex sequences which converge to $0$.\\

\noindent Moreover  in each previous situation $\Phi$ is an isometry.

It is well known that, for $1<p<2<q<\infty $, we have:

$$\ell^1(\mathbb{N})\hookrightarrow \ell^p(\mathbb{N}) \hookrightarrow \ell^2(\mathbb{N})\hookrightarrow \ell^q(\mathbb{N}) \hookrightarrow c_0(\mathbb{N} ) \hookrightarrow \ell^\infty(\mathbb{N})$$
where  each inclusion is continuous.
On the other hand, $c_{00}(\mathbb{N})\subset \ell^1(\mathbb{N})$ is a convenient space as a direct limit of ascending sequence of finite dimensional Banach spaces, say $\ell(n)$ of finite sequences defined  on $\{1,\dots, n\}$ (cf. proof of Proposition \ref{P_inclusion BoundedSets}). Moreover, $c_{00}(\mathbb{N})$ is dense in $ \ell^1(\mathbb{N})$ and its inclusion is continuous according to its convenient topology.
\end{remark}

As in Remark~\ref{R_linkWithSequences}, we have  a comparable  situation for the previous spaces of bounded operators:\\

\begin{proposition}
\label{P_inclusion BoundedSets} 
The set $\mathsf{L}_{\operatorname{fin}}(\Hc)$ has a structure of convenient space  which is dense in $\mathsf{L}_1(\Hc)$ and the inclusion is continuous according to the convenient topology on $\mathsf{L}_{\operatorname{fin}}(\Hc)$ and the Banach structure on $\mathsf{L}_{1}(\Hc)$. In particular, for $1<p<2<q<\infty$, we have the chain of continuous inclusions:
  \begin{equation}
  \label{eq_ChainInclusionLopComp} 
  \mathsf{L}_{\operatorname{fin}}(\Hc)\hookrightarrow  \mathsf{L}_1(\Hc)\hookrightarrow  \mathsf{L}_p(\Hc)\hookrightarrow \mathsf{L}_2(\Hc)\hookrightarrow \mathsf{L}_q(\Hc)\hookrightarrow  \mathsf{L}_K(\Hc)\hookrightarrow  \mathsf{L}_\infty(\Hc),
  \end{equation}
  for $1<p<2<q<\infty$,  
  and $\Phi^{-1}( \mathsf{L}_k(\Hc))=\ell^k(\mathbb{N})$, for any $1\leq k\leq \infty$
\end{proposition}

\begin{proof}
Let $\Hc_n$ be the vector subspace of $\Hc$ generated by $\{e_1,\dots,e_n\}$. We consider the canonical  inclusion $\iota_n:\Hc_n\hookrightarrow\Hc_{n+1}$.  If  $
\mathsf{L}(\Hc_n)$ is the Banach space of continuous endomorphisms of $\Hc_n$, we have an evident inclusion $\iota_n^L:\mathsf{L}(\Hc_n)\hookrightarrow \mathsf{L}(\Hc_{n+1})$. By the way, the sequence $\{\mathsf{L}(\Hc_n),\iota_n^L\}_{n\in \mathbb{N}}$ is an ascending sequence of Banach spaces and so its direct limit  $\underrightarrow{\lim}\mathsf{L}(\Hc_n)=\displaystyle\bigcup_{n\in \N}\mathsf{L}(\Hc_n)$ is  contained in a convenient subspace of  $\mathsf{L}_1(\Hc)$. Since  $\mathsf{L}(\Hc_n)$ is a Banach subspace, each inclusion $\iota_n^L$ of $\mathsf{L}(\Hc_n)$ into $\mathsf{L}_1(\Hc)$ is continuous and so is   the inclusion $\underrightarrow{\lim}\iota_n^L$ of $\underrightarrow{\lim}\mathsf{L}(\Hc_n)$ in  $\mathsf{L}_1(\Hc)$ (cf. \cite{CaPe23}, 5.4). \\
 
 Note that  $ \underrightarrow{\lim}\mathsf{L}(\Hc_n)=\bigcup_{n\in \mathbb{N}}\mathsf{L}(\Hc_n)$ is also contained in  $ \mathsf{L}_{\operatorname{fin}}(\Hc)$. Consider $A\in   \mathsf{L}_{\operatorname{fin}}(\Hc)$. Then the Kernel $K$ of $A$ is finite codimensional and its range $R=A(E)$ is finite dimensional. Therefore, thee exists $n\in \N$ such that $K+ \Hc_n=\Hc$ and $R\subset \Hc_n$. Thus the 
 restriction $A_n$ of $A$ to  $\Hc_n$  belongs to $ \mathsf{L}(\Hc_n)$.  Now, for any $m\geq n$, the restriction $A_m$ to $\Hc_m$ also belongs to $ \mathsf{L}(\Hc_n)$ and the restriction of 
 $A_m$ to $\Hc_n$ is $A_n$. It follows easily that $A=\underrightarrow{\lim}A_n$ and, finally, we have 
 $ \underrightarrow{\lim}\mathsf{L}(\Hc_n) = 
 \mathsf{L}_{fin}(\Hc)$, 
 which ends the proof of the  first part.\\
Relation (\ref{eq_ChainInclusionLopComp} ) is then a direct consequence of the previous result and (\ref{eq_ChainInclusionLop}). Finally, the last part is a consequence of Remark~\ref {R_linkWithSequences} and the definition of each Banach space $\mathsf{L}_k(\Hc)$ for $1\leq k\leq \infty$.
\end{proof}

\subsection{Von Neuman Algebra and Groupoid}

\subsubsection{Definition of a Von Neumann Algebra}${}$\\
We begin by some recalls on some usual topologies on $\mathsf{L}_\infty(\Hc)$.
\begin{enumerate}
\item[(1)] 
The \emph{norm topology} defined norm $||\;||_\infty$.
\item[(2)] 
The \emph{strong operator topology}. A basis of neighbourhoods of $A\in \mathsf{L}_\infty(\Hc)$ for this topology  is
\[
V(A,x_1\dots,x_n,\varepsilon):=\{B\;:\; ||B-Ax_i||_\infty<\varepsilon,\;\; \forall i=1,\dots, n\}.
\]
\item[(3)] 
The  \emph{weak operator topology}. A basis of neighbourhood of $A\in \mathsf{L}_\infty(\Hc)$ for this topology  is
\[
V(A,x_1\dots,x_n,\eta_1,\dots, \eta_n,\varepsilon):=\{B\;:\; |\langle B-Ax_i|\eta_i\rangle |<\varepsilon,\;\; \forall i=1,\dots, n\}.
\]
\item[(4)]
The  \emph{weak  topology} is the topology of pointwise convergence  on  $\Hc$ in the "weak topology" on $\Hc$ according to the inner product.
\end{enumerate}

We have the following order between these topologies
\[
\{\textrm{weak operator toplology}\}\;<\; \{\textrm{strong operator toplology}\}\;<\;\{\textrm{norm toplology}\}.
\]


\begin{definition}
\label{D_VonNEumanAlgebra} 
A *-subalgebra unital  $\mathfrak{M}$ \footnote{ This means that  $\mathfrak{M}$ is stable by adjoint operation and contains the identity.}   which is  closed for the topology of weak operator topology of $\mathsf{L}_\infty(\Hc)$ is called a \emph{Von Neuman algebra}\index{Von Neuman algebra}.
\end{definition}

\begin{remark}
\label{R_Sakai}
For any Banach space $\mathbb{E}$, a predual of $\mathbb{E}$ is a Banach space $\mathbb{E}_*$ such that $(\mathbb{E}_*)^\prime$ is isomorphic to $\mathbb{E}$. In general, when a predual exists,  it is not unique. We say that $\mathbb{E}_*$ is a \emph{strong unique predual} if $\mathbb{E}_*\subset \mathbb{E}^\prime$. S. Sakai (cf. \cite{Sak71}) has shown that there exists an  equivalence between weak closed *-subalgebras $\mathfrak{M}$ of $\mathsf{L}(\Hc)$ and $C^*$-algebras $\mathfrak{M}$ of $\mathsf{L}(\Hc)$ which having a strong unique predual $\mathfrak{M}_*$. For instance,  $\mathsf{L}_1(\Hc)$ (resp. $(\mathsf{L}_K(\Hc)$) is the predual of $\mathsf{L}_\infty(\Hc)$ (resp. $\mathsf{L}_1(\Hc)$).
\end{remark}

\begin{examples}
\label{Ex_ExampleVon Neumann}${}$ 
\begin{enumerate}
\item[1.] 
Any finite dimensional  *-subalgebra of $\mathsf{L}_\infty(\Hc)$ which contains $Id$.
\item[2.] 
$\mathsf{L}_\infty(\Hc)$.
\item[3.] 
The algebra of essentially bounded measurable function on $([0,1],dx)$
\end{enumerate}
\end{examples}

\subsection{Groupoid Associated to a Von Neumann Algebra}
\label{___GroupoidAssociatedToAVonNeumanAlgebra} 
For this paragraph, we refer to  \cite{OdSl16}.\\

Let $ \mathfrak{M}$  be a von Neumann
algebra. We denote by  $\mathcal{L}(\mathfrak{M})$  the set of projections of $\mathfrak{M}$ and by $\mathcal{U}(\mathfrak{M})$ the set of partial isometries of $\mathcal{U}(\mathfrak{M})$,  that is, 
\[
\mathcal{L}(\mathfrak{M}):= \{p \in \mathfrak{M}, \;:\;  p = p^2 = p^*\}
  \textrm{~~and~~}
\mathcal{U}(\mathfrak{M}) :=\{u \in \mathfrak{M},\;:\; u^*u = 1 = uu^*\}.
\]

\noindent $\mathcal{L}(\mathfrak{M})$ is an ordered set by
$p\leq q \Leftrightarrow  pq=p.$
In fact  $\mathcal{L}(\mathfrak{M})$ is a lattice 
\footnote{A partially ordered set 
${\displaystyle (L,\leq )}$ is called a \emph{lattice}\index{lattice} if each two-element subset 
${\displaystyle \{a,b\}\subseteq L}$ has a least upper bound, denoted by 
${\displaystyle a\vee b}$ and  greatest lower bound, denoted by 
${\displaystyle a\wedge b}$.}. \\

For $x\in \mathfrak{M}$, we consider its polar decomposition $x=u|x|$. In fact, since  $|x|=(x^*x)^\frac{1}{2}$, then $|x|$ belongs to $\mathfrak{M}^+:=\{ x\in \mathfrak{M},\;:\; x^*=x>0\}$ and $u\in\mathcal{U}(\mathfrak{M})$ and so ${\bf s}(|x|):= u^*u$ belongs to $\mathcal{L}(\mathfrak{M})$. Note that  ${\bf t}(| x^*|)=uu^*$ also belongs to $\mathcal{L}(\mathfrak{M})$.\\

We denote by $G(p\mathfrak{M}p)$ the group of invertible elements of the Von Neuman subalgebra $p\mathfrak{M}p$ of $\mathfrak{M}$ and $U(p\mathfrak Mp)$ the group of invertible elements of $p\mathfrak{M}p$. Note that for $p=\operatorname{Id}$ one has  $p\mathfrak{M}p=\mathfrak{M}$, so the corresponding groups are simply denoted $G(\mathfrak{M}) $ and $U(\mathfrak M)$  respectively. 

Note that for any $x\in \mathfrak{M}$, $|x|$ belongs to $p\mathfrak{M}$ where $p={\bf s}(|x|):=u^*u$. By the way, we consider the set
\[
\mathcal{G}(\mathfrak{M}):=\{x\in \mathfrak{M},\;:\; |x|\in \mathfrak{M}p,\;\; p={\bf s}(|x|)\}.
\]

The following theorem summarizes some essential results proved in \cite{OJS18}, 2:

\begin{theorem}
\label{T_GroupoidVon Neumann}${}$
\begin{enumerate} 
\item[1.]
$\mathcal{G}(\mathfrak{M})$ has a canonical structure of groupoid over $\mathcal{L}(\mathfrak{M})$ where, according to polar decompositions, we have:
\begin{enumerate}
\item[(i)] 
the identity section ${\bf 1}: \mathcal{L}(\mathfrak{M})\to \mathcal{G}(\mathfrak{M})$ is the inclusion;
\item[(ii)] 
the source ${\bf s}$ and the target ${\bf t}$ from $\mathcal{G}(\mathfrak{M})$ to  $\mathcal{L}(\mathfrak{M})$ are given by

${\bf s}(x)=u^*u$ and ${\bf t}(x)=uu^*$;
\item[(iii)] 
the multiplication ${\bf m}(x,y)$ is the multiplication  in the algebra $\mathfrak{M}$ if ${\bf s}(x)={\bf t}(y)$;
\item[[iv)] 
the inverse ${\bf i}:\mathcal{G}(\mathfrak{M})\to \mathcal{G}(\mathfrak{M})$ is defined by
 ${\bf i}(x)=|x|^{-1}u^*$ if the polar decomposition of  $x$ is $u|x|$.
\end{enumerate}
\item[2.] 
The set of partial isometries $\mathcal{U}(\mathfrak{M})$ is a wide subgroupoid $\mathcal{U}(\mathfrak{M})\tto \mathcal{L}(\mathfrak{M})$.
\end{enumerate}
\end{theorem}

\subsection{Orbits of  Von Neumann Algebra Groupoids} 
For this paragraph, we again  refer to  \cite{OJS18}.

Given a Von Neumann algebra $\mathfrak{M}$, recall that two projections $p,p'\in \mathfrak{M}$ are called \emph{Murray-von Neumann equivalent}  and denoted $p\sim p'$ if there exists  $u\in \mathfrak{M}$ such that $p=uu^*$ and $p'=u^*u$. Fix some $p_0\in \mathcal{L}(\mathfrak{M})$ and set
\[
\mathcal{L}_{p_0}:=\{p\in \mathcal{L}(\mathfrak{M})\;:\; p\sim p_0\}\textrm{ and } \mathcal{G}_{p_0}(\mathfrak{M}):={\bf t}^{-1}\left(\mathcal{L}_{p_0}(\mathfrak{M}\right)\cap {\bf s}^{-1}\left(\mathcal{L}_{p_0}(\mathfrak{M})\right) 
\]
Then $\mathcal{G}_{p_0}(\mathfrak{M})\tto \mathcal{L}_{p_0}(\mathfrak{M})$ is a subgroupoid of $\mathcal{G}(\mathfrak{M}\tto \mathcal{L}(\mathfrak{M})$. If $p\sim p_0$, then we have 
$\mathcal{G}_{p}(\mathfrak{M})\tto \mathcal{L}_{p}(\mathfrak{M})$ and $\mathcal{G}_{p_0}(\mathfrak{M})\tto \mathcal{L}_{p_0}(\mathfrak{M})$ coincide. From Theorem~\ref{T_OrbitBanachGroupoid}  and \cite{OJS18} on the other hand, we have the following illustration of Theorem~\ref{T_OrbitBanachGroupoid}  in this context:

\begin{theorem}
\label{P_OrbitsG(M)} ${}$
\begin{enumerate}
 
\item[1.]  
Each $\mathcal{G}_{p_0}(\mathfrak{M})\tto \mathcal{L}_{p_0}(\mathfrak{M})$ is open and closed in the groupoid  $\mathcal{G}(\mathfrak{M})\tto\mathcal{L}(\mathfrak{M})$. 
\item[2.] 
The Banach Lie groupoid $\mathcal{G}(\mathfrak{M})\tto \mathcal{L}(\mathfrak{M})$ is a disjoint union of Banach Lie groupoids $\mathcal{G}_{p_0}(\mathfrak{M})\tto \mathcal{L}_{p_0}(\mathfrak{M})$, $p_0\in \mathcal{L}\mathfrak{M})$. In particular, each set $\mathcal{L}_{p_0}(\mathfrak{M})$ is an orbit of the groupoid  $\mathcal{G}(\mathfrak{M})\tto\mathcal{L}(\mathfrak{M})$.
\item[3.] 
If $P_0=({\bf s}_{| \mathcal{G}_{p_0}(\mathfrak{M})})^{-1}(p_0)$ then $P_0\to \mathcal{L}_{p_0}(\mathfrak{M})$ is a principal bundle  with Lie group $G_0={\bf t}_{| P_0}^{-1}(p_0)$.
\end{enumerate}
\end{theorem}

To be more clear with the link between proofs in  \cite{OJS18} and direct applications of Theorem ~\ref{T_OrbitBanachGroupoid}, we end this  section with the interpretations of these results in terms of the Von Neumann algebra context according   to \cite{OJS18} results in section 1.\\

For fixed $p_0\in \mathcal{L}(\mathfrak{M})$,  $P_0$ is an open set of $\mathfrak{M}p_0$.  $G_0$ is the group $G(p_0\mathfrak{M}p_0)$ of invertible elements of the Von Neumann subalgebra $p_0\mathfrak{M}p_0$.
 We have a natural free right action of $G_0$ on $p_0\mathfrak{M}p_0$  which gives rise to a $G_0$ principal bundle $P_0\to P_0/G_0$ and $(P_0\times P_0)\to (P_0\times P_0)/G_0$  which  are isomorphic to  $P_0\to \mathcal{L}_{p_0}(\mathfrak{M})$.  In particular, $\mathcal{L}_{p_0}(\mathfrak{M})$ (resp. $\mathcal{G}_{p_0}(\mathfrak{M})$) is isomorphic to $P_0/G_0$   (resp. $(P_0\times P_0 )/G_0$). In fact, we have the following diagram of groupoid isomorphisms  (\cite{OJS18}, Proposition 1.4)
  \begin{equation*}
\xymatrix{
(P_0\times P_0)/G_0\ar[r] \ar@<-.5ex>[d] \ar@<.5ex>[d]
& \mathcal{G}_{p_0}(\mathfrak{M})\ar@<-.5ex>[d] \ar@<.5ex>[d] \\
   P_0/G_0 \ar[r] & \mathcal{L}_{p_0}(\mathfrak{M})
}
\end{equation*}


\subsection{Lie Poisson Structures on Von Neumann Algebra Groupoids}  
\label{__WeaksymplecticW}

A von Neumann algebra  also called a $W^*$-algebra is a $C^*$-algebra $\mathfrak{M}$ which has a Banach predual space $\mathfrak{M}_*$, i.e. $\mathfrak{M}=(\mathfrak{M}
_*)^*$. Denote by $\mathcal{L}(\mathfrak{M})$  the lattice  of  projections $p$ on $\mathfrak{M}$ such that $p=p^*=p^2$. If $\mathcal{U}(\mathfrak{M})$ is the set of all partial isometries  
$u$ in $\mathfrak{M}$ (i.e.  $u^* u\in \mathcal{L}(\mathfrak{M})$) for any $x\in \mathfrak {M}$ such that  $x=x^*$ let $s(x)$ be the smallest projection in $\mathcal{L}(\mathfrak{M})
$ such that $s(x)x=x$.  Let $\mathfrak{M}^+:=\{ x\in \mathfrak{M}\;:\; x^* x>0\}$; we then have a polar decomposition $x=u |x|$ where $u$ belongs to $\mathcal{U}(\mathfrak{M})$ and $|x|
\in \mathfrak{M}^+$. One has $s(  |x^*|)=uu^*$ and $s(|x|)=u^*u$. In fact, for any $x \in \mathfrak{M}$, we have  $|x|\in p\mathfrak{M}p$, where $p = s(|x|)$. \\
We consider the set $\mathcal{G}(\mathfrak{M})=\{x\in \mathfrak{M}: |x| \textrm{ invertible in } p\mathfrak{M}p\}$. From the previous notations, we have well defined  maps $\mathbf{s}:\mathcal{G}(\mathfrak{M})\to \mathcal{L}(\mathfrak{M})$ and $\mathbf{t}:\mathcal{G}(\mathfrak{M})\to \mathcal{L}(\mathfrak{M})$  respctively given by 

$\mathbf{s}(x)=u^* u$ and  $\mathbf{t}(x)=uu^*$. By the way we have a groupoid structure $\mathcal{G}\mathfrak{M}{\rightrightarrows}\mathcal{L}(\mathfrak{M})$(for such a 
complete description and results see \cite{OdSl16}).\\
 For fixed $p_0\in \mathcal{L}(\mathfrak{M})$ we set $P_0=\mathbf{s}^{-1}(p_0)$. Then $P_0$ is an open set of $\mathfrak{M}p_0$. Thus the tangent bundle $TP_0$ is isomorphic to  $P_0\times \mathfrak{M}p_0$ and its cotangent bundle $T^*P_0$ is isomorphic to $P_0\times p_0\mathfrak{M}^*$. 
Consider  the predual $(\mathfrak{M}p_0)_*=p_0\mathfrak{M}_*$ of $(\mathfrak{M}p_0)^*=p_0\mathfrak{M}^*$ and so the bundle $P_0\times p_0\mathfrak{M}_*$ is a weak Banach  subbundle of 
$P_0\times p_0\mathfrak{M}^*$ which defines a weak Banach subbundle $T_*P_0$ of $T^*P_0$. We denote by $\iota: T_*P_0\to T^*P_0$ the inclusion morphism.\\
Now according to  \cite{OJS18}, Proposition 5.2 and Remark 5.1, if $\omega$ denotes the canonical symplectic form on $T^*P_0$ then $\omega_0=\iota^*\omega$ is a weak symplectic form on $T_*P_0$. More precisely,  if we identify $T_*P_0$ with $P_0\times p_0\mathfrak{M}_*$ and $T^*P_0$ with $P_0\times p_0\mathfrak{M}^*$, then $T(T_*P_0)\equiv (P_0\times p_0\mathfrak{M}_*)\times(\mathfrak{M}p_0\times p_0 \mathfrak{M}_*)$ and the range of $\omega^\flat_0$ is then 
$$T^\flat(T_*P_0)\equiv (P_0\times p_0\mathfrak{M}_*)\times(\mathfrak{M}p_0\times p_0 \mathfrak{M}_*)\subset (P_0\times p_0\mathfrak{M}_*)\times (\mathfrak{M}p_0\times  (\mathfrak{M}p_0)^*)\equiv T^*(T_*P_0).$$
{\it Finally  the weak symplectic form $\omega$ defines a Lie Poisson structure on the algebra $\mathfrak{A}(T_*P_0))$ associated to $T^\flat (T_*P_0)$.
 In particular, ${\bf P}:=(\omega_0^\flat)^{-1}$ is a a Poisson anchor.}\\
  
 On  the other hand, we  have  the pair  groupoid  $T_*P_0\times T_*P_0\tto T_*P_0$ and a Poisson anchor 
 $${\bf P}\times {\bf P} :T^\flat(T_*P_0\times T_*P_0)\to T(T_*P_0\times T_*P_0)$$ 

In \cite{OJS18}, p. 40,  it is proved:\\ 
 \begin{proposition}
 \label{P_ExampleSubPoisson}   
The groupoid  $T_*P_0\times T_*P_0\tto T_*P_0$ has a sub-Poisson groupoid structure associated to  the following commutative diagram of groupoid morphisms:
 \begin{equation}
 \label{eq_DiagralSubPoisson}
 \xymatrix{
  T^\flat(T_*P_0\times T_*P_0)\ar@<-.5ex>[d] \ar@<.5ex>[d]\ar[r]^{{\bf P}\times{\bf P}}&T(T_*P_0\times T_*P_0)\ar@<-.5ex>[d] \ar@<.5ex>[d]\\
T^\flat(T_*P_0)\ \ar[r]^{{\bf P}} &T(T_*P_0)}
 \end{equation}
 \end{proposition}

In fact, the reader will find in \cite{OJS18} an exact sequence of sub-Poisson groupoids.

\subsection{Fredholm Groupoid Associated to the Von Neumann Algebra $ \mathfrak{M}=\mathsf{L}_\infty(\Hc)$}  
\label{__LinftyHc}${}$\\
 \emph{In this section, we consider the Von Neumann algebra $ \mathfrak{M}=\mathsf{L}_\infty(\Hc)$ and all results stated from \cite{OJS18}, 6.3.}\\

At first, recall that the Banach predual  $(\mathsf{L}_\infty(\Hc))_*$ is $\mathsf{L}_1(\Hc)$ and the pairing between $\mathsf{L}_1(\Hc)$ and  $\mathsf{L}_\infty(\Hc)$ is given by:
$$\langle A, B\rangle:=\mathsf{Tr}(AB)$$
 The      set  $\mathcal{L}(\mathsf{L}_\infty(\Hc))$ of orthogonal projections is isomorphic to $\mathcal{L}(\Hc)$ and the associated set $\mathcal{G}(\mathsf{L}_\infty(\Hc))$ is 
 $$\{A\in \mathsf{L}_\infty(\Hc),\;:\; \textrm{Im }A=\overline{\textrm{Im} A}\}$$
 By the way, the groupoid $\mathcal{G}(\mathsf{L}_\infty(\Hc))\tto \mathcal{L}(\mathsf{L}_\infty(\Hc))$ can be identified with the groupoid $\mathcal{G}\Hc)\tto \mathcal{L}(\Hc)$ where ${\bf s}(A)=(\ker A)^\perp$ and ${\bf t}(A)=\textrm{Im} A$.\\
 
 Let $\mathcal{L}_N(\Hc)$ be the set of projections of rank $N$ and we set
 $$\mathcal{L}_\mathsf{fin}(\Hc)=\bigcup_{N=1}^\infty \mathcal{L}_N(\Hc) \textrm{ and } \mathcal{G}_\mathsf{fin}(\Hc)={\bf  s}^{-1}(\mathcal{L}_\mathsf{fin}(\Hc))\cap{\bf  t}^{-1}(\mathcal{L}_\mathsf{fin}(\Hc)).$$
  Then $\mathcal{G}_\mathsf{fin}(\Hc)\tto \mathcal{L}_\mathsf{fin}(\Hc)))$ is the groupoid of finite rank operators\\
  
  On the other hand, consider the groupoid $\mathcal{G}_\infty(\Hc)\tto \mathcal{L}_\infty(\Hc)))$  of infinitely dimensional range partially invertible operators. We have an involution 
  
  $\perp: \mathcal{L}(\Hc)\to \mathcal{L}(\Hc)$ defined by $\perp(p):=p^\perp:=1-p$ for $p\in  \mathcal{L}(\Hc)$.\\
  
  \noindent  Then $\perp$ is an automorphism of $\mathcal{L}(\Hc)$ and $\perp\circ\perp=Id_{\mathcal{L}(\Hc)}$.\\
  
   We consider the  subgroupoid $\mathcal{G}_\mathsf{Fred}(\Hc)\tto \mathcal{L}_\mathsf{Fred}(\Hc))))$ of $\mathcal{G}_\infty(\Hc)\tto \mathcal{L}_\infty(\Hc)))$ defined in the following way:
   
   if $\mathcal{L}_\mathsf{Fred}(\Hc):=\perp(\mathcal{L}_\mathsf{fin}(\Hc))$ then $ \mathcal{G}_\mathsf{Fred}(\Hc)={\bf  s}^{-1}(\mathcal{L}_\mathsf{Fred}(\Hc))\cap{\bf  t}^{-1}(\mathcal{L}_\mathsf{Fred}(\Hc))$

 Now, as previously, for any fixed $p_0\in \mathcal{L}(\Hc)$, we can consider $P_0^\mathsf{Fred}:={\bf s}^{-1}(p_0)\cap \mathcal{G}_\mathsf{Fred}(\Hc)$.  To $P_0$ is associated  a  weak Banach subbundle $T_*P_0^\mathsf{Fred}$ of $T^*P_0^\mathsf{Fred}$  whose fibers are the predual of  fibers of $TP_0^\mathsf{Fred}$. Then by same arguments as in the proof of Proposition~\ref{P_ExampleSubPoisson} we have:

\begin{proposition}
\label{P_ExampleSubPoissonFred}
The groupoid  $T_*P_0^\mathsf{Fred}\times T_*P_0^\mathsf{Fred}\tto T_*P^\mathsf{Fred}_0$ has a sub-Poisson groupoid structure associated to  the following commutative diagram of groupoid morphisms:
 \begin{equation}
 \label{eq_DiagralSubPoisson}
 \xymatrix{
  T^\flat(T_*P_0^\mathsf{Fred}\times T_*P_0^\mathsf{Fred})\ar@<-.5ex>[d] \ar@<.5ex>[d]\ar[r]^{{\bf P}\times{\bf P}}&T(T_*P_0^\mathsf{Fred}\times T_*P_0^\mathsf{Fred})\ar@<-.5ex>[d] \ar@<.5ex>[d]\\
T^\flat(T_*P_0^\mathsf{Fred})\ \ar[r]^{{\bf P}} &T(T_*P_0^\mathsf{Fred})}
\end{equation}
\end{proposition}


\end{document}